\documentclass[twoside, 11pt]{article}

\title{On the long-time statistical behavior of smooth solutions of
    the weakly damped, stochastically-driven KdV equation}
\author{Nathan E. Glatt-Holtz, Vincent R. Martinez, Geordie H. Richards\\
  \scriptsize{emails: negh@iu.edu, 
                      vrmartinez@hunter.cuny.edu,
                      grichards@uoguelph.ca}}
\date{October 8, 2024}

\usepackage[margin=0.90in]{geometry}

\usepackage[nottoc,notlot,notlof]{tocbibind}
\usepackage{amsfonts, amssymb, amsmath, amsthm, multicol}
\usepackage[colorlinks=true, pdfstartview=FitV, linkcolor=blue, citecolor=blue, urlcolor=blue]{hyperref}
\usepackage[usenames]{xcolor}
\usepackage{accents}
\usepackage{comment}
\usepackage{float}
\usepackage{graphicx}
\usepackage{multirow}
\usepackage{parcolumns}
\usepackage{subcaption}
\usepackage{textcomp}
\usepackage{mathrsfs}
\usepackage{mathtools}
\usepackage{dsfont}
\usepackage{enumerate}
\usepackage[capitalize,nameinlink,noabbrev]{cleveref}
\usepackage{tocloft}
 
 \usepackage{tikz}

\newcommand{\submit}[1]{#1}

\definecolor{labelkey}{rgb}{0,0,1}

\usepackage{pdfsync}
\numberwithin{equation}{section}

\DeclareMathOperator{\esssup}{esssup}
\DeclareMathOperator{\vspan}{span}

\newtheorem{Thm}{Theorem}[section]
\newtheorem{Lem}[Thm]{Lemma}
\newtheorem{Prop}[Thm]{Proposition}
\newtheorem{Cor}[Thm]{Corollary}

\newtheorem{Def}[Thm]{Definition}
\newtheorem{Rmk}[Thm]{Remark}

\newtheorem*{Thm*}{Theorem}

\newcommand{\indFn}[1]{1 \! \! 1_{#1}}

\newcommand{\eps}{\epsilon}
\newcommand{\de}{\delta}
\newcommand{\De}{\Delta}

\newcommand{\RR}{\mathbb{R}}
\newcommand{\TT}{\mathbb{T}}
\newcommand{\ZZ}{\mathbb{Z}}
\newcommand{\E}{\mathbb{E}}
\newcommand{\Prb}{\mathbb{P}}

\newcommand{\bH}{\mathbf{H}}

\newcommand{\Mt}{\mathcal{M}}

\newcommand{\K}{\mathcal{K}}

\newcommand{\Mon}{\mathscr{M}}
\newcommand{\Ron}{\mathscr{R}}

\DeclareMathOperator{\rank}{rank}

\newcommand{\imb}{\hookrightarrow}

\newcommand{\ph}{\varphi}

\newcommand{\Ham}{\mathcal{H}}
\newcommand{\In}{\mathcal{I}}
\newcommand{\Kn}{\mathcal{K}}
\newcommand{\InG}{\tilde{\mathcal{I}}}

\newcommand{\abarm}{\bar{\alpha}_m}

\newcommand{\Wm}{\mathfrak{W}}

\def\gam{\gamma}
\def\s{\sigma}

\def\lam{\lambda}
\def\lb{\langle}
\def\rb{\rangle}
\def\bdy{\partial}

\def\T{\mathbb{T}}
\def\al{\alpha}
\def\be{\beta}

\def\z{\zeta}

\def\w{\omega}

\def\Om{\Omega}

\def\bdy{\partial}

\def\veps{\varepsilon}

\newcommand{\goesto}{\rightarrow}
\newcommand{\smod}{\setminus}

\def\R{\mathbb{R}}
\def\N{\mathbb{N}}

\newcommand{\til}[1]{{\tilde{#1}}}
\newcommand{\Sob}[2]{\lVert#1\rVert_{#2}}

\newcommand{\req}[1]{(\ref{#1})}

\newcommand{\Lip}{\mbox{Lip}}

\newcommand{\Atr}{\mathcal{A}}
\newcommand{\bSt}{\mathcal{B}}

\begin{document}

\sloppy

\markboth{N. Glatt-Holtz, V.R. Martinez, G. Richards} {Long-time
  statistics for the weakly damped stochastic KdV equation}

\maketitle

\begin{abstract}
  This paper considers the damped periodic Korteweg-de Vries (KdV) equation in
  the presence of a white-in-time and spatially smooth stochastic
  source term and studies the long-time behavior of solutions.  We show that the integrals of motion for KdV
  can be exploited to prove regularity and ergodic properties of
  invariant measures for damped stochastic KdV.  First, by considering non-trivial modifications of the integrals of motion, we establish
  Lyapunov structure by proving that moments of Sobolev norms of solutions at all orders
  of regularity are bounded globally-in-time; existence of
  invariant measures follows as an immediate consequence.  Next, we prove a weak Foias-Prodi type estimate for damped stochastic KdV, for which the synchronization occurs in expected value.  This estimate plays a crucial role throughout our subsequent analysis.  As a first novel application, we combine the Foias-Prodi estimate with the Lyapunov structure to establish that invariant
  measures are supported on $C^\infty$ functions provided
  that the external driving forces belong to $C^\infty$.  We then
  establish ergodic properties of invariant measures, treating the
  regimes of arbitrary damping and large damping separately. For arbitrary damping,
  we demonstrate that the framework of `asymptotic coupling' can be
  implemented for a compact proof of uniqueness of the invariant
  measure provided that sufficiently many directions in phase space
  are stochastically forced. Our proof is paradigmatic for SPDEs for which a weak Foias-Prodi type property holds.  Lastly, for
  large damping, we
  establish the existence of a spectral gap with respect to a
  Wasserstein-like distance, and exponential mixing and
  uniqueness of the invariant measure follows.  
\end{abstract}

{\noindent \small {\it {\bf Keywords: Stochastic PDEs, Korteweg-de Vries (KdV)
      equation, higher-order statistical moments, regularity of the invariant measure, unique
      ergodicity, exponential mixing}
  }  }
\begin{footnotesize}
\setcounter{tocdepth}{1}
\tableofcontents
\end{footnotesize}
\newpage

\section{Introduction}\label{sect:Intro}

This work presents a study of the damped-driven Korteweg-de Vries
(KdV) equation on the unit circle $\TT$, where the external source
term has a stochastic component.  The equation is given as
\begin{align}
   du + ( u Du + D^3u+ \gam u) dt = f dt + \sigma dW,
  \quad
  u(0) = u_0.
   \label{eq:s:KdV}
\end{align}
where $D = \partial/\partial x$, $\gam$ is a positive constant, the
external perturbation has a drift coefficient given by a
time-independent function, $f$, and the additive noise, $\sigma dW$,
is white-in-time and smooth in space. Note that we will work exclusively in the setting of functions with mean zero.

The KdV equation, \eqref{eq:s:KdV} with $\gamma=0$, $\sigma\equiv 0$,
is a canonical nonlinear dispersive Partial Differential Equation
(PDE) of significant mathematical and physical interest and has been
investigated through a vast literature.  It has Hamiltonian structure
and it is completely integrable.  As is typical for Hamiltonian
systems, the damped-driven version \eqref{eq:s:KdV} is also of wide
interest in a variety of mathematical and physical contexts.  See
\cref{sec:intro:prev} below for further details on the history of this
subject.  Notwithstanding this
extensive history, the long time behavior of solutions to
\eqref{eq:s:KdV} with $\gamma> 0$ and $\sigma\neq 0$ remains poorly
understood.

{{}Our interest in the ``weakly'' damped version of stochastic KdV with
spatially smooth noise is motivated by previous contributions in the
deterministic setting which elucidate a non-trivial structure in the
long time dynamics, namely the existence of a global attractor as
regular as the deterministic forcing \cite{Ghidaglia1988,
  MoiseRosa1997}.  By ``weak'' damping, we refer specifically to the fact that the mechanism for removing energy from the system \textit{does not} simultaneously regularize its solutions; this terminology is set in contrast to ``strong'' damping or dissipation, paradigmatically represented by the negative Laplacian operator, which does confer additional regularity to its solutions.} These previous works \cite{Ghidaglia1988,
  MoiseRosa1997} exposed the interaction between
complete integrability and weak damping that is specific to KdV to
illuminate mechanisms for understanding the long time dynamics which
are distinct from those that could be harnessed using strong
dissipation for a variety of infinite-dimensional dynamical systems
(see e.g. \cite{Temam1997}).  On the other hand, in the setting of
stochastic KdV with strong dissipation, some notable contributions
have been established in \cite{KP2008,Kuksin2010} which utilize
complete integrability, but any work of how integrability of KdV can
be harnessed to study \eqref{eq:s:KdV} remains absent from the
literature.  However, the reader is referred to
\cite{DebusscheOdasso2005} for related work in the context of the
weakly damped stochastic cubic nonlinear Schr\"{o}dinger (NLS)
equation.

Our primary mathematical objective is to exploit weak damping and
complete integrability of KdV to prove ergodic and mixing properties
of \eqref{eq:s:KdV} in the category of smooth solutions and to
illuminate key features for doing so in the broader context of weakly damped
systems.  We use the integrals of motion to establish a Lyapunov
structure for solutions of \eqref{eq:s:KdV} at all orders of Sobolev
regularity (\cref{thm:Lyapunov:intro}), from which the existence of
invariant measures (\cref{thm:exist:intro}) follows by applying a standard Krylov-Bogolyubov argument.  The next significant component of
our analysis is to establish a variant of a so-called
Foias-Prodi type estimate that holds in expectation
(\cref{thm:FP:est:Intro}), which is accomplished
through analysis of linearizations of the integrals of motion and by harnessing the Lyapunov structure.  Foias-Prodi estimates are commonly
exploited in the ergodic theory of SPDEs, but typically in a pathwise
formulation.  While such bounds are new for \eqref{eq:s:KdV} and
interesting in their own right, we use our modified Foias-Prodi
estimate to develop streamlined probabilistic arguments that
establish higher regularity of the invariant measure, for any damping
parameter $\gamma>0$ (\cref{thm:Regularity:Intro}), as well as
uniqueness of invariant measures provided that the noise activates a
sufficiently large number of modes
(\cref{thm:unique:ergodicity:Intro}).  We then proceed to establish another non-trivial consequence of the
Lyapunov structure: an exponential mixing result in the large
$\gamma$ asymptotic regime, via one-force-one-solution type arguments
(\cref{thm:large:damp:Mix:Intro}).  Finally, we identify how the approach we developed for a stochastic setting can be adapted to yield streamlined proofs of Foias-Prodi type estimates and regularity of the attractor for deterministic damped-driven KdV.

\subsection{Literature Review}
\label{sec:intro:prev}

Originally proposed by Boussinesq \cite{Boussinesq1871} in 1871 and
later popularized by Korteweg and de Vries \cite{KortewegdeVries1895}
in 1895, the classical KdV equation (\eqref{eq:s:KdV} with $\gam=0$
and $\sigma\equiv0$) was derived as a model for soliton behavior in
shallow water wave propagation.  It has since been realized as a
canonical model for unidirectional propagation of small-amplitude
waves across a wide range of contexts \cite{Miura1976}, and has been
studied to great depth in both physics and mathematics communities for
its rich structural properties.

In the deterministic setting, with smooth initial conditions belonging
to $H^m(\T)$, $m\geq2$, the well-posedness of KdV has been understood
since the works of \cite{Temam1969, Sjoberg1970, BonaSmith1975,
  BonaScott1976, SautTemam1976, Kato1979}.  The Hamiltonian
formulation of KdV is due to \cite{gardner1971,zakharov1971}, and the
existence of action-angle coordinates for KdV was first established in
\cite{miura1968I,miura1968II}, with important advances specific to the
periodic setting appearing in \cite{mckean1976,flaschka1976,lax1976}.
With respect to the action-angle coordinate system, the solution to
KdV can be expressed as linear flow on an infinite-dimensional torus,
providing substantial dynamical insight. For example, solutions evolve
almost periodically in time, and there exist large families of almost
periodic solutions that persist under small Hamiltonian perturbations
\cite{kappeler2003}.  Significant work on developing the
well-posedness theory of KdV in successively lower regularity settings
goes to back to the seminal works of \cite{KenigPonceVega1991,
  Bourgain1993}.  This development was advanced in
\cite{KenigPonceVega1996, CollianderKeelStaffilaniTakaokaTao2003},
establishing well-posedness in $H^s(\T)$ for $s$ negative, culminating
in endpoint global well-posedness in $H^{-1}(\T)$
\cite{KappelerTopalev2006, KillipVisan2019}. Note that the sharpness
of the well-posedness threshold $s=-1$ was established in
\cite{Molinet2011, Molinet2012}.  Another notable work in this regard
is an alternate proof of the result in \cite{Bourgain1993} provided in
\cite{BabinIlyinTiti2011}.

In the stochastic setting, research into well-posedness has focused on
rough stochastic forcing, where the low temporal regularity of
Brownian motion precludes a straightforward adaptation of harmonic
analysis techniques.  Nevertheless, local and global well-posedness
results have been established in a number of settings with either
additive or multiplicative noise \cite{deBouardDebussche1998,
  deBouardDebusscheTsutsumi1999, deBouardDebusscheTsutsumi2004,
  oh2010periodic}.  Omitted from these results is well-posedness in
the classical $H^m(\T)$ setting, where $m\geq2$, when the noise is
spatially smooth, essentially because the problem is much easier.
Indeed, the choice of smooth forcing allows one to rely on more
classical arguments, avoiding problems due to the temporal regularity
of Brownian motion.  For the sake of completeness the proof of global
well-posedness for \eqref{eq:s:KdV} is provided in the appendix (see
\cref{sect:apx:wp:SKdV}).

It must be emphasized that KdV appears as a model equation in physics,
where the full model may include dissipative effects due to friction,
and forcing terms, e.g. due to a noisy environment, among other
potential complications.  Some important contributions to the analysis
of KdV under deterministic perturbations with this structure have been
made.  In particular, with the form of damping considered here,
existence of a global attractor is known \cite{Ghidaglia1988}, and its
regularity has been studied in some detail \cite{MoiseRosa1997,
  Goubet2000, Goubet2018}. In fact, more recently, the dynamics of
damped-driven KdV was reduced to an ordinary
differential equation, referred to as the ``determining
form'' (cf.  \cite{JollySadigovTiti2017}).  See also
\cite{Johnson1970, AmickBonaSchonbek1989,
  BonaDougalisKarakashianMcKinney1996, ZahiboPelinovskySergeeva2009,
  ChehabSadaka2013a, ChehabSadaka2013b} for a sample of other
noteworthy contributions.  In contrast, the study of weakly damped KdV
with a stochastic driving force has received limited attention in the
literature.  In a recent paper \cite{EkrenKukavicaZiane2018}, the
issue of existence of invariant measures for damped stochastic KdV on
the real line was addressed in the functional setting of $H^1(\R)$.
Other interesting contributions were obtained in the setting of
strongly dissipative stochastic KdV \cite{KP2008,Kuksin2010}, where it
was shown that when the dissipation is suitably scaled with the noise,
then in the deterministic limit the integrals of motion evolve
according to certain averaged equations.

The existence of invariant measures for \eqref{eq:s:KdV} is fairly
direct in our setting of $H^m(\T)$, where $m\geq2$, by way of the
Krylov-Bogolyubov averaging procedure, provided we establish suitable
energy estimates to ensure compactness.  The uniqueness, regularity,
and ergodic properties of invariant measures is a more subtle issue.
In the regime of strong dissipation, {{}a pertinent result when viscosity is large} is found in 
\cite{MattinglyLargeViscosity1999} for the 2D stochastic Navier-Stokes
equations.  There, it is proven that the random attractor is
degenerate, i.e., it consists of a single trajectory, and is
ultimately parameterized by realizations of the noise paths; this
property immediately confers uniqueness of invariant measures.
Analogously, in the deterministic setting of \eqref{eq:s:KdV} it was
shown in \cite{CabralRosa2004} that a one-point global attractor
exists when damping is taken sufficiently large.  For arbitrary
damping, there are different approaches that could be considered.
Recent work in \cite{HairerMattinglyScheutzow2011,
  KuksinShirikyan2012, GlattHoltzMattinglyRichards2015,
  ButkovskyKulikScheutzow2019} identified an intuitive and
conceptually simple framework for proving uniqueness of the invariant
measure by an asymptotic coupling strategy.  Indeed, the work
\cite{GlattHoltzMattinglyRichards2015} surveyed a number of SPDEs for
which this framework led to streamlined proofs of uniqueness, where
the common thread between these systems was the existence of a finite
number of determining modes in the spirit of \cite{FoiasProdi1967},
and a sufficiently rich stochastic forcing to ensure that the
determining modes are excited. {{}In a seminal work by Hairer and Mattingly \cite{HairerMattingly2006}, the case of degenerate forcing was considered, where finitely many Fourier modes are directly excited, but the number of such modes is \textit{independent} of the size of viscosity.} We note that in light of recent
contributions regarding the dynamics of damped-driven KdV
\cite{Ghidaglia1988, MoiseRosa1997, Goubet2000, JollySadigovTiti2017},
among others, asymptotic coupling is a natural strategy to consider for proving uniqueness for
\eqref{eq:s:KdV} {{}in the slightly degenerate setting of forcing sufficiently many modes, dependent on the size of the damping parameter, that is considered here. Nevertheless,} the combined absence of strong dissipation and
the structure of the nonlinearity presents unique challenges {{}that must be overcome.}

\subsection{Overview of Main Results}
Let us now turn to provide a detailed overview of the main
mathematical results which we establish in this work.  An outline of
the strategy of our proofs are provided immediately below in
\cref{sec:intro:math}. {{}We point out that in statements that follow and throughout the manuscript, the damping parameter, $\gam$, is assumed to be positive and arbitrary unless otherwise specified.}

Our first theorem illuminates Lyapunov structure for Sobolev norms of
solution of \eqref{eq:s:KdV} at all degrees of regularity.
\begin{Thm}[Lyapunov structure for higher order Sobolev norms]
  \label{thm:Lyapunov:intro}
  Given sufficiently smooth $f$ and $\sigma$, and $m\geq 2$, let $u=u(t;u_0)$ be the solution of \eqref{eq:s:KdV}
  corresponding to initial condition $u_0\in{H}^m$.  There exist parameters
  $r,r'>0$ such that, for any $\eta > 0$
  \begin{align}
    \label{eq:Lyapunov:intro}
    \E\exp\left(\eta \|u(t)\|_{H^m}^{r}\right)
    \leq C\exp\left(Ce^{-\gamma r t}(\|u_0\|_{H^m}^{r} + \|u_0\|^{r'}_{L^2})\right),
  \end{align}
  for some $C>0$ where $C, r, r'$ are independent of $u_0$ and
  $t \geq 0$.
\end{Thm}
\noindent
The estimates provided by \cref{thm:Lyapunov:intro}, while of interest
in their own right, serve as a backbone for the rest of the analysis
herein as well as a foundation for future work on \eqref{eq:s:KdV} in
the category of smooth solutions.  Specifically
\cref{thm:Lyapunov:intro} will be crucial for analyzing regularity and ergodicity of invariant measures for \eqref{eq:s:KdV} as in
Theorems \ref{thm:exist:intro}--\ref{thm:Regularity:Intro} below.
Note that a number of additional moment bounds beyond
\eqref{eq:Lyapunov:intro} are proven to control the evolution of
higher Sobolev norms of solutions.  We refer the reader to
\cref{thm:Lyapunov} for these details and more precise statements.

As an immediate application of \cref{thm:Lyapunov:intro}, we establish
the existence of invariant measures via the classical
Krylov-Bogolyubov procedure. In order to state this and for use in
subsequent results, let us first introduce some additional notation
and terminology.  Let
\begin{align}
  P_t(u_0; A) := \Prb( u(t;u_0) \in A) \quad
   \text{ for any } u_0 \in H^m,
  \text{ and any Borel set } A \subseteq H^m,
  \label{eq:m:trans:intro}
\end{align}
denote the Markov transition semigroup on $H^m$ ($m\geq 2$)
corresponding to solutions $u = u(t; u_0)$ of \eqref{eq:s:KdV}.  This
semigroup $\{P_t\}_{t \geq 0}$ acts on observables
$\phi : H^m \to \RR$ and on (Borelean) probability measures $\mu$ on
$H^m$ according to
\begin{align*}
  P_t\phi(\cdot) :=  \int_{H^m} \phi(v) P_t(\cdot, dv),
  \quad
  \mu P_t(\cdot) := \int_{H^m} P_t(v, \cdot) \mu(dv),
\end{align*}
respectively.  In this Markovian formulation statistically steady
states of \eqref{eq:s:KdV} are defined by invariant measures, $\mu$,
that is, probability measures such that $\mu P_t = \mu$, for every
$t \geq 0$. Our result on existence of invariant measures is then
stated as follows.

\begin{Cor}[Existence of invariant measures]
\label{thm:exist:intro}
For sufficiently smooth $f$ and $\s$, there exists an invariant
measure for the Markov semigroup corresponding to \eqref{eq:s:KdV}.
\end{Cor}

 The second major result of this work is proof of a Foias-Prodi type estimate for \eqref{eq:s:KdV} that is valid for arbitrary damping.
 Roughly speaking, Foias-Prodi estimates quantify the following
 property for an infinite-dimensional dynamical system: given two
 solutions, if a sufficient number of components, i.e., low
 frequencies, of the two solutions synchronize in the limit as
 $t \to \infty$, then in fact all components asymptotically
 synchronize.  In our setting, the estimates provide a fundamental
 technical tool to address both the ergodic and regularity properties
 of invariant measures for \eqref{eq:s:KdV}.

 To formulate the result we fix any solution $u = u(t;u_0)$ of
 \eqref{eq:s:KdV} starting from a given initial state $u_0$.  Then,
 for any other initial state $v_0$, we consider the `nudged' system
 \begin{align}\label{eq:s:KdV:coupled:intro}
   dv+(vDv+D^3v+\gam v)dt
       =-\lam P_N(v-u)+fdt+\s dW,\quad v(0)= v_0.
 \end{align}
 Here, $P_N$ is Dirichlet projection to Fourier modes $|k|\leq N$ and
 $\lam$ is a positive real number.  As such the `nudging term'
 $-\lam P_N(v-u)$ drives $v$ towards $u$ on low frequencies, and a
 Foias-Prodi estimate will quantify how many modes need to be activated for this mechanism to synchronize the full solutions.
 \begin{Thm}[Foias-Prodi type  estimate]\label{thm:FP:est:Intro} 
   Consider any $f$ and $\sigma$ sufficently smooth. Then there exists a choice of $N$ and $\lambda$ in
   \eqref{eq:s:KdV:coupled:intro} such that, for any $r > 0$ and two initial conditions
   $u_0, v_0\in H^{2}$,
   \begin{align}
     \label{eq:f:p:Intro}
     \E \|u(t;u_0) - v(t;v_0) \|_{H^1}^2 \leq  \frac{C}{t^r},
   \end{align}
   where $C > 0$ is independent of $t > 0$ and $u$ and $v$ are the solutions of
   \eqref{eq:s:KdV} and \eqref{eq:s:KdV:coupled:intro} emanating from
   $u_0$ and $v_0$, respectively.
 \end{Thm}
 \noindent Classically, estimates like \eqref{eq:f:p:Intro} appear in
 the SPDE literature in a `pathwise formulation'.  Namely, for
 equations with a strong dissipative mechanism such as 2D
 Navier-Stokes equations, analogous statements to \eqref{eq:f:p:Intro}
 hold without expected values and are typically established with
 exponential rates of decay.  In contrast, relying on damping and
 complete integrability as the smoothing mechanism ultimately yields
 the weaker, non-pathwise form represented by \eqref{eq:f:p:Intro}.
 We note that non-pathwise Foias-Prodi structure was also exploited in
 \cite{DebusscheOdasso2005} in the context of the damped stochastic
 cubic NLS.  See \cref{sec:intro:math} below for further discussion.

 Our next result shows that the bounds in \cref{thm:Lyapunov:intro}
 extend to the statistically permanent regime by establishing
 $C^\infty$--regularity of the support of any invariant measure of
 \eqref{eq:s:KdV}.  {{}Let us summarize this regularity result now. For this, let $\Pr(H^2)$ denote the set of all probability measures on $(H^2,\mathcal{B}(H^2))$}.
 
\begin{Thm}[Higher order regularity of invariant
  measures]
  \label{thm:Regularity:Intro}
  Assume that $f$ and $\sigma$ have $C^\infty$--regularity.  Then any invariant
  measure $\mu\in\Pr(H^2)$ of \eqref{eq:s:KdV} is supported on smooth functions,
  namely
  \begin{align*}
    \mu\left(H^2\smod C^\infty\right) = 0.
  \end{align*}
  Furthermore, for any $m \geq 2$, there exists a corresponding $r > 0$,
  depending only on $f$ and $\sigma$, such that for any $\eta > 0$
 \begin{align*}
 	\int \exp(\eta \| u\|^r_{H^m}) \mu(du) \leq C < \infty,
 \end{align*}
 for a constant $C = C(f, \sigma, \eta)$ independent of $\mu$.
 In particular, any observable $\phi: H^m \to \RR$  such that
 \begin{align*}
 	\sup_{u} \left( |\phi(u)|  \exp( -\eta \| u\|_{H^m}^r) \right)< \infty
 \end{align*}
for some $\eta > 0$ is integrable with respect to any invariant state $\mu$.
\end{Thm}
\noindent See \cref{thm:regularity} below for the complete and
rigorous formulation of this result.  Here, it is important to
emphasize that \cref{thm:Regularity:Intro} does not follow as an
immediate or obvious consequence of \cref{thm:Lyapunov:intro}.  This
is because the damped stochastic KdV \eqref{eq:s:KdV} does not possess
a dissipative smoothing mechanism.  On the other hand, as
previously observed in the deterministic case, \cite{MoiseRosa1997},
the combination of damping and complete integrability of KdV does
provide a mechanism for smoothing at time infinity.  Our insight here
is that \cref{thm:FP:est:Intro} provides an avenue for a novel and less technically involved approach than the one taken in \cite{MoiseRosa1997} to
establish this smoothing mechanism.  Although \cref{thm:Regularity:Intro} exposes how this alternative approach to asymptotic smoothing can be applied to quantify the regularity of the support of invariant measures for damped stochastic KdV, we refer the reader to \cref{sect:det:case} for discussion on the application of our strategy to the deterministic damped-driven KdV.

For the next result, we address the question of  uniqueness of
the invariant measure for \eqref{eq:s:KdV} in the regime of an
arbitrary positive damping parameter.  For this situation, we will
require additional structure on the noise $\sigma$. As above, let $P_N$
denote projection onto Fourier modes $|k|\leq N$ and further let
$\mathcal{R}(\s)$ denote the range of $\s$ (i.e. the set of
stochastically forced modes, see \cref{sect:background} for a precise
definition). We prove the following.

\begin{Thm}[Unique ergodicity for arbitrary damping in the essentially
  elliptic case]
  \label{thm:unique:ergodicity:Intro}
  Assume that $f$ and $\s$ are sufficiently
  smooth. There exists $N=N(\gam,f,\s)$ such that if
  $\mathcal{R}(\s)\supset P_NL^2$, then there can be at most one
  invariant measure for \eqref{eq:s:KdV}.
\end{Thm}
\noindent The precise statement of this result is given in
\cref{thm:ergodicity:ess:elliptic} below.  Here, we emphasize that our
proof of uniqueness invokes the asymptotic coupling framework of
\cite{GlattHoltzMattinglyRichards2015} in a novel but direct
manner. Indeed, an expedient proof can be supplied once a Foias-Prodi
estimate in the spirit of \cref{thm:FP:est:Intro} is in place. The
argument we present exposes an additional flexibility to the framework
of \cite{GlattHoltzMattinglyRichards2015}, namely, that a weak Foias-Prodi
structure, which is satisfied only in expectation, is sufficient for the
approach to be implemented.

Our final result establishes exponential mixing for \eqref{eq:s:KdV} in the
large damping regime, i.e. $\gamma \gg 1$.

\begin{Thm}[Exponential mixing for large damping]
   \label{thm:large:damp:Mix:Intro}
   For any sufficiently smooth $f, \sigma$, there exists 
   $\bar{\gamma}>0$ such that if
   $\gamma \geq \bar{\gamma}$, then \eqref{eq:s:KdV} possesses a unique invariant
   measure $\mu$. Moreover
   \begin{align}
     \left| P_t \phi(u_0) - \int_{H^m} \phi(v) \mu(dv) \right|
     \leq C e^{-\delta t}
     \label{eq:exp:mix:intro:obs}
   \end{align}
   for all $u_0\in H^m$, for any sufficently regular observable $\phi: H^m \to \RR$, $m\geq2$.  Here,
   $C = C(u_0, \phi), \delta > 0$ are constants independent of $t > 0$,
   and moreover the exponential decay rate $\delta$ depends only on
   $\gamma, f, \sigma$ and universal quantities.
 \end{Thm}
 \noindent 
 In this regime, we identify that large damping confers a
 synchronization in expected value of solutions in the $L^2$-topology at time infinity, and use this
 synchronization to establish a spectral gap with respect to an
 $L^2$-based Wasserstein-like function on the space of probability
 measures (see \eqref{def:Wd} below for a precise definition).  Exponential mixing follows immediately, and we further extend mixing
 to higher Sobolev regularities.  For the precise statement see
 \cref{thm:spectralgap}.
 
 Generally speaking, the existence of a spectral gap has notable
 consequences for the observability of statistics with respect to the
 invariant measure such as a law of large numbers and continuity with
 respect to parameters (cf. \cite{HairerMattingly2008,
   HairerMattinglyScheutzow2011, Komorowski2012a, Komorowski2012b,
   KuksinShirikyan2012, Kulik2017}). {{}The ensuring the existence of a spectral gap when damping is sufficiently large is essentially due to the fact that the resulting system is strongly contractive; we refer the reader to \cite{SchmalfussBook} and \cite{Mattingly2003} for other notable works in the setting of strong contractivity.} Note that treatment of this case
 is independent of the structure of the noise.  That is, in the large damping regime, we impose no
 restrictions on $f$ and $\sigma$ beyond assuming
 sufficient spatial regularity.

\subsection{Framework, Proof Strategies, and Mathematical Challenges}
\label{sec:intro:math}

This section will present the strategies we employ to prove our main
results and review some of the technical challenges encountered
therein. We also take this opportunity to provide a detailed outline
of the paper. Following this section, we introduce mathematical
and notational background to be used throughout our manuscript
(\cref{sect:background}).

Our study of the Lyapunov structure for \eqref{eq:s:KdV}, stated in
its abbreviated form in \cref{thm:Lyapunov:intro}, will be presented
in \cref{sect:Lyapunov}, with the complete statement provided by
\cref{thm:Lyapunov}.  The proofs of these estimates are based on
demonstrating that for any $m$, the integral of motion for KdV of
order $m$, $\mathcal{I}_m$, can be modified to a functional,
$\mathcal{I}^{+}_m$, such that $\mathcal{I}^{+}_m$ simultaneously
controls the $H^m$--norm of the solution to \eqref{eq:s:KdV} and
remains bounded in a suitable probabilistic sense.  To be more
precise, recall (see e.g. \cite{miura1968II}) that the integrals of
motion for KdV take the form
  \begin{align}
    \In_m(v) 
    = \|D^m v\|_{L^2}^2 
      - \int_{\T} \left(\alpha_{m} v (D^{m-1} v)^2 
      - Q_m(v, Dv, \ldots, D^{m-2}v) \right)dx,
    \label{eq:inv:polynomials}
  \end{align}
  where $D=\frac{d}{dx}$, so that to `leading order' the equations
  preserve $H^m$ norms.  This algebraic structure has been exploited
  in the deterministic setting \cite{MoiseRosa1997} to establish the
  regularity of the global attractor for damped-driven KdV. In
  \cite{MoiseRosa1997}, however, the availability of uniform-in-time
  estimates on the solution allow one to largely ignore the finer
  structure of the lower-order quantities that appear in the analysis.
  In our stochastic setting, the lower order terms must be accounted
  for and handled more carefully.  In quantitative terms, we
  demonstrate that the lower order terms in \eqref{eq:inv:polynomials}
  can be controlled via interpolation estimates (see Theorem
  \ref{prop:mono:intp:bnd} below).  Here, it is crucial to track the
  rank of the monomials which constitute the lower order terms $Q_m$
  in order to obtain modest control over the exponents of
  interpolation.  Indeed, without some control over these exponents,
  the Lyapunov structure would not manifest in the estimates.
  Probabilistic estimates of the integrals of motion for a weakly
  damped and stochastically forced integrable PDE have appeared in the
  setting of the NLS equation (cf. \cite{DebusscheOdasso2005}).  In
  contrast to \cite{DebusscheOdasso2005}, however, we demonstrate how
  to control all higher Sobolev norms of solutions, whereas
  \cite{DebusscheOdasso2005} focuses only on $H^1$ estimates.  This
  higher order Lyapunov structure is applied in a crucial way to prove each of
  our subsequent results.

  In \cref{sect:FP:est} we present Foias-Prodi type estimates for
  \eqref{eq:s:KdV}, which effectively illustrate that synchronization
  of solutions at time infinity on a large number of low modes is a
  sufficient condition for full synchronization at time infinity
  (\cref{thm:FP:est:Intro}).  At a more technical level, we show that
  for certain stopping times $\tau_R$ that control the growth of
  solutions to \eqref{eq:s:KdV} and \eqref{eq:s:KdV:coupled:intro}, the
  restricted expectation
  $\E (\indFn{\tau_R = \infty} \|u(t) - v(t) \|_{H^1}^2)$ decays with
  an exponential rate in time (\cref{thm:FP:est}, \cref{cor:FP:est}),
  and use the Lyapunov structure to prove that the probability such
  stopping times remain finite decays algebraically in the ``cut-off
  parameter" $R$ (\cref{lem:FP:stoppingtime}).  This formulation is useful in applications of the Foias-Prodi structure, but also yields the simpler statement in \cref{thm:FP:est:Intro} (see \cref{cor:FP:est:Intro}).  As outlined above, the
  notable distinction of our Foias-Prodi type estimates from their
  more classical appearance in the dissipative SPDE literature is that
  our estimates hold in expectation, rather than in a pathwise
  formulation.  
  
  In the more classical situation, where one considers dissipative
  equations such as the 2D Navier-Stokes equations, the intuition is
  as follows: Source terms acting at large scales provide inputs which
  are redistributed across a range of length scales due to
  nonlinear effects; Strong dissipation acts as a sink preventing complex nonlinear
  dynamics below the dissipative length scale.  Thus, if the behavior
  of two solutions synchronize above the dissipative scale, then they synchronize at all scales.
  The dynamical structure leading to the Foias-Prodi effect has a different character for damped-driven KdV.
  Indeed, in our situation the damping acts
  uniformly across length scales.  On the other hand the complete
  integrability of KdV suggests a stronger
  constraint to the redistribution of source terms between length
  scales.  Here, compare the KdV invariants which constrain motion
  at all degrees of Sobolev regularity with the known invariants of 2D
  Euler which only appear at the $L^2$ and $H^1$ levels.  In summary, in order to achieve the Foias-Prodi bound in the absence of dissipation, one can appeal to stronger constraints on the redistribution of
  energy across length scales through the integrals of motion, but this ultimately results in a non-pathwise estimation.  At a technical level, when we estimate
  the difference $w$ between the solution $v$ of the controlled system
  \eqref{eq:s:KdV:coupled:intro} and the solution $u$ of the
  original system \eqref{eq:s:KdV}, we introduce a modified Hamiltonian
  functional (see \eqref{energy:funct:0:1}, \eqref{eq:mod:fn:base}, and \eqref{eq:FP:functional}),
  which is specifically designed to control the $H^1$--norm of $w$ (\cref{prop:FP:functional}).  Indeed, in order to remove problematic terms produced by nonlinear interaction of $u$ and $w$, the modified Hamiltonian
  incorporates the original solution $u$ through a direct dependence
  (not just through $w$). Consequently, the evolution of the modified
  Hamiltonian retains noise terms, which 
  naturally leads to the non-pathwise, i.e. weak form, of the Foias-Prodi estimate.  Furthermore, we note that controlling the modified Hamiltonian requires some bounds on $u$ in $H^2$ which are established by invoking the Lyapunov structure (see \cref{lem:FP:stoppingtime}).  This illustrates how the higher conservation laws of KdV are used in our Foias-Prodi estimate.

  \cref{sect:Regularity} is devoted to the proof of
  regularity of invariant measures, \cref{thm:Regularity:Intro}, which combines the Foias-Prodi
  estimate of \cref{thm:FP:est:Intro} in conjunction with the
  Lyapunov structure of our system.  It should be emphasized  that our
  proof of regularity makes no assumptions on either the damping or
  the noise terms beyond smoothness of the latter.  Note that previous work in the
  deterministically forced setting had established that the global
  attractor of damped-driven KdV inherits the regularity of the
  forcing term \cite{MoiseRosa1997}.  One may expect an analogous
  result for \eqref{eq:s:KdV} with a stochastic forcing, i.e. that the
  invariant measure is supported on functions with the regularity of
  $f$ and $\sigma$, and this is indeed the conclusion of
  \cref{thm:Regularity:Intro}.    To overcome the lack of dissipative regularization
  in the linear part of KdV, the proof of the regularity of the
  attractor in \cite{MoiseRosa1997} is based on a non-trivial
  decomposition of the solution.  Here, we identify a proof of
  regularity of invariant measures that relies on a stochastic control
  argument.  In particular, for each function $u_0$ in the support
  of the invariant measure, we use the assumption of invariance to
  propagate integration of a functional depending on the Sobolev norm of $u_0$ forward to integration relative to $u=u(t;u_0)$ solving \eqref{eq:s:KdV}, then
  decompose $u$ into the sum of $v=v(t;\til{u}_0)$ solving
  \eqref{eq:s:KdV:coupled:intro} with mollified data, $\til{u}_0$,
  which regularizes $v$, and the difference $w=v-u$.  We demonstrate
  that the modified system inherits the Lyapunov structure of
  \eqref{eq:s:KdV}, which supplies bounds uniform in the mollification
  to assess the `cost of control' represented by $v$. We then exploit
   the Foias-Prodi estimate \cref{thm:FP:est:Intro} to drive $w$ asymptotically to $0$.

  Next, in \cref{sect:UniqueErgodicity}, we present the proof of
  \cref{thm:unique:ergodicity:Intro}, that is, uniqueness of the
  invariant measure for \eqref{eq:s:KdV} with arbitrary damping and
  essentially elliptic stochastic forcing (see
  \cref{thm:ergodicity:ess:elliptic} for its more precise
  formulation).  Here, we emphasize the relative compactness of our
  proof.  {{}Indeed, we essentially exploit the asymptotic coupling framework in \cite{HairerMattinglyScheutzow2011} for
  proving unique ergodicity for SPDEs that was further developed in
  \cite{GlattHoltzMattinglyRichards2015, ButkovskyKulikScheutzow2019, glatt2024long}. The departure point of the current work is derives from \cite{GlattHoltzMattinglyRichards2015}, where a robust
  sufficient condition for uniqueness of the invariant measure is identified. In particular, the condition guarantees the}
  existence of a `modified SPDE' such that solution laws from the
  modified SPDE are absolutely continuous with respect to the laws from
  the original system, and such that, given distinct initial
  conditions, there is a positive probability that solutions of the
  modified and unmodified systems asymptotically synchronize (see \cref{thm:GHMR}).  By
  requiring that the additive noise in \eqref{eq:s:KdV} span the
  directions controlled in the Foias-Prodi estimate of
  \cref{thm:FP:est:Intro}, we are able to choose   \eqref{eq:s:KdV:coupled:intro} with a suitable cutoff on the control
  term as the `modified
  SPDE' such that the Novikov condition is satisfied, verifying the requisite absolute continuity condition by the Girsanov
  theorem.  In previous applications, demonstration of the second
  property of synchronization typically follows immediately from a
  \textit{pathwise}, i.e., `strong', Foias-Prodi estimate. The key
  observation here, however, is that the asymptotic coupling framework
  only requires one to enforce the synchronization property on a set
  of positive probability along a sampling of
    evenly-spaced discrete time trajectories. With this in mind, we
  are able to effectively leverage our `weak' Foias-Prodi structure (as manifested in \cref{cor:FP:est}) and our Lyapunov structure (\cref{lem:FP:stoppingtime}) 
  with the Borel-Cantelli lemma to establish the desired
  synchronization property (\cref{prop:asymptotic:coupling}). 

  Let us pause to remark on the broader context of our approach.  There are a handful of results in the literature on unique
  ergodicity for SPDEs with weak damping and spatially
  smooth noise, including examples of a damped stochastic Euler-Voigt
  model and damped stochastic Sine-Gordon equation presented in
  \cite{GlattHoltzMattinglyRichards2015}.  In these examples, however,
  it was possible to implement a pathwise Foias-Prodi control,
  essentially by virtue of the nonlinearity being smoothed for the
  Euler-Voigt model, and bounded for the Sine-Gordon equation.  In
  contrast, a Foias-Prodi structure in expected value and its utility
  to the analysis of ergodic properties was previously observed in
  \cite{DebusscheOdasso2005} in the context of the damped stochastic
  cubic NLS.  We emphasize that a distinction of our approach is the
  transparency of our nudging scheme for proving Foias-Prodi
  estimates, and the compactness of our subsequent proof of
  uniqueness.  {{}Furthermore, we believe our proof of regularity via Lyapunov and
  weak Foias-Prodi structure to be broadly
  applicable}.  For example, the techniques and approach taken here
  apply just as well to the the damped stochastic cubic NLS appearing
  in \cite{DebusscheOdasso2005}.  In this regard, a key contribution
  of our work is the identification, through study of \eqref{eq:s:KdV},
  of novel and flexible ways to establish and exploit weak Foias-Prodi
  structure to prove regularity and uniqueness of invariant measures
  for SPDEs which lack a strong dissipative mechanism.

Results in the large damping regime are presented in
  \cref{sect:Large:Damping}, namely the existence of a spectral gap with respect to a
  Wasserstein-like distance (\cref{thm:spectralgap}), which implies
  the mixing result stated in \cref{thm:large:damp:Mix:Intro} (see
  also \cref{rmk:spectralgap}).  These results are established by virtue of an $L^2$-synchronization of solutions in expected value at time infinity, reminiscent of a one-force one-solution principle.  The proof also hinges on
  exponential moment bounds for a functional that controls the
  $H^2$--norm of the solution (\cref{lem:LD:exp:mom}) which
  refine the estimates provided in \cref{sect:Lyapunov} by capturing
  dependence on the damping parameter. Due to the `nonlinear scaling'
  imposed by \eqref{eq:inv:polynomials} among the Sobolev norms, a careful accounting of dependence on
  the damping parameter is a delicate matter.  These moment bounds are
  crucial because, in the pursuit of $L^2$-synchronization, the
  damping parameter naturally induces a linear rate of decay in time,
  but it is counteracted by exponential powers of the $H^2$-level
  functional.  
  The functional appears with a specific $3/7$ power in the exponential moment (see \eqref{eq:stab:main}--\eqref{eq:int:fact:L2:LD:sp}), so we focus our refined analysis on bounding this specific power.
  The power $3/7$ leads to another delicate point of analysis, which is that we need to confirm that $L^2$-powers incorporated to positivize the $H^2$-level functional (to ensure that it controls the $H^2$-norm) can be reduced to a strength of the reciprocal $7/3$ power, in order to avoid unbounded moments (see \eqref{eq:I2p:exp:LD:eqv}).  {{}We refer the reader to \cref{rmk:large:damping:scaling} for a more in-depth discussion of this particular scaling.} Ultimately, we are able to verify this is possible, and to control the exponential rate in time of the required exponential moments
  uniformly with respect to the damping parameter, which is what allows us to close the proof of the spectral gap (via $L^2$-synchronization) provided that the
  damping is large.  
  
  Finally, in \cref{sect:det:case}, we present novel proofs of largely known results for the long-time dynamics of deterministic damped-driven KdV.  Indeed, by revisiting our analysis from previous sections, setting $\sigma \equiv 0$, we are able to quickly deduce a Foias-Prodi type estimate and regularity of the global attractor for \eqref{eq:s:KdV} with smooth deterministic forcing.  This section is included to highlight several benefits to our approach.  With respect to novelty in the results, in contrast to the seminal formulation in \cite{JollySadigovTiti2017}, our Foias-Prodi estimate (see \cref{prop:FP:det}) is established with no assumption that the solution lies on the global attractor.  Furthermore, our proof of regularity of the global attractor (\cref{thm:sm:glb:attr}) is established in the case of forcing in $H^2$, an endpoint that was presented as an open problem in \cite{MoiseRosa1997}.  More importantly for the regularity result, however, is that our proof is novel in its implementation of a nudging scheme for this purpose.  In particular, the proof is relatively compact and transparent by virtue of this strategy.

\section{Mathematical Background and Notational Conventions}
\label{sect:background}

This preliminary section gathers some general background on KdV and
the damped stochastically driven variant of KdV, \eqref{eq:s:KdV},
that we consider in this work.  The presentation will serve as a
foundation for the rest of our analysis below.  We also use this
section to fix some notational conventions which are maintained
throughout the manuscript.

\emph{Regarding the usage of constants:}
throughout what follows we denote by $c$ any constants which 
depend only on universal (namely equation independent) quantities.  Otherwise we take
$C$ to denote constants which contain dependencies
on \eqref{eq:s:KdV} through solution parameters such as $\gamma, f, \sigma, t$.
When appropriate we are as explicit as needed concerning these dependencies
in the statement of the bounds.

\subsection{The Functional and Stochastic Setting}\label{sect:functional:setting}

Let $\T=(0,2\pi]$ and denote by $\mathcal{M}_{per}$ the set of
real-valued, Borel measurable functions that are periodic on $\T$.  We
work exclusively in the setting of mean-free functions over $\T$.  For
$p\in[1,\infty]$, we define the \textit{Lebesgue spaces},
$L^p:=\left\{g\in\mathcal{M}_{per}(\T): \Sob{g}{L^p}<\infty\right\}$,
where $\Sob{g}{L^p}:= \left(\int_\T |g(x)|^pdx\right)^{1/p}$ for
$p \in [1,\infty)$ and $\|g\|_{L^\infty} :=\esssup_{x\in\T}|g(x)|$.
For any $m \geq 0$, we define the \textit{(homogeneous) Sobolev
  spaces}, $H^m$, by
$H^m :=\{g\in L^2:\Sob{g}{H^m}<\infty, \int_\T g(x)dx=0\}$ where we
let $\Sob{\cdotp}{H^m}$ denote the norm induced by the inner product
given by
\begin{align*}
  \lb g_1,g_2\rb_{H^m}
       :=\sum_{k\in\ZZ\smod\{\bf{0}\}}|k|^{2m}\hat{g}_1(k)\overline{\hat{g}_2(k)},
  \quad \hat{g}(k):=(2\pi)^{-1}\int_\T e^{-ikx}g(x)dx,
\end{align*}
where $\overline{\hspace{2pt}\cdotp\hspace{2pt}}$ denotes complex
conjugation. When $m$ is positive integer, we use the notation
$D^m u =\frac{d^m}{dx^m}u$ so that, by Plancherel's Theorem, one has
$\lb g_1,g_2\rb_{H^m}=\lb D^m g_1, D^m g_2\rb_{L^2}$ where
$\lb g_1,g_2\rb_{L^2}:=\int_{\T}g_1(x)\overline{g_2(x)}dx$.  For
convenience, we will often denote $\lb \cdotp,\cdotp\rb_{L^2}$ simply
as $\lb \cdotp,\cdotp\rb$.  We will also occasionally write $H^0$ for
$L^2$.

We turn next to recall some formalisms needed for the stochastic
elements in our problem.  Fix a stochastic basis
$\mathcal{S} =(\Omega, \mathcal{F}, \{\mathcal{F}_t\}_{t \geq 0},
\Prb, \{W_k\}_{k\geq1})$, namely a filtered probability space relative
to which $\{W_k\}_{k\geq1}$ is a collection of independent,
identically distributed, real-valued, standard, Brownian motions.  {{}Given a Hilbert space, $H$, we denote the space of probability measures on the Borel sets, $\mathcal{B}(H)$, of $H$ by $\Pr(H)$.}

We
adopt the notation $\bH^{m}$ for the sequence space $\ell_2(H^m)$ that
is
\begin{align}
  \bH^{m} := \{ g = \{g_k \}_{k \geq 0} : \| g \|_{H^m}  < \infty \},
  \quad \| g \|_{H^m}^2 := \sum_{k \geq 1} \|g_k\|_{H^m}^2,
  \label{eq:sigma:sp}
\end{align}
and sometimes write, for $g, \tilde{g} \in \bH^0$,
\begin{align}
  (g \cdot \tilde{g})(x)  := \sum_{k \geq 1} g_k(x) \tilde{g}_k(x),
  \quad
  \text{ for } x \in \TT.
  \label{eq:point:wise:mult}
\end{align}
Then we regard
\begin{align}
  \sigma W := \sum_{k \geq 1} \s_k W_k,
  \label{eq:sigW:def}
\end{align}
so that when $\sigma \in \bH^{m}$ for some $m \geq 0$, one has
$\sigma W \in C([0,\infty); H^m)$ almost surely.

In another standard terminology of stochastic partial differential
equations, when $\sigma \in \bH^{m}$, $\sigma W$ is regarded as a
Brownian motion on $H^m$ with covariance given as
\begin{align}\label{eq:covariance}
C_m g: = \sum_{k \geq 1} \lb g, \s_k \rb_{H^m} \s_k.
\end{align}
Then $C_m$ is a
symmetric, non-negative trace class operator on $H^m$ whenever
$\sigma \in \bH^{m}$. Indeed, $\mbox{Tr}(C_m) = \|\sigma\|_{H^m}^2$
where recall that
$\mbox{Tr}(C_m) := \sum_{k \geq 1} \lb C_m e_k, e_k \rb_{H^m}$, for
any orthonormal sequence $\{e_k\}_{k \geq 0}$ in $H^m$.  See e.g. \cite{prevot2007concise, da2014stochastic} for
further details and generalities regarding Brownian motion on a
Hilbert space. 

Note that we
are particularly interested in the case when $\sigma$ is degenerate,
that is when only finitely many of the elements comprising $\sigma$
are non-zero.  
In this case, the range of $C_m$ is finite
dimensional. {{}In particular, if we interpret $\s:\R^K\goesto L^2$, $x\mapsto \sum_{k=1}^Kx_k\s_k$, for some $K>0$, as a linear operator, and we let $N$ denote the largest integer for which $\vspan\{\s_1,\dots,\s_K\}\supset P_NL^2$, then the corresponding (restricted) pseudo-inverse, $\s^{-1}: P_NL^2\goesto \R^K$ is a bounded operator.} 

\subsection{Integrals of Motion and Related Functionals}
\label{sect:Integrals}

As already discussed in the introduction the KdV equation
\begin{align}
  \bdy_t u + u D u + D^3 u = 0,
  \label{eq:KdV:det}
\end{align}
has a special collection of conserved quantities $\In_m$ defined for
$m \geq 0$, which are commonly refered to as the `integrals of motion'
of \eqref{eq:KdV:det}.  These invariants are given as certain
polynomial functions of $(u, Du, \ldots, D^m u)$ whose leading order
terms correspond to the $H^m$ norm.   Such
functionals, $\In_m$, and certain modified counterparts (cf. \eqref{eq:mod:int:m:simple}) 
will play a significant role in our analysis throughout.

In order to specify these integrals of motion 
we start by taking\footnote{The integral of motion $\In_1$ corresponds to the
  Hamiltonian of \eqref{eq:KdV:det} so that the KdV equation can be
  regarded as infinite-dimensional Hamiltonian system.  See
  e.g. \cite[Example 2.1]{Kuksin2000}, \cite{MarsdenRatiu2013}
  for further details.}  $\In_0(v) :=  \int_\T v^2dx$, 
\begin{align}
  \In_1(v) :=  \int_\T \left((Dv)^2dx - \frac{1}{3}  v^3 \right) dx,
  \quad 
  \In_2(v) := \int_\T \left((D^2v)^2 - \frac{5}{3} u (Du)^2 + \frac{5}{36} u^4 \right),
  \label{energy:funct:0:1}
\end{align}
each of which is formally invariant under \eqref{eq:KdV:det}, as one can readily check by hand.
To proceed to the higher order invariants it is convenient to introduce, for each $k \geq 0$, the
vector-valued differential operators $D_k$ as
\begin{align}
  D_k: H^k  \goesto H^k\times H^{k-1}\times\dots\times L^2,
  \quad D_kv :=(v,D^1 v,\dots, D^k v).
  \label{eq:vect:der}
\end{align}
With this notation we consider
for each $m \geq 0$, functionals $\mathcal{I}_m:H^m\goesto\R$ of the form
\begin{align}\label{energy:funct}
  \In_m(v)
  :=\int_\T 
  \left((D^m v)^2-\al_{m}v (D^{m-1}v)^2+Q_{m}(D_{m-2} v)\right)
     dx,
\end{align}
for suitable $\alpha_m \in \RR$, $Q_{m}: \RR^{m-1} \to \RR$ starting from
\begin{align}\label{def:al0:Q0:Q1}
       \al_0=0, \;Q_0\equiv0,
       \qquad 
       \al_1=1/3, \;Q_1\equiv0,
       \quad \text{ and } \quad
       \al_2= 5/3, \; Q_2(v) = \frac{5}{36}v^4,
 \end{align}
as is consistent with \eqref{energy:funct:0:1}.

The values of the coefficient $\al_m$ and the structure $Q_{m}$
can be determined to so that each $\In_m$ is preserved by
\eqref{eq:KdV:det}.  To this end, for each $m \geq 0$, we consider
the operators $L_m:H^{2m}\goesto H^0$ defined as
\begin{align}\label{def:dispersive:J}
  L_m(v):=& 2 (-1)^m  D^{2m}v-\al_{m}(D^{m-1}v)^2
            +2\al_{m}(-1)^mD^{m-1}(vD^{m-1}v)\notag\\
          &+2\sum_{k=0}^{m-2}(-1)^kD^k
            \left(\frac{\partial Q_{m}}{\partial y_k}(D_{m-2}v)\right),
\end{align}
where the final sum in \eqref{def:dispersive:J} is omitted when $m=0,1$. 
Thus, when $u$ is sufficiently smooth and obeys the KdV equation \eqref{eq:KdV:det} we have
\begin{align}\label{I:J:relationship}
  \frac{d}{dt}\In_m(u)
  =\lb L_m(u),\bdy_t u\rb
  =-\lb L_m(u),u {D} u + {D}^3 u\rb,
\end{align}
for each $m\geq0$. As such, the coefficient $\al_m$ and structure of the terms $Q_{m}$
are determined iteratively by the relations
\begin{align}
  \lb L_m(v),vDv +D^3v \rb=0,
  \label{dispersive:id}
\end{align}
for all $v \in H^{2m}$, $m\geq0$.

The coefficients $\al_{m}$ and form for the lower order
terms $Q_m$, structured as a suitable sum of 
non-constant monomials in $D_{m-2}u$, needed to maintain \eqref{dispersive:id}, have been
obtained in e.g. \cite{miura1968II,kruskal1970korteweg}.  In order to recall this
result in a precise formulation, we introduce some further notations
and definitions.  Let $\N_0$ to denote the non-negative integers and
recall the following multi-index notations. For
$\beta = (\beta_0,\beta_1, \ldots, \beta_k) \in \N_0^{k+1}$ we take
\begin{align*}
  |\beta| = \beta_0 + \beta_1 + \cdots + \beta_k,
\end{align*}
and let
\begin{align}
    y^\beta := y_0^{\be_0}\dots y_k^{\be_k}  
    \quad \text{ for } y = (y_0, \ldots, y_{k}) \in \R^{k+1},
  \label{eq:mono:multi}
\end{align}
so that $y^\beta$ is a monomial in $k+1$ (real) variables.  When
needed, for such $\beta \in \N_0^{k+1}$ we will sometimes write
$\pi_\beta(y) = y^\beta$ 
and denote
\begin{align}
  \pi_\be(D_kv) = v^{\be_0} (Dv)^{\be_1} \cdots (D^kv)^{\be_k}  
  \quad \text{ for any } v \in H^k,
  \label{eq:ass:poly}
\end{align}
where the vector-valued differential operator, $D_k$, is given as in \eqref{eq:vect:der}. 
Similarly we take for any $\alpha \in  \N_0^{k+1}$
\begin{align}
  \partial_\alpha := \partial_0^{\alpha_0} \cdots \partial_k^{\alpha_{k}},
  \label{eq:partial:multi}
\end{align}
to be the $|\alpha|$th order partial differential operator on functions
of $k+1$ variables.
\begin{Def}[\cite{miura1968II}]
  \label{def:rank}
  Fix any $k \geq 0 $. For any multi-index
  $\be = (\be_0,\dots,\be_k) \in\N^{k+1}_0$ consider the associated
  monomial $\pi_\be(y) = y^\beta$ in $k+1$ variables.  We define the
  \emph{rank} of this monomial $\pi_\be$ according to
  \begin{align}
    \rank(\pi_\be):=\sum_{j=0}^{k}\left(1+\frac{j}2\right)\be_j.
    \label{eq:rank:def}
  \end{align}
\end{Def}
\noindent We adopt the notation 
\begin{align*}
  \Mon_k:=\{y^\beta = y_0^{\be_0}\dots y_k^{\be_k}: \be=(\be_0,\dots,\be_k)\in\N_0^{k+1}\},
\end{align*}
that is the class of monomials in $k+1$ variables.  For any   
$k,n\geq0$, we denote the collection of the rank $n$ monomials
in $k +1$ variables as
\begin{align}\label{def:Ronkn}
  \Ron_{k,n}:=\{\pi \in \Mon_k:\rank(\pi)=n\}.
\end{align}
Clearly, for each $k,n\geq0$, $\Ron_{k,n}$ is a finite set. 
Note that the rank of any derivatives of a monomial are bounded above 
by the rank of the monomial itself. 

The result we need from \cite{miura1968II} is now summarized 
as follows:
\begin{Thm}[\cite{miura1968II}]
  \label{thm:KdV:Poly:Rank}
    For each $m \geq 2$, there exist 
  polynomials $Q_m$ made up solely of rank $m +2$ monomials of $m-1$ variables,
  namely
  \begin{align}
    Q_{m}(y_0,\dots, y_{m-2})
    :=\sum_{\pi \in\Ron_{m-2,m+2}}\al^{(\pi)}_{m} \pi(y_0,\dots, y_{m-2}),
    \quad \text{ for suitable values } \al^{(\pi)}_{m} \in \RR,
    \label{def:Qm:Form}
  \end{align}
  such that with these
  chosen values for $\alpha_m^{(\pi)}$, the
  corresponding nonlinear operators $L_m$ defined in
  \eqref{def:dispersive:J} satisfy the identity \eqref{dispersive:id}. 
\end{Thm}

\begin{Rmk}\label{rmk:Im:convention}
  From now on we will set $Q_m$ to be specified by
  \eqref{def:al0:Q0:Q1} when $m=0,1$ and by \eqref{def:Qm:Form} when
  $m\geq2$. We then suppose that $\In_m$ is given by
  \eqref{energy:funct}. By \cref{thm:KdV:Poly:Rank}, this choice of
  $\In_m$ is an invariant of the KdV equation, \eqref{eq:KdV:det}, for
  all $m\geq0$.
\end{Rmk}

As is evident from \eqref{energy:funct}, the leading order term in
$\In_m(v)$ is proportional to $\|u\|_{H^m}^2$, while lower order terms
are not sign-definite.  Nevertheless, these lower order terms can be
estimated by appropriately interpolating between $L^2$ and $H^m$
norms. It is important to point out that interpolation causes
exponents to accumulate at the endpoints. In principle, this could
disrupt the underlying Lyapunov structure, unless one can control the
accumulation. It is, therefore, crucial to exploit the constraint that
the monomials representing the lower order terms are all of a given
rank. For this reason, we prove the following general estimate on
monomials of $D_kv$ of a particular rank dictated by the monomials in
\eqref{def:Qm:Form}.  These bounds will be used extensively when $Q_m$
and various derivatives of $Q_m$ appear in our analysis.
\begin{Prop}
  \label{prop:mono:intp:bnd}
  Let $k\geq0$ and consider any $\be\in\N_0^{k+1}$.  Suppose that
  taking $\pi_\beta = y^\beta$ we have that
  $\rank(\pi_{\be}) \leq k+4$ with the $\rank$ defined as in
  \eqref{eq:rank:def}.  Then, cf. \eqref{eq:ass:poly},
  \begin{align}
    \Sob{\pi_{\be}(D_kv)}{L^1}
    \leq c \Sob{D^{k+2}v}{L^2}^{r}\Sob{v}{L^2}^{|\be|-r}, 
    \label{eq:momomial:est:1}
  \end{align}
  for every $v \in H^{k+2}$, for some constants $r = r(k,\beta) \in [0,2)$
  and $c = c(k,\be)>0$ independent of $v$. 
\end{Prop}
\begin{proof}
  Begin by observing that
  \begin{align}
    \Sob{\pi_\be(D_kv)}{L^1}
    \leq \Sob{v}{L^\infty}^{\be_0}\cdots\Sob{D^{k-1}v}{L^\infty}^{\be_{k-1}}
    \int_{\T} |D^{k} v|^{\beta_k} dx.
    \label{eq:main:est:IP}
  \end{align}
  For $l = 0, \ldots, k-1$, Agmon's inequality and interpolation
  yield
  \begin{align}
    \| D^l v\|_{L^\infty}
    \leq c \| D^l v\|_{L^2}^{1/2} \| D^{l+1} v\|_{L^2}^{1/2}
    \leq c \| D^{k+2}\|_{L^2}^{\frac{2l +1}{2(k+2)}}
    \| v\|_{L^2}^{\frac{2(k+2) - (2l +1)}{2(k+2)}}.
		 \label{eq:Linf:intp:lo}
  \end{align}
  We now combine \eqref{eq:main:est:IP} and \eqref{eq:Linf:intp:lo} to
  address \eqref{eq:momomial:est:1} by considering three cases for
  $\beta_k$, namely $\beta_k = 0, 1$ and $\beta_k \geq 2$.
	
  Firstly suppose that $\beta_k \geq 2$.  Invoking the
  Gagliardo-Nirenberg-Sobolev inequality and then interpolating we
  find
  \begin{align}
    \int_{\T} |D^{k} v|^{\beta_k} dx
    \leq c \|D^{k} v\|_{L^2}^{\frac{\beta_k + 2}{2}}
           \|D^{k+1} v\|_{L^2}^{\frac{\beta_k -2}{2}}
    \leq c \|D^{k+2} v\|_{L^2}^{\beta_k \frac{2k +1}{2(k+2)} - \frac{1}{k+2}}
    \| v\|_{L^2}^{\frac{3 \beta_k + 2}{2(k+2)}}.
		 \label{eq:Lbeta:intp:lo}
  \end{align}
  Combining \eqref{eq:main:est:IP}, \eqref{eq:Linf:intp:lo},
  \eqref{eq:Lbeta:intp:lo} we find
  \begin{align}\label{pb:main:interpolation}
    \Sob{\pi_\be(D_kv)}{L^1}
    \leq c
    \Sob{D^{k+2}v}{L^2}^{\sum_{\ell=0}^k\be_\ell\frac{2\ell+1}{2(k+2)} -\frac{1}{k+2}}
    \Sob{v}{L^2}^{\sum_{\ell=0}^k\be_\ell\frac{2(k+2)-(2\ell+1)}{2(k+2)}+ \frac{1}{k+2}}
  \end{align}
  Now observe that, referring back to \cref{def:rank},
  \begin{align}
    \sum_{\ell=0}^k\be_\ell\frac{2\ell+1}{2(k+2)} -\frac{1}{k+2}	
    = \frac{1}{k+2}\left(2\rank(\pi_\be) - \frac{3}{2} |\be| - 1\right).
	\label{eq:smp:rearng}
  \end{align}
  According \eqref{eq:rank:def} it is clear that
  $ \rank(\pi_\be) \leq (1 + k/2) |\beta|$ i.e. that
  $- |\beta| \leq -\frac{2}{k+2} \rank(\pi_\beta)$.  Thus, with our
  assumption that $\rank(\pi_\be) \leq k+ 4$, we have
  \begin{align*}
    \frac{1}{k+2}\left(2\rank(\pi_\be) - \frac{3}{2} |\be| - 1\right)
    \leq \frac{1}{(k+2)^2} ( (2k + 1)\rank(\pi_\be) - (k+2))
    \leq 2 \frac{k^2 + 4k + 1}{(k+2)^2} < 2.
  \end{align*}
  Combining this upper bound, the identity \eqref{eq:smp:rearng} and
  \eqref{pb:main:interpolation} we conclude \eqref{eq:momomial:est:1}
  in this first case when $\beta_k \geq 2$.
	
  We turn to the remaining cases when $\beta_k = 0, 1$. Noting that
  \begin{align*}
    \|D^k v\|_{L^1} \leq c \|D^k v\|_{L^2}
    \leq c\|D^{k+2} v\|_{L^2}^{\frac{k}{k+2}} \| v\|_{L^2}^{\frac{2}{k+2}},
  \end{align*}	
  and invoking \eqref{eq:Linf:intp:lo}, \eqref{eq:main:est:IP} we have,
  \begin{align}
    \Sob{\pi_\be(D_kv)}{L^1}
    \leq c
    \Sob{D^{k+2}v}{L^2}^{\sum_{\ell=0}^{k-1}\be_\ell\frac{2\ell+1}{2(k+2)} + \beta_k \frac{k}{k+2}}
    \Sob{v}{L^2}^{\sum_{\ell=0}^{k-1}\be_\ell\frac{2(k+2)-(2\ell+1)}{2(k+2)}+ \beta_k \frac{2}{k+2}},
    \label{eq:set:up:bnd:bk12}
  \end{align}
  when either $\beta_k = 0, 1$.  On the other hand, similarly to
  \eqref{eq:smp:rearng}, we observe that
  \begin{align}
    r := \sum_{\ell=0}^{k-1}\be_\ell\frac{2\ell+1}{2(k+2)}
    + \beta_k \frac{k}{k+2}
    = \frac{1}{k+2}\left( 2 \rank(\pi_\beta)
    - \frac{3}{2}|\beta| - \frac{\beta_k}{2} \right).
    \label{eq:smp:rearng:0:1}
  \end{align}

  To complete the proof we now consider individually the cases when
  $\beta_k = 0, 1$ and index over $|\beta| = 0,1,2$ and finally
  $|\beta| \geq 3$. In each case \eqref{eq:set:up:bnd:bk12},
  \eqref{eq:rank:def} combines with \eqref{eq:smp:rearng:0:1} to
  establish that $r < 2$ and hence to conclude
  \eqref{eq:momomial:est:1}. {{}To see how, observe that $(2\ell+1)/2\leq k$ whenever $\ell\leq k-1$. Hence, \eqref{eq:smp:rearng:0:1} implies $r\leq |\be|\frac{k}{k+2}$, from which we deduce $r<2$ if $|\be|\leq2$.} 
  On the other hand, when $|\beta| \geq 3$ and
  either $\beta_k = 0,1$ our assumption that
  $\rank(\pi_\beta)\leq k+4$ yields $r \leq \frac{2k + 7/2}{k+2}$.
  The proof is now complete.
\end{proof}

Working from the bounds provided in \cref{thm:KdV:Poly:Rank} and
\eqref{eq:momomial:est:1} on lower order terms in $\In_m$ we now
introduce the following modification of $\In_m$ as
\begin{align}
  \In_{m}^+(v) := \In_m(v) + {\abarm}\left( \|v\|_{L^2}^{2}+1\right)^{\bar{q}_m},
  \label{eq:mod:int:m:simple}
\end{align}
where ${\abarm},\bar{q}_m\geq 0$ are to be specified.  In particular,
under suitable specifications for $\In_m^+$, we have the following
result.
\begin{Prop}
  \label{lem:equivalence}
  Fix any $m \geq 0$ and define $\In_m$ as in \eqref{energy:funct:0:1}
  and \eqref{energy:funct} with $\al_m, Q_m$ specified by
  \eqref{def:al0:Q0:Q1} and \cref{thm:KdV:Poly:Rank}.
  Then, for every value of $\abarm, \bar{q}_m \geq 1$  in \eqref{eq:mod:int:m:simple} sufficiently large,
  we have
  \begin{align}
    \frac{1}{2} \left(\| D^m v\|^2_{L^2} + {\abarm} \left(\|v\|_{L^2}^{2}+1\right)^{\bar{q}_m}\right)
    \leq \In_{m}^+(v)
    \leq \frac{3}{2} \left(\| D^m v\|^2_{L^2} + {\abarm} \left(\|v\|_{L^2}^{2}+1\right)^{\bar{q}_m}\right),
    \label{eq:Imp:equi}
  \end{align}
  for any $v \in H^m$.
\end{Prop}
\begin{proof}
   When $m=0$, there is nothing to show. 
   {{}
   Turning to the case $m\geq1$,
  we have as an immediate consequence of \cref{prop:mono:intp:bnd} and
  Young's inequality that
  \begin{align*}
     \left| \int_{\T}Q_{m}(D_{m-2}v)dx
          \right|
 &\leq \frac{1}{8}\|D^{m} v\|^2_{L^2} + c( \|v\|_{L^2}^2+1)^{q}. 
   \end{align*}
  Similarly, by H\"older's inequality and  interpolation, we estimate
  \begin{align*}
   \left| \int_{\T}v (D^{m-1}v)^2dx\right|\leq \Sob{v}{L^\infty}\Sob{D^{m-1}v}{L^2}^2\leq c\Sob{D^mv}{L^2}^\frac{4m-3}{2m}\Sob{v}{L^2}^{\frac{2m+3}{2m}}.
  \end{align*}
  It follows, upon another application of Young's inequality, that
    \begin{align*}
      \left| \int_{\T}\left(\al_{m}v (D^{m-1}v)^2 - 2Q_{m}(D_{m-2}v) \right)dx\right|\leq\frac{1}{2}\|D^{m} v\|^2_{L^2} + c( \|v\|_{L^2}^2+1)^{q}. 
    \end{align*}
}
   Thus upon recalling the definition of \eqref{energy:funct}, we have
  \begin{align*}
   \frac{1}{2} \| D^m v\|^2_{L^2} -  c (\|v\|_{L^2}^{2}+1)^{q} 
    \leq \In_m(v) \leq \frac{3}{2} \| D^m v\|^2_{L^2} + c(\|v\|_{L^2}^{2}+1)^{q}.
  \end{align*}
  Adding ${\abarm}(\Sob{v}{L^2}^{2}+1)^{\bar{q}_m}$, for any
  appropriately large $\abarm$ and $\bar{q}_m$ across the chain of
  inequalities produces the desired result.
\end{proof}

\subsection{Well Posedness of Solutions and 
  the Markovian Setting}
  \label{sec:well:pose}

We now turn to state the well-posedness of
\eqref{eq:s:KdV} and then introduce the associated Markovian
framework upon which our analysis is founded. 
We refer the reader to
to \cref{sect:apx:wp:SKdV} for proof of the following result.

\begin{Prop}\label{prop:exist:uniq}
  Let $m\geq2$ and $\gam>0$ and fix a stochastic basis
  $\mathcal{S} =(\Omega, \mathcal{F}, \{\mathcal{F}_t\}_{t \geq 0},
  \Prb, \{W_k\}_{k\geq1})$.
  Then, for any $f\in H^m$, $\sigma \in \bH^{m}$ and $u_0\in H^m$,
  relative to this stochastic basis $\mathcal{S}$: 
  \begin{itemize}
  \item[(i)] There exists a random process
    $u(\cdot) = u(\cdot;u_0,f, \sigma)$ such that
    $u \in C([0,\infty);H^m)$ which is $\mathcal{F}_t$-adapted and
    solves the initial value problem \eqref{eq:s:KdV} in the usual
    It\=o sense.\footnote{When $m = 2$, \eqref{eq:s:KdV} is defined in
      the usual weak sense against test functions. On the other hand,
      for $m \geq 3$, \eqref{eq:s:KdV} can be made sense of pointwise
      for almost every $x \in \TT$.}
  \item[(ii)] Solutions to \eqref{eq:s:KdV} are pathwise unique,
  namely, given two solutions $u$ and $\tilde{u}$ corresponding to the
  same data $u_0, f$ and $\sigma$ we have
  \begin{align*}
    \Prb(u(t) = \tilde{u}(t) \text{ for every } t \geq 0) = 1.
  \end{align*}
\item[(iii)] Solutions to \eqref{eq:s:KdV} are continuous with respect to the initial data and forcing terms, that is
  for any fixed $t \geq 0$,
    \begin{align}
      u(t; u_0^{(n)}, f^{(n)}) \to u(t; u_0, f) 
      \quad \text{ almost surely in } H^m \text{ as } n \to \infty,
      \label{eq:cont:dep}
    \end{align}
    whenever $(u_0^{(n)},f^{(n)})  \to (u_0,f)$ in $(H^m)^2$ as $n \to \infty$.
  \end{itemize}
\end{Prop}
\noindent Note that when the context is clear in what follows we will
often drop the dependence on $f \in H^m$ and $\sigma \in \bH^m$
writing $u(t) = u(t; u_0)$.

\cref{prop:exist:uniq} provides all the ingredients necessary to place
\eqref{eq:s:KdV} in a Markovian framework; we refer the reader to e.g.
\cite{DaPratoZabczyk1996,KuksinShirikyan2012} 
for more generalities around this formalism.  For fixed $f \in H^m$
and $\sigma \in \bH^m$ the Markov kernel
\begin{align}
      P_t(u_0,A) = \Prb( u(t;u_0,f, \sigma) \in  A),\ \
      \text{ for any}\ t \geq 0, \quad u_0 \in H^m,\quad  A \in \mathcal{B}(H^m),
  \label{eq:mark:kernel}
\end{align}
where $\mathcal{B}(H^m)$ is the Borel $\sigma$-algebra on $H^m$.
This kernel defines a semigroup  $\{P_t\}_{t \geq 0}$ that acts on
observables $\phi: H^m \to \RR$ which are bounded and measurable
as
\begin{align}
  P_t \phi(u_0) := \int_{H^m} \phi(v) P_t(u_0, dv)
  \quad \text{ for any } u_0 \in H^m
  \label{eq:mt:obs:act}
\end{align}
and on Borel probability measures $\mu \in\Pr(H^m)$ as
\begin{align}
  \mu P_t(A) := \int_{H^m}  P_t(u_0, A)\mu(du_0),
  \quad \text{ for any }  A \in \mathcal{B}(H^m).
    \label{eq:mt:msr:act}
\end{align}
Here, we say that $\mu \in Pr(H^m)$ is an \emph{invariant measure}
for $\{P_t\}_{t \geq 0}$ if $\mu P_t = \mu$ for every $t \geq 0$.
Note that in view of \eqref{eq:cont:dep}, it is easy to see that
$\{P_t\}_{t \geq 0}$ is \emph{Feller} on $H^m$ in that
$P_t : C_b(H^m) \to C_b(H^m)$ where $C_b(H^m)$ represents the
continuous bounded functions on $H^m$.

\section{Lyapunov Estimates and
  Existence of Invariant Measures}
\label{sect:Lyapunov}

This section establishes the higher order Lyapunov bounds which we
overviewed in \cref{thm:Lyapunov:intro}.  These bounds will provide
the foundation for all of our subsequent analysis. As an immediate
application, we will infer the existence of invariant measures for
\eqref{eq:s:KdV}.

Note that in what follows it will often be convenient to work in terms
of the functionals native to \eqref{eq:s:KdV} that were introduced in
\cref{sect:Integrals} rather in terms of Sobolev norms directly.  We
recall that for each $m \geq 0$, we defined the energy functionals
$\In_m^+$ in \eqref{eq:mod:int:m:simple} which are defined by a
modification of the KdV invariants $\In_m$ specified as in
\eqref{energy:funct}.  We remind the reader that
\cref{lem:equivalence} demonstrates that (for suitably large values of
the parameters ${\bar{q}_m}, \abarm$) these functionals are
comparable to Sobolev norms modulo a sufficiently large power
$\bar{q}_m > 0$ of the $L^2$ norm, namely,
\begin{align*}
  \In_{m}^+(v)\sim \Sob{v}{H^m}^2+(\Sob{v}{L^2}^2+1)^{\bar{q}_m}.
\end{align*}
This allows one to think of $\In_m^+$ as representative of the
corresponding Sobolev norm.  Indeed, we immediately deduce
\cref{thm:Lyapunov:intro} from the following theorem by further
invoking \cref{lem:equivalence}.  Recall the notation $a\vee b:=\max\{a,b\}$, for real numbers $a,b$.

\begin{Thm}
  \label{thm:Lyapunov}
  Let $m \geq 0$ and suppose that $f \in H^{m\vee 2}$ and
  $\sigma \in \bH^{m\vee 2}$. Given $u_0\in H^{m\vee2}$, let
  $u = u(u_0, f, \sigma)$ denote the corresponding solution of
  \eqref{eq:s:KdV} with $\gam>0$ as in \cref{prop:exist:uniq}. There
  exists an ${\abarm}, \bar{q}_m \geq 1$ defining $\In_m^+$ as in
  \eqref{eq:mod:int:m:simple} sufficiently large, depending only on
  $m$, such that the following bounds hold:
  \begin{itemize}
  \item[(i)]
     For any $p > 0$
  \begin{align}
    \E\In_m^+(u(t))^p&\leq C(e^{-\gam pt}\In_m^+(u_0)^p+1),
    \label{eq:Lyapunov:Bnd}
  \end{align}
  for all $t\geq0$, where
  $C = C(m, p, \gamma, \|f\|_{H^m}, \|\sigma\|_{H^m}, \bar{q}_m,
  \abarm)$.
\item[(ii)] 
   For any $p>0$,  $q > 2$
    \begin{align}\label{eq:poly:est:asymptotic}
      \Prb\left(\sup_{t\geq T}\left(\In_m^+(u(t))^p
      +\gam p\int_0^t\In_m^+(u(s))^pds -C(t + 1)\right)\geq \In_m^+(u_0)^p
      +R\right)\leq  \frac{C\In_m^+(u_0)^{pq}}{(T+R)^{q/2-1}},
    \end{align}
    for all $R > 0$ and $T\geq0$, where
    $C = C(m, p, q, \gamma, \|f\|_{H^m}, \|\sigma\|_{H^m}, \bar{q}_m,
    \abarm)$ is independent of $T, R$.
  \item[(iii)]
    There exists an exponent
    $p_0 := p_0(m, \gamma, \|f\|_{H^m}, \|\sigma\|_{H^m}, \bar{q}_m)$ with
    $0<p_0\leq 1$ such that for all $p\in(0,p_0]$ and for any $\bar{\eta} > 0$
    \begin{align}\label{eq:exp:est:asymptotic}
      \Prb\left(\sup_{t\geq0}\left(\In_m^+(u(t))^p+ \gam p \int_0^t\In_m^+(u(s))^pds-C t\right)
              \geq \In_m^+(u_0)^p+R\right)
      \leq {{}2}\exp\left(-\bar{\eta} R\right),
    \end{align}
    for every $R > 0$ where
    $C =C (\bar{\eta}, p, m,\gam,\Sob{f}{H^m},\Sob{\s}{H^m}, \bar{q}_m,
    \abarm)$.  Moreover, in the special case when $m = 0$,
    we have
    \begin{align}\label{eq:exp:est:asymptotic:L2}
      \! \! \! \Prb\biggl(\sup_{t\geq0}\biggl(\|u(t)\|^2_{L^2}
      + \gam \int_0^t \|u(s)\|^2_{L^2}ds
      - \biggl(\frac{2}{\gam}{{}\|f \|_{L^2}^2}+ \|\sigma\|^2_{L^2}\biggr)t \biggr)
        \geq& {{}\|u_0\|_{L^2}^2}+R\biggr) \leq {{}2}e^{-\frac{1}{4} \gamma R \|\sigma\|^{-2}_{L^2}},
    \end{align}
    for every $R > 0$.
  \item[(iv)] Finally, for each $p\in(0,p_0]$, where $p_0$ is the
    threshold identified in $(iii)$, and for any $\eta > 0$,
  \begin{align}
    \label{eq:Lyapunov:Bnd:exp}
    &\E\exp\left(\eta\sup_{t\in[0,T]}
      \left(\In_m^+(u(t))^p+ \gam p\int_0^t\In_m^+(u(s))^pds\right)\right)
      \leq  \exp\left(\eta\In_m^+(u_0)^p + C(T +1))\right),
    \end{align}
    for all $T\geq0$ where again $C =C (\eta, p, m,\gam,\Sob{f}{H^m},\Sob{\s}{H^m}, \bar{q}_m,
    \abarm)$ but is independent of $T \geq 0$.  Moreover for any such $p \in [0,p_0)$
    \begin{align}
    \label{eq:Lyapunov:Bnd:subquadexp}
     \E \exp\left(\eta \In_m^+(u(T))^p\right)\leq C\exp\left(\eta e^{-\gam p T}\In_m^+(u_0)^p\right),
  \end{align}
  for all $T \geq 0$, where again
  $C =C (\eta, p, m,\gam,\Sob{f}{H^m},\Sob{\s}{H^m}, \bar{q}_m,
  {{}\abarm})$ but is independent of $T \geq 0$.
  \end{itemize}
\end{Thm}

{{}
From \cref{thm:Lyapunov}, we immediately deduce \cref{thm:Lyapunov:intro} from the introduction.

\begin{proof}[Proof of \cref{thm:Lyapunov:intro}]
By \cref{lem:equivalence} and \cref{thm:Lyapunov} (iv) we see that
\begin{align}
    \E\exp\left(\eta \|u(t)\|_{H^m}^{r}\right)&\leq C\E\exp\left(\eta \In_m^+(u(t))^{r/2}\right)\notag\\
    &\leq C\exp\left(\eta e^{-\frac{\gam r}2 t}\In_m^+(u_0)^{r/2}\right)\leq C\exp\left(\eta e^{-\frac{\gam r}2 t}\left(\Sob{u_0}{H^m}^r+\Sob{u_0}{L^2}^{r'}+1\right)\right),\notag
\end{align}
for some $r'\geq1$, as desired.
\end{proof}

The proof of \cref{thm:Lyapunov} is given in \cref{sect:Poly:bnd}
after we take the preliminary step of establishing evolutions for the
functionals $\In_m^+$ in \cref{sect:time:evol}. Before turning to
these details let us also show an application of
\cref{thm:Lyapunov} to establishing the existence of an invariant
measure for the Markov semigroup $\{P_t\}_{t\geq0}$ corresponding to
\eqref{eq:s:KdV}.  Note that this statement is precisely \cref{thm:exist:intro} from the introduction; its proof is elementary and appeals to the
classical Krylov-Bogolyubov procedure.
}

\begin{Cor}\label{thm:existence}
  Let $m\geq2$, $f\in H^{m+1}$, and $\s\in\bH^{m+1}$. Let
  $\{P_t\}_{t\geq0}$ be the Markov semigroup corresponding to
  \eqref{eq:s:KdV} on $H^{m}$. Then there exists $\mu\in\Pr(H^{m})$
  such that $\mu P_t=\mu$ for every $t > 0$.
\end{Cor}

\begin{proof}
For each $T>0$, define the probability measure
\begin{align}\notag
		\mu_T(A):=\frac{1}T\int_0^TP_t(0,A)dt,\quad A\in\mathcal{B}(H^{m}).
\end{align}
It suffices to show that the family
$\{\mu_T\}_{T\geq1}\subseteq Pr(H^{m})$ is tight. Indeed, since $P_t$ is
Feller in $H^{m}$, it would then follow from the Krylov-Bogolyubov
theorem that $\mu_T$ converges weakly, as $T\goesto\infty$, to some
$\mu\in\Pr(H^{m})$ and, moreover, that $\mu$ is an invariant measure
for $\{P_t\}_{t\geq0}$.

For $\rho>0$, let
$B_\rho:=\{v\in H^{m+1}:\Sob{v}{{H}^{m+1}}\leq \rho\}$. Then
$B_\rho\in\mathcal{B}(H^{m+1})$ and ${B_\rho}$ is pre-compact in
$H^{m}$.  Let $u(t;0)$ denote the unique solution of \eqref{eq:s:KdV}
corresponding to $u_0\equiv0$. Since $ f\in H^{m+1}$, an application
of Chebyshev's inequality, \cref{thm:Lyapunov} $(i)$ specialized to
the case $p=1$, and \cref{lem:equivalence} yields
\begin{align}
  P_t(0,B_\rho^c)\leq \frac{\E\Sob{u(t;0))}{H^{m+1}}^2}{\rho^2}\leq \frac{C}{\rho^2},\notag
\end{align}
for all $t\geq0$, where
$C=C(m,\gam, \Sob{f}{H^{m+1}},\Sob{\s}{H^{m+1}})$. Given $\veps>0$, we
choose $\rho$ such that $\rho^2>C\veps^{-1}$. Then for $T\geq1$, we
have
\begin{align}
  \mu_T(B_{\rho}^c)=\frac{1}T\int_0^TP_t(0,B_{\rho}^c)dt<\veps.\notag
\end{align}
This shows that the family $\{\mu_T\}_{T\geq1}$ is tight in $H^{m}$,
as claimed.
\end{proof}

\subsection{Time Evolution for Functionals of $u$}\label{sect:time:evol}
In this section we provide a precise formulation for the time 
evolution of the integrals of motion $\In_m$ as in \eqref{energy:funct:0:1} and \eqref{energy:funct}
and of various higher order moments. Let us start with the special
case of $m=0$, that is, the $L^2$--norm. This case follows as a direct application of It\=o's
lemma.  

\begin{Lem}\label{lem:L2:evolution}
  Fix any $u_0, f \in H^2$ and $\sigma \in \bH^2$ and consider the
  associated $u$ solving \eqref{eq:s:KdV} as in \cref{prop:exist:uniq}.
  Then $\|u\|_{L^2}^2$ obeys
\begin{align}
  d\|u\|^2_{L^2} + 2 \gamma \|u\|_{L^2}^2 dt = 2\lb f, u \rb dt + \|\sigma\|_{L^2}^2 dt + 2\lb u , \sigma \rb dW.
  \label{eq:L2:evol}
\end{align}
Furthermore, for any $p > 0$
\begin{align}
  d(\|u\|^2_{L^2} + 1)^p + 2 p \gamma (\|u\|^2_{L^2}+ 1)^p dt 
  &= p(\|u\|^2_{L^2} + 1)^{p-1}(2 \gamma + 2 \lb f , u\rb + \|\sigma\|^2_{L^2}) dt  \notag\\
   &\quad +2p(p-1)(\|u\|^2_{L^2} + 1)^{p-2} |\lb \sigma, u \rb|^2dt
\notag\\
    &\quad +2 p(\|u\|^2_{L^2} + 1)^{p-1} \lb u , \sigma \rb dW.
  \label{eq:L2:evol:p}
\end{align}
\end{Lem}

We next turn to compute evolution for $\In_m(u)$ for each $m \geq 1$.
The following formal computations can be rigorously justified by
considering a suitable regularization, e.g., fourth-order parabolic
regularization of solutions (cf. \submit{\cite{Temam1997} and
  \cref{sect:apx:wp:SKdV} for details}), so that solutions lie in
$C^\infty$ and then passing to the limit.  Working from
\eqref{eq:s:KdV} and \eqref{energy:funct:0:1}, \eqref{energy:funct},
for $m\geq1$, we have
\begin{align}
  d \In_m(u) \! = \!\!
  &\int_{\T} \! \biggl( \! 2 D^m u d(D^m u) - \!\alpha_m (D^{m-1} u)^2 du - 2\alpha_m u D^{m-1} u d(D^{m-1}u)
    \! + \!\! \sum_{j =0}^{m-2}\!\frac{\partial Q_m}{\partial y_j}(D_{m-2} u) d(D^j u) \!  \biggr)dx
                \notag\\
  &+ \int_{\T} \biggl( d(D^mu)  d(D^mu) - 2\alpha_m D^{m-1} u d(D^{m-1} u) du 
    - \alpha_m u d(D^{m-1} u) d(D^{m-1}u) \biggr) dx
  \notag\\
  &+ {{}\frac{1}2}\sum_{j,k =0}^{m-2}
    \int_{\T}\frac{\partial^2 Q_m}{\partial y_j\partial y_k}(D_{m-2} u) d(D^j u) d(D^k u) dx
    \notag\\
  =&   \int_{\T} \biggl( L_m (u) (f - u Du - D^3 u  - \gamma u) 
     + |D^m\sigma|^2 - 2\alpha_m D^{m-1} u  D^{m-1} \sigma \cdot \sigma 
    - \alpha_m u |D^{m-1} \sigma|^2 \biggr)   dx dt
  \notag\\
  &+ {{}\frac{1}2}\sum_{j,k =0}^{m-2}
    \int_{\T}\frac{\partial^2 Q_m}{\partial y_j\partial y_k}(D_{m-2} u) 
    D^j \sigma \cdot D^k \sigma  dx
    +\int_{\T} L_m (u) \sigma dx dW,
    \label{eq:my:big:phat:Ito:comp}
\end{align}
where recall that the nonlinear operator $L_m$ is given as in
\eqref{def:dispersive:J} and note the use of the notational convention (\ref{eq:point:wise:mult}).  Notice that with the cancelation condition
\eqref{dispersive:id} we have
\begin{align}
  \int_{\T} L_m (u)  &(u Du + D^3 u  + \gamma u) dx 
                     \notag\\
  &=  \gamma \int_{\T} \biggl(  2 |D^m u|^2 - 3 \alpha_m u (D^{m-1} u)^2 
  + \sum_{j = 0}^{m-2} \frac{\partial Q_m}{\partial y_j}(D_{m-2} u) D^j u \biggr)dx
     \notag\\
  &= 2 \gamma \In_m(u) +   
  \gamma \int_{\T} \biggl( \sum_{j = 0}^{m-2} \frac{\partial Q_m}{\partial y_j}(D_{m-2} u) D^j u 
      -2Q_m(D_{m-2} u) - \alpha_m u (D^{m-1} u)^2 \biggr)dx.
     \label{eq:gamma:terms:Im:Ito}
\end{align}
Also we have
\begin{align}
\begin{split}
\int_{\T}  L_m(u) f dx 
   &=  \int_{\T} \biggl(  2 D^m u D^mf -  \alpha_m  (D^{m-1} u)^2 f
  - 2 \alpha_m  u D^{m-1} u D^{m-1} f\biggr)dx \\
  &\quad + \int_{\T}\sum_{j = 0}^{m-2} \frac{\partial Q_m}{\partial y_j}(D_{m-2} u) D^j f dx
     \label{eq:f:terms:Im:Ito}
\end{split}
\end{align}
and similarly
\begin{align}
\begin{split}
 \int_{\T} L_m(u) \sigma dx 
  &=  \int_{\T} \biggl(  2 D^m u D^m\sigma -  \alpha_m  (D^{m-1} u)^2 \sigma
  - 2 \alpha_m  u D^{m-1} u D^{m-1} \sigma\biggr)dx
  \\
  &\quad 
  + \int_{\T}\sum_{j = 0}^{m-2} \frac{\partial Q_m}{\partial y_j}(D_{m-2} u) D^j \sigma dx
  \label{eq:sigma:terms:Im:Ito}
\end{split}
\end{align}
Therefore combining  \eqref{eq:my:big:phat:Ito:comp} with observations 
\eqref{eq:gamma:terms:Im:Ito} and \eqref{eq:f:terms:Im:Ito} we conclude:
\begin{Lem}\label{lem:Im:evolution}
  Given $m\geq1$, fix any $u_0, f \in H^{m\vee 2}$ and 
  $\sigma \in \bH^{m\vee 2}$.  Then the associated solution to
  \eqref{eq:s:KdV} as in \cref{prop:exist:uniq} satisfies
  \begin{align}
    d\In_m(u)+ 2 \gam\In_m(u)dt
    =\Kn_m^D dt + \K_m^S dW,
    \label{eq:In:m:Evol}
 \end{align}
 where recall that $\In_m$ are given in \eqref{energy:funct}
 while
 \begin{align}
   &\Kn_m^D(u,f,\sigma) \notag\\
   &\ :=   {{}-}\gamma \int_{\T} \left( 
       \sum_{j = 0}^{m-2} \frac{\partial Q_m}{\partial y_j}(D_{m-2} u) D^j u 
      -2Q_m(D_{m-2} u) - \alpha_m u (D^{m-1} u)^2 \right)dx
               \notag\\
             &\ \quad +  \int_{\T} \biggl(|D^m\sigma|^2 - 2\alpha_m D^{m-1} u  D^{m-1}
               \sigma \cdot \sigma 
    - \alpha_m u |D^{m-1} \sigma|^2 
               + {{}\frac{1}2}\sum_{j,k =0}^{m-2}
               \frac{\partial^2 Q_m}{\partial y_j\partial y_k}(D_{m-2} u) 
    D^j \sigma \cdot D^k \sigma \biggr)  dx
     \notag\\
   &\ \quad +\int_{\T} \biggl(  2 D^m u D^mf -  \alpha_m  (D^{m-1} u)^2 f
  - 2u D^{m-1} u D^{m-1} f
  + \sum_{j = 0}^{m-2} \frac{\partial Q_m}{\partial y_j}(D_{m-2} u)
     D^j f \biggr)dx
    \notag\\
   &\ := \gamma \Kn_{m,1}^D(u)+ \Kn_{m,2}^D(u,\sigma) + \Kn_{m,3}^D(u,f),
     \label{eq:Im:LHS:mess:D}
 \end{align}
 and 
 \begin{align}
 \begin{split}
   \Kn_m^S(u, \sigma)
   &:=  \int_{\T} \biggl(  2 D^m u D^m\sigma -  \alpha_m  (D^{m-1} u)^2 \sigma
   - 2u D^{m-1} u D^{m-1} \sigma\biggr)dx \\
   &\quad
   + \int_{\T} \sum_{j = 0}^{m-2}
   \frac{\partial Q_m}{\partial y_j}(D_{m-2} u) D^j \sigma dx.
     \label{eq:Im:LHS:mess:S}
 \end{split}
 \end{align}
 
\end{Lem}
\noindent Finally for the modified integrals of motion, $\In_m^+$,
defined as in \eqref{energy:funct:0:1}, \eqref{energy:funct}, and
\eqref{eq:mod:int:m:simple}, we have the following evolution which we
obtain by simply combining \eqref{eq:L2:evol:p} with
\eqref{eq:In:m:Evol}.
\begin{Lem}\label{lem:Im:plus:evolution}
  Given $m\geq1$, fix any $u_0, f \in H^{m\vee 2}$ and 
  $\sigma \in \bH^{m\vee 2}$. Then the associated solution to
  \eqref{eq:s:KdV} as in \cref{prop:exist:uniq} satisfies
  \begin{align}
    d \In_{m}^+(u) + 2 \gamma \In_{m}^+(u) dt = \Kn_m^{D,+} dt + \K_m^{S,+} dW,
    \label{eq:Inplus:simp}
  \end{align}
  where $\In_{m}^+$ are the functionals defined in \eqref{eq:mod:int:m:simple}
  \begin{align}
    \Kn_m^{D,+} &:= \, \Kn_m^{D}
    + {\abarm} 2 (1 - {\bar{q}_m})  \gamma\left( \|u\|_{L^2}^2+1\right)^{\bar{q}_m}+ {\abarm}
   {\bar{q}_m}\left( \|u\|_{L^2}^2+1\right)^{\bar{q}_m-1}(2\gam+2 \lb f , u\rb + \|\sigma\|^2_{L^2})   
                   \notag\\
   &\quad + {\abarm}2{\bar{q}_m}({\bar{q}_m}-1)\left( \|u\|_{L^2}^2+1\right)^{\bar{q}_m-2}
     |\lb \sigma, u \rb|^2,
    \label{eq:Inp:simp:RHS:D}
  \end{align}
  recalling the definition of $\Kn_m^{D}$ in \eqref{eq:Im:LHS:mess:D} and
  that the constants $\bar{q}_m$, ${\abarm}$ are defined in
  \eqref{eq:mod:int:m:simple}. On the other hand $\K_m^{S,+}$ is
  given as
  \begin{align}
    \Kn_m^{S,+} :=  \Kn_m^{S} +  
    2 {\abarm} \bar{q}_m \left(\|u\|_{L^2}^2+1\right)^{\bar{q}_m-1}\lb u , \sigma \rb,
    \label{eq:Inp:simp:RHS:S}
  \end{align}
  referring back to \eqref{eq:Im:LHS:mess:S} for the definition of $\Kn_m^{S}$.
  Furthermore, for any $p > 0$
  \begin{align}
  \begin{split}
    d \In_{m}^+(u)^p + 2 \gamma p \In_{m}^+(u)^p dt
    &= \left( p \In_{m}^+(u)^{p-1} \Kn_m^{D,+} 
       + \frac{p (p-1)}{2}\In_{m}^+(u)^{p-2} |\Kn_m^{S,+}|^2 \right)dt \\
      &\quad + p \In_{m}^+(u)^{p-1} \Kn_m^{S,+} dW.
    \label{eq:Imp:mom:evol}
  \end{split}
  \end{align}
\end{Lem}

\subsection{Proof of \cref{thm:Lyapunov}}
\label{sect:Poly:bnd}

We turn now to the proof of \cref{thm:Lyapunov} which is divided into
  the subsections addressing the algebraic moment bounds \eqref{eq:Lyapunov:Bnd}, martingale estimates
  \eqref{eq:poly:est:asymptotic}, and finally the
  exponential moment bounds  \eqref{eq:Lyapunov:Bnd:exp} and \eqref{eq:Lyapunov:Bnd:subquadexp}.

\subsubsection{Algebraic Moment Bounds:
  Proof of \eqref{eq:Lyapunov:Bnd}}
\label{sect:Alg:mom:bnd}

First we consider the case $m=0$. We proceed from
\cref{lem:L2:evolution}. In particular, given any $p>0$, from
\eqref{eq:L2:evol:p}, we apply the Cauchy-Schwarz and Young's
inequalities to arrive at
\begin{align}\label{eq:L2:estimate1}
  d\left(\Sob{u}{L^2}^2+1\right)^p&+p\gam\left(\Sob{u}{L^2}^2+1\right)^pdt
  \leq C dt +2p\left(\Sob{u}{L^2}^2+1\right)^{p-1}\lb u,\s\rb dW,
\end{align}
where $C = C(\gamma, \|f\|_{L^2}, \|\sigma\|_{L^2}, p)$.  Thus, upon
taking expected values and applying Gronwall's inequality, we obtain
$\E\left(\Sob{u(t)}{L^2}^2+1\right)^p \leq C\left(e^{-\gam p
    t}\left(\Sob{u_0}{L^2}^{2}+1\right)^p+1\right)$ which implies
\eqref{eq:Lyapunov:Bnd} for $m=0$, as claimed.

We now proceed to the proof of \eqref{eq:Lyapunov:Bnd} for $m \geq
1$. This entails estimating each of the drift terms appearing in
\eqref{eq:Imp:mom:evol} via repeated applications of the monomial
bounds \eqref{eq:momomial:est:1}.  Referring back to the definition of
$\Kn_m^{D}$ as in \eqref{eq:Im:LHS:mess:D} through $\Kn_m^{D,+}$ in
\eqref{eq:Inp:simp:RHS:D}, we now estimate each of the terms
$\Kn_{m,1}^{D}$, $\Kn_{m,2}^{D}$, $\Kn_{m,3}^{D}$ and finally
$\Kn_{m}^{D,+}$ in succession.  Starting with $\Kn_{m,1}^{D}$, we
estimate the first set of terms appearing there by observing from
\cref{thm:KdV:Poly:Rank} that ${{}(\bdy_{y_j}Q_m(y))y_j}$ consists of
monomials of rank $m+2$ for each $j = 0, \ldots, m-2$ and then apply
\cref{prop:mono:intp:bnd}. The second term in $\Kn_{m,1}^{D}$ is
treated similarly while the final term can be handled directly with
interpolation {{}as in the proof of \cref{lem:equivalence}}.  We infer
\begin{align}
  |\Kn_{m,1}^{D}(u)|
  \leq C \sum_{k =1}^{M} \|D^m u\|^{r_k}_{L^2} \| u \|^{q_k}_{L^2}.
    \label{eq:Kn1:est}
\end{align}
Turning to $\Kn_{m,2}^{D}$, we observe that the terms involving the
Hessian of $Q_m$ consist of monomials of rank less than $m+2$ so that
we can once again invoke \cref{prop:mono:intp:bnd} to estimate
them. The remaining terms are bounded directly with
interpolation. With these observations we conclude
\begin{align}
  |\Kn_{m,2}^{D}(u, \sigma)| 
  &\leq C\sum_{k =1}^{M} \|D^m u\|^{r_k}_{L^2} \| u \|^{q_k}_{L^2}.
    \label{eq:Kn2:est}
\end{align}
Regarding $\Kn_{m,3}^{D}$ arguing in a similar fashion, we find
\begin{align}
  |\Kn_{m,3}^{D}(u,f)|
  \leq C\sum_{k =1}^{M} \|D^m u\|^{r_k}_{L^2} \| u \|^{q_k}_{L^2}.
    \label{eq:Kn3:est}
\end{align}
The remaining terms in $\Kn_m^{D,+}$, which depend on $u$ only through
$L^2$--norms, are estimated as
\begin{align}
  | \Kn_{m}^{D,+} -  \Kn_{m}^{D}|
  &\leq C \bar{q}_m \abarm^2 (1 + \|u\|^2_{L^2})^{\bar{q}_m -1/2},
  \label{eq:Kn:extra:est}
\end{align}
where we note our standing assumption that $\bar{q}_m \geq 1$.  {{}Indeed, note that the first term in $\Kn_{m}^{D,+} -  \Kn_{m}^{D}$ may be dropped in obtaining an upper bound since it is non-positive.}
Combining the estimates \eqref{eq:Kn1:est}--\eqref{eq:Kn:extra:est}
we conclude
\begin{align}
   |\Kn_{m}^{D,+}|
  \leq C\sum_{k =1}^{M} \|D^m u\|^{r_k}_{L^2} \| u \|^{q_k}_{L^2} +
  C \bar{q}_m \abarm^2 (1 + \|u\|^2_{L^2})^{\bar{q}_m -1/2},
   \label{eq:Kn:summary:0}
\end{align}
where, to emphasize, $C = C(m,\gamma, \|f\|_{H^m}, \|\sigma\|_{H^m})$,
$r_k = r_k(m) \in [0,2)$, $q_k = q_k(m) \geq 0$ and $M= M(m) \geq 1$
are all constants independent of $\bar{q}_m, \abarm \geq 1$.  Thus,
by applying Young's inequality and now selecting
$\bar{q}_m, \abarm$ appropriately large we finally conclude
\begin{align}
   |\Kn_{m}^{D,+}|
   \leq C+ {{}\frac{\gamma}{100}} \In_m^+(u),
   \label{eq:Kn:summary}
\end{align}
with now
$C = C(m, \gamma, \|f\|_{H^m}, \|\sigma\|_{H^m}, \bar{q}_m, \abarm)$.

Next, we estimate $|\Kn_m^{S,+}|^2$ in \eqref{eq:Imp:mom:evol}, which
is defined as in \eqref{eq:Inp:simp:RHS:S} and
\eqref{eq:Im:LHS:mess:S}. Starting with $\Kn_m^{S}$, which is the
portion of the stochastic term from $\In_m$, we estimate
 \begin{align}
   |\Kn_m^{S}| 
   \leq C \sum_{k =1}^{M} \|D^m u\|^{r_k}_{L^2} \|u\|_{L^2}^{q_k},
   \label{eq:KnS:1}
 \end{align}
 with appropriate direct use of interpolation and by again invoking
 \cref{thm:KdV:Poly:Rank}, and \cref{prop:mono:intp:bnd} to estimate
 the final terms. Here, $C = C(m, \|\sigma\|_{H^m})$,
 $M = M(m) \geq 1$, $r_k = r_k(m) \in (0,2)$ and
 $q_k = q_k(m) \geq 0$.  On the other hand, from
 \eqref{eq:Inp:simp:RHS:S} we have
 \begin{align}
   |\Kn_m^{S, +} -\Kn_m^{S}|
   \leq 2 \|\sigma\|_{L^2} \bar{q}_m \abarm (1 + \|u\|^2_{L^2})^{\bar{q}_m -1/2},
      \label{eq:KnS:2}
 \end{align}
 Thus by combining \eqref{eq:KnS:1}, \eqref{eq:KnS:2}, {{}as well applying \cref{lem:equivalence}}, then again noting
 carefully that $C$, $M$, $r_k$, $q_k$ in \eqref{eq:KnS:1} are
 independent of $\bar{q}_m, \abarm \geq 1$ we conclude that
\begin{align}
  \frac{p -1}{2}|\Kn_m^{S,+}|^2 
  &\leq C
      + {{}\frac{\gamma}{100}}\In_m^{+}(u)^2,
  \label{eq:KnSplus:squared}
\end{align}
where now $C = C(\gamma,m, p, \|\sigma\|_{H^m},\bar{q}_m, \abarm)$.
 
Upon combining \eqref{eq:Kn:summary} and \eqref{eq:KnSplus:squared}
with \eqref{eq:Imp:mom:evol}, then applying Young's inequality again,
we obtain after rearranging the stochastic differential inequality
\begin{align}
  d \In_m^+(u)^p + \frac{3\gam p}{2}\In_m^+(u)^pdt 
  \leq Cdt
  +p\In_m^+(u)^{p-1}\Kn_m^{S,+}dW.
  \label{eq:final:bound:Poly:moments}
\end{align}
Taking the expected value of both sides and applying Gronwall's inequality
yields \eqref{eq:Lyapunov:Bnd} completing the proof of the first item in \cref{thm:Lyapunov}.

\subsubsection{Martingale estimates:
  Proof of \eqref{eq:poly:est:asymptotic}}

For any $m \geq 1$ we find from \eqref{eq:final:bound:Poly:moments} that, for any $R,T,p > 0$,
\begin{align}
  \Prb&\left(\sup_{t\geq T}\left(\In_m^+(u(t))^p
    +\gam p\int_0^t\In_m^+(u(s))^pds-\In_m^+(u_0)^p-C (t+1)\right)\geq R\right) \notag \\ 
      &\qquad\qquad\qquad\qquad\qquad\qquad\qquad\qquad\qquad\qquad\qquad
        \leq  \Prb\left( \sup_{t \geq T} (\mathcal{M}(t) - t -2) \geq R\right),
    \label{eq:poly:MG:bnd:setup}
\end{align}
where we take
\begin{align}\label{def:Implus:power:mart}
  \Mt(t) :=p\int_0^t\In_m^+(u)^{p-1}\Kn_m^{S,+}dW,
\end{align}
with $\Kn_m^{S,+}$ is given by \eqref{eq:Inp:simp:RHS:S}.
Here, note that the constant
$C = C(\gamma,m, p, \|f\|_{H^m}, \|\sigma\|_{H^m},\bar{q}_m, \abarm)$
appearing in \eqref{eq:poly:MG:bnd:setup} is independent
of $T, R$ with \eqref{eq:poly:MG:bnd:setup} holding for any appropriately large values of
$\bar{q}_m$ and $\abarm$ for which \eqref{eq:final:bound:Poly:moments} is valid. 
Observe that for any $T\geq0$, $R > 0$
\begin{align}
  \left\{\sup_{t\geq T}\left(\Mt(t)- t -2\right)\geq R\right\}
  \subseteq \bigcup_{n\geq [T]}\left\{\sup_{t\in[n,n+1)}\left(\Mt(t)- t -2\right)\geq R\right\},
  \label{eq:T:split:M}
\end{align}
where $[T]$ denotes the largest integer $\leq T$.
On the other hand, notice that for $R>0$ and any $n \geq 0$,
\begin{align}\label{eq:sup:Mtstar}
  \left\{\sup_{t\in[n,n+1)}\left(\Mt(t)- t-2\right)\geq R\right\}\subseteq\left\{\Mt^*(n+1)\geq R+ n +2\right\},
\end{align}
where we adopt the notation $\Mt^*(t) := \sup_{s \in [0,t]} |\Mt(s)|$.

In order to obtain a suitable estimate for
\eqref{eq:poly:MG:bnd:setup} from 
\eqref{eq:T:split:M} and \eqref{eq:sup:Mtstar}, let us recall that the Burkholder-Davis-Gundy
inequality states
\begin{align}
  \E (\mathcal{N}^*(\tau)^q) \leq c \E ([\mathcal{N}](\tau)^{q/2}),
  \label{eq:BDG}
\end{align}
for any continuous, (locally) square-integrable, martingale
$\{\mathcal{N}(t)\}_{t \geq 0}$, any $q > 0$, and any stopping time
$\tau$.  Here, $c$ is a universal quantity depending only on $q > 0$
and $\mathcal{N}^*(\tau)$, $[\mathcal{N}](\tau)$ denote the
maximum and the quadratic variation of $\mathcal{N}$ up to
time $\tau \geq 0$, respectively. Note that upon writing
$\mathcal{N}(t) = \int_0^t g dW$ for its appropriate
$\mathcal{F}_t$-adapted representative $g$, we have
\begin{align}
  \mathcal{N}^*(\tau) := \sup_{t \in [0,\tau]} \left| \int_0^t g dW \right|,
  \qquad
  [\mathcal{N}](\tau) :=  \int_0^\tau |g|^2 ds.
  \label{eq:MG:Note}
\end{align}
See e.g. \cite[Theorem 3.28]{karatzas2014brownian} for further details.

With \eqref{eq:BDG} in hand, we proceed as follows. From
\eqref{def:Implus:power:mart}, \eqref{eq:MG:Note}, notice that for
any $t \geq 0$
\begin{align}
  [\Mt](t) = p^2\int_0^t\In_m^+(u(s))^{2p-2}|\Kn_m^{S,+}|^2ds.
  \label{eq:Imp:mart:QV}
\end{align}
Observe that from \eqref{eq:KnSplus:squared} and the fact that $\In_m^+\geq 1$, we
have $|\Kn_m^{S,+}|^2\leq C\In_m^+(u(s))^2$. Thus from
\eqref{eq:BDG}, H\"older's inequality and \eqref{eq:Lyapunov:Bnd}
already established above we find that, for all $q\geq2$,
\begin{align}   
  \E\Mt^*(t)^q
  \leq C\E\left(\int_0^t\In_m^{+}(u(s))^{2p}ds\right)^{q/2}
  \leq Ct^{(q-2)/2}\E\int_0^t\In_m^+(u(s))^{pq}ds \leq C(t+1)^{q/2}\In_m^+(u_0)^{pq}.
  \label{eq:Mtstar:bnd}
\end{align}
Next from \eqref{eq:T:split:M} followed by
\eqref{eq:sup:Mtstar}, Chebyshev's inequality, and \eqref{eq:Mtstar:bnd},
we have, for any $q \geq 2$
\begin{align}
  \Prb \left\{\sup_{t\geq T}\left(\Mt(t)- t-2\right)\geq R\right\}
  &\leq \sum_{n\geq[T]}\Prb\left\{\Mt^*(n+1)\geq R+ n +2\right\}
    \leq \sum_{n\geq[T]}\frac{\E\Mt^*(n+1)^q}{(R+ n +2)^q}\notag\\
  &\leq C\In_m^+(u_0)^{pq} \sum_{n\geq[T]}\frac{1}{(R+ n +2 )^{q/2}}
  \leq \frac{C \In_m^+(u_0)^{pq}}{(T + R)^{q/2-1}},\notag
\end{align}
where we once again emphasize that $C$ depends on
$m, p,q,\gam,\Sob{f}{H^m},\Sob{\s}{H^m}, \bar{q}_m, \abarm$ but is
independent of $T, R>0$.  We thus have established establish
\eqref{eq:poly:est:asymptotic} in the case $m \geq 1$.  The proof for
$m=0$ follows directly from \eqref{eq:L2:estimate1} by arguing along
precisely the same lines as we just detailed for $m \geq 1$.  The 
proof of \cref{thm:Lyapunov}, $(ii)$ is therefore complete.

\subsubsection{Exponential Moment Bounds:
  Proofs of \eqref{eq:exp:est:asymptotic}, \eqref{eq:Lyapunov:Bnd:exp},
  and \eqref{eq:Lyapunov:Bnd:subquadexp} }
  \label{sect:exp:mom:bnd}

For the proof of the final items $(iii)$ and $(iv)$ in \cref{thm:Lyapunov}, we will make
repeated use of the following the exponential martingale inequality: Given a
continuous, (locally) square-integrable martingale,
$\{\mathcal{N}(t)\}_{t \geq 0}$, we have for any $\eta>0$ and $R>0$
\begin{align}\label{eq:exp:Mt:gen}
  \Prb\left(\sup_{t\geq0}
  \left(\mathcal{N}(t) -
     \frac{\eta}2[\mathcal{N}](t)\right)
  \geq  R\right)\leq e^{-\eta R},
\end{align}
where $[\mathcal{N}]$ is the quadratic
variation of $\mathcal{N}$ as in \eqref{eq:MG:Note}. See
e.g. \cite[Chapter 6, Section 8]{krylov2002introduction} or
\cite[Proposition 3.1]{GlattHoltz2014} for a proof of this bound.

Regarding the proof of \eqref{eq:exp:est:asymptotic}, we return
again to \eqref{eq:final:bound:Poly:moments}, integrate in time,
then invoke \eqref{eq:exp:Mt:gen}.  We find that for any $\bar{\eta}, R > 0$
\begin{align}
  \Prb&\left(\sup_{t \geq 0}\left(\In_m^+(u(t))^p
        +\frac{3\gam p}{2}\int_0^t\In_m^+(u(s))^pds-\In_m^+(u_0)^p-C t
        - \frac{\bar{\eta}}{2} [\Mt](t)\right)\geq R\right)
        \notag \\ 
      &\qquad\qquad\qquad\qquad\qquad\qquad\qquad\qquad\qquad
        \leq  \Prb\left( \sup_{t \geq 0}
            (\Mt(t) - \frac{\bar{\eta}}{2} [\Mt](t)) \geq R\right)
        \leq e^{-\bar{\eta} R},
    \label{eq:poly:MG:bnd:setup:1}
\end{align}
where $\Mt$ is the martingale defined in \eqref{def:Implus:power:mart}
and $[\Mt]$ is the quadratic variation of $\Mt$ given explicitly as
\eqref{eq:Imp:mart:QV}.  Returning to \eqref{eq:KnS:1}, notice that, so long as $\bar{q}_m \geq 1$ in the
definition of $\In_m^+$ is sufficiently large, we can find
$\delta = \delta(m, \|\sigma\|_{H^m}, \bar{q}_m)$ with $\delta \in (0,1)$
such that
\begin{align}
  |\Kn_m^S| \leq C(\In_m^+)^\delta.
  \label{eq:careful:dlt}
\end{align}
Thus, for any $p > 0$
\begin{align}\notag
  [\Mt](t) \leq C \int_0^t\In_m^+(u(s))^{2p-2 + 2\delta}ds.
\end{align}
Note carefully that the value of $\delta \in (0,1)$ in
\eqref{eq:careful:dlt} is independent of $p$. {{}Thus, we may choose $p_0:=1-\de$. Since $\In_m^+\geq1$, it follows that for all $p\in(0,p_0)$}
\begin{align}\label{eq:careful:dlt:mt}
    {{}[\Mt](t)\leq C\int_0^t\In_m^+(u(s))^{2p_0-2+2\de}ds=Ct.}
\end{align}

Combining this observation with \eqref{eq:poly:MG:bnd:setup:1} we  conclude that \eqref{eq:exp:est:asymptotic} holds.

{{}We remark} that the proof of \eqref{eq:exp:est:asymptotic:L2} follows
closely along the lines we just outlined to obtain \eqref{eq:exp:est:asymptotic}
by using \eqref{eq:L2:evol}
with  \eqref{eq:exp:Mt:gen}. {{}In particular, we apply \eqref{eq:exp:Mt:gen} with $\mathcal{N}(t):=2\int_0^t\lb u(t),\s\rb dW$ and $\eta:=\frac{\gam}4\Sob{\s}{L^2}^{-2}$. Indeed, using \eqref{eq:L2:evol} we get
\begin{align*}
    [\mathcal{N}](t)=\int_0^t\sum_{k\geq1}4|\lb u(s),\s_k\rb|^2ds\leq \int_0^t4\Sob{\s}{L^2}^2\Sob{u(s)}{L^2}^2ds.
\end{align*}
Also
\begin{align*}
    \mathcal{N}(t)\geq\Sob{u(t)}{L^2}^2-\Sob{u_0}{L^2}^2+2\gam\int_0^t\Sob{u(s)}{L^2}^2ds-\frac{\gam}2\int_0^t\Sob{u(s)}{L^2}^2ds-\frac{2}\gam\Sob{f}{L^2}^2t-\Sob{\s}{L^2}^2t.
\end{align*}
Thus
\begin{align*}
    \mathcal{N}(t)-\frac{\eta}2[\mathcal{N}](t)\geq\Sob{u(t)}{L^2}^2-\Sob{u_0}{L^2}^2+\gam\int_0^t\Sob{u(s)}{L^2}^2ds-\frac{2}\gam\Sob{f}{L^2}^2t-\Sob{s}{L^2}^2,
\end{align*}
which together with \eqref{eq:exp:Mt:gen} implies \eqref{eq:exp:est:asymptotic}.
}

Turning to the proof of \eqref{eq:Lyapunov:Bnd:exp} we proceed from
\eqref{eq:exp:est:asymptotic} by invoking the elementary identity
\begin{align}
\label{eq:elem}
    \E X = \int_{0}^{\infty}\Prb( X \geq y) d y,
\end{align}
valid for any non-negative random variable $X \in L^1(\Omega)$.
Specifically, given $\eta >0$ in
\eqref{eq:Lyapunov:Bnd:exp}, we take
\begin{align*}
  X := \exp\left( {{}\eta} \sup_{t \geq 0}\left(\In_m^+(u(t))^p
        +\gam p\int_0^t\In_m^+(u(s))^pds-\In_m^+(u_0)^p-C t\right) \right),
\end{align*}
where we choose $C$ suitably large so that
\eqref{eq:exp:est:asymptotic} holds for this $C$ when
$\bar{\eta} = 2 \eta$.  For this choice clearly
\eqref{eq:exp:est:asymptotic} implies that
$\int_0^\infty \Prb(X \geq y )dy< \infty$, 
which thus immediately yields \eqref{eq:Lyapunov:Bnd:exp}
from \eqref{eq:elem}.

For the final bound, \eqref{eq:Lyapunov:Bnd:subquadexp}, observe that
\eqref{eq:Imp:mom:evol} with \eqref{eq:Kn:summary} and \eqref{eq:KnSplus:squared}
yield, for any $T, t, \bar{\eta} > 0$,
\begin{align}
  e^{-\gam p (T-t)}\In_m^+(u(t))^p
  +\frac{\gam p}{2}\int_0^te^{-\gam p(T-s)}\In_m^+(u(s))^p ds 
  &- \In_m^+(u_0)e^{-\gam pT}- C\int_0^te^{-\gam p(T-s)} ds
    - \frac{\bar{\eta}}{2} [\til{\Mt}](t) \notag\\
    &\leq \til{\Mt}(t) - \frac{\bar{\eta}}{2} [\til{\Mt}](t),
    \label{eq:next:bound:exp:moments}
\end{align}
where
\begin{align}\notag
  \til{\Mt}(t):=p\int_0^te^{-\gam p(T-s)}\In_m^+(u)^{p-1}\Kn_m^{S,+}dW,
  \quad
  [\til{\Mt}](t) =  p^2 \int_0^te^{-2\gam p(T-s)} \In_{m}^+(u)^{2(p-1)} |\Kn_m^{S,+}|^2 ds.
\end{align}
Note carefully that, for each fixed value of $T > 0$,  $\{\til{\Mt}(t)\}_{t \geq 0}$ is a continuous
square-integrable martingale.
We estimate the quadratic variation $[\til{\Mt}]$ of $\til{\Mt}$ using
\eqref{eq:careful:dlt} and {{}infer similarly to \eqref{eq:careful:dlt:mt} that, for any $0 < p < p_0$ with
$p_0=1-\de$ as in the previous item, we have} 
\begin{align}\label{eq:mg:QV:Imp}
    {{}[\til{\Mt}](t)
      \leq C \int_0^te^{-2\gam p(T-s)}\In_{m}^+(u)^{2p-2+2\de}ds\leq C,
      }
\end{align}
where once again $C = C(\bar{\eta}, m, p, \gam,\Sob{f}{H^m},\Sob{\s}{H^m}, \bar{q}_m, \abarm)$
is independent of $T, t > 0$.

Combining \eqref{eq:mg:QV:Imp} with \eqref{eq:next:bound:exp:moments}
and then using \eqref{eq:exp:Mt:gen} we infer that, for every
$R > 0$,
\begin{align}
  \Prb( \In_m^+(u(T))^p - \In_m^+(u_0)^p e^{-\gam p T}- C \geq R) 
  \leq \Prb \left( \sup_{t \geq 0} (\til{\Mt}(t)
  - \frac{\bar{\eta}}{2} [\til{\Mt}](t)) \geq R \right)
  \leq e^{-\bar{\eta} R},
  \label{eq:final:max:lvl}
\end{align}
for a suitable constant
$C= C(\bar{\eta}, m, p, \gam,\Sob{f}{H^m},\Sob{\s}{H^m}, \bar{q}_m,
\abarm)$ which is independent of $T > 0$. We now apply
\eqref{eq:final:max:lvl} in conjunction with \eqref{eq:elem} to infer
\eqref{eq:Lyapunov:Bnd:subquadexp}.  Specifically we take
$X := \exp(\In_m^+(u(T))^p - \In_m^+(u_0)^pe^{-\gam p T}- C)$ where we
choose the constant $C$ so that \eqref{eq:final:max:lvl} holds for
$\bar{\eta} = 2 \eta$ where $\eta$ is the desired rate specified in
\eqref{eq:Lyapunov:Bnd:subquadexp}.  The proof of this final item $(iv)$ in \cref{thm:Lyapunov} is
now complete.

\section{Foias-Prodi Estimates}
\label{sect:FP:est}

This section is devoted to establishing a Foias-Prodi type estimate
for \eqref{eq:s:KdV}, which we previewed in \cref{thm:FP:est:Intro}.
The complete and precise form is provided here in a sequence of results from \cref{thm:FP:est} to \cref{cor:FP:est:Intro}.  These results, which will serve as crucial components to our arguments establishing
regularity and uniqueness properties of invariant measures in \cref{sect:Regularity} and \cref{sect:UniqueErgodicity},
reveal that
\eqref{eq:s:KdV} possesses a nonlinear mechanism which enables an asymptotic coupling of
solutions by means of a finite-dimensional control.
Furthermore, as noted in the introduction and detailed in
\cref{sect:det:case}, the approach we lay
out in this section will yield novel and compact proofs of some existing
results in the deterministic setting.

We now describe more precisely our choice for the control.  Given
$f\in H^2$, $\s\in\bH^2$, and $u_0\in H^2$, let $u(u_0, f,\s)$ denote
the corresponding solution of \eqref{eq:s:KdV}. Given $\lam,N>0$ and
$v_0\in H^2$, let $v(v_0,u_0,f,\s)$ denote the corresponding solution
of 
\begin{align}
  dv
  + ( v {D} v +  {D}^3 v + \gam v) dt
  =  f dt + \sigma dW-\lambda P_N(v-u)dt,  \quad v(0) = v_0,
     \label{eq:s:KdV:ng}
\end{align}
where $P_N$ denotes the projection onto Fourier modes $|k|\leq N$. We
will often refer to \eqref{eq:s:KdV:ng} as the \textit{nudged system
  corresponding to \eqref{eq:s:KdV}} or simply the \textit{nudged
  system}.  Note the well-posedness results in \cref{prop:exist:uniq}
for \eqref{eq:s:KdV} can be trivially extended to \eqref{eq:s:KdV:ng} by replacing $f$ by $f+\lam P_Nu$ and observing that $\lam P_Nv$ does not crucially impact the well-posedness estimates. {{}We refer the reader to the recent work \cite{JollySadigovTiti2017}, where well-posedness of \eqref{eq:s:KdV:ng} is treated in the deterministic setting; the energy-based approach present there is adaptable to the stochastic setting which is parsed out in \cref{sect:apx:wp:SKdV} for the case of \eqref{eq:s:KdV}.}

The desired Foias-Prodi type estimate is captured by the following
theorem.

\begin{Thm}
  \label{thm:FP:est}
  Let $f \in H^2$ and $\sigma \in \bH^2$ and $\gamma > 0$.  Then there
  exists $\lam_0\geq1$, depending only
  $\gam,\Sob{f}{H^1},\Sob{\s}{L^2}$, and $N_0=N_0(\gam,\lam)$ such that for any $\lam\geq\lam_0$, $N\geq N_0$, and
  $u_0, v_0 \in H^2$
  \begin{align}\label{eq:FP1}
    \begin{split}
   \E \biggl(&
    \exp\biggl( \gamma (t \wedge \tau) - 
    \frac{c_0}{\lambda}
    \int_0^{t\wedge \tau} (1 + \|u(s)\|_{H^2}^2
              + \frac{1}{\lambda}\| v(s)\|^2_{L^2} )ds
        \biggr)
        \|u(t\wedge\tau) - v(t\wedge\tau) \|_{H^1}^2   \\
     &+ \gamma
    \int_0^{t\wedge \tau} \exp\biggl( \gamma s - 
    \frac{c_0}{\lambda} \int_0^s (1 + \|u(t')\|_{H^2}^2 
                           + \frac{1}{\lambda}\| v(t')\|^2_{L^2} )dt'
    \biggr)\|u(s) - v(s) \|_{H^1}^2 ds
    \biggr)\\
    &\qquad\qquad \leq C \left(\|u_0 - v_0\|_{H^1}^2
      +\Sob{u_0-v_0}{L^2}^{10/3}+(1+\Sob{u_0}{L^2}^2)\Sob{u_0-v_0}{L^2}^2\right),
    \end{split}
  \end{align}
  which holds for any stopping time $\tau \geq 0$ and any $t \geq 0$. In particular, $\lam_0, N_0$ may be chosen to satisfy \eqref{cond:lam0:N0}.
  Here, $u = u(u_0, f, \sigma)$ and $v = v(v_0, u_0, f, \sigma)$ obey
  \eqref{eq:s:KdV} and \eqref{eq:s:KdV:ng} respectively. The constant
  $c_0 > 0$ depends only on universal quantities and specifically is
  independent are of $\lam, N$ whereas
  $C = C(\gamma, \|f\|_{H^1}, \|\sigma\|_{L^2}, \lambda)$ so that both
  $c_0, C$ are independent of $u_0, v_0, t, \tau$.
\end{Thm}
\begin{Rmk}
\label{rmk:N:dep:FP}
  Examining the proof of \cref{thm:FP:est}, cf.
  \eqref{eq:param:tune:1}-\eqref{eq:param:tune:3}, we may select
  any
  \begin{align}\label{cond:lam0:N0}
    \lambda \geq  c(\|f\|_{H^1}+ \|\sigma\|_{L^2}^2), \quad
    N \geq c\max\{\gamma, \gamma^{-1/3}\}\lambda^{5/2} ,
  \end{align}
  for a universal constant $c>0$.  {{}Based on the manner in which the analysis is performed, do not expect the scaling}
  $N \sim \max\{\gamma, \gamma^{-1/3}\}(\|f\|_{H^1}+ \|\sigma\|_{L^2}^2)^{5/2}$
  {{}to be optimal. In the setting of large damping, however, the exponents that appear for the lower bound on the damping can be directly linked to the algebraic structure of the conservation laws and, therefore, expected to be optimal. We refer the reader to \cref{rmk:large:damping:scaling} for the explanation in that setting.}
\end{Rmk}

In order to control the integrating factor that appears in \eqref{eq:FP1}, we
make a judicious choice of stopping times.  For $R, \beta > 0$, let
\begin{align}
          \tau_{R,\beta} := \inf_{t \geq 0} 
     \left\{ 
  \frac{c_0}{\lambda}
  \int_0^t (1 + \| u(s)\|_{H^2}^2 + \frac{1}{\lambda}\| v(s)\|^2_{L^2})ds  
          - \frac{\gamma}{2} t  - \beta \geq R
    \right\},
  \label{def:FP:stoppingtime}
\end{align}
where, to emphasize, $c_0 \geq 1$ is precisely the constant appearing
in \eqref{eq:FP1}.  Here, it is convenient to include the extra
parameter $\beta > 0$, which we use to track dependence on initial
conditions $u_0, v_0$ in subsequent estimates on $\tau_{R,\beta}$ in
\cref{lem:FP:stoppingtime}, \cref{cor:FP:est:Intro}.  This explicit
dependence on the initial conditions is used in the proof of
\cref{thm:regularity} below.

{{}With $\tau_{R, \beta}$ in hand, we may immediately draw a corollary to \cref{thm:FP:est} that yields a conditional decay estimate on the difference $u-v$. Indeed, let $\Gamma(t) = \gamma t - \frac{c_0}{\lambda}\int_0^t(1 + \|u(s)\|_{H^2}^2
+ \frac{1}{\lambda}\| v(s)\|_{L^2}^{2})ds$. Then  \eqref{eq:FP1} implies
    \begin{align}\notag
        \E\indFn{\tau_{R,\be}=\infty}\exp\left(\Gamma(t\wedge\tau_{R,\be})\right)\|u(t)-v(t)\|_{H^1}^2\leq         \E\exp\left(\Gamma(t\wedge\tau_{R,\be})\right)\|u(t)-v(t)\|_{H^1}^2.\notag
    \end{align}
Since
    \begin{align}\notag
        \indFn{\tau_{R,\be}=\infty}\exp\left(\Gamma(t\wedge\tau_{R,\be})\right)\geq\indFn{\tau_{R,\be}=\infty}\exp\left(\frac{\gam}2t-(R+\be)\right).
    \end{align}
It then follows that
    \begin{align}\notag
         e^{\frac{\gam}2t-(R+\be)} \E\indFn{\tau_{R,\be}=\infty}\|u(t)-v(t)\|_{H^1}^2\leq\E\indFn{\tau_{R,\be}=\infty}\exp\left(\Gamma(t\wedge\tau_{R,\be})\right)\|u(t)-v(t)\|_{H^1}^2.
    \end{align}
We may then invoke \cref{thm:FP:est} to obtain the following result.
}

\begin{Cor}\label{cor:FP:est}
  Under the conditions of \cref{thm:FP:est}, for any
  $u_0,v_0 \in H^2$ and any $R, \beta \geq 0$
 \begin{align}
   &\E \indFn{\tau_{R,\beta} = \infty } \|u(t)-v(t) \|_{H^1}^2
   \notag\\
    & \leq C\exp\left({{}R + \beta}
      - \frac{\gamma}{2} t\right)
    \left(\|u_0-v_0\|_{H^1}^2+\Sob{u_0-v_0}{L^2}^{10/3}
      +(1+\Sob{u_0}{L^2}^2)\Sob{u_0-v_0}{L^2}^2\right),
       \label{eq:FP2}
 \end{align}
 for $t \geq 0$.  Here, $u = u(u_0,f, \sigma), v(v_0, u_0, f, \sigma)$
 and note that
  $c_0, C > 0$ are the
 constants appearing in \eqref{eq:FP1}.
\end{Cor}

In order to draw useful conclusions from \cref{cor:FP:est} and
ultimately to deduce the convergence we overviewed in
\cref{thm:FP:est:Intro} above in the introduction, we must show
that \eqref{eq:s:KdV:ng} has some suitable Lyapunov structure.
Indeed, since \eqref{eq:s:KdV:ng} is the same as \eqref{eq:s:KdV}
except for an additional term, we show that the Lyapunov structure
established in \cref{sect:Lyapunov} for \eqref{eq:s:KdV} is largely
preserved by \eqref{eq:s:KdV:ng}, but with upper bounds depending on
$\lam$. This is quantified precisely by the following theorem:
\begin{Prop}
  \label{thm:Lyapunov:ng}
  Fix any  $\gamma >0$, $f \in H^{ 2}$
  and $\sigma \in \bH^{ 2}$.  Given $u_0,v_0\in H^{2}$ and
  $\lam,N \geq 1$, let $u = u(u_0, f, \sigma)$, $v = v(v_0, u_0, f, \sigma)$ 
  denote the corresponding
  solution of \eqref{eq:s:KdV}, \eqref{eq:s:KdV:ng} respectively.
  \begin{itemize}
  \item[(i)]  Then, for any $R > 0$,
  \begin{align}\label{eq:L2:exp:Mt:ng}
    &\Prb
      \Bigg{(}\sup_{t\geq0}
      \bigg{(}\Sob{v(t)}{L^2}^2+ \gam\int_0^t\Sob{v(s)}{L^2}^2ds
      - \lambda C t\bigg{)}
      \geq\Sob{v_0}{L^2}^2+ \frac{\lambda}{\gamma} \Sob{u_0}{L^2}^2 + R\Bigg{)}
      \leq\exp(-\bar{\eta} R),
  \end{align}
  where $C = C(\gam, \Sob{f}{L^2},\Sob{\s}{L^2})$ is independent of
  $\lambda \geq 1$,
  $\bar{\eta} = {\eta}(\gam,\Sob{\s}{L^2}, \lambda) > 0$ and both
  constants are independent of $u_0, v_0$ and $R$.
  \item[(ii)] Moreover for each $m \geq 0$, assuming furthermore  
  $f, v_0 \in H^{m}$
  and $\sigma \in \bH^{m }$
  \begin{align}
   \label{eq:Lyapunov:Bnd:Hm:ng}
    \E \|v(t)\|_{H^m}^2
   \leq C (e^{-\gam t}  ( \|v_0\|_{H^m}^2  
       + \Sob{v_0}{L^2}^q+\Sob{u_0}{L^2}^{q}+1) + 1),
  \end{align}
  where
  $C = C(m, \gamma, \|f\|_{H^m}, \|\sigma\|_{H^m}, \lambda, N), q =
  q(m) > 0$ are independent of $u_0, v_0$ and of $t \geq 0$.   Here,
  we emphasize that the bound \eqref{eq:Lyapunov:Bnd:Hm:ng}
  \emph{does not} require $u_0 \in H^m$ when $m> 2$.
  \end{itemize}
\end{Prop}
\begin{Rmk}
It is not difficult to obtain higher order moment estimates for
  $\|v\|_{H^m}$ by further exploiting the bounds we developed in
  \cref{sect:Lyapunov}.  However since \eqref{eq:Lyapunov:Bnd:Hm:ng}
  is sufficent for what follows we omit further details.
\end{Rmk}

We now draw the following two crucial but immediate consequences from
\cref{thm:Lyapunov:ng}.  The first is a bound on the stopping times
$\tau_{R,\beta}$ introduced in \eqref{def:FP:stoppingtime}.
\begin{Cor}\label{lem:FP:stoppingtime}
  Fix any $\gam > 0$, $\lam, N \geq 1$ and $f \in H^2$,
  $\sigma \in \bH^2$.  Given $u_0, v_0 \in H^2$, let
  $u = u(u_0, f, \sigma)$ solve \eqref{eq:s:KdV} and
  $v = v(v_0, u_0, f, \sigma)$ obey \eqref{eq:s:KdV:ng}.
  For each $R, \beta > 0$ define $\tau_{R, \beta}$ as in
  \eqref{def:FP:stoppingtime}.  Then, there exists a
  constant $C = C(\gamma, \|f\|_{H^2}, \|\sigma\|_{H^2})$
  independent of $u_0$, $v_0$, such that, so long as
  \begin{align}
    \beta \geq C (1+ \|u_0\|^2_{H^2} + \|u_0\|^{14/3}_{L^2}  + \|v_0\|^2_{L^2}),
    \label{eq:beta:cond:stt}
  \end{align}
  we have, for each $p > 0$ and $\lam$ chosen sufficiently large depending only on $\gam$, that
  \begin{align}\label{eq:FP:stoppingtime:bound}
        \Prb(\tau_{R, \beta} <\infty)\leq 
        \frac{C(1 + \|u_0\|_{H^2}^{q})}{R^p}.
  \end{align}
  Here, the constants
  $C =C(p,\gam,\Sob{f}{H^2},\Sob{\s}{H^2}, \lambda)$, $q = q(p)$ do
  not depend on $R,\beta$, $u_0$ or $v_0$.
\end{Cor}
\noindent The second consequence is one which follows from \eqref{eq:FP2}
and \eqref{eq:FP:stoppingtime:bound}{{}upon making a judicious choice of $t$-dependence for $R$, e.g., $R=(\gam/100)t$; it asserts a precise and more quantitative
restatement of \cref{thm:FP:est:Intro} and is stated as follows}
\begin{Cor}\label{cor:FP:est:Intro}
  Under the conditions of \cref{thm:FP:est} select any $\lambda, N \geq 1$
  in \eqref{eq:s:KdV:ng} so that \eqref{eq:FP1} holds.
  Then for any $u_0,v_0 \in H^2$ and for any  $p > 0$, there exists {{}$q=q(p)>2$} such that
  \begin{align}
    \E\|u(t)-v(t)\|_{H^1} \leq \frac{\exp(C {{}(1+\|u_0\|_{H^2}+\|v_0\|_{H^1})^q)}}{t^p}
    \label{eq:algebraic:decay}
  \end{align} 
  for all $t\geq0$, where $u = u(u_0,f, \sigma)$ and $v = v(v_0, u_0, f, \sigma)$ 
  obey
  \eqref{eq:s:KdV} and \eqref{eq:s:KdV:ng} respectively.
  Here, the constant $C(p,\gam,\Sob{f}{H^2},\Sob{\sigma}{H^2}, \lambda) > 0$ is
  independent of $u_0$, $v_0$ and $t > 0$.
\end{Cor}

The remainder of \cref{sect:FP:est} is organized as follows. In
\cref{sect:mod:Ham} we introduce and study a functional related to the
$H^1$-norm which we use to establish \cref{thm:FP:est}.  In
\cref{sect:FP:thm:proof} we turn to the proof of \cref{thm:FP:est}.
Finally we prove \cref{thm:Lyapunov:ng} and its consequences
\cref{lem:FP:stoppingtime}, \cref{cor:FP:est:Intro} in
\cref{sect:Lyapunov:ng}.

\subsection{The Modified Hamiltonian Functional}
\label{sect:mod:Ham}

As discussed in the introduction, our proof of \cref{thm:FP:est} makes
crucial use of a `modified Hamiltonian' $\InG_1^+$ (given in
\eqref{eq:FP:functional} below), which is inspired by a functional
introduced previously in \cite{Ghidaglia1988}.  Before proceeding to
define $\InG_1^+$ and to compute its evolution we begin with some
heuristics.  These considerations will guide the estimates carried out
below in \cref{sect:FP:thm:proof}.

Given $u$ and $v$ satisfying \eqref{eq:s:KdV} and \eqref{eq:s:KdV:ng}
respectively and taking $w:=v-u$, we obtain the evolution of the
difference:
\begin{align}\label{eq:s:KdV:diff}
  \bdy_tw+\frac{1}{2} D(w^2)
  +D^3w +D(uw) {{}+} \gam w + \lam P_Nw= 0,
  \quad w(0)=v_0-u_0.
\end{align}
We might begin by computing the
evolution of the $L^2$ norm of $w$ and find
\begin{align}
   \frac{d}{dt}\| w\|^2_{L^2}
  + 2\gamma \| w\|^2_{L^2} + 2\lambda  \|P_N w\|^2_{L^2}
  = - 2\int_{\T} D (uw)  w dx = - \int_{\T} Du w^2 dx.
  \label{eq:w:L2}
\end{align}
{{}
By H\"older's inequality and Sobolev embedding we obtain
    \begin{align*}
        \left|\int_{\T} Du w^2dx\right|\leq \Sob{Du}{L^\infty}\Sob{w}{L^2}^2\leq c\Sob{u}{H^2}\Sob{w}{L^2}^2.
    \end{align*}
Hence
}
\begin{align}
\| w(t) \|^2_{L^2} \leq \|w_0 \|_{L^2}^2 \exp\left( - 2\gamma t + c\int_0^t
\|u\|_{H^2}ds\right).
  \label{eq:w:L2:no:decay}
\end{align}
Here, in contrast to the large damping case where $\gamma \gg 1$ (see
\eqref{eq:stab:main} and \cref{thm:spectralgap} in
\cref{sect:Large:Damping} below), this bound is not anticipated to
produce a desirable decay in the $L^2$ norm; we do not expect that
$- 2\gamma t + c\int_0^t \|u\|_{H^2}^2ds \to -\infty$ as
$t \to \infty$.  On the other hand, in contrast to the case of
dissipative equations such as the Navier-Stokes or Reaction-Diffusion
equations, we cannot apply the generalized Poincar\'e inequality as in
\cite{FoiasProdi1967} to produce an effective large damping on the
high frequencies from the combination of the terms
$2\gamma \| w\|^2_{L^2} + 2\lambda \|P_N w\|^2_{L^2}$ appearing in
\eqref{eq:w:L2}.  See e.g. \cite{GlattHoltzMattinglyRichards2015} for
further general context in an SPDE setting close to our present
considerations.

Inspired by \cite{Ghidaglia1988}, where a partial invariant of the
linearized KdV equation $\partial_t w + D(uw) + D^3w = 0$ was
identified, we now proceed as follows.  In a similar fashion to
\eqref{energy:funct:0:1}, \eqref{I:J:relationship}, we make the
observation that,
\begin{align}
	\int_\TT \left(D^3 w + \frac{1}{2} D(w^2) + D(uw)\right)
	           \left( D^2 w + \frac{1}{2} w^2 + uw \right) dx = 0
	           \label{eq:H1:can}
\end{align}
for \emph{any} suitably smooth $u,w$ and then compare with
\eqref{eq:s:KdV:diff}.  In order to incorporate the cancellation induced by the last term in the parentheses in \eqref{eq:H1:can}, it is therefore natural to introduce the
functional
\begin{align}
   \InG_1(w, u) := \In_1(w) - \int_{\T} u w^2 dx 
   = \int_{\T}\bigl( (Dw)^2 - \frac{1}{3} w^3 -u w^2 \bigr) dx. 
   \label{eq:mod:fn:base}
\end{align}
After some careful calculations, which are provided in detail in
\cref{sect:FP:thm:proof}, we find
\begin{align}
  d &\InG_1(w, u)+ 2 (\gamma \InG_1(w, u)   + \lambda \|D P_N w\|_{L^2}^2 ) dt 
             \label{eq:I1:bnd:FP:3}       \\
  &=\int_{\T}\biggl( \frac{\gamma}{3} w^3 + \lambda P_Nw w^2
    + u D uw^2  + D^3uw^2 + \gam u w^2 + 2u w\lam P_Nw -fw^2
    \biggr)dx dt 
    + \int_{\T}  w^2 \sigma dx dW.
    \notag
\end{align}
Crucially each of the terms on the right-hand side of
\eqref{eq:I1:bnd:FP:3} can be bounded in terms of quantities involving only the $L^2$-norm of $w$.  This structure allows one to take suitable advantage of
the nudging term on the low frequencies and the damping term at high
frequencies via the inverse Poincar\'e inequality.  

To be more concrete about the nonlinear bounds, consider, for example, the term
$\int_{T} D^3 uw^2 dx$ in \eqref{eq:I1:bnd:FP:3}.  Integrating by
parts, using Agmon's inequality, and then using splitting
$w = P_N w + Q_Nw$ into low and high frequencies we find
\begin{align}
	\left| \int_{T} D^3 uw^2 dx \right|
	&= 2\left| \int_{T} D^2 uDw w dx \right|
	\leq 2 \|u\|_{H^2}\|w\|_{H^1} \|w\|_{L^\infty}
	\leq c  \|u\|_{H^2}\|w\|_{H^1}^{3/2} \|w\|_{L^2}^{1/2}
	\notag\\
	&\leq c \|u\|_{H^2}\|w\|_{H^1}^{3/2} (\|P_N w\|_{L^2}^{1/2} + \|Q_N w\|_{L^2}^{1/2} )
	\notag\\
	&\leq \frac{c}{\lambda^{{{}1/3}}}\|u\|_{H^2}^{4/3}\|w\|_{H^1}^{2}  
	+ \frac{c}{N^{1/2}}
	\|u\|_{H^2}\|w\|_{H^1}^{2} +\lambda \|P_N w\|_{L^2}^{2},
	\label{eq:ex:FP:bounds}
\end{align}
for a constant $c$ depending only on universal quantities.  We then
observe that the last term in this upper bound can be absorbed in the
nudging term in \eqref{eq:I1:bnd:FP:3}.  On the other hand, by tuning
$N$ and $\lambda$ to be appropriately large as a function of $\gamma >0$, we
can hope to bound the first two terms using the $H^1$--like term
$\gamma \InG_1(w, u)$,  thereby obtaining a decay that could not be
achieved in \eqref{eq:w:L2:no:decay} by working at the $L^2$--level. {{}We point out that we ultimately treat this term slightly differently in the analysis below, but hope only to illustrate the overall principle that we apply.}

A number of additional complexities need to be addressed in order to
conclude some form of time decay in $w$ from \eqref{eq:I1:bnd:FP:3} in
conjunction with frequency decomposition estimates like
\eqref{eq:ex:FP:bounds}.  In particular, note that this analytic approach required subtraction of the term $u w^2$ with a direct dependence on $u$ in \eqref{eq:mod:fn:base} in order to invoke the cancellation \eqref{eq:H1:can}.  Then, in the calculation of $d\InG_1(w, u)$, the equation \eqref{eq:I1:bnd:FP:3} contains a martingale term which comes from the differential falling on $u$ in \eqref{eq:mod:fn:base}.
Ultimately, it is this direct dependence on $u$ in $\InG_1(w, u)$ that is the reason
why we are only able to conclude the bounds in \cref{thm:FP:est},
\cref{cor:FP:est:Intro} in expectation.

Next, let us note that $\InG_1$ fails to directly control the $H^1$
norm of $w$ or to be sign-definite.  This is a serious concern since
the martingale terms in \eqref{eq:I1:bnd:FP:3} prevent us from
carrying out our analysis in `pathwise fashion'.  Fortunately, in view
of \eqref{eq:w:L2}, the $L^2$ evolution for $w$ does not produce any
terms worse than those appearing in \eqref{eq:I1:bnd:FP:3}.  We can
proceed by drawing on the approach considered in
\cref{lem:equivalence} and in our estimates in \cref{sect:Lyapunov},
and we `positivize' $\InG_1(w, u)$ by introducing suitable additional
terms to $\InG_1$ involving only the $L^2$ norms of $u$ and $w$.  The positivized modified functional is defined as
\begin{align}
  \InG_1^+(w, u) &:= \int_{\T} \bigl( (Dw)^2 - \frac{1}{3} w^3 - uw^2   \bigr) dx
    + \bar{\alpha}(\|w\|_{L^2}^{10/3} + (1 + \|u\|_{L^2}^2) \|w\|^2_{L^2}),
    \label{eq:FP:functional}
\end{align}
for $w \in H^1$ and $u \in L^2$. Here, similar to
\cref{lem:equivalence}, we include the parameters $\bar{\alpha} \geq 1$
which is convenient to tune to be appropriately large as we proceed in
our estimates for $\InG_1^+$.

For the final point of difficulty, let us note several terms on the right hand
side of \eqref{eq:I1:bnd:FP:3} with cubic dependence on $w$.  Here, we
write $w^3 = (u- v)w^2$ and ultimately obtain a bound on $w$ balancing
growth involving time integrals of $\|v\|_{L^2}$.  This, in turn, requires
some care in terms of growth of these quantities as a function of $\lambda$.

\subsection{Proof of \cref{thm:FP:est}}
\label{sect:FP:thm:proof}
We now proceed to the rigorous proof of \cref{thm:FP:est}.  Before
determining the evolution for $\InG_1^+(w, u)$ we make the initial
observation that $\InG_1^+(w, u)$ is indeed strictly positive and
controls $\|w\|_{H^1}^2$.
\begin{Prop}
  \label{prop:FP:functional}
  Consider the functional $\InG_1^+:H^1 \times L^2 \to \RR$ defined as
  in \eqref{eq:FP:functional}.  Then for any value of $\bar{\alpha} > 0$,
  sufficiently large depending only on universal quantities,
  \begin{align}
    \frac{1}{2}&\left( 
       \|D w\|_{L^2}^2
       + \bar{\alpha}( \|w\|_{L^2}^{10/3} 
       + (1 + \|u\|_{L^2}^2) \|w\|^2_{L^2}) \right)
    \notag\\
    &\qquad \leq \InG_1^+(w, u) 
    \leq 
    \frac{3}{2}\left( 
       \|D w\|_{L^2}^2
       + \bar{\alpha}( \|w\|_{L^2}^{10/3} 
       +  (1 + \|u\|_{L^2}^2) \|w\|^2_{L^2}) \right),
         \label{eq:norm:equiv:FP:FN}
  \end{align}
  for any $u \in L^2$ and any $w \in H^1$.
\end{Prop}

\begin{proof}

We observe with interpolation and Young's inequality that
\begin{align*}
  \frac{1}{3}\|w \|_{L^3}^3 \leq c \|D w\|_{L^2}^{1/2}\|w\|_{L^2}^{5/2}
  \leq \frac{1}{4} \|D w\|_{L^2}^2  + c\|w\|_{L^2}^{10/3},
\end{align*}
for any $w \in H^1$.  Similarly
\begin{align*}
   \left| \int_{\T} u w^2 dx \right|
   \leq  \|u\|_{L^2} \|w\|_{L^2}\|w\|_{L^\infty}
   \leq c \|u\|_{L^2} \|w\|_{L^2}^{3/2}\|Dw\|_{L^2}^{1/2}
   \leq c(1 + \|u\|_{L^2}^2)\|w\|_{L^2}^{2}
       +\frac{1}{4}\|Dw\|_{L^2}^2,
\end{align*}
for any $u \in L^2$ and $w \in H^1$.  Combining these two bounds
with the definition of
$\InG_1$ as in \eqref{eq:mod:fn:base}, we have that for any $w \in H^1$ and $u \in L^2$, 
\begin{align*}
  \frac{1}{2}\| D w \|^2_{L^2} - c(\|w\|_{L^2}^{10/3} + (1 + \|u\|_{L^2}^2) \|w\|^2_{L^2})
  \leq 
  \InG_1(w, u)
  &\leq
  \frac{3}{2}\| D w \|^2_{L^2} + c(\|w\|_{L^2}^{10/3} + (1 + \|u\|_{L^2}^2) \|w\|^2_{L^2}),
\end{align*}
which holds for some universal constant $c>0$ independent of $w, u$.  Combining this estimate with
\eqref{eq:FP:functional} yields \eqref{eq:norm:equiv:FP:FN} completing the proof.
\end{proof}

We turn now to determine the evolution of $\InG_1^+(w,u)$.
We begin by verifying \eqref{eq:I1:bnd:FP:3} namely the evolution for the principal part $\InG_1(w,u)$ of  $\InG_1^+(w,u)$, 
 given by \eqref{eq:mod:fn:base}.  
Regarding the first component $\In_1(w) := \int_{\T} ( (Dw)^2 - \tfrac{1}{3} w^3 )dx$ in $\InG_1(w,u)$ we observe, cf. \eqref{energy:funct:0:1}, \eqref{def:dispersive:J}, \eqref{dispersive:id},
that $\In_1$ is the first integral of motion of the KdV equation \eqref{eq:KdV:det} with the associated nonlinear operator defined by 
\begin{align}
	L_1(z) := -2 D^2 z - z^2.
	\label{eq:level:1:multiplier}
\end{align}
As such, from \eqref{eq:s:KdV:diff}, we have
\begin{align}
  \frac{d}{dt}\In_1(w)
  =& -\int L_1(w)(\gamma w  + D (uw)+\lam P_Nw) dx
     \notag\\
       =& - 2 \gamma \In_1(w) - 2\lambda \|D P_N w\|_{L^2}^2 + \int_{\T}\biggl( \frac{\gamma}{3} w^3
                       -  L_1(w)D(uw) + \lambda P_Nw w^2
                    \biggr)dx.
  \label{eq:I1:bnd:FP:1}
\end{align}
For the second term in $\InG_1(w,u)$ involving $u w^2$, we compute
 \begin{align}
   d \int_{\T} u w^2 dx
   =& \int_{\T} \left( w^2 du + 2 u w \partial_t w  \right) dx
      \notag\\
    =& \int_{\T} \biggl(w^2(f - u D u - D^3u - \gam u)
         - 2 u w (\gam w + \frac{1}{2} D(w^2) +D^3 w
         +D (uw) + \lam P_Nw) \biggr) dx dt
      \notag\\
     &+ \int_{\T}  w^2 \sigma dx dW.
         \label{eq:I1:bnd:FP:2}
 \end{align}
Note that $\int_{\T} 2 u w D (uw) dx =   \int_{\T} D ((uw)^2) dx = 0$.
Moreover, referring back to \eqref{eq:level:1:multiplier}, we have $\int_{\T}  u w (D(w^2) + 2D^3w) dx
  = \int_{\T}  L_1(w)D(uw)dx$.
Therefore, combining these two identities with \eqref{eq:I1:bnd:FP:1}, \eqref{eq:I1:bnd:FP:2}, 
and \eqref{eq:mod:fn:base} we conclude \eqref{eq:I1:bnd:FP:3}.

Next, we determine the evolution for the $L^2$-based terms in $\InG_1^+(w,u)$. 
Using \eqref{eq:w:L2}, we find that for any $q \geq 2$,
\begin{align}
   \frac{d}{dt}\| w\|^q_{L^2}
  +q\gamma \| w\|^q_{L^2} + q \lambda  \|P_N w\|^2_{L^2} \| w\|^{q-2}_{L^2} 
  &= -\frac{q}{2} \| w\|^{q-2}_{L^2}  \int_{\T} Du  w^2 dx. 
  \label{eq:w:L2:p}
\end{align}
On the other hand, notice that from \eqref{eq:L2:evol} and \eqref{eq:w:L2}, we may compute
\begin{align}
  d((1+ \|u\|_{L^2}^2)& \|w\|_{L^2}^2) 
  + (2 \gamma(1+\Sob{u}{L^2}^2) \|w\|_{L^2}^2  + 2 \gamma  \|u\|_{L^2}^2\|w\|^2_{L^2}
  + 2 \lambda (1+\|u\|_{L^2}^2 ) \|P_N w\|_{L^2}^2) dt
    \notag\\
  =& \left( 2\lb u,f\rb\Sob{w}{L^2}^2+\| \sigma \|^2_{L^2} \| w \|^2_{L^2}  
     - ( 1 + \|u\|^2_{L^2} ) \int_{T} Du w^2 dx  \right) dt
    + 2\|w\|^2_{L^2} \langle \sigma, u \rangle dW.
     \label{eq:L2:Prod:pos}
\end{align}
Then with this expression, from \eqref{eq:w:L2:p} and \eqref{eq:I1:bnd:FP:3}, we have
\begin{align}
  d \InG_1^+(w,&u) + 2 \gamma \InG_1^+(w,u)dt
                  + ({{}\frac{4}{3}\gamma + \frac{10}3\lambda}) \bar{\alpha} \| w\|_{L^2}^{10/3}dt
                  + 2 (\gamma + \lambda)\bar{\al} \|u\|_{L^2}^2\|w\|^2_{L^2}dt
                    \notag\\
                  &\quad 
                  + 2 \lam\bar{\alpha} \| w\|^2_{L^2}dt
                    + 2 \lambda \|D P_N w\|^2_{L^2}dt
                    \notag\\
  &= \int_{\T}\biggl( \frac{\gamma}{3} w^3 + \lambda P_Nw w^2
    + u D uw^2  + D^3uw^2 + 3\gam u w^2 + 2u w\lam P_Nw -fw^2
    \biggr)dx dt  
     \notag\\
  &\quad + \left( \bar{\alpha} \| \sigma \|^2_{L^2} \| w \|^2_{L^2}dt 
         -\frac{5\bar{\alpha} }{3} \| w\|_{L^2}^{4/3} \int_{\T} Du w^2 dx 
         - \bar{\alpha} (1+\| u\|_{L^2}^{2} )  \int_{\T} Du w^2 dx \right) dt
    \notag\\
   & \quad+ \left({{}2\bar{\alpha}\lb u,f\rb\Sob{w}{L^2}^2}+\frac{10}{3} \lambda \bar{\alpha} \| w\|_{L^2}^{4/3} \|Q_N w\|^2_{L^2} 
     + 2 \lambda \bar{\alpha} (1+ \|u\|^2_{L^2}) \|Q_N w\|^2_{L^2}\right) dt
     \notag\\
   & \quad+ \left(2\bar{\alpha} \|w\|^2_{L^2} \langle \sigma, u \rangle 
     + \lb w^2, \sigma\rb  \right)dW 
  =: \sum_{j = 1}^{{{}13}} K_j^D dt + K^S dW.
  \label{eq:FB:motherfn:evo}
\end{align}
Note that we have completed the square and rearranged accordingly in
our accounting of the terms
$2 \lambda (1+\|u\|_{L^2}^2 ) \|P_N w\|_{L^2}^2$ and
$(10\lambda/3) \|P_N w\|^2_{L^2} \| w\|^{4/3}_{L^2}$ from
\eqref{eq:L2:Prod:pos} and \eqref{eq:w:L2:p} respectively.

Let us now estimate each of the terms in \eqref{eq:FB:motherfn:evo} in
turn.  We remind the reader that we follow the convention of using $c > 0$
for constant depending only on universal i.e. equation independent quantities
and $C > 0$ for constants depending on equation dependent parameters, i.e., $f$
$\sigma$, $N$ etc, while emphasizing these dependencies when relevant. 
Starting with $K^D_1, K_2^D$ we have,
\begin{align}
  |K^D_1| + |K_2^D| 
  &\leq (\frac{\gamma}{3}\| w\|_{L^\infty} + \lambda \| P_N w\|_{L^\infty}) \|w\|_{L^2}^2
  \leq c(\gamma + \lambda) \|Dw\|^{1/2}_{L^2}\|w\|^{5/2}_{L^2}
  \notag\\
  &\leq \frac{\gamma}{1000}\|Dw\|_{L^2}^2 + c\frac{(\gamma + \lambda)^{4/3}}{\gamma^{1/3}}\|w\|^{10/3}_{L^2}\notag\\
  &\leq \frac{\gam}{100}\InG_1^+(w,u)
    + \frac{c}{\bar{\al}}\left(1+\frac{\lam}{\gam}\right)^{1/3}(\gam+\lam)\bar{\al}\Sob{w}{L^2}^{10/3},
    \label{eq:fp:est:K1:K2}
\end{align}
where we used \eqref{eq:norm:equiv:FP:FN} for the final inequality.
Next we have, with Agmon's inequality, interpolation 
and another invocation of \eqref{eq:norm:equiv:FP:FN},
\begin{align}
   |K^D_3| &\leq \| u Du \|_{L^\infty}\| w\|_{L^2}^2
   	    \leq c\| u\|_{L^2}^{1/2} \|D u\|_{L^2}\|D^2 u\|_{L^2}^{1/2} \| w\|_{L^2}^2
             \leq c \|u\|_{L^2}\|u\|_{H^2} \|w\|_{L^2}^2
             \notag\\
          \leq& \frac{c}{(\gam+\lam) \bar{\alpha}} \|u\|_{H^2}^2 \|w\|_{L^2}^2  
               +  
               \frac{\gam+\lam}{100} \bar{\alpha}\|u\|_{L^2}^2\|w\|_{L^2}^2   
               \notag\\
           &\leq \frac{c}{\lambda (\bar{\alpha})^2} \left(1+\|u\|_{H^2}^2\right) \InG_1^+(w,u)
           +\frac{\gam+\lam}{100} \bar{\alpha}\|u\|_{L^2}^2\|w\|_{L^2}^2.
  \label{eq:fp:est:K3}
\end{align}
Regarding $K^D_4$ we have, again invoking Agmon's inequality,
\begin{align}
  |K^D_4| &= 2 \left| \int_\T D^2 u w D wdx   \right|
  \leq c \|D^2 u\|_{L^2} \|w\|_{L^\infty}\|Dw\|_{L^2}
            \leq c\|D^2 u\|_{L^2} \|w\|_{L^2}^{1/2}\|Dw\|_{L^2}^{3/2}
            \notag\\
  &\leq \frac{c}{(\lam\bar{\alpha})^{1/3}} \|u\|_{H^2}^{4/3}\|Dw\|_{L^2}^2 
       +  \frac{\lambda}{100} \bar{\alpha} \|w\|_{L^2}^2\notag\\
          &\leq  \frac{c}{(\lam\bar{\al})^{1/3}}\left(1+\Sob{u}{H^2}^{2}\right)
            \InG_1^+(w,u)+\frac{\lam}{100}\bar{\al}\Sob{w}{L^2}^2.
           \label{eq:fp:est:K4}
\end{align}
Turning to $K^D_5$, we have, similarly to the previous two bounds
\begin{align}
  |K^D_5| &\leq  c\gamma  \|u\|_{L^2}^{1/2}\|Du\|_{L^2}^{1/2} \|w\|_{L^2}^2
            \leq c\gam\Sob{u}{L^2}^{3/4}\Sob{D^2u}{L^2}^{1/4}\Sob{w}{L^2}^2\notag\\
          &\leq c\frac{\gamma^{8/5} }{((\gam+\lam)\bar{\alpha})^{3/5}}   \|u\|_{H^2}^{2/5}\|w\|_{L^2}^2 
              + \frac{\gam+\lam}{100}\bar{\alpha}\|u\|_{L^2}^2 \|w\|_{L^2}^2\notag\\
              &\leq c\frac{\gamma^{8/5} }{\lam^{3/5}\bar{\alpha}^{8/5}}   \left(1+\|u\|_{H^2}^{2}\right)\InG_1^+(w,u) 
              + \frac{\gam+\lam}{100}\bar{\alpha}\|u\|_{L^2}^2 \|w\|_{L^2}^2.
           \label{eq:fp:est:K5}
\end{align}
Next, for $K^D_6$, we estimate
\begin{align}
  |K^D_6| &\leq 2\lambda \|u\|_{L^\infty}\|w\|_{L^2}^2 
  \leq c\lambda  \|Du\|_{L^2}^{1/2}\|u\|_{L^2}^{1/2}\|w\|_{L^2}^2 
  \leq c\lam\Sob{u}{L^2}^{3/4}\Sob{D^2u}{L^2}^{1/4}\Sob{w}{L^2}^2\notag\\
     &\leq c\frac{\lambda }{\bar{\alpha}^{3/5}} \|u\|_{H^2}^{2/5}\|w\|_{L^2}^2
                      + \frac{\lambda \bar{\alpha}}{100}\|u\|_{L^2}^2\| w\|_{L^2}^2
            \notag\\
  &\leq   c\frac{\lambda }{\bar{\alpha}^{8/5}} \left(1+\|u\|_{H^2}^2\right)\InG_1^+(w,u)
                      + \frac{\gam+\lam}{100}\bar{\alpha}\|u\|_{L^2}^2\| w\|_{L^2}^2.
           \label{eq:fp:est:K6}
\end{align}
Regarding $K^D_7$ and $K^D_{8}$ we simply observe that,
recalling that $\bar{\alpha} \geq 1$
\begin{align}
  |K^D_7| +|K^D_{8}| 
  \leq \frac{(\|f\|_{L^\infty} + \|\sigma\|^2_{L^2})}{\lambda} \lambda\bar{\alpha} \|w\|^2_{L^2}.
           \label{eq:fp:est:K7:K8}
\end{align}
For $K^D_9$ we have
\begin{align}
  |K^D_9| 
  &\leq \frac{5\bar{\alpha} }{3} \| Du\|_{L^\infty} 
    \| w\|_{L^2}^{10/3}
  \leq c\bar{\al}\Sob{u}{H^2}\Sob{w}{L^2}^{10/3}
  \leq 
    c\frac{\bar{\alpha}}{\gam+\lam}\|u\|_{H^2}^2\| w\|_{L^2}^{10/3} 
    + \frac{\gam+\lam}{100} \bar{\alpha} \| w\|_{L^2}^{10/3}
    \notag\\
  &\leq     \frac{c}{\lam}\left(1+\|u\|_{H^2}^2\right)\InG_1^+(w,u)
    + \frac{\gam+\lam}{100} \bar{\alpha} \| w\|_{L^2}^{10/3}.
           \label{eq:fp:est:K9}
\end{align}
Similarly for $K^D_{10}$ we observe
\begin{align}
  |K^D_{10}| 
  &\leq 
    \bar{\alpha} \|Du\|_{L^\infty} (1+\| u\|_{L^2}^{2} ) \| w\|_{L^2}^2
   \leq c\bar{\al}\Sob{u}{H^2}(1+\Sob{u}{L^2}^2)\Sob{w}{L^2}^2\notag\\
  &\leq c\frac{\bar{\alpha}}{\lam} \Sob{u}{H^2}^2 (1+\| u\|_{L^2}^{2} ) \| w\|_{L^2}^2  
        +\frac{\lam}{100} \bar{\alpha} (1+\| u\|_{L^2}^{2} ) \| w\|_{L^2}^2 
  \notag\\
  &\leq \frac{c}{\lambda}\left(1+ \|u\|^2_{H^2} \right)\InG_1^+(w,u)
    +\frac{\lam}{100}\bar{\al}\Sob{w}{L^2}^2
    +\frac{\gam+\lam}{100}\bar{\al}\Sob{u}{L^2}^2\Sob{w}{L^2}^2.
             \label{eq:fp:est:K10}
\end{align}
Regarding the final two terms, $K^D_{{{}12}}$ and $K^D_{{{}13}}$, we invoke 
the inverse Poincare inequality.  Indeed, we estimate
\begin{align}
  |K^D_{{{}12}}|
  &\leq c\frac{ \lambda \bar{\alpha}}{N^2} \|w\|_{{{}L^2}}^{4/3} \|D Q_N w\|_{L^2}^2
    \leq c\frac{ \lambda \bar{\alpha}}{N^2} 
  (\|u\|^{4/3}_{L^2} + \|v \|^{4/3}_{L^2}) \InG_1^+(w,u)
  \notag\\ &\leq 
  c\frac{ \lambda \bar{\alpha}}{N^2} 
  (1+{\|u\|^{2}_{L^2} + \|v \|^{2}_{L^2}}) \InG_1^+(w,u),
           \label{eq:fp:est:K11}
\end{align}
where for the second inequality we used that $w =  v-u$, where
$v$ is the solution of \eqref{eq:s:KdV:ng}.  On the
other hand, for $K^D_{{{}13}}$, we estimate
\begin{align}
  |K^D_{{{}13}}| \leq c\frac{\lambda \bar{\alpha}}{N^2} (1 + \|u\|_{L^2}^2) \|D Q_N w\|_{L^2}^2
            \leq c\frac{ \lambda \bar{\alpha}}{N^2} (1 + \|u\|_{L^2}^2) \InG_1^+(w,u).
           \label{eq:fp:est:K12}
\end{align}

Gathering the proceeding estimates and referring back to
\eqref{eq:FB:motherfn:evo} we now tune
$\lambda, \bar{\alpha}, N \geq 1$ as follows.  First in order to
control the terms $K_7^D$ and $K_8^D$ in \eqref{eq:fp:est:K7:K8} we
select $\lambda$ sufficiently large, taking say
\begin{align}
  \lambda \geq {{}100\left(\|f\|_{L^\infty} + \|\sigma\|_{L^2}^2\right)}.
  \label{eq:param:tune:1}
\end{align}
Next in order to control {{}$K_1^D, \ldots, K_6^D$} we choose any $\bar{\alpha}$,
compatible with the requirements for \cref{prop:FP:functional} appropriately 
as a function of $\gamma$ and $\lambda$.  For example we consider, 
\begin{align}
  \bar{\alpha} \geq 2c\max\{\lambda^2, \gamma^2, \gamma^{-2/3}\}
    \label{eq:param:tune:2}
\end{align}
where $c$ is the universal constant in \eqref{eq:fp:est:K1:K2}.  Finally
we take $N$ large enough so that the final two terms $K_{{{}12}}, K_{{{}13}}$
are suitably bounded, taking say
\begin{align}
  N\geq \lambda^{3/2} \bar{\alpha}^{1/2}.
  \label{eq:param:tune:3}
\end{align}
With these selections and referring back to
\eqref{eq:fp:est:K1:K2}--\eqref{eq:fp:est:K12} we find
\begin{align*}
  \sum_{j=1}^{{{}13}} |K_j^D|
  \leq& \left(\frac{\gamma}{2} + \frac{c_0}{\lambda}\left(1+ \|u\|^2_{H^2}
        + \frac{1}{\lambda} \|v\|^2_{L^2}\right)\right)\InG_1^+(w,u)
        \notag\\
  &+ \frac{4}{3}(\gamma + \lambda) \bar{\alpha} \|w\|^{10/3}
  + (\gamma + \lambda)\bar{\alpha} \|u\|^2_{L^2}\|w\|^2_{L^2}
  + \lambda\bar{\alpha} \|w\|^2_{L^2}.
\end{align*}
where $c_0 > 0$ is precisely universal constant appearing in
\eqref{eq:FP1}.  By now comparing this bound against
\eqref{eq:FB:motherfn:evo} we conclude
\begin{align}
  d \InG_1^+(w,u) 
  + \left(\frac{3\gamma}{2}
  - \frac{c_0}{\lambda}
  \left(1 + \|u\|_{H^2}^2 + \frac{1}{\lam}\| v\|_{L^2}^2\right) \right)
      \InG_1^+(w,u)dt  
  \leq \left(2\bar{\alpha} \|w\|^2_{L^2} \langle \sigma, u \rangle  + 
      \lb w^2,\s\rb \right)dW.
        \label{eq:mother:fn:rearr:1}
\end{align}
Now, by letting
$\Gamma(t) = \gamma t - \frac{c_0}{\lambda}\int_0^t(1 + \|u(s)\|_{H^2}^2
+ \frac{1}{\lambda}\| v(s)\|_{L^2}^{2})ds$, and considering the
integrating factor $e^{\Gamma}$ we infer from
\eqref{eq:mother:fn:rearr:1} that
\begin{align*}
 d (e^{\Gamma}\InG_1^+(w,u)) + \frac{\gam}{2}(e^{\Gamma}\InG_1^+(w,u))dt \leq
  e^{\Gamma} \left(2\bar{\alpha} \|w\|^2_{L^2} \langle \sigma, u \rangle  + 
      \lb w^2 \sigma\rb  \right)dW.
\end{align*}
Integrating this bound in time, taking an expected value and then
using the pointwise bound \eqref{eq:norm:equiv:FP:FN} we infer
\eqref{eq:FP1}, as claimed.  The proof of \cref{thm:FP:est} is
complete.

\subsection{Lyapunov Structure for the Nudged Equation: Proofs of
  \cref{thm:Lyapunov:ng}, \cref{lem:FP:stoppingtime},
and  \cref{cor:FP:est:Intro} }
\label{sect:Lyapunov:ng}

As previously mentioned, in order to exploit the Foias-Prodi property we
must ensure that the good Lyapunov structure of the original system is
largely preserved by the nudged system. The proof of this is
entirely analogous to the proof of \cref{thm:Lyapunov}, except that we carefully track the dependence of the constants on $\lam$ and $N$.

\subsubsection{Proof of \cref{thm:Lyapunov:ng}}
We proceed by observing that \eqref{eq:s:KdV:ng} may be viewed as
satisfying \eqref{eq:s:KdV} with external forcing $f$ replaced by
$\til{f}$ defined as $\til{f}:=-\lam P_N(v-u)+f$.  With this in mind
we turn first to establish \eqref{eq:L2:exp:Mt:ng}.  Working from
\eqref{eq:L2:evol} we find
\begin{align}
  d \|v\|^2 + \left(2 \gamma \|v\|^2 + 2 \lambda \|P_N v\|^2\right)dt
  =\left(2 \langle f, v\rangle
  + 2 \lambda \langle P_N v, P_N u \rangle
  + \|\sigma\|_{L^2}^2\right)dt + 2 \langle v, \sigma \rangle dW.
  \label{eq:L2:evo:Nudge}
\end{align}
{{}Now observe that
    \begin{align}
        2|\lb f,v\rb|\leq \frac{2}{\gam}\Sob{f}{L^2}^2+\frac{\gam}8\Sob{v}{L^2}^2,\quad \frac{7}8\gam\Sob{v}{L^2}^2\geq\frac{\gam}{8\Sob{\s}{L^2}^2}4\int_0^t|\lb\s,v\rb|^2ds.\notag
    \end{align}
}
As such, for every $t \geq 0$, we have the estimate
\begin{align*}
 &\| v(t)\|_{L^2}^2 + \gamma \int_0^t\|v\|^2ds
  - \left(1+ \frac{\lambda}{\gamma}\right)\left( \frac{2}{\gamma} \|f \|^2_{L^2} + \|\sigma\|^2_{L^2}\right)t
  - \|v_0\|^2 - \frac{\lambda}{\gamma} \| u_0\|^2_{L^2}\\
  &\leq
  2\int_0^t \langle \sigma, v \rangle dW
  - \frac{\gam}{8\Sob{\s}{L^2}^2}4
  \int_0^t |\langle \sigma, v \rangle|^2ds
  + \lambda \left[\int_0^t \|u\|^2_{L^2}{{}dt} - \frac{1}{\gamma} \left( \frac{2}{\gamma} \|f \|^2_{L^2} + \|\sigma\|^2_{L^2}\right)t
  - \frac{1}{\gamma} \| u_0\|^2_{L^2}\right].
\end{align*}
We therefore obtain the first bound from \eqref{eq:exp:est:asymptotic:L2},
\eqref{eq:exp:Mt:gen}.

We turn now to the second bound \eqref{eq:Lyapunov:Bnd:Hm:ng}.  Here,
as in \cref{sect:Lyapunov}, we make use of the functionals $\In_m^+$
as defined in \eqref{eq:mod:int:m:simple}.  Following precisely the
same computations as in \cref{lem:Im:evolution} we see that
$\In_m^+(v)$ satisfies {{}an evolution equation} of the form \eqref{eq:Inplus:simp}
with $f$ replaced appropriately by $\til{f} :=-\lam P_N(v-u)+f$ in
$\Kn_{m}^{D,+}$.  From here the proof proceeds in a similar manner to
the proof of \eqref{eq:Lyapunov:Bnd} carried out in
\cref{sect:Alg:mom:bnd}.  The bounds for $\Kn_{m,1}^D(v)$ and
$\Kn_{m,2}(v)$ are precisely as given in \eqref{eq:Kn1:est},
\eqref{eq:Kn2:est}.  We {{}revisit} \eqref{eq:Kn3:est} involving
$\tilde{f}$ and find again from \eqref{eq:f:terms:Im:Ito} with
\cref{thm:KdV:Poly:Rank}, \cref{prop:mono:intp:bnd} that
\begin{align}
  |\Kn_{m,3}^{D}(v,\tilde{f})|
  \leq C(1 + \| u\|_{L^2})\sum_{k =1}^{M} \|D^m v\|^{r_k}_{L^2} (1+ \| v \|^2_{L^2})^{q_k}.
  \label{eq:Kn3:est:mod}
\end{align}
{{} Indeed, one sees from \eqref{eq:f:terms:Im:Ito} that the only potentially problematic term in $\mathcal{K}^D_{m,3}$ is given by $\int_{\T} 2 D^m v D^m\til{f} dx$, as all the other terms that arise are either of lower order or else do not appear quadratically in $D^mv$, and are, hence, ``good'' terms; upon closer inspection, one sees that
\[
\int_{\T}2 D^m v D^m\til{f}=-2\lam\Sob{D^mP_Nv}{L^2}^2+2\lam\lb D^mP_Nv,P_Nu\rb.
\]
Thus $-2\lam\Sob{D^mP_Nv}{L^2}^2$ may dropped in obtaining an upper bound for $\mathcal{K}^D_{m,3}$; all other terms may be estimated using \cref{thm:KdV:Poly:Rank} and \cref{prop:mono:intp:bnd}; in comparison to the analysis carried out previously in \cref{sect:Alg:mom:bnd}, the only major difference is that estimates now depend on $\lam$ and $\Sob{u}{H^k}$, $k=1,\dots, m$ since the role of $f$ there is replaced with $\lam P_Nu$ here. We emphasize that ultimately, for each $k =1, \ldots, M(m)$, the resulting exponents and constants satisfy
$r_k = r_k(m) \in (0,2)$, $q_k = q_k(m) >0$ and
$C = C(m, \|f\|_{H^m}, N, \lambda) > 0$. In particular, they are all independent of $u$ and $v$. 

Now, in place of
\eqref{eq:Kn:extra:est} we have the bound
\begin{align}
  | \Kn_{m}^{D,+} -  \Kn_{m}^{D}|
  &\leq C(1 + \|u\|_{L^2}) \bar{q}_m \abarm^2 (1 + \|v\|^2_{L^2})^{\bar{q}_m -1/2},
    \label{eq:Kn:extra:est:mod}
\end{align}
where in this case
$C = C(\gamma,\|f\|_{L^2}, \|\sigma\|_{L^2}, \lambda) > 0$ is once
again independent of $u$, $v$ and the parameters
$\bar{q}_m, \abarm \geq 1$ which are part of the defintion of
$\In_m^+$. Indeed, in estimating the set of terms in $\Kn_{m}^{D,+} -  \Kn_{m}^{D}$, we observe that the only potentially problematic term that arises is the sole term in which $\til{f}$ appears: ${\abarm}{\bar{q}_m}\left( \|v\|_{L^2}^2+1\right)^{\bar{q}_m-1}\lb\til{f} , v\rb$; all other terms are treated exactly as we did in obtaining \eqref{eq:Kn:extra:est}. Specifically, upon inspection we see that
\begin{align}
&{\abarm}{\bar{q}_m}\left( \|v\|_{L^2}^2+1\right)^{\bar{q}_m-1}\lb\til{f} , v\rb\notag\\
&=-\lam{\abarm}{\bar{q}_m}\left( \|v\|_{L^2}^2+1\right)^{\bar{q}_m-1}\Sob{P_Nv}{L^2}^2+\lam{\abarm}{\bar{q}_m}\left( \|v\|_{L^2}^2+1\right)^{\bar{q}_m-1}\lb P_Nu,P_Nv\rb.\notag
\end{align}
We note that the term that is quadratic in $v$ is non-positive. We therefore drop it in obtaining an upper bound.
}

Thus by combining \eqref{eq:Kn1:est}, \eqref{eq:Kn2:est},
\eqref{eq:Kn3:est:mod} \eqref{eq:Kn:extra:est:mod}, applying Young's
inequality appropriately and then making a sufficiently large choice
of $\bar{q}_m, \abarm$ used to define $\In_m^+$, commensurate with
\cref{lem:equivalence} we obtain, similarly to
\eqref{eq:Kn:summary} above, that
\begin{align}
  | \Kn_{m}^{D,+} | \leq C(1 + \|u\|^q_{L^2}) + \frac{\gamma}{100} \In_m^+(v).
  \label{eq:Kn:summary:mod}
\end{align}
Here, again we note that $C(m, \gamma, \|f\|_{H^m}, \|\sigma\|_{H^m}, \lambda, N) > 0$
and $q= q(m) > 0$ are independent of $u,v$ 

With \eqref{eq:Kn:summary:mod} in hand we find with \eqref{eq:Inplus:simp}
that $\frac{d}{dt} \E \In_m^+(v) + \gamma \E \In_m^+(v) \leq C(1 + \E\|u\|^q_{L^2})$
so that with Gr\"onwall and \cref{lem:equivalence}
\begin{align}\label{est:Hm:ng}
  \E\|v(t)\|_{H^m}^2 \leq C e^{-\gamma t}(\|v_0\|_{H^m}^2 + \|v_0\|^q_{L^2} +1)
  + C\int_0^t e^{-\gamma(t -s)} \E\|u(s)\|^q_{L^2}ds.
\end{align}
With this observation and then, invoking \eqref{eq:Lyapunov:Bnd} in
the case $m=0$ with appropriately large value for $p$,  the desired bound
\eqref{eq:Lyapunov:Bnd:Hm:ng}, now follows.

\subsubsection{Stopping Time Estimates:
  Proof of \cref{lem:FP:stoppingtime}}
\label{sect:FP:stoppingtime}
Observe that, for $R, \beta > 0$ and any $\lambda \geq 1$, using
\cref{lem:equivalence}, we have
\begin{align*}
   &\left\{\sup_{t\geq0}\left( 
   \frac{c_0}{\lambda} \int_0^t 
   (1+ \|u(s)\|_{H^2}^2+\frac{1}{\lam}  \|v(s)\|_{L^2}^2) ds 
     - \frac{\gamma}{2} t  -\be  \right)\geq R\right\} \\
  &\quad\subseteq
   \left\{\sup_{t\geq0}\left( 
     \frac{2c_0}{\lambda}
     \int_0^t \In_{2}^{+}(u)ds -
     t \frac{\gamma}{4}
   -\frac{\be}{2}  \right)\geq \frac{R}{2}\right\}\\
  &\quad \quad \quad\cup
  \left\{\sup_{t\geq0}\left( 
    \frac{c_0}{\lambda} \int_0^t \|v\|^2_{L^2}ds
       -  t \frac{\lam\gam}{4}
  -\frac{\lambda\be}{2}  \right)\geq \frac{\lambda R}{2}\right\},
\end{align*}
Here, $c_0$ is the constant appearing in \eqref{eq:FP1},
\eqref{def:FP:stoppingtime} and we choose the parameters
$\bar{\alpha}_2, \bar{q}_2 \geq 1$ that define $\In_2^+$ sufficiently
large so that \cref{lem:equivalence}, \cref{thm:Lyapunov} both apply.
We now consider any $\lambda \geq 1$ large enough so that
$2c_0/\lambda^{1/2} \leq \gamma$ 
and that $\lambda^{1/2}\gamma/8$ is
greater than both of the $\lambda$-independent constants $C$ which
appear in the bounds \eqref{eq:poly:est:asymptotic} and
\eqref{eq:L2:exp:Mt:ng}.  Thus, with any such appropriately large
choice of $\lambda \geq 1$, we find
\begin{align}\label{eq:stopping:set}
  \Prb(\tau_{R,\beta} <\infty)
   \leq& \Prb\left( \sup_{t \geq 0} \left( \gamma \int_0^t \In_2^+(u) ds -
   Ct
   - {{}\frac{\lambda^{1/2} \beta}{2}} \right)  \geq \frac{\lambda^{1/2} R}{2}\right)
   \notag\\
       &\quad \quad
    + \Prb\left( \sup_{t \geq 0} \left( \gamma \int_0^t \|v\|^2_{L^2} ds -
         C\lam t
         - {{}\frac{\lambda^{3/2}\beta}{2}} \right)
         \geq \frac{\lambda^{3/2} R}{2}\right),
\end{align}
which holds for any $R, \beta > 0$. We estimate the first term appearing on the right-hand side by using \eqref{eq:poly:est:asymptotic} and the second by using \eqref{eq:L2:exp:Mt:ng} under a suitable choice of $\be$. Indeed, let us first observe that
\begin{align}
\left\{\sup_{t \geq 0} \left( \gamma \int_0^t \In_2^+(u) ds -
   Ct
   - \frac{\lambda^{1/2} \beta}{2} \right)  \geq \frac{\lambda^{1/2} R}{2}\right\}
   \subset\left\{\sup_{t \geq 0} \left(\In_2^+(u(t)) \gamma \int_0^t \In_2^+(u) ds -
   C(t+1) \right)  \geq R_1\right\}\notag,
\end{align}
where
\[
R_1:=\In_2^+(u_0)+\frac{\lambda^{1/2} R}{2}+\left(\frac{\lam^{1/2}\be}2-\left(C+\In_2^+(u_0)\right)\right).
\]
Upon choosing $\be$ such that $\be\geq\frac{2}{\lam^{1/2}}\left(C+\In_2^+(u_0)\right)$, we see that $R\geq \frac{\lam^{1/2}}2R$. Applying \eqref{eq:poly:est:asymptotic} yields
\begin{align}\label{eq:prop:stop:est1}
\Prb\left( \sup_{t \geq 0} \left( \gamma \int_0^t \In_2^+(u) ds -
   Ct
   - \frac{\lambda^{1/2} \beta}{2} \right)  \geq \frac{\lambda^{1/2} R}{2}\right)\leq \left(\frac{2}{\lam^{1/2}}\right)^{q/2-1}\frac{C \In_2^+(u_0)^q}{R^{q/2-1}},
\end{align}
for all $q>2$. For the second term in \eqref{eq:stopping:set}, assuming $\lam\geq 2^{2/3}$, we have
\begin{align}
&\left\{\sup_{t \geq 0} \left( \gamma \int_0^t \|v\|^2_{L^2} ds -
         C\lam t
         - \frac{\lambda^{3/2}\beta}{2} \right)\geq\frac{\lam^{3/2}R}2\right\}\notag\\
         &\subset\left\{\sup_{t\geq0}\left(\Sob{v(t)}{L^2}^2+\gam\int_0^t\Sob{v}{L^2}^2ds-\lam Ct\right)\geq R+\frac{\lam^{3/2}\be}2\right\}.\notag
\end{align}
{{}Hence, if $\be$ satisfies $\be\geq\frac{2}{\lam^{3/2}}\left(\Sob{v_0}{L^2}^2+\frac{\lam}\gam\Sob{u_0}{L^2}^2\right)$}, then we may apply \eqref{eq:L2:exp:Mt:ng} to deduce
\begin{align}\label{eq:prop:stop:est2}
\Prb\left(\sup_{t \geq 0} \left( \gamma \int_0^t \|v\|^2_{L^2} ds -
         C\lam t
         - \frac{\lambda^{3/2}\beta}{2} \right)\geq\frac{\lam^{3/2}R}2\right)\leq \exp(-\bar{\eta}R).
\end{align}
Upon returning to \eqref{eq:stopping:set} and applying \eqref{eq:prop:stop:est1}, \eqref{eq:prop:stop:est2}, we deduce \eqref{eq:FP:stoppingtime:bound}. Note that from \eqref{eq:I2p:exp:LD:eqv}, we may simplify the condition on $\be$ to \eqref{eq:beta:cond:stt}.

\subsubsection{Proof of \cref{cor:FP:est:Intro}}

Recalling $w(t) := u(t;u_0)-v(t;v_0)$ and invoking \cref{cor:FP:est},
\cref{lem:FP:stoppingtime} and finally \cref{thm:Lyapunov:ng}, we have
\begin{align*}
  &\E\|w(t)\|_{H^1}
  =\E(\indFn{\tau_{R,\beta} =\infty}\|w(t)\|_{H^1})
     +\E(\indFn{\tau_{R,\beta}<\infty}\|w(t)\|_{H^1})
     \notag\\
   &\leq C{{}\exp\left({{}R + \beta}
      - \frac{\gamma}{2} t\right)}
     (\|u_0\|^{q}_{H^1} + \|v_0\|_{H^1}^{q} +1) 
         + \Prb(\tau_{\beta,R} < \infty)^{1/2}
         (\E(\|u(t)\|^2_{H^1} + \|v(t)\|^2_{H^1}))^{1/2}
        \notag\\
   &\leq C\left(\exp\left(R + \beta
         - \frac{\gamma}{2} t\right) + \frac{1}{R^{p/2}}\right)
         (\|u_0\|_{H^2}+\|v_0\|_{H^1}+1)^{{}\til{q}}
\end{align*}
for any $p > 0$, {{}where $\til{q}=\til{q}(p)$ is sufficiently large}.  Here, recall that the stopping times $\tau_{R,\beta}$
are given in \eqref{def:FP:stoppingtime} and note carefully that the
constants $C = C(p,\gamma, \|f\|_{H^2}, \|\sigma\|_{H^2}, \lambda)$,
$\til{q} = \til{q}(p)$ are independent of $u_0, v_0$ and $t, R, \beta$.  Selecting
now $R = {{}\frac{\gamma}{4}} t$ for each $t > 0$ and taking $\beta$
according to the lower bound \eqref{eq:beta:cond:stt} and $p > 0$
appropriately large the desired result follows.

\section{Regularity Results for Stationary Solutions}
\label{sect:Regularity}

In this section we prove that any invariant measure is supported in
$H^m$, provided that $f\in H^m$ and $\s\in\bH^m$, where $m\geq2$, that
is, that the support of the invariant measure is contained in a space
that is as smooth as the forcing. The precise statement is given as
follows.

\begin{Thm}\label{thm:regularity}
  Fix any $m\geq2$ and suppose that $\gam>0$, $f\in H^m$ and
  $\s\in \bH^m$.  Let $\{P_t\}_{t\geq0}$ be the Markov semigroup
  defined on $H^2$ corresponding to \eqref{eq:s:KdV} with this
  data. Then
    \begin{align}\label{eq:reg:im}
        \mu(H^2\smod H^m)=0,
    \end{align}
    holds for all invariant measures $\mu\in\Pr(H^2)$ of
    $\{P_t\}_{t\geq0}$. Moreover, for any such $\mu$ there exists
    a $p = p(m) > 0$ such that, for any $\eta > 0$,
  \begin{align}
    \label{eq:reg}
    \int_{H^2}\exp(\eta \Sob{u}{H^m}^{2p})\mu(du) \leq C<\infty,
  \end{align}
  where $C = C(\eta, m, \gamma, \|f\|_{H^m}, \|\sigma\|_{H^m})$ is independent
  of $\mu$.
\end{Thm}

{{}Note that when $f\in H^m$, for all $m\geq2$, then \cref{thm:regularity} immediately implies the first claim in \cref{thm:Regularity:Intro} from the introduction, while \eqref{eq:reg} implies the second claim; the third claim there is simply a restatement of \eqref{eq:reg} by duality.}

\begin{proof}[Proof of \cref{thm:regularity}]
For $K, M \geq 1$ let
\begin{align}\label{def:Hm:regularized}
        \phi_{K,M}(u):=K\wedge\Sob{P_Mu}{H^m},
\end{align}
where recall that $P_M$ denotes the spectral projection onto
wavenumbers $\leq M$.  Note that $\phi_{K,M}$ is continuous and
bounded on $H^2$.  Fix any invariant measure $\mu\in \Pr(H^2)$ for
\eqref{eq:s:KdV}. Given $\rho>0$, let $B_\rho\subseteq H^2$ denote the
$H^2$-ball of radius $\rho>0$ centered at $0$.  Then, invoking
invariance, we have
\begin{align}
    	\int_{H^2} \phi_{K,M}(u_0) \mu(du_0) &= \int_{H^2} \phi_{K,M}(u_0) \mu P_t(du_0) 
        = \int_{H^2} P_t \phi_{K,M}(u_0) \mu (du_0) 
        =\int_{H^2}\E\phi_{K,M}(u(t;u_0))\mu(du_0)\notag\\
        &=\int_{B_\rho^c} \E\phi_{K,M}(u(t;u_0))\mu(du_0)
             + \int_{B_\rho} \E\phi_{K,M}(u(t;u_0))\mu(du_0),
            \label{eq:near:far:split}
\end{align}
valid for any $t \geq 0, \rho > 0$

By appropriately selecting $\rho$ and $t$ we now use
\eqref{eq:near:far:split} to bound
$\int_{H^2} \phi_{K,M}(u_0) \mu(du_0)$ independently of $K,M$.  For
the first term we simply observe that
\begin{align}
        \int_{B_\rho^c}\E\phi_{K,M}(u(t;u_0))\mu(du_0)\leq K\mu(B_\rho^c).
        \label{eq:far:field}
\end{align}    
In order to bound the second integral over ${B_\rho}$ in
\eqref{eq:near:far:split} we select $\lam, N \geq 1$ such that
\cref{cor:FP:est:Intro} applies for the given data $f \in H^m$,
$\sigma \in \bH^m$ and $\gamma > 0$.  For each $u_0\in B_\rho$ we
consider $v(0,u_0) := v(0, u_0, f, \sigma)$ satisfying
\eqref{eq:s:KdV:ng} with $v_0\equiv0$ and this choice of $\lambda, N$.
Here, note carefully that this $v$ satisfies the $H^m$ bound
\eqref{eq:Lyapunov:Bnd:Hm:ng} for every $u_0 \in H^2$.  Thus, with the
reverse Poincar\'e inequality, \cref{cor:FP:est:Intro} and
\cref{thm:Lyapunov:ng}, $(ii)$ we find that, for any $t > 0$,
\begin{align}
	\int_{B_\rho} &\E\phi_{K,M}(u(t;u_0))\mu(du_0) 	        \notag\\
	 &\leq \int_{B_\rho} \E \| P_M (u(t;u_0) - v(t,0,u_0)) \|_{H^m}  \mu(du_0)
	        +  \int_{B_\rho} \E \| v(t;0,u_0)\|_{H^m}\mu(du_0) \notag\\
        	 &\leq M^{m-1} \int_{B_\rho} \E \| u(t;u_0) - v(t,0,u_0) \|_{H^1}  \mu(du_0)
	        + \int_{B_\rho} \E \| v(t;0,u_0)\|_{H^m}\mu(du_0) \notag\\
	 &\leq M^{m-1} \int_{B_\rho} \frac{\exp(C(1 + \|u_0\|^5_{H^2}))}{t} \mu(du_0)
	        +  C\int_{B_\rho}(e^{-\gamma t} 
	        (\|u_0\|^q_{L^2} + 1)+1)^{1/2}\mu(du_0)
	        \notag\\
	 &\leq M^{m-1} \frac{\exp(C(1 + \rho^5))}{t} + C(e^{-(\gam/2) t}\rho^{q/2}+1)\leq C\left(M^{m-1} \frac{\exp(C(1 + \rho^5))}{t} + 1\right),
	 \label{eq:near:field}
\end{align}
where, to emphasize, the constant
$C = C(m, \lambda, N, \gamma, \|f\|_{H^m}, \|\sigma\|_{H^m}) $ is
independent of $M \geq 1$, $\rho, t > 0$ and $\mu$.  By combining
\eqref{eq:far:field} and \eqref{eq:near:field} we find
\begin{align}
	\int_{H^2} \phi_{K,M}(u) \mu(du) \leq K\mu(B_\rho^c)+ C\left(M^{m-1} \frac{\exp(C(1 + \rho^5))}{t} + 1\right),
	\label{eq:free:range:params}
\end{align}
where again $C$ is independent of $K, M, \rho$ and $t$.

Thus, by selecting first $\rho >0$ large enough such that
$K\mu(B_\rho^c) \leq 1$ and then selecting $t >0$ sufficiently large
so that $M^{m-1} t^{-1}\exp(C(1 + \rho^5))\leq 1$, we conclude that
$\int_{H^2} \phi_{K,M}(u) \mu(du)$ is bounded independently of
$K, M \geq 1$ and $\mu$.  Thus, sending $M,K \to \infty$ in
\eqref{def:Hm:regularized} and invoking the monotone convergence
theorem we infer that $\int_{H^2} \|u\|_{H^m} \mu(du) < \infty$.  This
in particular establishes the first claim \eqref{eq:reg:im}.

Regarding the second claim, \eqref{eq:reg} we invoke the bound
\eqref{eq:Lyapunov:Bnd:subquadexp}.  Indeed, considering any
$p = p(m)$ commensurate with \cref{thm:Lyapunov} so that
\eqref{eq:Lyapunov:Bnd:subquadexp} holds, we now define
\begin{align*}
	\psi_{M,K}(u) = \exp( \eta \| P_M u\|^{2p}_{H^m}) \wedge K
\end{align*}
for every $M, K \geq 1$ and every $\eta > 0$.  Here, as with
$\phi_{M,K}$ previously, $\psi_{M,K}$ is continuous and bounded on
$H^2$.  Take $\tilde{B}_\rho :=\{ u\in H^m : \|u\|_{H^m} < \rho\}$.
Arguing similarly to \eqref{eq:near:far:split}, \eqref{eq:far:field}
but now with an appropriate use \eqref{eq:reg:im} at the intermediate
step we find
\begin{align}
	\int_{H^2} \psi_{K,M}(u_0) \mu(du_0) 
	=& \int_{H^2} \E \psi_{K,M}(u(t;u_0)) \mu(du_0) 
	= \int_{H^m} \E \psi_{K,M}(u(t;u_0)) \mu(du_0) \notag\\
	\leq& \int_{\tilde{B}_\rho} \E \psi_{K,M}(u(t;u_0)) \mu(du_0)  + 
	        K \mu(H^m \setminus \tilde{B}_\rho),
	         \label{eq:far:field:2:0}
\end{align}
for any $t, \rho > 0$.  On the other hand, with \cref{lem:equivalence}
and \eqref{eq:Lyapunov:Bnd:subquadexp} we find
\begin{align} 
	 \int_{\tilde{B}_\rho} \E \psi_{K,M}(&u(t;u_0)) \mu(du_0) 
	\leq \int_{\tilde{B}_\rho} \E \exp( \eta \| u(t; u_0)\|^{2p}_{H^m}) \mu(du_0) 
	\leq \int_{\tilde{B}_\rho} \E \exp( \eta \In_m^+(u(t; u_0))^{p}) \mu(du_0) 
	\notag\\
	&\leq \int_{\tilde{B}_\rho} C\exp\left(\eta e^{-\gam p t}\In_m^+(u_0)^p \right)\mu(du_0) 
	\leq  \int_{\tilde{B}_\rho} \exp\left(C (e^{-\gam p t}  \| u(t; u_0)\|^{q}_{H^m} +1)\right)\mu(du_0) 
	\notag\\
	&\leq \exp\left(C (e^{-\gam p t}  \rho^q +1)\right)
	 \label{eq:near:field:2:0}
\end{align}
where once again
$C = C(m, \gamma, \|f\|_{H^m}, \|\sigma\|_{H^m}, \eta, p )$,
$q = q(m,p)$ are both independent of $t, M, K,\rho$ and $\mu$.
Combining \eqref{eq:far:field:2:0} and \eqref{eq:near:field:2:0} and
adjusting the free parameters $t, \rho >0$ appropriately as we did for
\eqref{eq:near:field:2:0} obtain a bound for
$\int_{H^2} \psi_{K,M}(u_0) \mu(du_0)$ which is independent of
$K, M \geq 1$ and of $\mu$.  Thus another invocation of the monotone
convergence theorem now yields \eqref{eq:reg}.  The proof is complete.
\end{proof}

\section{Unique Ergodicity with Arbitrary Damping and Degenerate
  Forcing}
\label{sect:UniqueErgodicity}

{{}This section addresses unique ergodicity in the case where
$\gamma > 0$ is arbitrary, while on the other hand imposing some
requirements on the range of $\sigma$. In particular, for $N>0$, let $P_N$ denote the Dirichlet projection onto frequencies in the range $[-N, N]$. Given $K>0$, we will assume throughout this section that $\s$ satisfies}
\begin{align*}
  \s:\R^K\goesto L^2,\quad \mathcal{R}(\sigma) \supset P_NL^2,
\end{align*}
{{}where $N$ is the largest integer for which $\mathcal{R}(\s)\supset P_NL^2$ holds. We recall from \cref{sect:functional:setting} that the associated pseudo-inverse $\s:P_NL^2\goesto\mathbb{R}^K$ is bounded; we will crucially invoke these below in order to apply Girsanov's theorem. The main result of this section is \cref{thm:unique:ergodicity:Intro} from the introduction, which we restate here in a more precise form as \cref{thm:ergodicity:ess:elliptic}}.

\begin{Thm}
  \label{thm:ergodicity:ess:elliptic}
  Given $\gam>0$, $f\in H^2$, and $\s\in\bH^2$, there exists
  $N=N(\gam,\Sob{f}{H^1},\Sob{\s}{L^2})$ such that if
  $\mathcal{R}(\s)\supset P_NH^2$, then $\{P_t\}_{t \geq 0}$ possesses at most one ergodic invariant probability measure $\mu\in\Pr(H^2)$, where $\{P_t\}_{t\geq0}$ is the Markov
  semigroup corresponding to \eqref{eq:s:KdV} on $H^2$.
\end{Thm}
\begin{Rmk}
  An explicit estimate on the required size for $N$ as a function of
  $\gam,\Sob{f}{H^1},\Sob{\s}{L^2}$ can be found in
  \cref{rmk:N:dep:FP}.  Note that the requirement for $N$ is
  dictated by our analysis below, so that we can obtain the
  bounds \eqref{eq:FP2} and \eqref{eq:FP:stoppingtime:bound}.  Our estimate on $N$ in \cref{rmk:N:dep:FP} is likely far from optimal,
   and we ultimately conjecture that uniqueness of
  invariant measures should hold even if just a few well-selected
  frequencies are forced with the number of frequencies independent of
  $f$, $\|\sigma\|_{L^2}$ and $\gamma >0$.  While such a
  `hypoelliptic' setting \'a la \cite{HairerMattingly2008,
    HairerMattingly2011, FoldesGlattHoltzRichardsThomann2013} is
  beyond the scope of the analysis presented here, we intend to take up
  this more difficult setting in future work.
\end{Rmk}

The proof of this theorem will ultimately rely on an application of an
abstract result from \cite{GlattHoltzMattinglyRichards2015} (see also
\cite{HairerMattinglyScheutzow2011}), which identifies a set of
conditions under which the Markov semigroup is guaranteed to have at
most one invariant measure (\cref{thm:GHMR} below). In particular, one
need only ensure that the process couples asymptotically on a set of
positive probability. The statement of this abstract result requires a
few more technical definitions, which we now recall.

Let $H$ be a Polish space with a metric $\rho$.  Let
$\mathcal{B}(H)$ denote the Borel $\sigma$-algebra on $H$. Suppose that
$P$ is a Markov transition kernel on $H$ namely we suppose that
$P: H \times \mathcal{B}(H) \to [0,1]$ such that $P(u, \cdot)$ is a
probability measure for any given $u \in H$ and that $P(\cdot, A)$ is
a measurable function for any fixed $A \in \mathcal{B}(H)$.  As above
in \eqref{eq:mt:obs:act}, \eqref{eq:mt:msr:act}, $P$ acts on bounded
measurable observables $\phi: H \to \RR$ and Borel probability
measures $\mu$ as $P\phi(u) = \int \phi(v)P(u,dv)$ and
$\mu P(A) = \int P(u, A) \mu(du)$ respectively.  A probability measure $\mu$ on
$\mathcal{B}(H)$ is invariant if $\mu P = \mu$.

Let $P^{\mathbb{N}}$ be the associated Markov transition kernel on the
space of one-sided infinite sequences $H^{\mathbb{N}}$, namely, for
any $u_0 \in H$ and any $A \in \mathcal{B}(H^{\mathbb{N}})$,
$P^{\mathbb{N}}(u_0, A) := \Prb((u(1;u_0), u(2; u_0), \ldots) \in
A)$, where $u(n;u_0) \sim P(u(n-1;u_0), dv)$ for each $n \geq 1$,
starting from $u(0;u_0) = u_0$.  Denote by $\Pr(H)$ and
$\Pr(H^{\mathbb{N}})$ the collections of Borel probability measures on
$H$ and $H^{\mathbb{N}}$, respectively.  Then, for any
$\mu \in\Pr(H)$, we define the \textit{suspension of $\mu$ to $H^\N$}
by $\mu P^{\mathbb{N}}\in\Pr(H^{\mathbb{N}})$ by
$\mu P^{\mathbb{N}}(\cdot) :=
\int_{H^{\mathbb{N}}}P^{\mathbb{N}}(u,\cdot)\mu(du)$.

For any $\mu_1,\mu_2\in\Pr(H^{\mathbb{N}})$, consider the
\textit{space of asymptotically equivalent couplings}
\begin{align*}
  \til{\mathcal{C}}(\mu_1,\mu_2)
  := \{\Gamma \in\Pr(H^{\mathbb{N}}\times H^{\mathbb{N}}):
  \Gamma\Pi_{i}^{-1} \ll \mu_i \ \text{for each} \  i=1,2\}, 
\end{align*}
where $\Pi_{i}$ denotes projection onto the $i^{\text{th}}$ coordinate
and $\nu \ll \mu$ means that a measure $\nu$ is absolutely continuous
with respect to another measure $\mu$, namely, $\nu(A) = 0$ whenever
$\mu(A) = 0$.  The notation $\mu f^{-1}$ signifies the push forward of
a measure $\mu$ under a function $f$, that is
$(\mu f^{-1})(A) = \mu(f^{-1}(A))$ so that, in terms of observables
$\phi$, we have
$\int \phi(u)(\mu f^{-1})(du) = \int \phi(f(u)) \mu(du)$.  Next, with
respect to a possibly different choice of distance ${\til{\rho}}$ on
$H$, consider the subspace of mutually convergent sequence-pairs
\begin{align*}
  D_{{\til{\rho}}}:= \{(\mathbf{u},\mathbf{v})\in H^{\mathbb{N}}\times H^{\mathbb{N}}:
  \lim_{n\rightarrow \infty} {\til{\rho}}(u_n,v_n) = 0\}.
\end{align*}
Consider also the test functions
\begin{align}\label{def:determine:measures}
  \mathcal{G}_{{\til{\rho}}}
  = \left\{\phi \in C_{b}(H):
  \sup_{u\neq v}\frac{|\phi(u) - \phi(v)|}{{\til{\rho}}(u,v)}<\infty \right\},
\end{align}
and note that a set $\mathcal{G}$ of bounded, real-valued, measurable
functions is said to \textit{determine measures} on $H$ provided that,
if $\mu_1,\mu_2 \in\Pr(H)$ satisfy
$\int_H \phi(u)\mu_1(du) = \int_H \phi(u) \mu_2(du)$ for all
$\phi \in \mathcal{G}$, it follows that $\mu_1=\mu_2$.

With these preliminaries in hand, we are ready to state the abstract
result from \cite{GlattHoltzMattinglyRichards2015}.

\begin{Thm}[\cite{GlattHoltzMattinglyRichards2015}]\label{thm:GHMR}
  Suppose ${\til{\rho}}$ is a metric such that $\mathcal{G}_{{\til{\rho}}}$
  determines measures on $(H,\rho)$, and that
  $D_{{\til{\rho}}}\subseteq H^{\mathbb{N}}\times H^{\mathbb{N}}$ is
  measurable.  If, for each $u_0,v_0\in H$, there exists a corresponding
  $\Gamma_{u_0,v_0} \in \til{\mathcal{C}}(\delta_{u_0} P^{\mathbb{N}},\delta_{v_0}
  P^{\mathbb{N}})$ such that $\Gamma_{u_0,v_0}(D_{{\til{\rho}}})>0$, then there
  is at most one $P$-invariant probability measure $\mu\in\Pr(H)$.
\end{Thm}

Any such measure
$\Gamma_{u_0,v_0} \in\til{\mathcal{C}}(\delta_{u_0} P^{\mathbb{N}},\delta_{v_0}
P^{\mathbb{N}})$ for which a set
$D_{\til{\rho}}\subseteq H^\N\times H^\N$ exists is referred to as an
\textit{asymptotic coupling} for
$\delta_{u_0} P^{\mathbb{N}},\delta_{v_0} P^{\mathbb{N}}$. In the context of
SPDEs (and SDEs), \cref{thm:GHMR} provides a straightforward criteria
for establishing the uniqueness of an invariant measure.  Indeed, as
highlighted in \cite{GlattHoltzMattinglyRichards2015} (see also
\cite{ButkovskyKulikScheutzow2019}), to establish unique ergodicity
for a given SPDE, it suffices to identify a modified SPDE such that:
$(i)$ the laws of solutions to the new SPDE remain absolutely
continuous with respect to the laws of the original; and $(ii)$ there exists an infinite sequence of evenly spaced times such that for any pair of distinct initial conditions, there is a positive probability that solutions to
these systems converge at time infinity, possibly in a weaker
topology. Points $(i)$, $(ii)$ then constitute conditions for
establishing a so-called \textit{asymptotic coupling}, i.e., the
existence of a measure $\Gamma$ with the property stated in
\cref{thm:GHMR}.

For our application here, we consider $H=H^2$ with $\rho$ given by the
usual $H^2$ distance and take $\til{\rho}$ to denote the metric
induced by the $H^1$ norm on the set $H^2$.  In this setting one
indeed has that $\mathcal{G}_{\til{\rho}}$ determines measures on
$H^2$.  To see this we begin by recalling that $\mathcal{G}_\rho$
determines measures on $Pr(H^2)$; see e.g. \cite[Theorem
1.2]{billingsley2013convergence}. Next notice that any element
$\psi \in \mathcal{G}_{\rho}$ can be approximated pointwise by
$\psi_N(u) := \psi(P_Nu)$ where ${{}\psi_N\in \mathcal{G}_{\til{\rho}}}$ for
each $N \geq 1$.  To see that indeed
$\psi_N \in \mathcal{G}_{\til{\rho}}$ for each $N$, observe that, for
any $u \not= v \in H^2$, by invoking the reverse Poincar\'e
inequality, one obtains that
$|\psi_N(u) - \psi_N(v)| \leq C\|P_Nu - P_Nv\|_{H^2}\leq
CN\|u-v\|_{H^1}$ where
$C = \sup_{\bar{u} \not = \bar{v}}\frac{|\psi(\bar{u}) -
  \psi(\bar{v})|}{\rho(\bar{u},\bar{v})}$. Now suppose that we are
given two measures $\mu, \nu \in Pr(H^2)$ which agree on
$\mathcal{G}_{\til{\rho}}$, namely let us suppose that $\mu$ and $\nu$
are such that $\int \phi(u) \mu(du) = \int \phi(u) \nu(du)$ for every
$\phi \in \mathcal{G}_{\til{\rho}}$.  Then, in fact we have, for any
$\psi \in \mathcal{G}_{\rho}$, that
$\int \psi(u) \mu(du) = \lim_{N \to \infty}\int \psi(P_Nu) \mu(du) =
\lim_{N \to \infty} \int \psi(P_Nu) \nu(du) = \int \psi(u) \nu(du)$
where we invoke the bounded convergence theorem to justify the first
and last equalities.  Thus, in summary, if $\mu$ and $\nu$ agree on
$\mathcal{G}_{\til{\rho}}$ then indeed these two measure must also
agree on the larger determining set $\mathcal{G}_{\rho}$ and so we
conclude that $\mathcal{G}_{\til{\rho}}$ is itself determining.

For the damped stochastic KdV, \eqref{eq:s:KdV}, our candidate for the
modified system is constructed in two steps.  First, for any two
initial conditions $u_0, v_0 \in H^{2}$, we consider the shifted
stochastic KdV equation starting from $v_0$ given by
\eqref{eq:s:KdV:ng} from \cref{sect:FP:est} where the $u$ appearing in
the `nudging term' $P_N(v-u)$ obeys \eqref{eq:s:KdV} starting from
$u_0$.  Here, the integer $N$ and the `nudging strength' $\lambda$
appearing in \eqref{eq:s:KdV:ng} are parameters to be chosen.  Given
$K>0$, the modified system upon which we will build an asymptotic
coupling is then given by
\begin{align}
  dv^*
  + ( v^* {D} v^* +  {D}^3 v^* + \gam v^*) dt
  = -\lambda P_N(v^*-u)\indFn{t\leq \tau^*_K}dt + f dt
  + \sigma dW,  \quad v(0) = v_0,
     \label{eq:s:KdV:ng2}
\end{align}
where $\tau_K^*$ is the stopping time defined by
\begin{align}\label{def:ng:stoppingtime}
  \tau^*_K := \inf_{t \geq 0} 
  \left\{
  \int_0^t \|P_N (v^*(s) - u(s))\|_{L^2}^2ds \geq K
  \right\}.
\end{align}
and as with \eqref{eq:s:KdV:ng}, $u$ in the now cut-off nudging term 
solves \eqref{eq:s:KdV} starting from $u_0$.

For a fixed initial conditions $u_0, v_0\in H^2$, we denote by
$\Phi_{u_0}$ and ${\Phi^*_{v_0, u_0}}$ the measurable maps induced by
solutions of \eqref{eq:s:KdV} and \eqref{eq:s:KdV:ng2}, respectively,
mapping an underlying probability space
$(\Omega,\mathcal{F},\mathbb{P})$ to $C([0,\infty);H^2)$.  It follows
that the law of solutions of \eqref{eq:s:KdV} is given by
$\mathbb{P}\Phi_{u_0}^{-1}$, and the law associated to
\eqref{eq:s:KdV:ng2} is $\mathbb{P}(\Phi^*_{v_0,u_0})^{-1}$.

\begin{Prop}\label{prop:asymptotic:coupling}
  For any {$\lam,K,N>0$ and $\s\in\bH^2$ satisfying $\mathcal{R}(\s)\supset P_NL^2$}, the laws of solutions to
  \eqref{eq:s:KdV} and \eqref{eq:s:KdV:ng2} are equivalent, that is,
  for any $u_0, v_0 \in H^2$,
  $\mathbb{P}\Phi_{v_0}^{-1}\sim \mathbb{P}(\Phi^*_{v_0, u_0})^{-1}$
  (mutually absolutely continuous) as measures on $C([0,\infty);H^2)$.
  Furthermore, there exists a choice $\lam,N$, depending only on
  $\gam, \Sob{f}{H^2}, \Sob{\s}{H^2}$, such that for any $u_0,v_{0}\in H^2$, one has
    \begin{align}
        \mathbb{P}(\lim_{n\rightarrow \infty}\|v^*(n)-u(n)\|_{H^1}= 0)>0,
        \label{eq:dcy:pos:set}
    \end{align}
    where $u$ and $v^*$ are the solutions of \eqref{eq:s:KdV} and \eqref{eq:s:KdV:ng2}
    corresponding to $u_0$ and $v_0$, respectively.
\end{Prop}

We are now positioned to show that the criteria for unique ergodicity
are satisfied. In particular we will now show that
\cref{prop:asymptotic:coupling} implies that the hypotheses of
\cref{thm:GHMR} are verified, and hence, that
\cref{thm:ergodicity:ess:elliptic} is indeed true.

\begin{proof}[Proof of \cref{thm:ergodicity:ess:elliptic}]
  Let $P_{1}$ be the Markov kernel associated
  to \eqref{eq:s:KdV} via \eqref{eq:mark:kernel} at time $t=1$.  We
  aim to show that $P_{1}$ has at most one invariant measure.
  For any $u_0, v_0 \in H^2$ consider the measure $\Gamma_{u_0,v_0}$
  on $(H^2)^{\mathbb{N}}\times (H^2)^{\mathbb{N}}$ given by the law of
  the associated random vector
  $(u({n}),v^*({n}))_{n\in \mathbb{N}}$ where $u, v$
  solve \eqref{eq:s:KdV}, \eqref{eq:s:KdV:ng2} initialized at
  $u_0,v_0$.  Observe that, invoking \cref{prop:asymptotic:coupling}
  to obtain the equivalence of solution laws, we have
  $\Gamma_{u_0,v_0}
  \Pi_{2}^{-1}\ll\delta_{v_0}P_{1}^{\mathbb{N}}$, and in
  particular that
  $\Gamma_{u_0,v_0} \in
  \til{\mathcal{C}}(\delta_{u_0}P_{1}^{\mathbb{N}},\delta_{v_0}P_{1}^{\mathbb{N}})$,
  where we recall that $P_{1}^{\mathbb{N}}$ is the Markov
  transition kernel associated to $P_{1}$ on $H^{\mathbb{N}}$.
  Furthermore, recalling that ${\til{\rho}}(u,v) := \|u-v\|_{H^1}$, it
  follows from the definition of $\Gamma$ and
  \cref{prop:asymptotic:coupling} that, for a suitable choice of
  $\lam, N$ defining $v^*$, we have
    \begin{align*}
      \Gamma (D_{{\til{\rho}}})
      = \mathbb{P}\left((u({n}),v^*({n})\right)_{n\in \mathbb{N}}
                     \in D_{{\til{\rho}}}\big)
      = \mathbb{P}\left(\lim_{n\rightarrow \infty}
                  \|v^*({n})-u(n)\|_{H^1}= 0\right)>0.
    \end{align*}
    
    Upon observing that the test functions $G_{{\til{\rho}}}$
    determine measures on $H^2$, an application \cref{thm:GHMR} then
    implies that, there is at most one
    probability measure in $\Pr(H^2)$ that is invariant for
    $P_{1}$.  This in turn means that any two invariant
    measures for the entire semigroup $\{P_t\}_{t \geq 0}$ must 
    coincide since they are both invariant measures for $P_{1}$. 
    The proof of \cref{thm:ergodicity:ess:elliptic} is complete.

\end{proof}

We are therefore left to prove \cref{prop:asymptotic:coupling}.  We do
so by appealing to the Girsanov theorem for absolute continuity, and to the Foias-Prodi and Lyapunov structure established in previous sections for proving asymptotic coupling. 
Indeed, with the choice of stopping time in \eqref{eq:FP:stoppingtime:bound}, we invoke \cref{cor:FP:est} to establish an asymptotic coupling for $\{(u(n),v^*(n))\}_{n>0}$, bounding the probability of convergence from below as described in \eqref{eq:dcy:pos:set} by invoking \cref{lem:FP:stoppingtime} in combination with our Lyapunov structure, and further argumentation involving the Borel-Cantelli lemma.

\begin{proof}[Proof of \cref{prop:asymptotic:coupling}]
We begin by establishing the equivalence of solution laws. 
We define a ``Girsanov shift'' by
\begin{align*}
    a(t):=\sigma^{-1}\lambda P_{N}(v^*-u)\indFn{t\leq \tau^*_K},
\end{align*}
where $\tau_K^*$ is given by \eqref{def:ng:stoppingtime}, and let
\begin{align}\label{def:W:star}
    W^*(t) := W(t)-\int_0^t a(s)ds.
\end{align}
Then consider the continuous martingale
\begin{align*}
    \Mt(t):=\exp\left(\int_0^t a(s)dW - \frac{1}{2}\int_0^t\|a(s)\|_{L^2}^{2}ds\right).
\end{align*}
Observe that $a(t)$ satisfies the Novikov condition \cite[Proposition
1.15]{RevuzYor1999}:
\begin{align*}
    \E\exp\left(\frac{1}{2}
    \int_0^{\infty}\|a(s)\|_{L^2}^{2}ds\right)
    = \E\exp\left(\frac{1}{2}
    \int_0^{\tau^*_K}\|a(s)\|_{L^2}^{2}ds\right)
    \leq e^{\frac{1}{2}K} < \infty.
\end{align*}
{{}Note that we have invoked the boundedness of the pseudo-inverse $\s^{-1}$ of $\s$, which holds owing to the fact that $\mathcal{R}(\s)\supset P_NL^2$}. In particular, we deduce that $\{\Mt(t)\}_{t\geq 0}$ is uniformly integrable.  Thus,
by the Girsanov theorem \cite[Propositions 1.1-1.4]{RevuzYor1999}
there is a probability measure $\mathbb{Q}$ on
$(\Omega,\mathcal{F},(\mathcal{F})_{t\geq 0})$ such that the
Radon-Nikodym derivative of $\mathbb{Q}$ with respect to $\Prb$ on
$\mathcal{F}_t$ is $\Mt(t)$, and moreover $W^*$ is a standard Wiener
process with respect to $\mathbb{Q}$.  In particular,
$\Prb \sim \mathbb{Q}$ on
$\mathcal{F}_{\infty}=\sigma(\cup_{t\geq 0}\mathcal{F}_t)$.

Next, let us denote by $u=u(t;u_0,W)$ and $v^* = v^*(t;v_0,u_0,W)$ the
solutions to \eqref{eq:s:KdV} and \eqref{eq:s:KdV:ng2} corresponding
to Brownian motion $W$. By uniqueness of solutions
(cf. \cref{prop:exist:uniq}) and the definition of $W^*$ in
\eqref{def:W:star} above, one sees immediately that
$v^*(t;v_0, u_0, W) = u(t;v_0,W^*)$ almost surely with respect to
$\Prb$ or $\mathbb{Q}$. Thus, for any
$\mathcal{F}_{\infty}$-measurable $A\subseteq C([0,\infty);H^2)$, we
have
\begin{align*}
    \Prb(\Phi^*_{v_0,u_0})^{-1}(A) 
    = \Prb\left(v^*(\cdot;v_0,u_0,W)\in A\right)
    = \Prb\left(u(\cdot;v_0,W^*)\in A\right).
\end{align*}
On the other hand, since $W^*$ is a standard Wiener process with
respect to $\mathbb{Q}$, we have
\begin{align*}
    \Prb\Phi_{v_0}^{-1}(A) 
    = \Prb\left(u(\cdot;v_0,W)\in A\right)
    = \mathbb{Q}\left(u(\cdot;v_0,W^*)\in A\right).
\end{align*}
Hence, from the equivalence $\Prb \sim \mathbb{Q}$ on
$\mathcal{F}_{\infty}$, it follows that
$\Prb\Phi_{v_0}^{-1}\sim \Prb(\Phi^*_{v_0,u_0})^{-1}$ on
$C([0,\infty);H^2)$, as claimed.

It remains to establish that, with a positive probability, solutions
of \eqref{eq:s:KdV} and \eqref{eq:s:KdV:ng2} emanating from distinct
initial conditions converge to each other on any evenly spaced
infinite sequence of time points.  For this, let $w:=v-u$, where $v$
is the corresponding solution of \eqref{eq:s:KdV:ng}.  We define
\begin{align*}
  B_{n} :=
  \left\{  \| w(n) \|_{H^1}^2 
      + \gam\int_{n}^{{{}n+1}} \| w(s) \|_{L^2}^2 ds
  \geq \frac{1}{n^2}  \right\}.
\end{align*}
Then set
\begin{align}
  B := \bigcap_{m =1}^\infty \bigcup_{n = m}^\infty B_n.
  \label{eq:good:bad:ass:cup}
\end{align}

Next we fix any suitably large values of $N, \lambda \geq 1$ in
\eqref{eq:s:KdV:ng}, \eqref{eq:s:KdV:ng2} so that \cref{cor:FP:est}
holds and consider the stopping times $\tau_{R,\beta}$ defined in
\eqref{def:FP:stoppingtime}.  With Fubini's theorem, we immediately
infer that
\begin{align*}
 & \E \left(\! \indFn{\tau_{R,\beta} = \infty}\biggl( \| w(n) \|_{H^1}^2 
      + \gam\int_{n}^{n+1} \! \! \! \!\! \! \!\| w(s) \|_{L^2}^2 ds \biggr)\! \right)
      \! =\E (\indFn{\tau_{R,\beta} = \infty}\| w(n) \|_{H^1}^2)
      + \gam\int_{n}^{n+1} \! \!\! \! \!\E (\indFn{\tau_{{{}R,\beta}} = \infty} \| w(s) \|_{L^2}^2) ds
      \notag\\
      &\leq C(1 + \gam)
      \exp\left(R+\be-\frac{\gamma}{2} n \right)
       \left(\| w_0\|_{H^1}^2+\Sob{w_0}{L^2}^{10/3}+(1+\Sob{u_0}{L^2}^2)\Sob{w_0}{L^2}^2\right),
\end{align*}
for any $R, \beta > 0$.
Hence with the Markov inequality we have
\begin{align*}
  &\Prb( B_n \cap \{ \tau_{R,\beta} = \infty\}) \notag\\
  &\leq  C(1 + \gam)\exp\left({{}R+ \beta}\right)
    \left(\| w_0\|_{H^1}^2+\Sob{w_0}{L^2}^{10/3}
    +(1+\Sob{u_0}{L^2}^2)\Sob{w_0}{L^2}^2\right)
    n^2e^{-\frac{\gamma}{2}n},
\end{align*}
so that
$\sum_{n = 1}^\infty \Prb( B_n \cap \{ \tau_{R,\beta} = \infty\}) <
\infty$.  As such, according to the first Borel-Cantelli lemma,
$\Prb(B \cap \{\tau_{R,\beta} = \infty\}) = 0$, for any $\beta, R >0$,
where $B$ is given in \eqref{eq:good:bad:ass:cup}.

Thus, along with the estimate \eqref{eq:FP:stoppingtime:bound} from
\cref{lem:FP:stoppingtime}, we can select a value of
$\beta = \beta(u_0,v_0)$ so that, for any $R> 0$ we have
\begin{align*}
\Prb(B)=\Prb(B\cap\{\tau_{R,\beta} < \infty\}) 
  \leq  \Prb(  \tau_{R,\beta} < \infty) \leq  \frac{C}{R},
\end{align*}
where
$C = C(\gam, \lambda, \| u_0\|_{H^2},\Sob{f}{H^2},\Sob{\s}{H^2})$ is
independent of $R$.  Upon choosing
$R^*=R^*(\gam,\lambda, \Sob{u_0}{H^2},\Sob{f}{H^2},\Sob{\s}{H^2})$
sufficiently large, we may therefore invoke continuity from below to
choose $m^*>0$ sufficiently large so that
\begin{align}
  \Prb\left(\bigcap_{n = m^*}^\infty B_{n}^c\right) >  \frac{1}{2}.
  \label{eq:good:set:Ass:cup:1}
\end{align}
On the other hand, for this fixed value of $m^*$, we define
\begin{align*}
  E_{R} := \left\{ \int_0^{{{}m^*}} \|P_N w(s)\|_{L^2}^2 ds \geq R\right\}.
\end{align*}
Upon combining the Markov inequality with \eqref{eq:Lyapunov:Bnd} from
\cref{thm:Lyapunov} and \eqref{eq:Lyapunov:Bnd:Hm:ng} from
\cref{thm:Lyapunov:ng}, we obtain $\Prb(E_{R})\leq C{{}m^*}R^{-1}$ for
some constant $C=C(\gam,\|u_0\|_{H^2},\|v_0\|_{L^2})$.  It then
follows from \eqref{eq:good:set:Ass:cup:1} that, upon taking $R^*>0$
possibly larger
\begin{align}\notag
  \Prb\left(\bigcap_{n = m^*}^\infty B_{n}^c \cap E_{R^*}^c\right)>  \frac{1}{4}.
\end{align}
Furthermore, on the set $\bigcap_{n = m^*}^\infty B_{n}^c \cap E_{R^*}^c$, we find
\begin{align*}
    \int_0^\infty \|P_N w(s)\|_{L^2}^2 ds  
    \leq {R^*} + {{}\frac{1}\gam}\sum_{n=m^*}^{\infty}\frac{1}{n^2}:= K.
\end{align*}
It follows that
$\bigcap_{n = m^*}^\infty B_{n}^c \cap E_{R^*}^c\subseteq \{\tau^*_K =
\infty\}$, that is, $v(t)=v^*(t)$, for all $t\geq0$, on
$\bigcap_{n = m^*}^\infty B_{n}^c \cap E_{R^*}^c$. Since $w=v-u$, it
follows that
\begin{align*}
    \Prb\left(\lim_{n\rightarrow \infty}
    \|v^*({ n})-u({n})\|_{H^1}= 0\right)
    \geq \Prb\left(\bigcap_{n = m^*}^\infty B_{n}^c \cap E_{R^*}^c\right)>0,
\end{align*}
as claimed in \eqref{eq:dcy:pos:set}.  The proof of
\cref{prop:asymptotic:coupling} is now complete.
\end{proof}

\begin{Rmk}\label{rmk:ue:proof}
{{}We point out that the definition of $B_n$ above requires both $ \| w(n) \|_{H^1}^2$ and $\int_{n}^{{{}n+1}} \| w(s) \|_{L^2}^2 ds$. The first quantity is required to enforce convergence, while the second quantity is required to help enforce the Novikov condition. Indeed, we see that over the set $\bigcap_{n=m^*}^\infty B_n^c$, one has $\Sob{w(n)}{H^1}<n^{-2}$, for all $n\geq m^*$, due to the first quantity, so that $\lim_{n\goesto\infty}\Sob{w(n)}{H^1}=0$. On the other hand, the second quantity ensures $\int_{m^*}^\infty\Sob{w(s)}{L^2}^2ds=\sum_{n=m^*}^\infty\int_n^{n+1}\Sob{w(s)}{L^2}^2ds<\sum_{n=m^*}^\infty n^{-2}$, which is arbitrarily small for $m^*$ sufficiently large.
}
\end{Rmk}

\begin{Rmk}\label{rmk:triv:mod}
  The proof of \cref{thm:ergodicity:ess:elliptic} can be adjusted with
  trivial modifications to apply to the semigroup $\{P_t\}_{t \geq 0}$
  for \eqref{eq:s:KdV} restricted to $H^m$ for any $m > 2$.  However,
  for this restriction of $\{P_t\}_{t \geq 0}$ to $H^m$ to be well
  defined, we need that $f \in H^m$ and $\sigma \in \bH^m$. Now, under
  the assumption of such smoother source terms, consider $\mu^{(2)}$ and
  $\mu^{(m)}$ to be the unique measures for $\{P_t\}_{t \geq 0}$
  defined on $H^2$ and $H^m$, respectively.  We know from
  \eqref{thm:regularity} that, in fact $\mu^{(2)}(H^m) =1$.  In turn
  this means $\mu^{(2)}$ is actually invariant for our semigroup
  restricted to this smoother class of functions $H^m$ and by uniqueness we conclude that $\mu^{(2)} = \mu^{(m)}$.
\end{Rmk}

\section{Results in the Case of a Large Damping}
\label{sect:Large:Damping}

In this section, we present our mixing results for the large damping
regime which we previewed above in \cref{thm:large:damp:Mix:Intro}.
These results reflect the fact that, for $\gamma$ sufficiently large,
any two solutions of \eqref{eq:s:KdV} emanating from different initial
data asymptotically synchronize in expectation, indicative of a `one-force one-solution'
phenomenon. Before proceeding to the complete and rigorous formulation of
\cref{thm:large:damp:Mix:Intro}, given as \cref{thm:spectralgap}
below, we first exhibit some crucial preliminary estimates.  Note that
throughout this section we refine our convention and
insist that all constants $C$ which appear are taken to be
independent of the damping parameter $\gamma \geq 1$.  As
previously constants labeled $c$ depend solely on universal, that is
equation independent, quantities.

Our mixing analysis proceeds starting from an estimate on the $L^2$ of
the difference of solutions.  Given any $u_0, \tilde{u}_0$ we take
$u$ and $\tilde{u}$ to be the corresponding solutions of
\eqref{eq:s:KdV} and let $w = u-\tilde{u}$. Observe that $w$ satisfies
\begin{align*}
    \bdy_t w+\gam w+{D}^3w+\frac{1}{2}{D}(w^2)=-{D} (\tilde{u} w).
\end{align*}
Upon taking the $L^2$ inner product with $w$ and integrating by parts
and noting that $\lb {D}w,w^2\rb = 0$, we obtain
\begin{align}\label{L2:stab:balance}
  \frac{1}{2}\frac{d}{dt}\Sob{w}{L^2}^2
  +\gam\Sob{w}{L^2}^2
  =-\frac{1}{2}\lb {D}\tilde{u},w^2\rb
  = -\frac{1}{4}\lb {D}(u +  \tilde{u}),w^2\rb.
\end{align}
Take $\bar{u} = \frac{1}{4}(u + \tilde{u})$ and observe that by
Agmon's inequality and interpolation
\begin{align}
    |\lb {D} \bar{u},w^2\rb|
    &\leq\Sob{{D}\bar{u}}{L^\infty}\Sob{w}{L^2}^2
    \leq c\Sob{{D}^2\bar{u}}{L^2}^{1/2}
      \Sob{{D}\bar{u}}{L^2}^{1/2}\Sob{w}{L^2}^2
      \leq c\Sob{{D}^2\bar{u}}{L^2}^{3/4}
         \Sob{\bar{u}}{L^2}^{1/4}\Sob{w}{L^2}^2.
      \label{L2:stab:interpolate}
\end{align}
Thus, upon combining \eqref{L2:stab:balance} and
\eqref{L2:stab:interpolate} and then applying Gr\"onwall's inequality,
we find, for any $u_0, \tilde{u}_0$ and $t \geq 0$
\begin{align}
  \label{eq:stab:main}
  \| u(t)-\tilde{u}(t)\|_{L^2}
  \leq \Sob{u_0-\til{u}_0}{L^2}
  \exp\left(- \gamma t
  + c  \int_0^t\Sob{{D}^2(u+\tilde{u})}{L^2}^{3/4}
             \Sob{u+\tilde{u}}{L^2}^{1/4}ds
  \right),
\end{align}
where, to emphasize, $c> 0$ depends only on universal quantities.

At this point we would like to combine \eqref{eq:stab:main} with the
bounds in \cref{thm:Lyapunov} in order to obtain an exponential in
time decay for $\E\| u(t)-\tilde{u}(t)\|_{L^2}$.  To this end notice
that \eqref{eq:exp:est:asymptotic:L2} and \eqref{eq:elem} yield that,
for any $\gamma > 0$,
\begin{align}
  \E \exp\left( \frac{\gamma^2}{8 \|\sigma\|^2_{L^2}}  \int_0^t \|u\|_{L^2}^2 ds\right) 
  \leq 2 \exp\left( \|u_0\|_{L^2}^2 +
           \left(\frac{2}{\gamma} \|f\|^2_{L^2} + \|\sigma\|^2_{L^2}\right)t\right).
           \label{eq:exp:L2}
\end{align}
On the other hand we have no reason to expect finite exponential
moments for $\int_0^t \|u\|^p_{L^2} ds$ when $p> 2$.  As such we
proceed by estimating the integrating factor in \eqref{eq:stab:main}
as
\begin{align}
 \| D^2(u+\tilde{u})\|_{L^2}^{3/4}\|u+\tilde{u}\|_{L^2}^{1/4}
    \leq   c(\| D^2(u+\tilde{u})\|_{L^2}^{6/7} + \|u+\tilde{u}\|_{L^2}^{2}).
    \label{eq:int:fact:L2:LD:sp}
\end{align}
Thus, in order to leverage the bounds \eqref{eq:stab:main},
\eqref{eq:int:fact:L2:LD:sp} to produce the desired exponential decay,
and in view of the approach in \cref{sect:Lyapunov} we proceed next to
establish suitable exponential moments for $\In_2^+(u)^{3/7}$ which in
turn controls $\|u\|^{6/7}_{H^2}$ according to \cref{lem:equivalence}.

While we have already established estimates like
\eqref{eq:Lyapunov:Bnd:exp}, \eqref{eq:Lyapunov:Bnd:subquadexp} for
$\In_2^+(u)$ above, here we need
to be able to ensure suitable exponential moment bounds for
$\int_0^t \In_2^+(u)^{p} ds$ precisely when $p = 3/7$ and to more
carefully track bounds for this quantity in terms $\gamma$
dependencies.  Referring back to \eqref{energy:funct:0:1} and
\eqref{eq:mod:int:m:simple} and motivated by \eqref{eq:exp:L2} 
we now explicitly define
\begin{align}
  \In_2^+(v)
  := \int_\T \left((D^2v)^2 - \frac{5}{3} u (Du)^2
            + \frac{5}{36} u^4 \right) dx
  + \bar{\alpha}_2(\|v\|_{L^2}^2 + 1)^{7/3}.
  \label{eq:I2p:exp:LD}
\end{align}
Crucially we need to justify the specific $7/3$ power on the
$L^2$ norm in the functional so that \cref{lem:equivalence} holds for
this specific choice of $\bar{q}_2$ in $\In_2^+$ namely that
\begin{align}
  \frac{1}{2}( \|v\|_{H^2}^2
  +  \bar{\alpha}_2(\|v\|_{L^2}^{2} +1)^{7/3}) \leq \In_2^+(v)
  \leq \frac{3}{2}( \|v\|_{H^2}^2
  +    \bar{\alpha}_2(\|v\|_{L^2}^{2} +1)^{7/3})
  \label{eq:I2p:exp:LD:eqv}
\end{align}
is maintained when we fix $\bar{\alpha}_2$ to be any value greater
larger than a universal quantity.  Indeed, with the Agmon and Sobolev
inequalities and interpolation we observe that
\begin{align*}
  \left|\int \left(\frac{5}{3} v (Dv)^2 -  \frac{5}{36} v^4\right) dx \right|
  \leq& c (\|v\|_{L^\infty} \|Dv\|^2_{L^2} + \|v\|_{L^4}^4)
  \leq c (\|v\|_{L^2}^{1/2}\|Dv\|^{5/2}_{L^2} + \|v\|_{L^2}^{3}\|Dv\|_{L^2})\\
  \leq& c(\|v\|_{L^2}^{7/4}\|D^2v\|^{5/4}_{L^2} + \|v\|_{L^2}^{7/2}\|D^2v\|_{L^2}^{1/2})
  \leq  \frac{1}{2}\|v\|^2_{L^2}  +  c\|v\|^{14/3}_{L^2}.
\end{align*}
which then yields \eqref{eq:I2p:exp:LD:eqv} for any $\bar{\alpha}_2$ greater
the universal constant $c > 0$.

Let us next observe that since we do not have the same flexibility in
choose of the $L^2$ based portion of $\In_2^+$ as we did above in
\cref{sect:Lyapunov}, a second consideration enters in our choice of
the parameter $\bar{\alpha}_2 > 0$ in \eqref{eq:I2p:exp:LD}. Namely we
select $\bar{\alpha}_2$ to effectively weight a critical term that
appears when we compute the evolution for $\In_2^+$.  Let us make this
explicit.  As we derived above in \cref{sect:time:evol},
$\In_2^+(u(t))^{p}$ for $p > 0$ obeys \eqref{eq:Imp:mom:evol}.  Here,
comparing \eqref{energy:funct:0:1} against \eqref{eq:Im:LHS:mess:D},
\eqref{eq:Inp:simp:RHS:D} we obtain the explicit expressions
\begin{align}
    \Kn_2^{D,+} &:= \gamma \Kn_{2,1}^D + \Kn_{2,2}^D +  \Kn_{2,3}^D
    - {\bar{\al}_2} \gamma \frac{8}{3} \left( \|u\|_{L^2}^2+1\right)^{7/3}+ {\bar{\al}_2}
   \frac{7}{3}\left( \|u\|_{L^2}^2+1\right)^{4/3}(2\gam+2 \lb f , u\rb + \|\sigma\|^2_{L^2})   
                   \notag\\
   &\quad + {\bar{\al}_2}\frac{56}{9}\left( \|u\|_{L^2}^2+1\right)^{1/3}
     |\lb \sigma, u \rb|^2.
    \label{eq:In:Kp:LD:p}
\end{align}
where
\begin{align}
  \Kn_{2,1}^D(u)&:= \int_\T\left( \frac{5}{18} u^4 -\frac{5}{3}u({D}u)^2\right)dx, 
  \notag\\
  \Kn_{2,2}^D(u,\s)&:=\int_\T\left(|{D}^2\s|^2
                  -\frac{10}{3}{D}u({D}\s)\cdotp\s-\frac{5}{3}u|{D}\s|^2
                   + \frac{5}{3} u^2 |\sigma|^2 \right)dx,
  \notag\\
  \Kn_{2,3}^D(u,f)&:=\int_\T\left(2{D}^2u{D}^2f- \frac{5}{3}\left({D}u\right)^2f
                    -2u{D}u{D}f+\frac{5}{9}u^3{D}f\right)dx.
        \label{eq:Ln:K:LD:123}
\end{align}
Meanwhile comparing \eqref{energy:funct:0:1} with
\eqref{eq:Im:LHS:mess:S}, \eqref{eq:Inp:simp:RHS:S} yields
  \begin{align}
    \mathcal{K}_2^S=\int_{\T} \biggl(  2 {D}^2 u {D}^2\sigma
    -  \frac{5}{3} ({D} u)^2 \sigma
   - 2u {D} u {D} \sigma
    + \frac{5}{9} u^3\sigma \biggr)dx
    +  {\abarm} \frac{14}{3} \left(\|u\|_{L^2}^2+1\right)^{4/3}\lb u , \sigma \rb.
   \label{eq:Ln:K:LD:s}
\end{align}
Now, perusing \eqref{eq:In:Kp:LD:p}, \eqref{eq:Ln:K:LD:123},
\eqref{eq:Ln:K:LD:s} and comparing against \eqref{eq:Imp:mom:evol} we
see that the critical term needed to close estimates in the evolution
for $(\In_2^+(u))^p$ is $\Kn^{D}_{2,1}$.  Here, observe that the Agmon and
Sobolev inequalities, interpolation yield
\begin{align*}
  |\Kn^{D}_{2,1}(u)|
  \leq& c (\|u\|_{L^4}^4 + \|u\|_{L^\infty}\|Du\|_{L^2}^2)
     \leq c( \|u\|_{H^{1/4}}^4 + \|u\|_{L^2}^{1/2}\|Du\|_{L^2}^{5/2})
    \leq c (\|u\|^{7/2}_{L^2} \| u\|^{1/2}_{H^2} 
   +\|u\|^{7/4}_{L^2} \| u\|^{5/4}_{H^2}) \\
   \leq& \frac{1}{1000} \|u\|_{H^2}^2 + c \|u\|^{14/3}_{L^2},
\end{align*}
where to emphasize $c> 0$ depends only on universal quantities.  Thus
we can make a choice of $\bar{\alpha}_2$ in \eqref{eq:I2p:exp:LD}
determined exclusively by universal quantities such that both
\eqref{eq:I2p:exp:LD:eqv} and
\begin{align}
   |\Kn^{D}_{2,1}(u)|  \leq \frac{1}{1000}\In_2^+(u),
 \label{eq:LD:crit:al:s}
\end{align}
hold simultaneously.

With these preliminaries in hand we are now ready to state the main
result of this section.  Our convergence results are carried out with
respect to a suitable Wasserstein-like distance. To formulate this,
let us define for $k\geq0$ and $\eta_0 > 0$
\begin{align}\label{def:distancelike}
  d_k(u,v):=\exp\left(\eta_0 \left(\In_2^+(u)^{3/7}+ \In_2^+(v)^{3/7} \right)\right)
       \Sob{u-v}{H^k}.
\end{align}
Here, our definition $\bar{\alpha_2} >0$ specifying $\In_2^+$ is
determined independently of $\gamma > 0$ so that \eqref{eq:I2p:exp:LD},
\eqref{eq:LD:crit:al:s} both hold.  Note furthermore that $d_k$ is a
distance-like function on $H^m\times H^m$ for $0 \leq k \leq m$ since
it is symmetric, lower semi-continuous, and satisfies $d_k(u,v)=0$ if
and only if $u=v$. Given $\mu,\nu\in\Pr(H^m)$, the coupling distance
corresponding to $d_k$, for $0\leq k\leq m$, is given by
\begin{align}\label{def:Wd}
  \Wm_{d_k}(\nu_1,\nu_2)
  :=\inf_{\pi\in\mathcal{C}(\nu_1,\nu_2)}\int_{H^m\times H^m}d_k(u,v)\pi(du,dv),
\end{align}
where $\mathcal{C}(\nu_1,\nu_2)$ denotes the set of all couplings of
$\nu_1,\nu_2$, that is, all the measures on $H^m\times H^m$ with marginals given
by $\nu_1$ and $\nu_2$, respectively.

The main result of this section is now stated as follows:
\begin{Thm}\label{thm:spectralgap}
  Fix any $f \in H^2$, $\sigma \in \bH^2$. For such $f$ and $\sigma$
  and any $\gamma > 0$ take $\{P_t\}_{t \geq 0}$ to be the Markov
  semigroup \eqref{eq:mark:kernel} for \eqref{eq:s:KdV} corresponding
  to this data.  Then, there exist
  $\gam_0=\gam_0(\Sob{f}{H^2},\Sob{\s}{H^2}) \geq 1$ and a rate
  $\eta_0= \eta_0 (\Sob{f}{H^2},\Sob{\s}{H^2}) > 0$ defining
  $\Wm_{d_k}$ for $0 \leq k < m$ via \eqref{def:Wd},
  \eqref{def:distancelike} such that
  \begin{itemize}
  \item[(i)] for any $\gamma \geq \gamma_0$ and any $t \geq 0$,
  \begin{align}\label{eq:spectral:gap}
    \Wm_{d_0}(\nu_1 P_t,\nu_2 P_t)\leq 2e^{-\frac{\gam t}{2}} \Wm_{d_0}(\nu_1,\nu_2),
  \end{align}
  for any $\nu_1,\nu_2 \in\Pr(H^2)$.  
\item[(ii)] If we furthermore assume that
  $f \in H^m, \sigma \in \bH^m$ for some $m \geq 2$ then, for each
  $0<k<m$,
  \begin{align}\label{eq:mixing:Hm}
    \! \!  \! \! \!   \! \!  \!   \!   \! \!
    \mathfrak{W}_{d_k}(\nu_1 P_t,\nu_2 P_t)\leq Ce^{-\frac{m-k}{2m}\gam t} \!
    \int_{\substack H^m}
    \int_{\substack H^m} \! \! \! \!  \! \!
    (1+ \Sob{v}{H^m} + \Sob{\tilde{v}}{H^m})
    e^{2\eta_0 (\In_2^+(v)^{3/7} +  \In_2^+(\tilde{v})^{3/7})}
    \nu_1 (dv)
    \nu_2(d\tilde{v}),
  \end{align}
  holds for some constant $C = C(m,\gam,\Sob{f}{H^m},\Sob{\s}{H^m})$
  which is independent of $\nu_1, \nu_2$.
\item[(iii)] There exists a unique invariant measure
  $\mu \in \mbox{Pr}(H^2)$.  Additionally, so long as
  $f \in H^m, \sigma \in \bH^m$ for some $m \geq 2$, we have, for each
  $0 \leq k < m$
  \begin{align}\label{eq:mixing:Hm:2}
    \mathfrak{W}_{d_k}(\mu,\nu P_t)
    \leq Ce^{-\frac{m-k}{2m}\gam t} \left(1 + \int_{\substack H^m} \Sob{v}{H^m}
       e^{2\eta_0 \In_2^+(v)^{3/7}}
    \nu (dv)\right),
  \end{align}
  for any $\nu$ where $C = C(m,\gam,\Sob{f}{H^m},\Sob{\s}{H^m})$ is independent of $\nu$ and $t$.
\end{itemize}
\end{Thm}

\begin{Rmk}
\label{rmk:spectralgap}
\cref{thm:spectralgap} has immediate implications for
observables of solutions of \eqref{eq:s:KdV}.
Indeed it is easily seen from \eqref{def:Wd} that
\begin{align}\label{eq:partial:duality}
  \left|\int\phi(v)\nu_1(dv)-\int \phi(v)\nu_2(dv)\right|
  \leq \Sob{\phi}{\Lip(d_k)}\mathfrak{W}_{d_k}(\nu_1,\nu_2),
\end{align}
holds for all $\phi\in \Lip(d_k)$ where
\begin{align*}
  \mbox{Lip}(d_k) =\left\{\phi:H^k\rightarrow \mathbb{R}\,\Bigg|\,\Sob{\phi}{\Lip(d_k)}:
  = \sup_{u \not = v\in H^k}\frac{|\phi(u)-\phi(v)|}{d_k(u,v)}<\infty\right\}.
\end{align*}
Thus, by applying \eqref{eq:mixing:Hm:2} with $\nu=\de_{u_0}$ and then
invoking \eqref{eq:partial:duality}, we obtain the formulation in
\cref{thm:large:damp:Mix:Intro} from the introduction which thus holds over
a rich class of observables $\phi\in \Lip(d_k)$ for
$0 \leq k \leq m-1$.\footnote{Here, as an illustrative example,
  observe that if $\phi: H^2 \to \RR$ is such that, for every
  $u, v \in H^2$,
  $|\phi(u) - \phi(v)| \leq C(\|u\|^q_{H^2} + \|v\|^{q}_{H^2}) \|u
  -v\|_{H^2}$ for some $q \geq 0$ then $\phi \in \mbox{Lip}(\rho_2)$.
  Hence in particular the $m$-point correlation functions
  $\phi(u) := u(x_1) u(x_2) \ldots u(x_m)$ for any
  $x_1, \ldots x_m \in \TT$ all belong to $\mbox{Lip}(\rho_2)$.}
Moreover note that `spectral-gap' estimates as in
\eqref{eq:spectral:gap} can be used to produce law of large numbers
and central limit theorem type results concerning observables of $u$ 
obeying \eqref{eq:s:KdV}.  See e.g. \cite{Komorowski2012a,
  Komorowski2012b, KuksinShirikyan2012, Kulik2017,
  GlattHoltzMondaini2020} for further details on this point.
\end{Rmk}

Before turning to the proof of the main result and in view of
\eqref{eq:stab:main}, \eqref{eq:int:fact:L2:LD:sp},
\eqref{eq:I2p:exp:LD} we take the preliminary step of revisit the
proof of \cref{thm:Lyapunov}. The following lemma provides the
desirable refinements of \eqref{eq:Lyapunov:Bnd:exp},
\eqref{eq:Lyapunov:Bnd:subquadexp} as needed below:
\begin{Lem}
  \label{lem:LD:exp:mom}
  Fix any $f \in H^2, \sigma \in \bH^2$, $\gamma \geq 1$ and $u_0 \in H^2$. Let $u$
  be the corresponding solution of \eqref{eq:s:KdV}.  Then there exists an
  $\eta^* = \eta^*(\|\sigma\|_{H^2}) > 0$ independent of $\gamma \geq 1$
  \begin{align}\label{eq:expmoments:I2plusform}
    \E \exp\left(\eta\,\In_2^+(u(t))^{3/7}
    +\frac{3 \eta\gam}{7}\int_0^t\In_2^+(u(s))^{3/7}ds\right)
    \leq 2\exp\left(\eta\left(\In_2^+(u_0)^{3/7}+ C\gam t\right)\right).
  \end{align}
  which holds for any $0 \leq \eta \leq \eta^*$ and
  where the constants $C = C( \|f\|_{H^2}, \|\sigma\|_{H^2})$ are
  independent of $\gamma \geq 1$ and $t \geq 0$ and $\eta \geq 0$.\footnote{Here, 
  to reiterate, $\In_2^+$ is defined according to \eqref{eq:I2p:exp:LD}
  with $\bar{\alpha}_2 > 0$ choosen independently of $f, \sigma, \gamma$
  such that \eqref{eq:Ln:K:LD:s}, \eqref{eq:LD:crit:al:s} both hold.  Furthermore
  it is not hard to see that $\eta^*$ can be specified as $c \|\sigma\|^{-1}_{H^2}$
  where $c$ is a universal constant.}
  Furthermore, for any $t \geq 0$,
  \begin{align}\label{eq:expmoments:I2plusform:subquad}
        &\E \exp\left(\eta\,\In_2^+(u(t))^{3/7}\right)
    \leq 
     C\exp\left(\eta e^{-\frac{3}{7}\gam t}\In_2^+(u_0)^{3/7}\right).
  \end{align}
  where again $C = C( \|f\|_{H^2}, \|\sigma\|_{H^2})$ is
  independent of $\gamma \geq 1$, $t > 0$ and $0\leq \eta \leq 	\eta^*$.
\end{Lem}

\begin{Rmk}\label{rmk:large:damping:scaling}
{{}In light of \eqref{eq:I2p:exp:LD} and \cref{lem:LD:exp:mom}, the exponent $3/7$ plays a distinguished role in determining a lower bound on the size of $\gam$ to ensure the existence of a spectral gap (see \cref{thm:spectralgap} (i)). This exponent arises as the natural exponent for controlling $\bdy_xu$ in $L^\infty$ by $\bdy_x^2u$ in $L^2$ and $u$ in $L^2$; this can be seen from \eqref{eq:stab:main} and \eqref{eq:int:fact:L2:LD:sp}. 

On the other hand, for entirely separate reasons, the conserved quantity $\In_2$ can be controlled entirely in terms of $\Sob{u}{H^2}^2$ and $\Sob{u}{L^2}^{14/3}$. Since the KdV has been written in non-dimensional variables, we may add these two quantities and observe that $\Sob{u}{H^2}^{6/7}+\Sob{u}{L^2}^2\sim(\Sob{u}{H^2}^2+\Sob{u}{L^2}^{14/3})^{3/7}$. From this point of view, the KdV obeys a ``nonlinear scaling'' in the sense that $\Sob{u}{L^2}^2$ scales like $\Sob{u}{H^2}^{6/7}$;  morally speaking, in terms of the corresponding conserved quantity, $\Sob{u}{L^2}^2$ scales like $\In_2^{3/7}$. 

In other words, the existence of a spectral gap is possible when damping is sufficiently large precisely for this reason: the quantity that damping is required to control to ensure global dissipativity of the system is ``critical'' with respect to the scaling of the corresponding conserved quantity of the KdV evolution. We point out that this same phenomenon is absent for the weakly damped cubic NLS equation since the quantity required to be controlled there is found to be ``supercritical'' from this point of view. We refer the reader to the recent work \cite{BrzezniakFerrarioZanella2023}, where the analogous large damping result is proved for the cubic NLS by exploiting dispersive estimates.
}
\end{Rmk}

\begin{proof}
  As in the proofs of \eqref{eq:Lyapunov:Bnd:exp},
  \eqref{eq:Lyapunov:Bnd:subquadexp} in \cref{thm:Lyapunov}, we proceed
  by working from the identity \eqref{eq:Imp:mom:evol} now in the
  special case $p = 3/7$, $m =2$ and begin by estimating the terms
  involving $\Kn^{D,+}_2$ and $\Kn^{S,+}_2$.  Pursing the explicit
  expressions \eqref{eq:Ln:K:LD:123}, making use of
  \eqref{eq:LD:crit:al:s} and employing further routine usages of
  Agmon, Sobolev and interpolation inequalities we find
  \begin{align}
    |\Kn^{D}_2(u,f, \sigma)| 
    \leq \frac{\gamma}{1000} \In_2^+(u)
    + C(1+ \|u\|_{H^2} + \|u\|_{L^2} \|u\|_{H^2} + \|u\|^{11/4}_{L^2} \|u\|^{1/4}_{H^2})
    &\leq\frac{\gamma}{100} \In_2^+(u) + C,
	\label{eq:exp:bnd:LD:1}
  \end{align}
  where $C = C(\|f\|_{H^2}, \|\sigma\|_{H^2})$ is independent of
  $\gamma \geq 1$.  Meanwhile from \eqref{eq:In:Kp:LD:p} and noting
  that $\bar{\alpha}_2$ has been chosen based on universal quantities
  independently of $\gamma$ we find
  \begin{align}
  	 |\Kn_2^{D,+}  - \Kn^{D}_2|\leq \frac{\gamma}{100} \In_2^+(u) + C \gamma,
	 	\label{eq:exp:bnd:LD:2}
  \end{align}
  where $C = C(\|f\|_{L^2}, \|\sigma\|_{L^2})$.  Turning to the It\=o
  terms $\mathcal{K}_2^{S,+}$, we find from \eqref{eq:Ln:K:LD:s} that
  \begin{align}
   |\mathcal{K}_2^{S,+}| \leq& C( 1+ \|u\|_{H^2} +  \|u\|_{L^2} \|u\|_{H^2} 
   + \|u\|^{11/4}_{L^2}\|u\|^{1/4}_{H^2} + \|u\|^{11/3}_{L^2})\leq C(\In_2^+(u))^{11/14},
   	\label{eq:exp:bnd:LD:3}
  \end{align}
  where again $C = C(\|\sigma\|_{H^2})$ is independent of $\gamma \geq 1$.

  Next, with these bounds on $\Kn_2^{D,+}$ and $\Kn_2^{S,+}$ in hand
  we establish both \eqref{eq:Lyapunov:Bnd:exp},
  \eqref{eq:Lyapunov:Bnd:subquadexp} by making use of the exponential
  martingale inequality \eqref{eq:exp:Mt:gen}. Start with
  \eqref{eq:expmoments:I2plusform}.  Proceeding similarly to
  \eqref{eq:poly:MG:bnd:setup:1} above, we gather
  \eqref{eq:exp:bnd:LD:1}, \eqref{eq:exp:bnd:LD:2} and compare against
  \eqref{eq:Imp:mom:evol}.  With \eqref{eq:exp:Mt:gen} we find, for
  $\bar{\eta}, R > 0$,
  \begin{align}
  \Prb&\left(\sup_{t \geq 0}\left(\In_2^+(u(t))^{3/7}
        +\frac{5\gam}{7} \int_0^t\In_2^+(u(s))^{3/7}ds-\In_2^+(u_0)^{3/7}-C \gamma t
        - \frac{\bar{\eta}}{2} [\Mt](t)\right)\geq R\right)
        \leq e^{-\bar{\eta} R},
        \label{eq:MG:est}
  \end{align}
  where, cf.  \eqref{def:Implus:power:mart}, \eqref{eq:Imp:mart:QV}
  $[ \mathcal{M}](t) =
  \frac{9}{49}\int_0^t(\In_2^+)^{-8/7}|\Kn_m^{S,+}|^2 ds$. Here,
  to emphasize, the constant $C = C(\|f\|_{H^2}, \|\sigma\|_{H^2})$
  does not depend on $\gamma$.  Now in view of \eqref{eq:exp:bnd:LD:3}
  we find that
  \begin{align}
    [ \mathcal{M}(t)] \leq C\int_0^t \In_2^+(u(s))^{3/7}ds,
    \label{eq:Imp:mg:LD}
  \end{align}
  where again $C = C(\|\sigma\|_{H^2})$ is independent of
  $\gamma \geq 1$. We now consider $\bar{\eta} = \frac{4}{7C}$, with $C$
  precisely the constant appearing in \eqref{eq:Imp:mg:LD}, in the
  expression \eqref{eq:MG:est}.  With this choice for $\bar{\eta}$ in
  \eqref{eq:MG:est} and then making use of \eqref{eq:elem} in a
  similar fashion to \cref{sect:exp:mom:bnd} we obtain the first
  desired bound \eqref{eq:expmoments:I2plusform} for any $\eta$
  sufficiently small depending only on $\|\sigma\|_{H^2}^{-1}$.

  We turn to the second bound, \eqref{eq:expmoments:I2plusform:subquad}.
  Drawing once again from \eqref{eq:exp:bnd:LD:1}, \eqref{eq:exp:bnd:LD:2}
  and otherwise arguing similarly to \eqref{eq:next:bound:exp:moments}
  we have for, for any $t > 0$,
  \begin{align}
  \Prb&\left(\In_2^+(u(t))^{3/7} + \frac{ 3 \gam }{14}\int_0^t \! \! e^{-\frac{ 3 \gam}{7}(t-s)}\In_m^+(u(s))^{3/7} ds 
       - e^{-\frac{\gamma 3}{7} t}\In_2^+(u_0)^{3/7}
        -C \gamma \int_0^t \! \! e^{-\frac{3\gamma}{7} (t -s)} ds
       \! - \frac{\bar{\eta}}{2} [\tilde{\Mt}](t) \geq R\right) \notag\\
      &\qquad \qquad \leq \Prb\left( \sup_{s \geq 0}
        (\tilde{\Mt}(s) -\frac{\bar{\eta}}{2}[\tilde{\Mt}](s)) > R\right)
        \leq e^{-\bar{\eta}R},
        \label{eq:MG:est:LD:Decay}
  \end{align}
  where, for any $s \geq 0$,
  \begin{align}\notag
    \til{\Mt}(s):=\frac{3}{7}\int_0^s e^{-\frac{3\gam}{7} (s-r)}\In_2^+(u)^{-4/7}\Kn_2^{S,+}dW,
    \quad
    [\til{\Mt}](s)= \frac{9}{49} \int_0^s e^{-\frac{6\gam}{7} (s-r)} \In_{2}^+(u)^{-\frac{8}{7}} |\Kn_2^{S,+}|^2 dr.
  \end{align}
  Now, arguing as in \eqref{eq:Imp:mg:LD} we have that
  $[\til{\Mt}](t) \leq C \int_0^t e^{-\frac{3\gam}{7} (t-s)}
  \In_m^+(u(s))^{3/7}\,ds$ with $C = C(\|\sigma\|_{H^2})$.  Thus, by now
  taking $\bar{\eta} = \frac{3}{14C}$, for this $\gamma$-independent
  constant $C$, in \eqref{eq:MG:est:LD:Decay} and then making a second
  invocation of \eqref{eq:elem}, we obtain
  \eqref{eq:expmoments:I2plusform:subquad}.  The proof is now
  complete.
\end{proof}

With \cref{lem:LD:exp:mom} now in hand we now turn to the proof of the
main result of the section:

\begin{proof}[Proof of \cref{thm:spectralgap}]
  We begin with the useful observation that, since $d_k$ is
  non-negative and lower-semicontinuous on $H^m$ for $0 \leq k < m$,
  we have
  \begin{align}
    \Wm_{d_{k}}(\mu P_t,\nu P_t)
    \leq \int_{H^m\times H^m}\Wm_{d_k}(\de_{u_0}P_t,\de_{\til{u}_0}P_t)
    \pi(du_0,d\til{u}_0),
    \label{eq:Wd:diracs}
  \end{align} 
  for any $\pi\in\mathcal{C}(\mu,\nu)$ (cf. \cite[Theorem
  4.8]{villani2008optimal}).  With this observation our task for
  proving $(i)$ therefore reduces to showing that
  \begin{align}
    \Wm_{d_0}(\de_{u_0}P_t,\de_{\til{u}_0}P_t) \leq 2 e^{-\frac{\gamma}{2} t} d_0(u_0, v_0),
    \label{eq:cond:1:red:LD}
  \end{align}
  for any pair $u_0, v_0 \in H^2$.

  To this end, given any $u_0, \tilde{u}_0\in H^2$, let $(u(t), \tilde{u}(t))$ be the
  coupling of the laws $\delta_{u_0}P_t$, $\delta_{\tilde{u_0}}P_t$ by
  solving \eqref{eq:s:KdV} starting from $u_0$, $\tilde{u}_0$ using the
  same Brownian motion.  Then, by invoking the definition of the
  coupling distance \eqref{def:Wd}, and combining
  \eqref{eq:stab:main} with \eqref{eq:int:fact:L2:LD:sp}, \eqref{eq:I2p:exp:LD},
  \eqref{eq:I2p:exp:LD:eqv} we obtain,
  \begin{align}
  &\!\! \Wm_{d_0}(\de_{u_0}P_t,\de_{\til{u}_0}P_t)
    \leq\E \left(\exp\left(\eta_0 (\In_2^+(u(t))^{3/7}+\In_2^+(\til{u}(t))^{3/7})\right)
        \Sob{u(t)- \tilde{u}(t)}{L^2} \right)  \label{eq:sync:avg}
        \\
        &  \!\!\leq e^{-\gamma t}\|u_0 - \tilde{u}_0\|_{L^2} \E
        \exp \! \left( \eta_0 (\In_2^+(u(t))^{3/7} \! +\In_2^+(\til{u}(t))^{3/7})+ 
              c \! \int_0^t \! (\In_2^+({u}(s))^{3/7}\! + \In_2^+(\til{u}(s))^{3/7})ds \!\right). 
              \notag
  \end{align}
  Here, note that this estimate \eqref{eq:sync:avg} holds for any
  choice of $\eta_0 > 0$ and $\gamma \geq 1$.

We now apply the bound \eqref{eq:expmoments:I2plusform} in
\cref{lem:LD:exp:mom} to \eqref{eq:sync:avg} and, in the process,
specify values for $\eta_0 > 0$ and $\gamma_0> 0$.  First, notice that
we select an $0 < \eta_0 \leq \eta^*/4$ so that $C\eta_0 < 1/4$ where
$C$ is precisely the constant appearing in
\eqref{eq:expmoments:I2plusform}.  We then determine $\gamma_0 \geq 1$
so that $\frac{3 \eta_0 \gamma _0}{14} \geq c$ where $c$ is the
universal constant appearing in \eqref{eq:sync:avg}.  With these
choices, \eqref{eq:expmoments:I2plusform}, \eqref{eq:sync:avg}, and
H\"older's inequality produce \eqref{eq:cond:1:red:LD}, completing the
proof the first item.

We next establish the higher-order mixing bounds as desired for
\eqref{eq:mixing:Hm}.  Here, we observe with interpolation that, for
$0 < k < m$,
\begin{align*}
	d_k(u,v) \leq  \|u - v\|_{L^2}^{\frac{m-k}{m}} \|u - v\|_{H^m}^{\frac{k}{m} }
	\exp\left(\eta_0 \left(\In_2^+(u)^{3/7}+ \In_2^+(v)^{3/7} \right)\right) 
	\leq d_0(u,v)^{\frac{m-k}{m}}  d_m(u,v)^{\frac{k}{m}},
\end{align*}
so that
\begin{align}
      \Wm_{d_k}(\de_{u_0}P_t,\de_{\til{u}_0}P_t) 
      \leq  (\E (d_0(u(t), \tilde{u}(t))))^{\frac{m-k}{m}}  (\E(d_m(u(t), \tilde{u}(t))))^{\frac{k}{m}}. 
      \label{eq:dk:nm:interp}
\end{align}  
Now observe that with
H\"older's inequality, \eqref{eq:expmoments:I2plusform:subquad}
commensurate with our choice of $\eta_0$, \eqref{eq:Lyapunov:Bnd}, and
finally \cref{lem:equivalence}, we obtain
\begin{align}
   \E(d_m(u(t), \tilde{u}(t))) \leq&
	2(\E \| u(t)\|_{H^m}^2+ \E \| u(t)\|_{H^m}^2 )^{1/2} 
                                     \left( \E \exp(4\eta_0 \In_2^+(u(t))^{3/7})
                                     \E \exp(4\eta_0 \In_2^+(\tilde{u}(t))^{3/7})\right)^{1/4}
	\notag\\
  \leq&  C( \|u_0\|_{H^m} + \|\tilde{u}_0\|_{H^m} + 1)
        e^{2\eta_0 (\In_2^+(u_0)+ \In_2^+(u_0))},
	      \label{eq:dm:nm:mom}
\end{align}
for a constant $C = C(m,\gamma, \|f\|_{H^m}, \|\sigma\|_{H^m}) > 0$
independent of $t, u_0, \tilde{u}_0$.  Thus, combining
\eqref{eq:dk:nm:interp} with \eqref{eq:dm:nm:mom} noting that the
upper bound in \eqref{eq:dm:nm:mom} is independent of the coupling and
finally invoking \eqref{eq:cond:1:red:LD} we conclude that, for any
$u_0, \tilde{u}_0$
\begin{align}
\Wm_{d_0}(\de_{u_0}P_t,\de_{\til{u}_0}P_t) \leq e^{-\frac{m-k}{2m} \gamma t}
 C( \|u_0\|_{H^m} + \|\tilde{u}_0\|_{H^m} + 1) e^{2\eta_0 (\In_2^+(u_0)+ \In_2^+(u_0))}.
\end{align}
Hence, with \eqref{eq:Wd:diracs}, we infer the second item \eqref{eq:mixing:Hm}. 

We turn to the final item $(iii)$.  Here, we make the preliminary
observation that, for any $\gamma \geq 1$ $f \in H^2$,
$\sigma \in \bH^2$ and any invariant measure $\mu$ of the
corresponding Markov semigroup $\{P_t\}_{t \geq 0}$ of
\eqref{eq:s:KdV} that
\begin{align}
  \int_{H^2} (\|u\|_{L^2}e^{\eta_0 \In_2^+(u)^{3/7}}
  + e^{4\eta_0 \In_2^+(u)^{3/7}} ) \mu(du) \leq C <\infty
	\label{eq:IM:bnd:LD}
\end{align}
where $C = C( \|f\|_{H^2}, \|\sigma\|_{H^2})$ is independent of $\mu$
and $\gamma$.  To see this we observe that
$\|u\|_{L^2}\exp(\eta_0 \In_2^+(u)) \leq c\exp(2\eta_0 \In_2^+(u))$
and then argue as in \eqref{eq:far:field:2:0} and
\eqref{eq:near:field:2:0}, replacing \eqref{eq:Lyapunov:Bnd:subquadexp}
appropriately with \eqref{eq:expmoments:I2plusform}, while noting that
$4 \eta_0$ is an admissible choice for $\eta$ in this later bound.
Observe, furthermore, that if we suppose that $f \in H^m$,
$\sigma \in \bH^m$ for $m \geq 2$ we have that
\begin{align}
	\int_{H^2} \|u\|_{H^m}e^{2\eta_0 \In_2^+(u)^{3/7}} \mu(du) \leq C <\infty
		\label{eq:IM:bnd:HReg:LD}
\end{align}
for a constant $C = C(\gamma, \|f\|_{H^m}, \|\sigma\|_{H^m})$
independent of $\mu$.  To obtain \eqref{eq:IM:bnd:HReg:LD}, we simply
combine \eqref{eq:reg}, \eqref{eq:IM:bnd:LD} with H\"older's
inequality.

With \eqref{eq:IM:bnd:LD}, \eqref{eq:IM:bnd:HReg:LD} in hand we now
fix any two invariant measures $\mu, \tilde{\mu}$ of
$\{P_t\}_{t \geq 0}$ and maintain the standing assumptions on $f, \sigma$
and $\gamma$.  With invariance, \eqref{eq:spectral:gap}, and
\eqref{eq:IM:bnd:LD}, we find for any $t \geq 0$, that
\begin{align*}
	\Wm_{d_0}( \mu, \tilde{\mu})  \leq \Wm_{d_0}(\mu P_t, \tilde{\mu} P_t)
	\leq 2e^{-\frac{\gam t}{2}} \Wm_{d_0}(\mu,\tilde{\mu})
	\leq  C e^{-\frac{\gam t}{2}},
\end{align*}
where $C$ is independent of $\mu, \tilde{\mu}$ and $t \geq 0$.  We
thus infer $\Wm_{d_0}( \mu, \tilde{\mu}) =0$ so that\footnote{Here, we
  may argue by combining \eqref{eq:partial:duality} with the fact that
  $L^2$ determines measures on $H^2$; see
  \cref{sect:UniqueErgodicity}.}  $\mu = \tilde{\mu}$ as desired for
uniqueness. On the other hand, the bound \eqref{eq:mixing:Hm:2} follows
from \eqref{eq:mixing:Hm} combined with \eqref{eq:IM:bnd:HReg:LD}.  This establishes the final item $(iii)$, thus completing the proof of
\cref{thm:spectralgap}.
\end{proof}

\section{Remarks on the Deterministic Case}
\label{sect:det:case}

In the final section of this manuscript we revisit some now well-established results concerning the synchronization of small scales via large
scales and regarding the regularity of global attractors for the
deterministically forced and weakly damped KdV equation.  Note that
this deterministic setting may be seen as a special case of
\eqref{eq:s:KdV} when we simply set $\s\equiv0$.  With this in mind we
illustrate how the analysis above provides novel proofs and in some cases
leads to stronger statements of established
results.  Since the material covered
in this section is mostly classical our exposition is somewhat
aggressive with details.  We refer the reader to
e.g. \cite[Chapter IV, Section 7]{Temam1997} for a self contained
presentation close to our current setting.
 
To fix notations we consider the Cauchy initial value problem over
$\T=(0,2\pi]$ given by
\begin{align}
  \bdy_tu+\frac{1}{2}D(u^2)+D^3u+\gam u=f,\quad u(0,x)=u_0(x),\quad x\in\T,\quad t>0,
  \label{eq:KdV}
\end{align}
with periodic boundary conditions, where we assume that
$u_0, f\in H^m$ for some $m \geq 2$.  In particular, $u_0,f$ are
periodic and mean-free over $\T$.  We begin by recalling the basic
well-posedness and Lypuanov structure of \eqref{eq:KdV} in the context
of smooth solutions.
\begin{Prop}\label{prop:wp:det:case}
Fix any $f \in H^m$ for $m \geq 2$ and any $\gamma  \geq 0$. 
\begin{itemize}
\item[(i)] For any $u_0 \in H^m$ there exist a unique
  $u = u(u_0,f) \in C(\RR, H^m)$ obeying \eqref{eq:KdV} such that
  $u(0) = u_0$.
\item[(ii)] Define $\{S(t)\}_{t \geq 0}$ as collection maps on $H^m$
  according to $S(t) u_0 = u(t;u_0,f)$.  Then $\{S(t)\}_{t \geq 0}$ is
  a group, namely $S(t+s) = S(t) S(s)$ for any $s, t \in \RR$.
  Furthermore, for any $t \geq 0$, $S(t)$ is a weakly and strongly
  continuous map on $H^m$.
\item[(iii)]  For any $k \geq 0$, so long as $f \in H^k$, we have
  \begin{align}
    \|u(t; u_0)\|_{H^k} \leq C(e^{-\gamma t} (\|u_0\|_{H^k} + \|u_0\|^q_{L^2}) + 1),
    \label{eq:lypunov}
  \end{align}
  where $C = C(\gamma, k, \|f\|_{H^k})> 0$ and $q = q(k) > 0$ are both
  independent of $t \geq 0$ and $u_0$.
\end{itemize}
\end{Prop}

\noindent The well-posedness of \eqref{eq:KdV}, namely items $(i)$ and
$(ii)$ of \cref{prop:wp:det:case}, is classical; see \cite{Temam1969,
  Sjoberg1970, BonaSmith1975, BonaScott1976, SautTemam1976, Kato1979}.
Observe furthermore that this result is essentially a special case of
\cref{prop:exist:uniq} above when we take $\sigma \equiv 0$, so that we
refer the reader to \cref{sect:apx:wp:SKdV} below aside from the
aforementioned references.  The final item $(iii)$ in
\cref{prop:wp:det:case} is can be found in \cite[Proposition
3.1]{MoiseRosa1997} and later in \cite[Chapter IV, Section
7]{Temam1997}.  Here, however, the approach we developed in
\cref{sect:Lyapunov} provides a different and perhaps more direct
proof of \eqref{eq:lypunov} as follows:
\begin{proof}[Proof of \cref{prop:wp:det:case}, $(iii)$]
  We make use of the positivized functionals $\In_m^+$ defined as in
  \eqref{eq:mod:int:m:simple}.  Here, $\In_m^+(u)$ satisfies
  \eqref{eq:Inplus:simp} where we simply set the terms involving
  $\sigma$ to zero.  The required estimates for $\Kn_{m,1}^D$,
  $\Kn_{m,2}^D$, and hence for $\Kn_{m}^{D,+}$ which remain in
  \eqref{eq:Inplus:simp} in this special case
  then proceed precisely as in \eqref{eq:Kn1:est}, \eqref{eq:Kn3:est},
  \eqref{eq:Kn:extra:est}, leading to \eqref{eq:Kn:summary:0}.  Now, by
  appropriately selecting the parameters $\bar{q}_m$ and
  $\bar{\alpha}_m$ which appear in the definition of $\In_m^+$ we now
  obtain  \eqref{eq:Kn:summary}.  In summary we have that
  $\frac{d}{dt} \In_m^+(u) + \gamma \In_m^+(u) \leq C$ which in turn
  yields with Gr\"onwall the bound
  \begin{align}
    \In_m^+(u(t)) \leq e^{-\gamma t} \In_m^+(u_0) + C.
    \label{eq:Im:Inq}
  \end{align}
  Here, note that the constant $C$ which appears here can be seen to
  depend only on $\|f\|_{H^m}$ and $\gamma > 0$ and universal
  quantities.  Finally, by invoking \cref{lem:equivalence} this bound
  \eqref{eq:Im:Inq} now yields \eqref{eq:lypunov} completing the
  proof.
\end{proof}

\subsection{Foias-Prodi and Nudging Estimates}
  
The first results we now revisit concern the existence of finitely
many determining modes, that is, the classical Foias-Prodi estimate
\cite{FoiasProdi1967} for \req{eq:KdV}. We remark that a version of
such an estimate was already established for solutions on the global
attractor in \cite{JollySadigovTiti2017}.  We
provide an alternate, direct proof of the classical version of the
Foias-Prodi estimate based on the approach developed in
\cref{sect:FP:est}.  At the end of this subsection we revisit the
Foias-Prodi type nudging estimate \cref{thm:FP:est} reformulated
appropriately for the deterministic case.

\begin{Thm}\label{prop:FP:det}
  Let $u, \tilde{u} \in C([0,\infty);H^2)$ be two solutions of
  \req{eq:KdV} corresponding to initial data
  $u_0, \tilde{u}_0 \in H^2$, external forcing $f \in H^2$ and damping
  $\gamma > 0$.  Then, there exists an integer $N>0$, depending only on
  $\gam,f$, such that
    \begin{align}\label{eq:FP:det}
        \lim_{t\goesto\infty}\Sob{P_N(u(t)-\tilde{u}(t))}{L^2}=0,
        \quad\text{implies}\quad
        \lim_{t\goesto\infty}\Sob{u(t)-\tilde{u}(t)}{H^1}=0.
    \end{align}
\end{Thm}

\begin{proof}
  We proceed by making use of the `modified-Hamiltonian' functional
  $\InG_1^+$ which we introduced above in \cref{sect:mod:Ham}.  Now,
  taking $w = u_2 - u_1$, we can follow precisely the computations
  leading to \eqref{eq:FB:motherfn:evo} with $\lam=0$ and $\s\equiv0$
  and obtain that
  \begin{align}
    \frac{d}{dt} &\InG_1^+(w,u)
         + 2 \gamma \InG_1^+(w,u)
                   + \frac{4}{3}\gamma \bar{\alpha} \| w\|_{L^2}^{10/3}
                  + 2 \gamma \bar{\al} \|u\|_{L^2}^2\|w\|^2_{L^2}
                    \notag\\
   &= \int_{\T}\biggl( \frac{\gamma}{3} w^3 
            + u D uw^2  + D^3uw^2 + 3\gam u w^2 -fw^2
            - \bar{\alpha} \left(  1+ \frac{5}{3} \| w\|_{L^2}^{4/3}
           +\| u\|_{L^2}^{2} \right)Du w^2
     \biggr)dx \notag\\
   &:= K_1^D + K_3^D + K_4^D + K_5^D + K_7^D + K_9^D + K_{10}^D.
     \label{eq:FB:mfn:evo:DC:low}
  \end{align}
  Here, note that the strange numbering convention on the right hand
  side terms is taken in order to be consistent with
  \eqref{eq:FB:motherfn:evo} above.  Moreover, note that in contrast to its
  more involved role in the proof of \cref{thm:FP:est}, we simply set
  the parameter $\bar{\alpha}$ in the functional $\InG_1^+$ to be
  sufficiently large so that \eqref{eq:norm:equiv:FP:FN} holds.

  Next, with only minor adjustments to the estimates
  \eqref{eq:fp:est:K1:K2}--\eqref{eq:fp:est:K10} carried out above in
  \cref{sect:FP:thm:proof} we find
  \begin{gather*}
    |K^D_1|  \leq \frac{\gam}{100}\Sob{Dw}{L^2}^2  + C\Sob{w}{L^2}^{10/3},
    \;
    |K^D_3| \leq c\Sob{u}{H^2}^2\Sob{w}{L^2}^2,
    \;
    |K^D_4| \leq \frac{\gam}{100}\Sob{Dw}{L^2}^2+c\Sob{u}{H^2}^4\Sob{w}{L^2}^2,
    \\
    |K^D_5| \leq C\Sob{u}{H^2}\Sob{w}{L^2}^2,
    \;
    |K^D_7| \leq c \Sob{f}{L^\infty}\Sob{w}{L^2}^2,
    \;
  |K^D_9 + K^D_{10}| \leq c(1 + \|u\|_{H^2}^3 + \|\tilde{u}\|^{8/3}_{L^2})\|w\|^2_{L^2}.
\end{gather*} 
Upon combining these estimates, making use of the bound \eqref{eq:Im:Inq}, invoking the 
inverse Poincar\'e inequality and finally making use of \eqref{eq:norm:equiv:FP:FN}  we
find
\begin{align}
     \frac{d}{dt}\InG_1^+(w,u) &
     + {\frac{3}2\gamma}\InG_1^+(w,u) 
     \notag\\
     \leq& C(1+\|u\|_{H^2}^4 + \| \tilde{u}\|_{L^2}) \|w\|_{L^2}^2
     \leq C(1+ e^{-\gamma t}( \| u_0\|_{H^2}^q + \|\tilde{u}_0\|_{L^2}^q)) (\|Q_N w\|_{L^2}^2  + \|P_Nw\|_{L^2}^2).
     \notag\\
     \leq& C(1+ e^{-\gamma t}( \| u_0\|_{H^2}^q + \|\tilde{u}_0\|_{L^2}^q)) \left( \frac{1}{N} \InG_1^+(w,u)  + \|P_Nw\|_{L^2}^2\right),
     \label{eq:main:BND:DFP}
\end{align}
for a suitable constant $C = C(\|f\|_{H^2}, \gamma), q > 0$ independent of $t \geq 0$ and $u_0, \tilde{u}_0$
and $N$.  

Thus we impose the condition on $N$ that $C/N \leq \gamma$ with $C$ 
precisely the constant appearing in  \eqref{eq:main:BND:DFP}.  Under this 
condition we find, upon applying Gr\"onwall's  inequality, that
\begin{align}
	\InG_1^+(w(t),u(t))
	\leq&  \exp( C (1+ \|u_0\|_{H^2}^q + \|\tilde{u}_0\|_{L^2}^q)) \left(
	e^{-t\gamma}\InG_1^+(w_0,u_0)
	 + \int_0^t  e^{-\gamma(t-s)} \|P_Nw(s)\|_{L^2}^2 ds \right),
	 \label{eq:FP:Gron:1}
\end{align}
where again the constant $C = C(\gamma, \|f \|_{H^2})$ is independent of 
$t \geq 0$, $u_0$ and $\tilde{u}_0$.
Now, on the other hand 
\begin{align}	
	\int_0^t  e^{-\gamma(t-s)} \|P_Nw(s)\|_{L^2}^2 ds  
	\leq& 
      \int_0^{t/2} \! \! \!  e^{-\gamma(t-s)} ( \|u(s)\|_{L^2}^2 + \|\tilde{u}(s)\|_{L^2}^2 )ds
	 + \! \! \sup_{s \in [t/2,\infty)} \| P_Nw(s)\|_{L^2}^2\int_{t/2}^t  e^{-\gamma(t-s)} ds 
	 \notag\\
	\leq& C \left(e^{- \gamma t/2} (1 + \|u_0\|_{L^2}^2+ \|\tilde{u}_0\|^{2}_{L^2}) 
	+ \sup_{s \in [t/2,\infty)} \| P_Nw(s)\|_{L^2}^2\right),
		 \label{eq:FP:Gron:2}
\end{align}
where we made use of \eqref{eq:lypunov} for the final bound.   By now combining \eqref{eq:FP:Gron:1},  
\eqref{eq:FP:Gron:2} with \eqref{eq:norm:equiv:FP:FN} we obtain the desired implication \eqref{eq:FP:det}. 

\end{proof}

We conclude this subsection by providing the promised reformulation of \cref{thm:FP:est}.
We use this formulation to provide a novel proof of
the regularity of the global attractor in the sequel \cref{sec:higher:reg:atr} immediately below.
\begin{Thm}
  \label{prop:NG:det}
  Fix any $f \in H^2$ and any $u_0, v_0 \in H^2$.    Let $u$ be the solution
  of \eqref{eq:KdV} corresponding to $u_0$ and take $v$ to be the solution of
  \begin{align}
  \bdy_t v+\frac{1}{2}D(v^2)+D^3v+\gam v =f + \lambda P_N(u - v), \quad v(0)=v_0,
  \label{eq:KdV:ng:Det}
\end{align}
for some choice of $\lambda > 0$ and $N \geq 0$ where recall that $P_N$ denotes the projection 
onto the Fourier frequencies in the range $[-N, N]$.    Then 
\begin{itemize}
\item[(i)] for any $m \geq 0$ and so long as $f \in H^m$ we have, for any $t \geq 0$,
\begin{align}
\|v(t)\|_{H^m}^2
   \leq C (e^{-\gam t}  ( \|v_0\|_{H^m}^2  
       + \Sob{v_0}{L^2}^q+\Sob{u_0}{L^2}^{q}) + 1),
       	\label{eq:shft:bnd:det}
  \end{align}
  where
  $C = C(m, \gamma, \|f\|_{H^m}, \lambda, N), q =
  q(m) > 0$ are independent of $u_0, v_0$ and of $t \geq 0$.  
\item[(ii)] There exists a choice of $N, \lambda \geq 1$
depending only on $f$ and $\gamma$ such that
\begin{align}
 	\|u(t) - v(t)\|_{H^1} \leq  C\|u_0 - v_0\|_{H^1} 
	\exp(-t \gamma + C(\|u_0\|^2_{H^2} + \|u_0\|_{L^2}^q + \|v_0\|_{L^2}^2)),
	\label{eq:ng:bnd:det}
\end{align}
where $C = C(\gamma,\|f \|_{H^2})$, $q > 0$ are independent of $u_0, v_0$ and $t \geq 0$.
\end{itemize}
\end{Thm}
\begin{proof}
    For item $(i)$, we appeal to the proof of \cref{thm:Lyapunov:ng} $(ii$), where we set $\s\equiv0$ throughout and arrive at \eqref{est:Hm:ng}. We then invoke \cref{prop:wp:det:case} $(iii)$ to control the time integral of $e^{-\gam(t-s)}\Sob{u(s)}{L^2}^q$, and the proof of $(i)$ is complete.

	Regarding the second item, $(ii)$, we once again argue using the functional $\InG_1^+(w, u)$ 
	defined as in \eqref{eq:FP:functional}.
	We let $w = v - u$ and observe that $\InG_1^+(w, u)$ obeys \eqref{eq:FB:motherfn:evo},
	where we can delete the terms involving $\sigma$ namely $K_8^D$ and $K^S$.
	The estimates for each of the remaining $K^D_j$ proceed precisely as in 
	\eqref{eq:fp:est:K1:K2}--\eqref{eq:fp:est:K12}.  After appropriately tuning $\bar{\alpha}$
	in $\InG_1^+$ and selecting $\lambda,N$ sufficiently large as in \eqref{eq:param:tune:1}--\eqref{eq:param:tune:3} 
	we infer
	\begin{align}
	  \frac{d}{dt} \InG_1^+(w,u) 
  + \left(\frac{3\gamma}{2}
  - \frac{c_0}{\lambda}
  \left(1 + \|u\|_{H^2}^2 + \frac{1}{\lam}\| v\|_{L^2}^2\right) \right)
      \InG_1^+(w,u)
  \leq 0.
        \label{eq:mother:fn:det1}
\end{align}

Thus, to complete the proof we need a suitable estimate for $\|v\|_{L^2}$.
Here, comparing with \eqref{eq:L2:evo:Nudge}, we find that 
$\frac{d}{dt} \|v\|^2_{L^2} + \gamma \|v\|^2_{L^2} \leq \frac{1}{\gamma} \|f\|^2_{L^2} 
+ \lambda \|u\|^2_{L^2}$. Hence, with \eqref{eq:Im:Inq} we obtain the bound
\begin{align}
  \| v(t)\|_{L^2}^2 
  \leq& e^{-\gamma t} \|v_0\|^2_{L^2} 
  +  \int_0^t e^{-\gamma(t -s)}\left(\frac{1}{\gamma} \|f\|^2_{L^2} + \lambda \|u\|^2_{L^2} \right)ds
  \notag\\
  \leq&   e^{-\gamma t} \|v_0\|^2_{L^2} + \lambda C ( e^{-\frac{\gamma}{2}t }\|u_0\|_{L^2}^2  + 1),
  \label{eq:mother:fn:det2}
\end{align}
where we emphasize $C = C(\|f\|_{L^2}, \gamma)$ is independent of $\lambda \geq 1$.
Thus applying the Gr\"onwall lemma to \eqref{eq:mother:fn:det1} and making use
of \eqref{eq:Im:Inq} and \eqref{eq:mother:fn:det2}, we find
\begin{align}
 	\InG_1^+(w(t),u(t)) 
  \leq& \exp\left( - \frac{3\gamma}{2}t +  \frac{c_0}{\lambda}\int_0^t
  \left(1 + \|u(t)\|_{H^2}^2 + \frac{1}{\lam}\| v(t)\|_{L^2}^2\right) ds\right)\InG_1^+(w_0,u_0)  
  \notag\\
  \leq& \exp\left( - \frac{3\gamma}{2}t +  \frac{C}{\lambda}t 
  + C(\|u_0\|^2_{H^2} + \|u_0\|_{L^2}^q + \|v_0\|_{L^2}^2)  \right)\InG_1^+(w_0,u_0),\notag
\end{align}
where $C = C(\gamma, \|f\|_{H^2})$ are independent of $\lambda > 0$. We now tune $\lambda > 0$ sufficiently large and deduce the desired result.
\end{proof}

\subsection{Higher Regularity of the Global Attractor}
\label{sec:higher:reg:atr}

In this final section we revisit results from \cite{Ghidaglia1988, MoiseRosa1997}
concerning the long time behavior of \eqref{eq:KdV}.
In these works it was shown that \eqref{eq:KdV} possesses a global attractor 
$\mathcal{A}$.  Moreover, while \eqref{eq:KdV} lacks 
an instantaneous smoothing mechanism, \cite{MoiseRosa1997} established that
$\mathcal{A}$ is as smooth as the forcing $f$ driving the dynamics in \eqref{eq:KdV}. 
In particular $\mathcal{A} \subseteq C^\infty$, so long as $f \in C^\infty$.

The higher regularity result established in \cite{MoiseRosa1997} 
relies on a delicate decomposition of solutions of
\eqref{eq:KdV}.  The decomposition that they identified concerns 
the high-frequency component of $u$ namely $Q_N u = q+z$.  
Here, $q$ asymptotically deteriorates to zero while the complementary
component $z$ starts from a smooth initial condition and remains
bounded in each higher order Sobolev norm up to the regularity of the 
forcing.    

Our proof here considers a different decomposition of solutions based on 
\cref{prop:NG:det} and is similar to the approach we employed in \cref{sect:Regularity} in pursuit of the higher regularity of invariant
measures for \eqref{eq:s:KdV}.   
Given any $u_0 \in H^2$ we write the corresponding solution $u$
as $u = v + w$ where $v$ obeys \eqref{eq:KdV:ng:Det} starting
from $v(0) = 0$ and $w$ satisfies \eqref{eq:s:KdV:diff}.
Note that, as a technical matter, our proof of the regularity of the global 
attractor for \eqref{eq:KdV} covers the case $m=2$ that was
left open in the previous work \cite{MoiseRosa1997}.

Before stating the main result in a precise formulation let us recall a 
few generalities.    A \emph{global attractor} for a semigroup $\{S(t)\}_{t \geq 0}$ on 
a metric space $(H,d)$ is a compact set $\Atr \subseteq H$ such that $S(t) \Atr
= \Atr$ for every $t \geq 0$ and such that, for any bounded set $\bSt$, 
$\lim_{t \to \infty} d(S(t) \bSt, \Atr) = 0$.  Here,
$d$ is the Hausdorff semi-distance namely $d(A, B) :=  
\sup_{v \in A} \inf_{w \in B} d(v,w)$.    Recall furthermore that a 
set $\bSt_0$ is said to be an absorbing set of $\{S(t)\}_{t \geq 0}$ 
if for any bounded set $\bSt$ there exist a $t^* = t^*(\bSt) > 0$ such
that $S(t) \bSt \subseteq \bSt_0$ for every $t \geq t^*$. We refer to e.g.
\cite{Temam1997} for generalities concerning attractors
for infinite-dimensional dynamical systems.

\begin{Thm}
	\label{thm:sm:glb:attr}
	Suppose that $f \in H^2$ and $\gamma > 0$ and let $\{S(t)\}_{t \geq 0}$ 
	to be the corresponding solution semigroup for \eqref{eq:KdV} posed on $H^2$.  Then
	\begin{itemize}
	\item[(i)] $\{S(t)\}_{t \geq 0}$ has a \emph{global attractor} $\Atr$ given as
	\begin{align}
		\Atr = \bigcap_{s \geq 0} \overline{\bigcup_{t \geq s} S(t) \bSt_0},
	\label{eq:atr:omg:lim}
	\end{align}
	where the closure is taken in the $H^2$-topology where $\bSt_0$ is an $H^2$-bounded absorbing set for $\{S(t)\}_{t \geq 0}$.
	\item[(ii)] If $f \in H^m$, for some $m \geq 2$, then there exists
		an $C = C( \gamma, \|f\|_{H^m})$ such that 
		\begin{align}
			\Atr \subseteq \{ u \in H^m : \|u\|_{H^m} \leq C\}.
			\label{eq:atr:smooth:set}
		\end{align}
		In particular, $\Atr \subseteq C^\infty$ whenever
		$f \in C^\infty$.
	\end{itemize}
\end{Thm}
\begin{proof}
The first item can be established using an asymptotic compactness criteria,
the weak continuity properties of $S(t)$, and the evolution equation for $\In_m$ as found above
in \cref{lem:Im:evolution} with $\sigma \equiv 0$.    We refer the interested reader
to \cite{Ghidaglia1988, MoiseRosa1997} and also \cite[Chapter IV, Section 7]{Temam1997}
for further details.

For the second item $(ii)$, fix any $u_* \in \Atr$. According to the characterization \eqref{eq:atr:omg:lim},
we can find a sequence in $\{u_0^{(n)}\}_{n\geq0}$ in $\bSt_0$ (namely a sequence bounded independently of $u_*$) 
and $\{t_n\}_{n \geq 0}$ with 
$\lim_{n \to \infty} t_n = \infty$ such that $\lim_{n \to \infty} S(t_n)u_0^{(n)} = u_*$, where this limit is carried out in $H^2$.
Take $u^{(n)}$ to be the solution of \eqref{eq:KdV} corresponding to $u_0^{(n)}$.  Let $v^{(n)}$ 
solve \eqref{eq:KdV:ng:Det} starting from $v(0) = 0$, nudged towards $u^{(n)}$ with $\lambda$ and $N$ 
in the nudging term $\lambda P_N(v^{(n)} - u^{(n)})$
specified solely as a function of $\|f\|_{H^2}$ and $\gamma >0$ such that the bound \eqref{eq:ng:bnd:det}
is maintained.

Now fix any $M \geq 1$ and observe that, by applying the bounds \eqref{eq:shft:bnd:det} and \eqref{eq:ng:bnd:det}
\begin{align}
	\| P_M u_*\|_{H^m}  =& \lim_{n \to \infty} \|P_M u^{(n)}(t_n)\|_{H^m}
	\leq \limsup_{n \to \infty} \|P_M (u^{(n)}(t_n) - v^{(n)}(t_n))\|_{H^m}
	       +  \limsup_{n \to \infty} \|P_Mv^{(n)}(t_n)\|_{H^m} \notag \\
	\leq& \limsup_{n \to \infty} M^{m-1} \| (u^{(n)}(t_n) - v^{(n)}(t_n))\|_{H^1}
	       +  \limsup_{n \to \infty} \|v^{(n)}(t_n)\|_{H^m} \notag \\
	\leq&  \limsup_{n \to \infty} C M^{m-1}\|u_0^{(n)} \|_{H^1} 
	\exp(-t_n \gamma + C(\|u_0^{(n)}\|^2_{H^2} + \|u_0^{(n)}\|_{L^2}^q )) \notag\\
	&+  \limsup_{n \to \infty} C (e^{-\gam t_n}  \Sob{u_0^{(n)}}{L^2}^{q} + 1)\leq C,\notag
\end{align}
where $C=C(\gam,\|f\|_{H^m})$.
We have, therefore, found an upper bound for $\|P_M u_*\|_{H^m}$ which is independent
of $M$, depending only on $\|f\|_{H^m}$ and $\gamma > 0$.  The proof is complete.
\end{proof}

Let us make some concluding remarks concerning the $t = \infty$ regularity of \eqref{eq:KdV}.
\begin{Rmk}
\mbox{}
\begin{itemize}
\item[(i)] So long as $f \in H^m$ for some $m \geq 2$, \eqref{eq:KdV} defines a semigroup $\{S(t)\}_{t \geq 0}$
on $H^m$ according to \cref{prop:wp:det:case}.  In turn, as in \cref{thm:sm:glb:attr}
$(i)$, one can establish the existence of a global attractor $\Atr_{m}$ for the semigroup $\{S(t)\}_{t \geq 0}$
at each allowable degree of regularity.  These attractors in fact all coincide due to regularity as in \cref{thm:sm:glb:attr} $(ii)$.
Indeed, for each $m\geq2$, $\Atr_m$ is the maximal (with respect to inclusion) $H^m$--bounded set that is invariant for the semigroup $\{S(t)\}_{t\geq0}$. Since $\Atr_m$ is also bounded in $H^2$, we automatically have $\Atr_m\subseteq\Atr_2$.  On the other hand, according to \cref{thm:sm:glb:attr} $(ii)$, $\Atr_2$ is a bounded subset of $H^m$. By virtue of also being invariant for the semigroup, it follows that $\Atr_2\subseteq\Atr_m$. In particular, $\Atr_2=\Atr_m$.

\item[(ii)] The results in \cref{sect:Large:Damping} can be translated to the deterministic setting we 
are considering here to show that $\{S(t)\}_{t \geq 0}$ has a one point attractor when $\gamma \gg 0$
as a function of $\|f\|_{H^2}$ (cf. \cite{CabralRosa2004}).  To see this, one may use the characterization of 
the attractor \eqref{eq:atr:omg:lim} with the estimate \eqref{eq:stab:main} 
and suitable translation of the bounds in \cref{lem:LD:exp:mom} to the deterministic 
case when $\sigma \equiv 0$.   Here, we furthermore observe that if we parameterize
the attractor $\Atr_\gamma$ as a function of $\gamma$ so
$\Atr_\gamma = \{u_\gamma\}$ for all $\gamma$ sufficiently large, then $u_\gamma \to 0$
as $\gamma \to \infty$ in $L^2$.   Indeed by invariance we have that $u_\gamma$
obeys $\frac{1}{2}D(u^2_\gamma)+D^3u_\gamma+\gam u_\gamma=f$, so that 
\begin{align*}
  \|u_\gamma\|_{L^2}^2 = \frac{1}{\gamma}  \langle f, u_\gamma \rangle \
  \leq \frac{1}{2\gamma} \|f\|_{L^2}^2 +\frac{1}{2\gamma} \|u_\gamma\|_{L^2}^2.
\end{align*}
This convergence should also occur in all higher Sobolev 
norms but this would require one to establish certain $\gamma \geq 1$ 
independent bounds using the the integrals of motion $\In_m$
and identities analogous  \eqref{eq:gamma:terms:Im:Ito} and \eqref{eq:f:terms:Im:Ito}.
We leave further details to the interested reader. 
\item[(iii)] Invariant measures of \eqref{eq:KdV}, namely measures $\mu$
such that $\mu S(t)^{-1} = \mu$  
are readily seen to exist\footnote{Indeed in the deterministic case
we define the Markov semigroup as $P_t(u_0,A) := \delta_{S(t)u_0}(A)$ so that
this definition coincides with the definition of invariant measures given in the 
general case for \eqref{eq:s:KdV} in \cref{sec:well:pose}.} from the arguments
in \cref{thm:existence}.
 According to e.g.
\cite{chekroun2012invariant}, \cite{foias2001navier} any such measure $\mu$
is supported on $\Atr$ so that $\mu(\Atr) = 1$ is consistent with \cref{thm:regularity}.
\end{itemize}
\end{Rmk}

\setcounter{equation}{0}

\appendix

\section{Well Posedness Results for the Stochastic KdV Equation}
\label{sect:apx:wp:SKdV}

The well-posedness result stated in \cref{prop:exist:uniq} 
can be proven by subtracting away the additive noise terms
 $\sigma W$ in \eqref{eq:s:KdV} and carrying out a pathwise
 analysis of the equations for $v = u - \sigma W$. 
Indeed, let $v,\xi$ satisfy
	\begin{align}
		\begin{split}\label{v:eqn}
		&\bdy_tv+\gam v+D^3v+D(\xi v)+vDv=f-\gam \xi-D^3\xi-\xi D\xi,\quad v(0)=u_0,\\
		&d\xi=\s dW,\quad \xi(0)=0.
		\end{split}
	\end{align}
Then upon setting $u=v+\xi$, it follows that $u$ satisfies \eqref{eq:s:KdV}.

This section is organized as follows.  Following \cite{Temam1997}, we consider a regularized version of \eqref{v:eqn} parametrized by $\eps>0$ (see \eqref{v:par:reg:eqn} below), whose solutions converge almost surely to solutions of \eqref{v:eqn} as $\eps\goesto0$.
In \cref{subsect:apriori:est}, we establish pathwise a priori bounds on the regularized system.  The convergence of solutions to the regularized system is then established in \cref{subsect:cauchyinHm}.
The proof of \cref{prop:exist:uniq} is then presented in \cref{subsect:wellposed}.

\subsection{A priori estimates}\label{subsect:apriori:est}

 Fix $m\geq2$ and $T>0$. Let $u_0\in{H}^m$, $\xi\in L^\infty_{loc}([0,\infty);H^{m+3})$, and $f\in H^m$. Let $\eta$ be a polynomial in $5$ variables of any nonzero degree. It will be convenient to consider a more general form of the equation \eqref{v:eqn}.
In particular, we consider the following:
	\begin{align}\label{v:gen:eqn}
		\bdy_tv+\gam v+D^3v+D(\xi v)+vDv=\eta(D_{3}\xi,f),\quad v(0)=u_0,
	\end{align}
where $D_k\xi=(\xi,D\xi,\dots, D^k\xi)$. Observe that when 
    \begin{align}\label{eq:eta}
        \eta=-\gam\xi-D^3\xi-\xi D\xi+f,
    \end{align}
then \eqref{v:gen:eqn} reduces to \eqref{v:eqn}. Thus, we will further assume that $\eta$ is linear in the variable corresponding to $D^3\xi$. For convenience, we will often write $\eta(D_3\xi,f)$ simply as $\eta$.

To obtain existence of solutions to \eqref{v:gen:eqn}, we will consider a sequence of parabolic regularizations:
	\begin{align}\label{v:par:reg:eqn}
		\bdy_tv^\eps+\gam v^\eps+D^3v^\eps+D(\xi v^\eps)+v^\eps Dv^\eps+\eps{D}^4v^\eps=\eta^\eps,\quad v^\eps(0)=u_0^\eps,
	\end{align}
where $u_0^\eps=P_{\eps^{-p}}u_0$ and $\eta^\eps=P_{\eps^{-p}}\eta(D_3\xi,f)$, for some $p\in(0,1]$, and where $P_\de$, for $\de>0$, denotes projection onto Fourier modes $|k|\leq\de$ and $P_{\de_1,\de_2}=P_{\de_2}-P_{\de_1}$, for $\de_1<\de_2$. We point out that we intentionally do not regularize $\xi$ appearing in the linear term $D(\xi v^\eps)$; this allows for some simplification in the analysis since estimation of this term in $H^m$ will only ever require control of at most $D^{m+2}$ derivatives on $\xi$ in $L^2$. We emphasize that the estimates we perform below are formal. They can, however, be justified at the level of the Galerkin approximation.  Indeed, the Galerkin system to be considered is given by:
	\begin{align}\label{eq:galerkin}
		\bdy_tv_M^\eps+\gam v_M^\eps+D^3v_M^\eps+P_MD(\xi v_M^\eps)+P_M(v_M^\eps Dv_M^\eps)+\eps D^4v_M^\eps=\eta^\eps_M,\quad v_M^\eps(0)=P_Mu_0^\eps,
	\end{align}
for $M>0$, where $P_M$ denotes the projection onto Fourier modes, $|k|\leq M$. 
With this in mind, one can show that the equation \eqref{v:par:reg:eqn} has a unique solution, $v^\eps$, satisfying 
    \begin{align}\label{sol:v:eqn}
        v^\eps\in C([0,T];{H}^m)\cap L^2(0,T;{H}^{m+2}),\quad \text{for all } T>0.
    \end{align}

It will be convenient to define the following quantities: Given $T>0$ and integer $\ell\geq0$, we define
	\begin{align}\label{def:ZF}
		\Xi_\ell(T):={\Sob{\xi}{L_T^\infty H^\ell}}+1,\quad  \Phi_{\ell}(T):={\Sob{\eta(D_3\eta,f)}{L^\infty_TH^\ell}}+1,
	\end{align}
 where $L^p_TX=L^p(0,T;X)$, for $p\in[1,\infty]$. Since $\eta $ is a polynomial, we observe the following elementary estimate for $\eta$, which follows from the chain rule, mean value theorem, and Poincar\'e's inequality:
    \begin{align}\label{est:eta:gen}
        \begin{split}
        \Phi_k&\leq (\Xi_{k+3}(T)+\Sob{f}{H^k}+1)^p\\
        \Sob{\eta(D_3\xi_1,f_1)-\eta(D_3\xi_2,f_2)}{H^k}&\leq C\left(\Sob{\xi_1-\xi_2}{H^{k+3}}+\Sob{f_1-f_2}{H^k}\right),    
        \end{split}
    \end{align}
for some sufficienty large power $p$, depending only on $\eta$ and $k$, and some constant $C>0$ that depends algebraically on $k,\Sob{f}{H^k},\Xi_{k+3}^{(1)}(T),\Xi_{k+3}^{(2)}(T)$, where $\Xi_{\ell}^{(j)}(T)=\Sob{\xi_j}{L^\infty_TH^\ell}+1$.

Throughout, we set forth the following convention for constants:
    \begin{itemize}
        \item We drop the superscript $\eps$ in $v^\eps$ in performing the apriori estimates.
        \item We drop the dependence on $T$ in \eqref{def:ZF}, that is,
$\Xi_\ell=\Xi_\ell(T)$ and $\Phi_{\ell}=\Phi_{\ell}(T)$.
        \item Similar to the convention for constants set in the beginning of \cref{sect:background}, given $m\geq2$, $c$ will denote a generic constant depending on $m$ and universal quantities, while $C$ will denote constant that may also depend on $m, \gam,T,\Xi_{m+3}$, as well as $\Sob{u_0}{H^{m-1}}$, and $\Phi_{m-1}$, but \textit{not} $\Sob{u_0}{H^m},\Phi_m$, as we will need to track the dependence on $\Sob{u_0}{H^m}, \Phi_m$ carefully. We note that the dependence on $T$ and $\Sob{u_0}{H^\ell},\Phi_\ell,\Xi_\ell$, will always be such that the constant $C$ is increasing in these quantities.
    \end{itemize}

\begin{Lem}\label{lem:L2v}
Let $T>0$ and $\eps>0$. There exists a constant $C\geq1$, independent of $\eps$, such that
	\begin{align}\label{eq:L2v}
		\Sob{v^\eps}{L_T^\infty L_x^2}+\eps^{1/2}\Sob{v^\eps}{L^2_TH^2}\leq C.
	\end{align}
\end{Lem}

\begin{proof}
We multiply \eqref{v:par:reg:eqn} by $v$ and integrate over $\T$ to obtain
	\begin{align}\notag
		\frac{1}2\frac{d}{dt}\Sob{v}{L^2}^2+\gam\Sob{v}{L^2}^2+\eps\Sob{{D}^2v}{L^2}^2=-\lb{D}(\xi v),v\rb+\lb \eta,v\rb.\notag
	\end{align}
Observe that  by the Sobolev embedding $H^1\imb L^\infty$, the Cauchy-Schwarz inequality, and Young's inequality, we estimate
	\begin{align}\notag
		|\lb{D}(\xi v),v\rb|\leq c\Sob{\xi}{L^\infty}\Sob{v}{L^2}^2\leq c\Xi_2\Sob{v}{L^2}^2,
		\quad|\lb \eta,v\rb|\leq\Sob{\eta}{L^2}\Sob{v}{L^2}
			\leq c\Phi_0+\frac{\gam}{100}\Sob{v}{L^2}^2.
	\end{align}
Thus, using \eqref{def:ZF}, we have
	\begin{align}
		\frac{d}{dt}\Sob{v}{L^2}^2+\eps\Sob{{D}^2v}{L^2}^2\leq  C\Sob{v}{L^2}^2+C.\notag
	\end{align}
An application of Gronwall's inequality yields
	\begin{align}
		\Sob{v(t)}{L^2}^2+2\eps\int_0^te^{-C (t-s)}\Sob{{D}^2v(s)}{L^2}^2ds
		\leq C\exp\left(C T\right)\left(\Sob{u_0}{L^2}^2+1\right),\notag
	\end{align}
for some $C>0$, for all $0\leq t\leq T$. It follows that
    \begin{align}\label{est:Linfty:L2}
        \sup_{0\leq t\leq T}\Sob{v(t)}{L^2}^2\leq C\exp\left(C T\right)\left(\Sob{u_0}{L^2}^2+1\right),
    \end{align}
and
    \begin{align}\label{est:L2:H2}
        2\eps\int_0^t\Sob{{D}^2v(s)}{L^2}^2ds\leq C\exp\left(C T\right)\left(\Sob{u_0}{L^2}^2+1\right).
    \end{align}
The desired result follows from adding \eqref{est:Linfty:L2} and \eqref{est:L2:H2}.
\end{proof}

\begin{Lem}\label{lem:H1v}
Let $T>0$ and suppose $\eps\in(0,1]$. There exists a constant $C\geq1$, independent of $\eps$, such that 
	\begin{align}\label{eq:Hamv}
		\sup_{0\leq t\leq T}\Ham(v^\eps(t))+{\eps}\Sob{v^\eps}{L_T^2H^3}^2\leq C.
	\end{align}
In particular
    \begin{align}\label{eq:H1v}
		\Sob{v^\eps}{L_T^\infty H^1_x}+\eps^{1/2}\Sob{v^\eps}{L^2_TH^3_x}\leq C.
	\end{align}
\end{Lem}

\begin{proof}
We consider the Hamiltonian, $\Ham(v)$, and compute
	\begin{align}
		\frac{d}{dt}\Ham(v)+{\gam}\Sob{{D}v}{L^2}^2+\eps\Sob{{D}^3v}{L^2}^2
		=&\frac{\gam}2\int v^3dx+\lb{D}\eta,{D}v\rb-\frac{1}2\lb \eta,v^2\rb\notag\\
		&-\lb{D}^2(\xi v),{D}v\rb+\frac{1}2\lb{D}(\xi v),v^2\rb-\frac{\eps}2\lb{D}^4v,v^2\rb\notag.
	\end{align}
Observe that by H\"older's inequality, interpolation, the Sobolev embedding $H^1\imb L^\infty$, Young's inequality, \cref{lem:L2v}, and the fact that $\eps\leq1$, we have
	\begin{align}
		\frac{\gam}2\Sob{v}{L^3}^3&\leq c\gam\Sob{{D}v}{L^2}^{1/2}\Sob{v}{L^2}^{5/2}
		\leq\frac{\gam}{100}\Sob{{D}v}{L^2}^2+ C,\label{Ham:v3}\\
		\frac{\eps}2|\lb {D}^4v,v^2\rb|&\leq \frac{\eps}2\Sob{{D}^3v}{L^2}\Sob{v}{L^\infty}\Sob{{D}v}{L^2}
		\leq\frac{\eps}{100}\Sob{{D}^3v}{L^2}^2+c\eps\Sob{{D}v}{L^2}^{4}\notag\\
		&\leq\frac{\eps}{100}\Sob{{D}^3v}{L^2}^2+c\eps\Sob{{D}^3v}{L^2}^{4/3}\Sob{v}{L^2}^{8/3}
		\leq\frac{\eps}{50}\Sob{{D}^3v}{L^2}^2+C.\label{Ham:visc}
	\end{align}
Also observe that by H\"older's inequality, Young's inequality, \cref{lem:L2v},  the Sobolev embedding $H^1\imb L^\infty$, the Cauchy-Schwarz inequality, and \eqref{def:ZF} we have
	\begin{align}
		|\lb{D}\eta,{D}v\rb|+|\lb \eta,v^2\rb|
		&\leq c\Sob{{D}\eta}{L^2}^2+\frac{\gam}{100}\Sob{{D}v}{L^2}^2+\Phi_1\Sob{v}{L^2}^2\leq \frac{\gam}{100}\Sob{{D}v}{L^2}^2+C\label{Ham:zf}.
	\end{align}
On the other hand, observe that
	\begin{align}
		\lb{D}^2(\xi v),{D}v\rb=-\frac{1}2\lb{D}^3\xi,v^2\rb+\frac{3}2\lb{D}\xi,({D}v)^2\rb,\quad \lb{D}(\xi v),v^2\rb=\frac{2}3\lb{D}\xi,v^3\rb.\notag
	\end{align}
Thus, by H\"older's inequality, the Sobolev embedding $H^1\hookrightarrow L^\infty$, \cref{lem:L2v}, \eqref{def:ZF}, and Young's inequality, we have
	\begin{align}
		|\lb{D}^2(\xi v),{D}v\rb|+ |\lb{D}(\xi v),v^2\rb|&\leq c\Xi_3\Sob{v}{L^2}^2+c\Xi_2\left(\Sob{{D}v}{L^2}^2+\Sob{v}{L^3}^3\right)\notag\\
		&\leq C\left(\Sob{v}{L^2}^2+\Ham(v)+\Sob{v}{L^3}^3\right)\notag\\
		&\leq C\left(\Sob{v}{L^2}^2+\Ham(v)+\Sob{{D}v}{L^2}^{1/2}\Sob{v}{L^2}^{5/2}\right)\notag\\
		&\leq C\Ham(v)+\frac{\gam}{100}\Sob{{D}v}{L^2}^2+C\label{Ham:lin1}.
	\end{align}
	
Upon combining \eqref{Ham:v3}, \eqref{Ham:visc}, \eqref{Ham:zf}, \eqref{Ham:lin1}, we arrive at
	\begin{align}
	\frac{d}{dt}\Ham(v)+\frac{\eps}2\Sob{{D}^3v}{L^2}^2
	&\leq C\Ham(v)+ C.\notag
	\end{align}
An application of Gronwall's inequality thus yields
	\begin{align}
		\Ham(v(t))+\frac{\eps}2\int_0^te^{-C(t-s) }\Sob{{D}^3v(s)}{L^2}^2ds
		&\leq C\exp\left(CT\right)\left(|\Ham(u_0)|+1\right),\notag
	\end{align}
for some $C$ and all $t\in[0,T]$. We then argue as in \eqref{est:Linfty:L2}, \eqref{est:L2:H2} and use interpolation to deduce the desired result.
\end{proof}

\begin{Lem}\label{lem:I2v}
Let $T>0$ and suppose $\eps\in(0,1]$. There exists a $C\geq1$, independent of $\eps$, such that
	\begin{align}\label{eq:I2v}
		\sup_{0\leq t\leq T}\In_2(v^\eps(t))+\eps\Sob{v^\eps}{L^2_TH^4_x}^2\leq C.
		\end{align}
In particular
    \begin{align}\label{eq:H2v}
		\Sob{v^\eps}{L^\infty_TH^2}+\eps^{1/2}\Sob{v^\eps}{L^2_TH^4_x}
		\leq C.
	\end{align}
\end{Lem}

\begin{proof}
We consider the functional $\In_2(v)$, and find
	\begin{align}
		\frac{d}{dt}&\In_2(v)+\frac{9\gam}5\Sob{{D}^2v}{L^2}^2+\frac{\gam}2\Sob{v}{L^4}^4+\eps\Sob{{D}^4v}{L^2}^2\notag\\
		&=3\gam\lb v,({D}v)^2\rb+\frac{9}5\lb \bdy_{xx}\eta,{D}^2v\rb-\frac{3}2\lb \eta,({D}v)^2\rb-\frac{3}2\lb{D}v^2,{D}\eta\rb+\frac{1}2\lb v^3,\eta\rb\notag\\
		&\quad +\frac{3\eps}2\lb{D}^4v,({D}v)^2\rb+\frac{3\eps}2\lb{D}v^2,{D}^5v\rb-\frac{\eps}2\lb v^3,{D}^4v\rb\notag\\
		&\quad -\frac{9}5\lb{D}^3(\xi v),{D}^2v\rb+\frac{3}2\lb{D}(\xi v),({D}v)^2\rb+\frac{3}2\lb{D}v^2,{D}^2(\xi v)\rb-\frac{1}8\lb v^3,{D}(\xi v)\rb\notag.
	\end{align}
Applying H\"older's inequality, interpolation, Young's inequality, and \cref{lem:L2v}, we estimate
	\begin{align}\label{I2:mid}
		3\gam|\lb v,({D}v)^2\rb|&\leq3\gam\Sob{v}{L^2}\Sob{
		{D}v}{L^4}^2\leq c\gam\Sob{v}{L^2}\Sob{{D}^2v}{L^2}^{1/2}\Sob{{D}v}{L^2}^{3/2}\notag\\
			&\leq c\gam\Sob{v}{L^2}^{7/4}\Sob{{D}^2v}{L^2}^{5/4} \leq C
			+\frac{\gam}{100}\Sob{{D}^2v}{L^2}^2.
	\end{align}
Similarly
	\begin{align}\label{I2:f1}
		\frac{3}2|\lb \eta,({D}v)^2\rb|\leq c\Phi_0\Sob{{D}v}{L^4}^2\leq C
		+\frac{\gam}{100}\Sob{{D}^2v}{L^2}^2.
	\end{align}
On the other hand, upon integrating by parts, using H\"older's inequality, applying \eqref{def:ZF}, \cref{lem:L2v}, and Young's inequality, we obtain
	\begin{align}\label{I2:f2}
		\frac{9}5|\lb D^2\eta,{D}^2v\rb|+\frac{3}2|\lb{D}v^2,{D}\eta\rb|+\frac{1}2|\lb v^3,\eta\rb|
		&\leq c\Phi_2\Sob{{D}^2v}{L^2}+C\Sob{v}{L^4}^2\Phi_2+C\Sob{v}{L^6}^3\Phi_0\notag\\
		&\leq C+\frac{\gam}{1000}\Sob{{D}^2v}{L^2}^2+\frac{\gam}{100}\Sob{v}{L^4}^4+C\Sob{{D}^2v}{L^2}^{1/2}\Sob{v}{L^2}^{5/2}\notag\\
		&\leq C
		+\frac{\gam}{100}\Sob{{D}^2v}{L^2}^2+\frac{\gam}{100}\Sob{v}{L^4}^4.
	\end{align}
We treat the remaining terms. In particular, observe that integrating by parts yields
	\begin{align}
		\lb{D}^3(\xi v),{D}^2v\rb
		= \lb{D}^3\xi,v{D}^2v\rb  -\frac{3}2\lb{D}^3\xi,({D}v)^2\rb+\frac{5}2\lb{D}\xi,({D}^2v)^2\rb.\notag
	\end{align}
Then by H\"older's inequality, interpolation, the Sobolev embedding $H^1\imb L^\infty$, Young's inequality, \cref{lem:L2v}, and \eqref{def:ZF}, we obtain
	\begin{align}
		\frac{9}5|\lb{D}^3(\xi v),{D}^2v\rb|&\leq c\Xi_4\Sob{v}{L^2}\Sob{{D}^2v}{L^2}+c\Xi_4\Sob{{D}v}{L^2}^2+c\Xi_2\Sob{{D}^2v}{L^2}^2\leq C
		+C\Sob{{D}^2v}{L^2}^2\notag.
	\end{align}
Furthermore, by interpolation and \cref{lem:L2v}, we have
	\begin{align}
		C\Sob{{D}^2v}{L^2}^2&\leq C\In_2(v)+C|\lb v,({D}v)^2\rb|+C\Sob{v}{L^4}^4\leq C\In_2(v)+C
				\Sob{{D}^2v}{L^2}^{5/4}+C\Sob{{D}^2v}{L^2}^{1/2}
				.\notag
	\end{align}
Thus, by Young's inequality, we obtain
	\begin{align}\label{I2:lin1}
		&\frac{9}5|\lb{D}^3(\xi v),{D}^2v\rb|\leq
		\frac{\gam}{100}\Sob{{D}^2v}{L^2}^2+C\In_2(v)+C.
	\end{align}
	
Similarly
	\begin{align}
		\lb{D}(\xi v),({D}v)^3\rb&=-3\lb\xi,v({D}v)^2{D}^2v\rb,\notag\\
		\lb{D}v^2,{D}^2(\xi v)\rb&=2\lb{D}^2\xi,v^2{D}v\rb+4\lb{D}\xi,v({D}v)^2\rb+2\lb \xi,v{D}v{D}^2v\rb\notag.
	\end{align}
Thus, by H\"older's inequality, interpolation, the Sobolev embedding, Young's inequality, \cref{lem:L2v}, and \eqref{def:ZF}, we obtain
	\begin{align}\label{I2:lin2}
		|\lb{D}(\xi v),({D}v)^3\rb|&\leq c\Xi_1\Sob{v}{L^\infty}\Sob{{D}v}{L^4}^2\Sob{{D}^2v}{L^2}\leq c\Xi_1\Sob{v}{L^2}^{5/2}\Sob{{D}^2v}{L^2}^{3/2}\leq C
		+\frac{\gam}{100}\Sob{{D}^2v}{L^2}^2.
	\end{align}
We similarly estimate
	\begin{align}\label{I2:lin3}
		|\lb{D}v^2,{D}^2(\xi v)\rb|&\leq c\Xi_3\Sob{v}{L^4}^2\Sob{{D}v}{L^2}+c\Xi_2\Sob{v}{L^2}\Sob{{D}v}{L^4}^2+c\Xi_1\Sob{v}{L^4}\Sob{{D}v}{L^4}\Sob{{D}^2v}{L^2}\notag\\
		&\leq C
		\Sob{{D}^2v}{L^2}^{3/4}\Sob{v}{L^2}^{9/4}+C\Sob{v}{L^2}^{5/2}\Sob{{D}^2v}{L^2}^{1/2}+C\Sob{v}{L^2}^{5/4}\Sob{{D}^2v}{L^2}^{7/4}\notag\\
		&\leq C
		+\frac{\gam}{100}\Sob{{D}^2v}{L^2}^2
	\end{align}
	
Lastly, we treat the terms with $\eps$. Upon integrating by parts and applying H\"older's inequality, interpolation, \cref{lem:L2v},  Young's inequality, and the fact that $\eps\leq1$, we obtain
	\begin{align}
		&\frac{3\eps}2|\lb{D}^4v,(Dv)^2\rb|+\frac{3\eps}2|\lb{D}v^2,{D}^5v\rb|+\frac{\eps}2|\lb v^3,{D}^4v\rb|\notag\\
		&\leq C\eps\Sob{{D}^4v}{L^2}\Sob{{D}v}{L^4}^2+C\eps\Sob{{D}^2v^2}{L^2}\Sob{{D}^4v}{L^2}+C\eps\Sob{v}{L^6}^3\Sob{{D}^4v}{L^2}\notag\\
		&\leq \frac{\eps}{1000}\Sob{{D}^4v}{L^2}^2+C\eps\Sob{{D}^2v}{L^2}^{5/4}\Sob{v}{L^2}^{3/4}+C\eps\Sob{v}{L^2}^{11/4}\Sob{{D}^4v}{L^2}^{5/4}+C\eps\Sob{{D}^2v}{L^2}\Sob{v}{L^2}^{5}\notag\\
		&\leq \frac{\eps}{100}\Sob{{D}^4v}{L^2}^2+\frac{\eps}{100}\Sob{{D}^2v}{L^2}^{2}+C.\notag
	\end{align}

There exists $C\geq1$ such that upon combining \eqref{I2:mid}, \eqref{I2:f1}, \eqref{I2:f2}, \eqref{I2:lin1}, \eqref{I2:lin2},  \eqref{I2:lin3}, and using the fact that $\eps\leq1$, we arrive at
	\begin{align}
		\frac{d}{dt}\In_2(v)+\frac{\eps}2\Sob{{D}^4v}{L^2}^2
		\leq C\In_2(v)+C\notag.
	\end{align}
Thus, by Gronwall's inequality, we have
	\begin{align}\notag
		\In_2(v(t))+\frac{\eps}2\int_0^te^{-C(t-s) }\Sob{{D}^4v(s)}{L^2}^2ds\leq C\exp\left({C T}\right)\left(|\In_2(v_0)|+1\right).
	\end{align}
We complete the proof by arguing as in \eqref{est:Linfty:L2}, \eqref{est:L2:H2} and by interpolation.
\end{proof}

\begin{Lem}\label{lem:Hmv}
Let $m\geq3$ and $T>0$. Suppose $p\in(0,1]$ and $0<\eps\leq1$. There exists a constant $C\geq1$, independent of $\eps$, such that
	\begin{align}\label{eq:Hmv}
		\Sob{v^\eps}{L^\infty_TH^m}+\eps^{1/2}\Sob{v^\eps}{L^2_TH^{m+2}}\leq C\left(\Sob{u_0^\eps}{H^m}+\Sob{\eta^\eps}{L^\infty_TH^m}\right).
	\end{align}
Moreover, for all $0\leq\ell\leq m$, we have
    \begin{align}\label{eq:Hmv:eps}
        	\Sob{v^\eps}{L^\infty_TH^m}+\eps^{1/2}\Sob{v^\eps}{L^2_TH^{m+2}}\leq C\left(\Sob{u_0}{H^\ell}+\Phi_\ell\right)\eps^{-(m-\ell)p}.
    \end{align}
\end{Lem}

\begin{proof}
We have
	\begin{align}\notag
		&\frac{1}2\frac{d}{dt}\Sob{{D}^mv}{L^2}^2+\gam\Sob{{D}^mv}{L^2}^2+\eps\Sob{{D}^{m+2}v}{L^2}^2=-\frac{1}2\lb{D}^{m+1}v^2,{D}^mv\rb-\lb{D}^{m+1}(\xi v),{D}^mv\rb+\lb{D}^m\eta,{D}^mv\rb.
	\end{align}
Observe that
	\begin{align}
		\frac{1}2\lb{D}^{m+1}v^2,D^mv\rb=\lb\frac{1}2{D}^{m+1}v^2-v{D}^{m+1}v,{D}^mv\rb-\frac{1}4\lb {D}v,({D}^mv)^2\rb.\notag
	\end{align}
Thus, using commutator estimates, H\"older's inequality, \cref{lem:I2v}, \eqref{def:ZF}, and Young's inequality we obtain
	\begin{align}
		\frac{1}2|\lb{D}^{m+1}v^2,D^mv\rb|
		\leq c\Sob{{D}v}{L^\infty}\Sob{{D}^mv}{L^2}^2\leq C\Sob{{D}^mv}{L^2}^2.\notag
	\end{align}
Similarly, by the Leibniz rule 
	\begin{align}
		|\lb{D}^{m+1}(\xi v),{D}^mv\rb|&\leq c\Xi_{m+2}\Sob{v}{L^2}\Sob{{D}^mv}{L^2}+c\Xi_2\Sob{{D}^mv}{L^2}^2\notag\\
			&\leq C
			+C\Sob{{D}^mv}{L^2}^2+\frac{\gam}{100}\Sob{{D}^mv}{L^2}^2.\notag
	\end{align}
On the other hand, we have
    \begin{align}\notag
        |\lb D^m\eta, D^mv\rb|\leq \frac{\gam}{100}\Sob{D^mv}{L^2}^2+C\Sob{\eta^\eps}{L^\infty_TH^m}^2.
    \end{align}
Therefore
	\begin{align}
		\frac{d}{dt}\Sob{{D}^mv}{L^2}^2+2\eps\Sob{{D}^{k+2}v}{L^2}^2
		\leq C
		\Sob{{D}^mv}{L^2}^2+C\left(\Sob{u_0}{L^2}^2+\Sob{\eta^\eps}{L^\infty_TH^m}^2\right).\notag
	\end{align}
By Gronwall's inequality, we obtain
	\begin{align}
		\Sob{{D}^mv(t)}{L^2}^2+2\eps\int_0^t\exp\left(-C(t-s)\right)
		\Sob{{D}^{m+2}v(s)}{L^2}^2ds
		\leq C\exp\left(CT\right)
		\left(\Sob{u_0^\eps}{H^m}^2+\Sob{\eta^\eps}{L^\infty_TH^m}^2\right).\notag
	\end{align}
We complete the proof by arguing as in \eqref{est:Linfty:L2} and \eqref{est:L2:H2}, and using the fact that $u_0^\eps,\eta^\eps$ are spectrally localized to shell of wave-numbers $|k|\leq\eps^{-p}$.
\end{proof}	

Next we establish an $\eps$-dependent stability estimate in $L^\infty_TH^m$ for the approximating sequence $v^\eps$ of \eqref{v:gen:eqn}.

\begin{Prop}\label{prop:veps:stab}
Let $\eps\in(0,1]$ and $m\geq2$. Let $u_0^1,u_0^2\in{H}^m$, $\eta_1,\eta_2\in L^\infty(0,T;H^m)$, and $\xi_1,\xi_2\in L^\infty_{loc}([0,\infty);H^{m+3})$. Let $v_1^\eps$ and  $v_2^\eps$ denote the unique solutions of \eqref{v:par:reg:eqn} corresponding to data $(u_0^1)^{\eps},\eta^\eps,\xi$ and $(u_0^2)^{\eps},\eta^\eps,\xi$, respectively.  There exists a constant $C\geq1$, independent of $\eps$, such that
	\begin{align}
	\Sob{v_1^\eps-v_2^\eps}{L^\infty_TH^m}\leq C\eps^{-2/5}\left(\Sob{u_0^1-u_0^2}{H^m}+\Sob{\eta_1-\eta_2}{L^2_TH^m}+\Sob{\xi_1-\xi_2}{L^2_TH^{m+1}}\right).\notag
	\end{align}
\end{Prop}

\begin{proof}
Let 
    \[
    w=v_1-v_2, \quad q=(v_1+v_2)/2,\quad  \psi=\eta_1-\eta_2,\quad \z=\xi_1-\xi_2,\quad w_0=u_0^1-u_0^2.
    \]
Suppose $t\in(0,T]$. Then, for each $\ell=0,\dots, m$, we have
	\begin{align}
		\bdy_t{D}^\ell w+\gam{D}^\ell w+{D}^3{D}^\ell w+\eps{D}^4{D}^\ell w=-D^{\ell+1}((q+\xi_1)w)-D^{\ell+1}(\z v_2)+\psi.\notag
	\end{align}

Consider $\ell=0$. Then
	\begin{align}
		\frac{1}2\frac{d}{dt}\Sob{w}{L^2}^2+\gam\Sob{w}{L^2}^2+\eps\Sob{{D}^2w}{L^2}^2=-\lb({Dq}+{D\xi_1}),w^2\rb-\lb D(\z v_2)+\psi,w\rb.\notag
	\end{align}
By H\"older's inequality, the Sobolev embedding $L^\infty\imb H^1$, \cref{lem:I2v}, and Young's inequality, it follows that
	\begin{align}
		|\lb {Dq}+{D\xi_1},w^2\rb|&\leq \left(\Sob{Dv_1}{L^\infty}+\Sob{Dv_2}{L^\infty}+\Sob{D\xi_1}{L^\infty}\right)\Sob{w}{L^2}^2\leq C\Sob{w}{L^2}^2.\notag\\
		|\lb D(\z v_2)+\psi,w\rb|&\leq C\left(\Sob{\z}{H^1}\Sob{v_2}{H^1}+\Sob{\psi}{L^2}\right)\Sob{w}{L^2}\leq C\left(\Sob{\z}{H^1}^2+\Sob{\psi}{L^2}^2\right)+\gam\Sob{w}{L^2}^2\notag.
	\end{align}
Thus, by Gronwall's inequality, we have
	\begin{align}\label{veps:stab:l2}
		\Sob{w(t)}{L^2}&\leq C\exp(CT)\left(\Sob{w_0}{L^2}+\Sob{\z}{L^\infty_TH^1}+\Sob{\psi}{L^\infty_TL^2}^2\right)\notag\\
		&\leq C\left(\Sob{w_0}{L^2}+\Sob{\z}{L^2_TH^1}+\Sob{\psi}{L^2_TL^2}^2\right).
	\end{align}
	
Now consider $\ell=1$, we have
	\begin{align}
		\frac{1}2\frac{d}{dt}\Sob{{Dw}}{L^2}^2+\gam\Sob{{Dw}}{L^2}^2+\eps\Sob{D^3w}{L^2}^2=&-\lb {D^2q}+{D^2\xi},w{Dw}\rb-\frac{3}2\lb {Dq}+{D\xi},{Dw}^2\rb\notag\\
		&-\lb D^2(\z v_2)+D\psi, Dw\rb=I+II+III.\notag
	\end{align}
We estimate $I, II$ using H\"older's inequality, interpolation, and Poincar\'e's inequality to obtain
	\begin{align}
		|I|+|II|&\leq C\Sob{w}{L^\infty}\Sob{{Dw}}{L^2}+\frac{3}2\left(\Sob{{Dq}}{L^\infty}+\Sob{{D\xi_1}}{L^\infty}\right)\Sob{{Dw}}{L^2}^2\notag\\
		&\leq C\Sob{{Dw}}{L^2}^{3/2}\Sob{w}{L^2}^{1/2}+C\Sob{{Dw}}{L^2}^2\leq C\Sob{{Dw}}{L^2}^2.\notag
	\end{align}
Similarly, for $III$, by additionally applying \cref{lem:I2v} and Young's inequality, we have
    \begin{align}
        |III|&\leq C\left(\Sob{\z}{H^2}\Sob{v_2}{H^2}+\Sob{\psi}{H^1}\right)\Sob{Dw}{L^2} \leq C\left(\Sob{\z}{H^2}^2+\Sob{\psi}{H^1}^2\right)+\gam\Sob{Dw}{L^2}^2.\notag
    \end{align}
It then follows from Gronwall's inequality 
	\begin{align}\label{veps:stab:h1}
		\Sob{{Dw}}{L^2}\leq C\exp(CT)\left(\Sob{{D}w_0}{L^2}+\Sob{\z}{L^2_TH^2}+\Sob{\psi}{L^2_TH^1}\right).
	\end{align}

Finally, suppose $\ell\geq2$. Then
	\begin{align}
		\frac{1}2\frac{d}{dt}\Sob{{D}^\ell w}{L^2}^2+\gam\Sob{{D}^\ell w}{L^2}^2+\eps\Sob{{D}^{\ell+2}w}{L^2}^2=-\lb{D}^{\ell+1}((q+\xi)w),{D}^\ell w\rb-\lb D^{\ell+1}(\z v_2)+D^\ell \psi, D^\ell w\rb.\notag
	\end{align}
We rewrite the right-hand side as in \eqref{III:leibniz}, then  proceed similarly to \eqref{veps:w:Hm:toporder}, \eqref{veps:w:Hm:secondorder}, and \eqref{veps:w:Hm:commutator}, estimating with H\"older's inequality, \eqref{veps:stab:l2}, \eqref{veps:stab:h1}, Young's inequality, interpolation, and Poincar\'e's inequality to obtain 
	\begin{align}
	&|\lb{D}^{\ell+1}((q+\xi_1)w),{D}^\ell w\rb|\notag\\
	&\leq |\lb (D^{\ell+1}(q+\xi_1))w,D^\ell w\rb|+c|\lb (q+\xi)(D^{\ell+1}w), D^\ell w\rb|+c\sum_{k=1}^{\ell-1}|\lb (D^{\ell+1-k}(q+\xi_1))(D^kw),D^\ell w\rb|\notag\\
	&\leq c\bigg[\Sob{q+\xi_1}{H^{\ell+1}}\Sob{Dw}{L^2}^{1/2}\Sob{w}{L^2}^{1/2}+\Sob{q+\xi}{H^2}\Sob{D^\ell w}{L^2}
	\notag \\
	& \quad \quad
	+\sum_{k=1}^{\ell-1}\Sob{q+\xi_1}{H^{\ell+1}}^{\frac{\ell-k}{\ell}}\Sob{q+\xi_1}{H^1}^{\frac{k}{\ell}}\Sob{D^\ell w}{L^2}^{\frac{2k+1}{2\ell}}\Sob{w}{L^2}^{\frac{2(\ell-k)-1}{2\ell}}\bigg]\Sob{D^\ell w}{L^2}\notag\\
	&\leq c\bigg[(\Sob{q+\xi_1}{H^{\ell+1}}+1)^{\ell+1}\Sob{Dw}{L^2}+(\Sob{q+\xi_1}{H^1}+1)^{\frac{2(\ell-1)}{2(\ell-1)+1}}\Sob{D^\ell w}{L^2}
	\bigg]\Sob{D^\ell w}{L^2}+C\Sob{D^\ell w}{L^2}^2\notag\\
	&\leq C\left(\Sob{v^1}{L^\infty_TH^{\ell+1}}+\Sob{v^2}{L^\infty_TH^{\ell+1}}+1\right)^{2(\ell+1)}\Sob{w_0}{H^1}^2+C\Sob{D^\ell w}{L^2}^2.\label{est:lin:term:Hm}
    \end{align}
Estimating as above and copiously applying Poincar\'e's inequality, we also have
    \begin{align}
    &|\lb D^{\ell+1}(\z v_2)+D^\ell\psi,D^\ell w\rb|\notag\\
    &\leq c\bigg[\Sob{\z}{H^{\ell+1}}\Sob{Dv_2}{L^2}+\Sob{\z}{H^2}\Sob{D^{\ell+1} v_2}{L^2}
	+\Sob{\z}{H^{\ell+1}}\Sob{D^\ell v_2}{L^2}+\Sob{D^\ell\psi}{L^2}\bigg]\Sob{D^\ell w}{L^2}\notag\\
	&\leq c\left(\Sob{\z}{H^{\ell+1}}\Sob{v_2}{H^{\ell+1}}+\Sob{D^\ell\psi}{L^2}\right)\Sob{D^\ell w}{L^2}\leq C\Sob{\z}{H^{\ell+1}}^2\Sob{v_2}{H^{\ell+1}}^2+C\Sob{\psi}{H^\ell}^2+\gam\Sob{D^\ell w}{L^2}^2\label{est:forcing:term:Hm}.
    \end{align}
Hence, by Gronwall's inequality and \cref{lem:Hmv} it follows that
	\begin{align}
		&\Sob{{D}^\ell w(t)}{L^2}^2\leq C\Sob{{D}^\ell w_0}{L^2}^2+C\left(\Sob{v^1}{L^\infty_TH^{\ell+1}}+\Sob{v^2}{L^\infty_TH^{\ell+1}}+1\right)^{2(\ell+1)}\Sob{w_0}{H^1}^2\notag\\
		&+C\left(\Sob{\z}{L^2_TH^{\ell+1}}^2\Sob{v_2}{L^\infty_TH^{\ell+1}}^2+\Sob{\psi}{L^2_TH^\ell}^2\right)\notag\\
		&\leq C\left(\Sob{{D}^\ell w_0}{L^2}^2+\Sob{\psi}{L^2_TH^\ell}^2\right)+\frac{C}{\eps^{2p(\ell+1)}}\left(\Sob{u_0^1}{H^\ell}^2+\Sob{u_0^2}{H^\ell}^2+\Phi_\ell^2\right)^{\ell+1}\left(\Sob{w_0}{H^1}^2+\Sob{\z}{L^2_TH^{\ell+1}}\right),\notag
	\end{align}
holds for all $\ell=2,\dots, m$, provided that $p\in(0,1/5]$.
The result follows upon choosing $p=2/(5(m+1))$ and using the fact that $\eps\leq1$.
\end{proof}

\subsection{Convergence of Approximations}\label{subsect:cauchyinHm}

We now show that for $m\geq2$ and $T>0$, the sequence $\{v^\eps\}_{\eps>0}$ is Cauchy in the topology of $C([0,T];H^m)$, whenever that $u_0\in {H}^m$, $\eta\in L^\infty(0,T;H^m)$, and $\xi\in L^\infty_{loc}([0,T);H^{m+3})$. Note that since $v^\eps\in C([0,T];H^m)$, it will suffice to show that $\{v^\eps\}_{\veps>0}$ is Cauchy in $L^\infty(0,T;H^m)$. We will consider the evolution of the difference $w=v^{\eps}-v^{\eps'}$. For this, it will be convenient to let
    \begin{align}\label{def:q:psi}
        q=(v^{\eps}+v^{\eps'})/2,\quad \psi=\eta^\eps-\eta^{\eps'},\quad \z=\xi^\eps-\xi^{\eps'},
    \end{align}
where $\xi^\eps=P_{\eps^{-p}}\xi$.

\begin{Lem}\label{lem:cauchy:L2}
Let $p\in(0,1/4]$ and $0<\eps'\leq\eps\leq1$. Let $T>0$ and $v^\eps, v^{\eps'}$ denote the unique solutions of \eqref{v:par:reg:eqn} on $(0,T]$ corresponding to data $u_0^{\eps}, \eta^\eps,\xi$, and $u_0^{\eps'}, \eta^{\eps'},\xi$, respectively. There exists a constant $C\geq1$, independent of $\eps,\eps', p$, such that
	\begin{align}\label{eq:cauchy:L2}
		\Sob{v^\eps-v^{\eps'}}{L^\infty_TL_x^2}\leq C
		\eps^{2p},
	\end{align}
In particular, $\{v^\eps\}_{\eps>0}$ is Cauchy in $L^\infty_TL^2_x$.
\end{Lem}

\begin{proof}
We have
	\begin{align}\notag
		\bdy_tw+\gam w+D^3w+\eps D^4w=-(\eps-\eps')D^4v^{\eps'}-{D}((q+\xi^{\eps})w)+\psi.
	\end{align}
Then, upon taking the $L^2$-inner product with $w$ and integrating by parts, we obtain
	\begin{align}\notag
		\frac{1}2\frac{d}{dt}\Sob{w}{L^2}^2+\gam\Sob{w}{L^2}^2+\eps\Sob{D^2w}{L^2}^2=-(\eps-\eps')\lb D^2v^{\eps'},D^2w\rb-\frac{1}2\lb Dq+D\xi^{\eps'},w^2\rb-\lb D(\z v^\eps),w\rb+\lb \psi,w\rb.
	\end{align}
Applying H\"older's inequality, the Cauchy-Schwarz inequality, Young's inequality, and \cref{lem:I2v} yields
	\begin{align}
		\frac{d}{dt}\Sob{w}{L^2}^2+\gam\Sob{w}{L^2}^2
		\leq C{|\eps-\eps'|}
		+C\Sob{w}{L^2}^2+C\Sob{\z}{H^1}^2 +C\Sob{\psi}{L^2}^2.\notag
	\end{align}
Since $\eps'\leq\eps$, by Gronwall's inequality and \cref{lem:I2v} we have
	\begin{align}\notag
		\Sob{w(t)}{L^2}^2
		\leq C\exp\left(CT\right)
		\left(\Sob{w_0}{L^2}^2+
		+\Sob{\z}{H^1}^2+\Sob{\psi}{L^2}^2+\eps\right).
	\end{align}
Also, observe that $w_0=P_{\eps^{-p},(\eps')^{-p}}u_0$, which implies $\Sob{w_0}{L^2}\leq c\eps^{2p}\Sob{u_0}{H^2}$. Similarly $\Sob{\psi}{L^2}\leq c\eps^{2p}\Phi_2\leq C\eps^{2p}$, upon making use of \eqref{est:eta:gen}. Since $p\leq1/4$, we have $\eps\leq\eps^{4p}$. Upon combining these estimates, we establish the desired inequality.
\end{proof}

\begin{Lem}\label{lem:cauchy:H1}
Let $p\in(0,1/4]$ and $0<\eps'\leq\eps\leq1$. Let $T>0$ and $v^\eps, v^{\eps'}$ be the unique solutions of \eqref{v:par:reg:eqn} on $[0,T]$ corresponding to data $u_0^{\eps}, \eta^\eps,\xi$ and $u_0^{\eps'}, \eta^{\eps'},\xi$.
There exists a constant $C\geq1$, independent of $\eps,\eps'$, such that
	\begin{align}
	\Sob{v^\eps-v^{\eps'}}{L^\infty_TH^1_x}&\leq C
	\eps^{p}.\notag
	\end{align}
In particular $\{v^\eps\}_{\eps>0}$ is Cauchy in $L^\infty_TH^1_x$.
\end{Lem}

\begin{proof}
We have
	\begin{align}
		\bdy_t{Dw}+\gam {Dw}+{D}^{3}{Dw}+\eps{D}^{5}w=-(\eps-\eps'){D}^{5}v^{\eps}-{D}^{2}((q+\xi)w)+{D\psi}.\notag
	\end{align}
Upon taking the $L^2$-inner product with ${Dw}$ and integrating by parts, we obtain
	\begin{align}\notag
		&\frac{1}2\frac{d}{dt}\Sob{{Dw}}{L^2}^2+\gam\Sob{{Dw}}{L^2}^2+\eps\Sob{{D}^3w}{L^2}^2\notag\\
		&=(\eps-\eps')\lb{D}^4v^{\eps'},{D}^2w\rb-\lb {D^2q}+{D^2\xi},w{Dw}\rb-\frac{1}2\lb {Dq}+{D\xi},({Dw})^2\rb+\lb {D\psi},{Dw}\rb.\notag
	\end{align}
We estimate the right-hand side term by term with H\"older's inequality, Cauchy-Schwarz, Young's inequality, interpolation, and the fact that $\eps'\leq\eps$ to obtain
	\begin{align}
		&\frac{1}2\frac{d}{dt}\Sob{{Dw}}{L^2}^2+\gam\Sob{{Dw}}{L^2}^2+\eps\Sob{{D}^3w}{L^2}^2\notag\\
		&\leq |\eps-\eps'|\Sob{D^4v^{\eps'}}{L^2}\Sob{D^3w}{L^2}^{1/2}\Sob{Dw}{L^2}^{1/2}+\left(\Sob{|D^2v^\eps|+|D^2v^{\eps'}|+|D^2\xi|}{L^2}\right)\Sob{Dw}{L^2}^{3/2}\Sob{w}{L^2}^{1/2}\notag\\
		&\qquad\qquad +\Sob{D\psi}{L^2}\Sob{Dw}{L^2}\notag\\
		&\leq C\eps^{3/2}\Sob{D^4v^{\eps'}}{L^2}^2+\frac{\eps}{100}\Sob{D^3w}{L^2}^2+C\Sob{D\psi}{L^2}^2+\frac{\gam}{100}\Sob{Dw}{L^2}^2+C\Sob{w}{L^2}^2.\notag
	\end{align}
Then by Gronwall's inequality, \cref{lem:Hmv} with $m=2$, and the fact that $w_0=P_{\eps^{-p},(\eps')^{-p}}u_0$, $\psi=P_{\eps^{-p},(\eps')^{-p}}\eta$, \eqref{est:eta:gen}, and $p\leq1/4$, we have
	\begin{align}
		\Sob{{Dw}(t)}{L^2}^2&\leq C\eps^{1/2}+C\eps^{2p}+C\eps^{4p}\leq C\eps^{2p},\notag
	\end{align}
as desired.

\end{proof}

\begin{Lem}\label{lem:cauchy:Hm}
Let $p\in(0,1/5]$, $0<\eps'\leq\eps\leq1$, and $m\geq2$. Let $T>0$ and $v^\eps, v^{\eps'}$ be the global smooth solutions of \eqref{v:par:reg:eqn} on $[0,T]$ corresponding to data $u_0^{\eps}, \eta^\eps,\xi$, and $u_0^{\eps'}, \eta^{\eps'},\xi$, respectively.
There exists a constant, $C\geq1$, independent of $\eps',\eps$, such that
\begin{align}\notag
		\Sob{v^\eps-v^{\eps'}}{L^\infty_TH^m}\leq C\left(\Sob{u_0^\eps-u_0^{\eps'}}{H^m}+\Sob{\xi^\eps-\xi^{\eps'}}{L^\infty_TH^{m+3}}+\max\{\eps,\eps^{\frac{8}{4m-5}}\}^{p/2}\left(\Sob{u_0}{H^m}+\Phi_m\right)\right).
	\end{align}
In particular, $\{v^\eps\}_{\eps>0}$ is Cauchy in $L^\infty_TH^m$.
\end{Lem}

\begin{proof}
Observe that
	\begin{align}
		\bdy_t{D}^mw+\gam{D}^mw+{D}^{m+3}w+\eps{D}^{m+4}w&=-(\eps-\eps'){D}^{m+4}v^{\eps'}+\frac{1}2D^{m+1}(w^2)-{D}^{m+1}((v^\eps+\xi)w)+{D}^m\psi\notag\\
		&=I+II+III+IV.\notag
	\end{align}
	
To treat $I$, we apply the Cauchy-Schwarz inequality, \cref{lem:Hmv} with $m\mapsto m+2$, and Young's inequality to estimate
	\begin{align}\label{est:I:Hm}
		|\lb I,{D}^mw\rb|&\leq(\eps-\eps')\Sob{{D}^{m+2}v^\eps}{L^2}\Sob{{D}^{m+2}w}{L^2}\leq c\eps\Sob{v^\eps}{H^{m+2}}^2+\frac{\eps}{1000}\Sob{D^mw}{L^2}^2\notag\\
		&\leq c\left(\Sob{u_0}{H^m}+\Phi_m\right)^2\eps^{1-4p}+\frac{\eps}{1000}\Sob{D^mw}{L^2}^2.
	\end{align}

For $II$, we first apply the Leibniz rule to rewrite it as
	\begin{align}
		II=w{D}^{m+1}w+\frac{1}2\sum_{\ell=1}^{m}\begin{pmatrix}m+1\\ \ell\end{pmatrix}({D}^{m+1-\ell}w)({D}^\ell w).\notag
	\end{align}
Observe that if $m=2$, then upon integrating by parts, applying H\"older's inequality, interpolating, then applying \cref{lem:Hmv} we have 
    \begin{align}\label{est:II:H2}
        |\lb II,D^mw\rb|&=\frac{5}2|\lb Dw,(D^mw)^2\rb|
        \leq c\Sob{Dw}{L^\infty}\Sob{D^mw}{L^2}^2\leq C\Sob{D^mw}{L^2}^2.
    \end{align}
Suppose then that $m\geq3$. then upon integrating by parts, it follows that
	\begin{align}
		\lb II,{D}^mw\rb=(m+1/2)\lb {Dw},({D}^mw)^2\rb+\frac{1}2\sum_{\ell=2}^{m-1}\begin{pmatrix}m+1\\ \ell\end{pmatrix}\lb({D}^{m+1-\ell}w)({D}^\ell w),{D}^mw\rb.\notag
	\end{align}
By H\"older's inequality, the Sobolev embedding $H^{1/4}\imb L^4$, interpolation, the Poincar\'e inequality, and applying \cref{lem:Hmv}, for $2\leq\ell\leq m-1$  we have
	\begin{align}\label{est:II:Hm}
		|\lb({D}^{m+1-\ell}w)({D}^\ell w),{D}^mw\rb|&\leq\Sob{{D}^{m+1-\ell}w}{L^4}\Sob{{D}^\ell w}{L^4}\Sob{{D}^mw}{L^2}\notag\\
		&\leq c\Sob{w}{H^{\ell+1/4}}\Sob{w}{H^{m-\ell+5/4}}\Sob{D^mw}{L^2}\notag\\
		&\leq c\Sob{D^mw}{L^2}^{\frac{2m-9/2}{m-2}}\Sob{D^2w}{L^2}^{\frac{m-3/2}{m-2}}\notag\\
		&\leq c\left(\Sob{D^2v^\eps}{L^2}+\Sob{D^2v^{\eps'}}{L^2}\right)\Sob{D^mw}{L^2}^2\notag\\
		&\leq C\Sob{D^mw}{L^2}^2.
	\end{align}
Thus, by \eqref{est:II:H2} and \eqref{est:II:Hm}, for each $m\geq2$, we have the estimate
    \begin{align}\notag
        |\lb II, D^mw\rb|\leq C\Sob{D^mw}{L^2}^2.
    \end{align}
	
For $III$, we again use the Leibniz rule to write
	\begin{align}\label{III:leibniz}
		\lb{D}^{m+1}(v^{\eps} w),{D}^mw\rb
		&=\lb({D}^{m+1}v^{\eps})w,{D}^mw\rb+(m+1/2)\lb Dv^{\eps},({D}^mw)^2\rb \notag \\
		&\quad +\sum_{\ell=1}^{m-1}\begin{pmatrix}m+1\\ \ell\end{pmatrix}\lb({D}^{m+1-\ell}v^{\eps})({D}^\ell w),D^mw\rb.
	\end{align}
For the first term, we apply H\"older's inequality, interpolation, Young's inequality, the Poincar\'e inequality, and \cref{lem:Hmv}, \cref{lem:cauchy:L2}, \cref{lem:cauchy:H1} to estimate   
    \begin{align}\label{veps:w:Hm:toporder}
        |\lb (D^{m+1}v^{\eps})w,D^mw\rb|&\leq \Sob{D^{m+1}v^{\eps}}{L^2}\Sob{w}{L^\infty}\Sob{D^mw}{L^2}\notag\\
        &\leq \Sob{D^{m+1}v^{\eps}}{L^2}\Sob{Dw}{L^2}^{1/2}\Sob{w}{L^2}^{1/2}\Sob{D^mw}{L^2}\notag\\
        &\leq C\eps^{3p}\Sob{v^{\eps}}{H^{m+2}}\Sob{v^\eps}{H^m}+\frac{\gam}{1000}\Sob{D^mw}{L^2}^2\notag\\
        &\leq C\left(\Sob{u_0}{H^m}+\Phi_m\right)^2\eps^p+\frac{\gam}{1000}\Sob{D^mw}{L^2}^2.
    \end{align}
On the other hand, by H\"older's inequality, the Sobolev embedding $H^1\imb L^\infty$, and \cref{lem:Hmv}, we have
    \begin{align}\label{veps:w:Hm:secondorder}
        |\lb Dv^{\eps}, (D^mw)^2\rb|\leq \Sob{Dv^{\eps}}{L^\infty}\Sob{D^mw}{L^2}^2\leq C\Sob{v^\eps}{H^2}\Sob{D^mw}{L^2}^2\leq C\Sob{D^mw}{L^2}^2.
    \end{align}
Observe that for $\ell=1,\dots, m-1$, by interpolation, Young's inequality, the Sobolev embedding $H^1\imb L^\infty$, and \cref{lem:Hmv}, \cref{lem:cauchy:L2}, and \cref{lem:cauchy:H1} we have
	\begin{align}\label{veps:w:Hm:commutator}
		|\lb{D}^{m+1-\ell}v^{\eps}{D}^\ell w,{D}^mw\rb|&\leq\Sob{{D}^{m+1-\ell}v^{\eps}}{L^4}\Sob{{D}^\ell w}{L^4}\Sob{{D}^mw}{L^2}\notag\\
		&\leq \Sob{{D}^{m+2}v^{\eps}}{L^2}^{\frac{4(m-\ell)-3}{4m}}\Sob{D^2v^{\eps}}{L^2}^{\frac{4\ell+3}{4m}}\Sob{{D}^mw}{L^2}^{\frac{4(m+\ell)+1}{4m}}\Sob{w}{L^2}^{\frac{4(m-\ell)-1}{4m}}\notag\\
		&\leq C\left(\Sob{u_0}{H^m}+\Phi_m\right)^{\frac{4(m-\ell)-3}{4m}}\eps^{\frac{p}m}\Sob{D^mw}{L^2}^{\frac{4(m+\ell)+1}{4m}}\notag\\
		&\leq C\left(\Sob{u_0}{H^m}+\Phi_m\right)^{\frac{8(m-\ell)-6}{4(m-\ell)-1}}\eps^{\frac{8p}{4(m-\ell)-1}}+\frac{\gam}{1000}\Sob{D^mw}{L^2}^2.
	\end{align}
Observe that $8p/(4(m-\ell)-1)\geq8p/(4m-5)$, for all $\ell=1,\dots, m-1$. Upon returning to \eqref{III:leibniz} and combining \eqref{veps:w:Hm:toporder}, \eqref{veps:w:Hm:secondorder}, \eqref{veps:w:Hm:commutator}, we obtain 
	\begin{align}
		|\lb D^{m+1}(v^\eps w),{D}^mw\rb|\leq C\left(\Sob{u_0}{H^m}+\Phi_m\right)^2\max\{\eps,\eps^{\frac{8}{4m-5}}\}^p+C\Sob{D^mw}{L^2}^2+\frac{\gam}{100}\Sob{D^mw}{L^2}^2.\notag
	\end{align}
Arguing similarly for the second summand in $III$, since $p<1$ we therefore have
	\begin{align}\notag
		|\lb III,{D}^mw\rb|&\leq C\left(\Sob{u_0}{H^m}+\Phi_m\right)^2\eps^p+C\eps^{3p}+C\Sob{D^mw}{L^2}^2+\frac{\gam}{100}\Sob{D^mw}{L^2}^2\notag\\
		&\leq C\left(\Sob{u_0}{H^m}+\Phi_m\right)^2\eps^{p}+C\Sob{D^mw}{L^2}^2+\frac{\gam}{100}\Sob{D^mw}{L^2}^2.\notag
	\end{align}
	
Lastly, to estimate $IV$, we apply Cauchy-Schwarz, Young's inequality, and the fact that $\eta$ is linear in the variable corresponding to $D^3\xi$, along with \eqref{est:eta:gen}, to obtain
	\begin{align}
		|\lb IV,{D}^mw\rb|&\leq C\Sob{{D}^m\psi}{L^2}^2+\frac{\gam}{100}\Sob{{D}^mw}{L^2}^2\leq C\Sob{\xi^\eps-\xi^{\eps'}}{H^{m+3}}+\frac{\gam}{100}\Sob{D^mw}{L^2}^2.\notag
	\end{align}
	
Finally, combining the estimates for $I$--$IV$, we deduce
	\begin{align}
		&\frac{d}{dt}\Sob{{D}^mw}{L^2}^2+\frac{\gam}{10}\Sob{D^mw}{L^2}^2+\frac{\eps}{10}\Sob{D^{m+2}w}{L^2}^2\notag\\
		&\leq c\left(\Sob{u_0}{H^m}+\Phi_m\right)^2\eps^{1-4p}+ C\Sob{D^mw}{L^2}^2+C\left(\Sob{u_0}{H^m}+\Phi_m\right)^2\eps^{p}+C\Sob{\xi^\eps-\xi^{\eps'}}{H^{m+3}}^2.\notag
	\end{align}
By Gronwall's inequality, we have
    \begin{align}
        \Sob{D^mw(t)}{L^2}^2
        &\leq C\left(\Sob{u_0}{H^m}+\Phi_m\right)^2\eps^{1-4p}+C\left(\Sob{u_0}{H^m}+\Phi_m\right)^2\eps^{p}+C\left(\Sob{D^mw_0}{L^2}^2+\Sob{\xi^\eps-\xi^{\eps'}}{H^{m+3}}^2\right)\notag.
    \end{align}
Therefore, since $p\leq 1/5$, we have
    \begin{align}
        \Sob{D^mw(t)}{L^2}^2\leq C\left(\Sob{u_0}{H^m}+\Phi_m\right)^2\max\{\eps,\eps^{\frac{8}{4m-5}}\}^p+C\left(\Sob{D^mw_0}{L^2}^2+\Sob{{D}^m\psi}{L^2_TL^2_x}^2\right),\notag
    \end{align}
for $0\leq t\leq T$, as claimed.
\end{proof}

Next, we deduce convergence.

\begin{Prop}\label{prop:cauchy:Hm}
Let $0<\eps'\leq\eps\leq1$ and $m\geq2$. Let $u_0\in{H}^m$, $\eta\in L^\infty_{loc}([0,\infty)H^m)$, $\xi\in L^\infty_{loc}([0,\infty);H^{m+3})$. Let $v^\eps, v^{\eps'}\in C([0,\infty);H^m)\cap L^2_{loc}(0,\infty);H^{m+2})$ be the unique solutions of \eqref{v:par:reg:eqn} corresponding to $u_0^{\eps}, \eta^{\eps},\xi$, and $u_0^{\eps'}, \eta^{\eps'},\xi$, respectively. Then for each $\mu,\nu,\rho,T>0$ 
	\begin{align}\label{eq:cauchyinmeasure:Hm}
		\lim_{\eps\goesto0}	\sup_{\Sob{\xi}{L^\infty_TH^{m+3}}<\rho}\sup_{\Sob{\eta}{L^\infty_TH^m}<\nu}\sup_{\Sob{u_0}{H^m}<\mu}\Sob{v^\eps-v^{\eps'}}{L^\infty_TH^m}=0.
	\end{align}
In particular, there exists a subsequence $\{\eps''\}_{\eps''>0}\subseteq\{\eps\}_{\eps>0}$, and $v\in C([0,\infty);H^m)$ such that  $v(0)=u_0$ and inherits the bounds stated in \cref{lem:L2v}, \cref{lem:H1v} and \cref{lem:Hmv} evaluated at $\eps=0$.
and $v^{\eps''}(\cdotp; u_0^{\eps''},\eta^{\eps''},\xi)|_{[0,T]}\goesto v$ in $L^\infty_TH^m$ as $\eps''\goesto0$, for all $u_0\in{H}^m$, $\eta\in L^\infty_{loc}([0,\infty)H^m)$, $\xi\in L^\infty_{loc}([0,\infty);H^{m+3})$, and $T>0$.
\end{Prop}

\begin{proof}
Observe that the estimates established in  \cref{lem:L2v}, \cref{lem:H1v}, \cref{lem:I2v}, \cref{lem:Hmv}, and \cref{lem:cauchy:L2}, \cref{lem:cauchy:H1}, \cref{lem:cauchy:Hm} are uniform in $u_0,\eta,\xi$, for $u_0\in\{\Sob{u_0}{H^m}\leq \mu\}$, $\eta\in\{\Sob{\eta}{L^\infty_TH^m}\leq \nu\}$, and $\xi\in\{\Sob{\xi}{L^\infty_TH^{m+3}}\leq \rho\}$, for any fixed $\mu,\nu,\rho>0$ and $m\geq2$. Therefore, by \cref{lem:cauchy:Hm} with any $p<1/5$
    \begin{align}\notag
    		\sup_{\Sob{\xi}{L^\infty_TH^{m+3}}<\rho}\sup_{\Sob{\eta}{L^\infty_TH^m}<\nu}&\sup_{\Sob{u_0}{H^m}<\mu}\Sob{v^\eps-v^{\eps'}}{L^\infty_TH^m}\notag\\
    		&\leq C(\mu,\nu,\rho,T)\left(\Sob{u_0^\eps-u_0^{\eps'}}{H^m}+\Sob{\xi^\eps-\xi^{\eps'}}{H^{m+3}}+\max\{\eps,\eps^{\frac{8}{4m-5}}\}^{1/10}\right).\notag
    \end{align}
Upon sending $\eps,\eps'\goesto0$, we obtain \eqref{eq:cauchyinmeasure:Hm}. The second claim follows upon diagonalizing the sequence along increasingly large values of $\mu,\nu,\rho$.
\end{proof}

Finally, we specialize to the case that $\xi=\s W$, where $\s\in\bH^{m}$ and $W=\{W_k\}_{k\geq1}$ denotes i.i.d., real-valued, standard Brownian motions (see \eqref{eq:sigW:def}). Observe that 
$\xi\in C([0,\infty),H^{m})$ a.s. In particular, $\xi$ is a continuous, square-integrable $H^{m}$-valued martingale, so that (cf. \cite[Theorem 3.9]{DaPratoZabczyk1996})
    \begin{align}\label{eq:mart:ineq}
        \Prb
        \left(\sup_{t\in[0,T]}\Sob{\xi(t)}{H^{m}}\geq\rho\right)\leq\rho^{-2}\sup_{t\in[0,T]}\E\Sob{\xi(t)}{H^{m}}^2<\infty,
    \end{align}
holds for all $\rho>0$.

\begin{Prop}\label{prop:cauchyinmeasure:Hm}
Let $0<\eps'\leq\eps\leq1$ and $m\geq2$. Let $u_0\in{H}^m$, $\eta\in L^\infty_{loc}([0,\infty)H^m)$, and $\xi:=\s W$, where $\s\in\bH^{m+3}$, and $\{W_k\}_{k\geq1}$ is a sequence of identically distributed, independent, one-dimensional, standard Brownian motions. Let $v^\eps, v^{\eps'}$ be the unique solutions of \eqref{v:par:reg:eqn} corresponding to $u_0^{\eps}, \eta^{\eps},\xi$, and $u_0^{\eps'}, \eta^{\eps'},\xi$, respectively. Then for each $\nu,\mu,\de,T>0$ 
	\begin{align}\label{cauchyinmeasure:Hm}
		\lim_{\eps\goesto0}\Prb\left[\sup_{\Sob{\eta}{L^\infty_TH^m}<\nu}\sup_{\Sob{u_0}{H^m}<\mu}\Sob{v^\eps-v^{\eps'}}{L^\infty_TH^m}\geq \de\right]=0.
	\end{align}
In particular, there exists a subsequence $(\eps'')_{\eps''>0}\subseteq(\eps)_{\eps>0}$, and $v\in C([0,\infty);H^m)$ such that  $v(0)=u_0$, $v$ inherits the bounds stated in \cref{lem:L2v}, \cref{lem:H1v}, \cref{lem:Hmv} evaluated at $\eps=0$.
and $v^{\eps''}(\cdotp; u_0^{\eps''},\eta^{\eps''},\xi)|_{[0,T]}\goesto v$ in $L^\infty_TH^m$ a.s., for all $u_0\in{H}^m$, $\eta\in L^\infty_TH^m$, and $T>0$. In particular, for $\mu,\nu,T>0$
	\begin{align}
		\lim_{\eps''\goesto0}\sup_{\Sob{\eta}{L^\infty_TH^m}<\nu}\sup_{\Sob{u_0}{H^m}<\mu}\Sob{v^{\eps''}-v}{L^\infty_TH^m}=0,\quad {\text{a.s.}}\notag
	\end{align}
\end{Prop}

\begin{proof}
Given $0\leq \ell\leq m+3$, and $\rho,T>0$, consider  $E(\ell,\rho,T)^c:=\{\Xi_\ell(T)< \rho\}\subseteq\Om$, where $\Xi_\ell$ is defined by \eqref{def:ZF} with $\xi$. Observe that \eqref{eq:mart:ineq} implies $\lim_{\rho\goesto\infty}\Prb(E(\ell,\rho, T)^c)=1$, for all $0\leq \ell\leq m+3$.

For $\veps>0$, choose $\rho=\rho(\veps,T)$ such that 
    \begin{align}\label{def:exceptional}
        \Prb(E(m+3,\rho,T))<\veps.
    \end{align}
Let $w^{\eps,\eps'}:=v^\eps-v^{\eps'}$, $w^{\eps,\eps'}_0=u_0^\eps-u_0^{\eps'}$, and $\eta^{\eps,\eps'}=\eta^\eps-\eta^{\eps'}$. Observe that for $\w\in E(m+3,\rho,T)^c$, $u_0\in \{\Sob{u_0}{H^m}<\mu\}$, and $\eta\in\{\Sob{\eta}{L^\infty_TH^m}<\nu\}$, \cref{lem:cauchy:Hm} with any $p<1/5$ implies that
	\begin{align}
	 	\Sob{w^{\eps,\eps'}(\w)}{L^\infty_TH^m}\leq C(\mu,\nu,\rho,T)\left(\Sob{w^{\eps,\eps'}_0}{H^m}+\Sob{\xi^{\eps}-\xi^{\eps'}}{L^\infty_TH^{m+3}}+\max\{\eps,\eps^{\frac{8}{4m-5}}\}^{1/10}\right).\notag
	\end{align}
	
Now given $\de>0$, choose $\eps_0=\eps_0(\de, \veps, \mu,\nu, T)>0$ sufficiently small, so that
    \begin{align}\label{eq:cauchy:small}
        C(\mu,\nu,\rho,T)\left(\Sob{w^{\eps,\eps'}_0}{H^m}+\Sob{\xi^{\eps}-\xi^{\eps'}}{L^\infty_TH^{m+3}}+\eps^{\frac{3}{40m-50}}\right)\indFn{E(m+3
        ,\rho,T)^c}< \frac{\de}2,
    \end{align}
whenever $0<\eps'\leq\eps\leq\eps_0$. Then for $\de>0$, if $0<\eps'\leq\eps\leq\eps_0$, then 
	\begin{align}
		&\Prb\left[\Sob{w^{\eps,\eps'}}{L_T^\infty H^{m}}\geq \de\right]\notag\\
		&\leq \Prb\left[C(\mu,\nu,\rho,T)\left(\Sob{w^{\eps,\eps'}_0}{H^m}+\Sob{\xi^{\eps}-\xi^{\eps'}}{L^\infty_TH^{m+3}}+\eps^{\frac{3}{40m-50}}\right)\indFn{E(m+3,\rho,T)^c}\geq\frac{\de}2\right]+\Prb(E(m+3,\rho,T))\notag\\
		&<\veps,\notag
	\end{align}
which establishes \eqref{cauchyinmeasure:Hm}.

Furthermore, it follows that there exists a subsequence of $\{v^{\eps}(\cdotp;u_0^\eps,\eta^\eps,\xi)\}_{\eps>0}$, depending on $\mu,\nu, T$, which converges a.s. in $L^\infty_TH^m$ to some $v_{T}(\cdotp;u_0,\eta,\xi)\in L^\infty(0,T;H^m)$, where $v_T(\cdotp;u_0,\eta,\xi)$ inherits the bounds satisfied by $v^\eps$ when formally evaluating at $\eps=0$. Owing to the structure of the bounds obtained for $v^\eps$ being such that they are increasing in $\mu,\nu,T$, we may diagonalize these sequences along increasing values of $\mu,\nu,T>0$. Upon relabelling the diagonal sequence as $\{v^\eps\}_{\eps>0}$, we obtain a single sequence that is independent of $\mu,\nu, T>0$. Therefore, we ultimately have  $v^\eps(\cdotp;u_0^\eps,\eta^\eps,\xi)|_{[0,T]}\goesto v_T(\cdotp;u_0,\eta,\xi)$ a.s. in $L^\infty_TH^m$, for all $u_0\in\bigcup_{\mu>0}\{\Sob{u_0}{H^m}<\mu\}={H}^m$ and $\eta\in \bigcup_{\nu>0}\{\Sob{\eta}{L^\infty_TH^m}<\nu\}=L^\infty_TH^m$, and each $T>0$. In particular, since $v^\eps\in C([0,T];H^m)$ for all $T$, it follows that there exists $v\in C([0,\infty);H^m)$ such that $v^\eps(\cdotp;u_0^\eps,\eta^\eps,\xi)|_{[0,T]}\goesto v(\cdotp;u_0,\eta,\xi)$ a.s. in $L^\infty_TH^m$, for all $u_0\in H^m$, $\eta\in L^\infty_{loc}([0,\infty)H^m)$, and $T>0$.
\end{proof}

\subsection{Proof of Well-posedness}\label{subsect:wellposed}

With the a priori and convergence estimates in place we are now prepared to present the proof of \cref{prop:exist:uniq}.
We begin by establishing that the limit identified by a convergent subsequence determined by \cref{prop:cauchyinmeasure:Hm} is in fact the unique solution of \eqref{v:eqn}.

\begin{Prop}\label{prop:wp:veqn}
The limit, $v$, identified in \cref{prop:cauchyinmeasure:Hm} is the unique solution of \eqref{v:eqn} corresponding to data $u_0,\eta,\xi$, where $\eta$ is given by \eqref{eq:eta}.
\end{Prop}

\begin{proof}
Let $v$ denote the limit identified by the a.s. convergent subsequence of $\{v^\eps(\cdotp;u_0^\eps,\eta^\eps,\xi)\}_{\eps>0}$ guaranteed by \cref{prop:cauchyinmeasure:Hm}. For convenience, we relabel the subsequence as $\{v^\eps\}_{\eps>0}$. When $m=2$, $v$ satisfies the weak form of \eqref{v:eqn}. Indeed, upon testing \eqref{v:par:reg:eqn} by $\ph\in{H}^1$, and integrating by parts once in the dispersive term, we obtain
    \begin{align}\label{wk:form:eps}
        \frac{d}{dt}\lb v^\eps,\ph\rb-\lb D^2v^\eps, D\ph\rb+\lb D(\xi v^\eps),\ph\rb+\lb v^\eps Dv^\eps,\ph\rb+\gam\lb v^\eps,\ph\rb+\eps\lb D^4v^\eps,\ph\rb=\lb \eta^\eps,\ph\rb,
    \end{align}
where $\lb\cdotp, \cdotp\rb$ denotes inner product in ${L}^2$. Owing to the a priori bounds on $v^\eps$ and the fact that $\eta^\eps\goesto\eta$ in $L^\infty_TH^{m}$ a.s., we may pass to the limit as $\eps\goesto0$, so that
	\begin{align}\label{wk:form:H2}
	\frac{d}{dt}\lb v,\ph\rb-\lb D^2v,D\ph\rb+\lb D(\xi v),\ph\rb+\lb vDv,\ph\rb+\gam\lb v,\ph\rb=\lb \eta,\ph\rb.
	\end{align}
Note that $\frac{dv}{dt}\in L^2(0,T;H^{-1})$. In particular, $v$ satisfies \eqref{v:eqn} in the weak sense above. On the other hand, if $m>2$, then $v$ satisfies \eqref{v:eqn} in the strong sense.

To obtain uniqueness, suppose that $v_1, v_2$ are two solutions of \eqref{v:eqn} on $[0,T]$ such that $v_1(0)=v_2(0)=u_0$. Let $w=v_1-v_2$ and $q=\frac{1}2(v_1+v_2)$. Upon taking the respective differences of $v_1, v_2$  in \eqref{wk:form:H2}, then setting $\ph=w$, we obtain
	\begin{align}
		\frac{1}2\frac{d}{dt}\Sob{w}{L^2}^2+\gam\Sob{w}{L^2}^2=-\lb {D}((q+\xi)w),w\rb.\notag
	\end{align}
Observe that
	\begin{align}
		|\lb {D}((q+\xi)w),w\rb|&\leq(\Sob{{Dq}}{L^\infty}+\Sob{{D\xi}}{L^\infty})\Sob{w}{L^2}^2\notag\\
		&\leq c\left(\Sob{v_1}{L^\infty_TH^2}+\Sob{v_2}{L^\infty_TH^2}+\Sob{\xi}{L^\infty_TH^2}\right)\Sob{w}{L^2}^2.\notag
	\end{align}
Thus, by Gronwall's inequality we have that
	\begin{align}
		\Sob{w(t)}{L^2}^2\leq \exp\left[c\left(\Sob{v_1}{L^\infty_TH^2}+\Sob{v_2}{L^\infty_TH^2}+\Sob{\xi}{L^\infty_TH^2}\right)T\right]\Sob{w_0}{L^2}^2,\notag
	\end{align}
holds a.s., for all $t\in[0,T]$. Since $w_0\equiv0$, we have that $w\equiv0$ in $[0,T]$ a.s., as desired.
\end{proof}

\begin{Prop}\label{prop:feller}
Let $m\geq2$, $f\in{H}^m$, and $\s\in\bH^{m+3}$, and consider $u = v+\xi$, where $v$ is the corresponding unique solution to \eqref{v:eqn}.  Then for each $t>0$, $u(t;u_0,f,\s)$ is continuous with respect to $(u_0, f, \s)$ in  $({H}^m)^2\times\bH^{m+3}. $
\end{Prop}

\begin{proof}
Suppose $\{u_{0,n}\}_{n>0}, \{f_n\}_{n>0}\subseteq{H}^m$  such that $u_{0,n}\goesto u_0$, $f_n\goesto f$ in $H^m$ and $\{\s_n\}_{n>0}\subseteq \bH^{m+3}$ such that $\s_n\goesto \s$ in $\bH^{m+3}$. Let $\xi_n=\s_nW$ and $\xi=\s W$. Observe that $\xi_n,\xi\in L^\infty_{loc}([0,\infty);H^{m+3})$ a.s. and $\xi_n\goesto\xi$ in $L^\infty_TH^{m+3}$ a.s. Since $\eta$ is a polynomial in $5$ variables, it follows that $\eta(D_m\xi_n,f_n)\goesto \eta(D_m\xi,f)$ in $L^\infty_TH^{m+3}$ a.s. 

For $n=0$, let us denote $u_{0,0}:=u_0$, $f_0=f$, and $\s_0=\s$.  Observe that for $n$ sufficiently large, we have $\Sob{u_{0,n}}{H^m}\leq 2\Sob{u_0}{H^m}$, $\Sob{f_n}{H^m}\leq2\Sob{f}{H^m}$, and $\Sob{\s_n}{H^m}\leq 2\Sob{\s}{H^m}$, for $n$ sufficiently large. Without loss of generality, let us thus assume that $\Sob{u_{0,n}}{H^m}\leq 2\Sob{u_0}{H^m}$, $\Sob{f_n}{H^m}\leq 2\Sob{f}{H^m}$, and $\Sob{\s}{H^m}\leq 2\Sob{\s}{H^m}$, for all $n\geq0$. Let $\De_n=(u_{0,n},\eta_n,\xi_n)$, where $\xi_n=\s_n W$, and $\De_n^\eps=(u_{0,n}^\eps, \eta_n^\eps,\xi_n)$, where $\xi_n^\eps=\s_n^\eps W$ and $\eta_n^\eps=P_{\eps^{-p}}\eta_n(D_3\xi_n^\eps,f_n^\eps)$, where $p=2/(5(m+1))$, chosen as in the proof of \cref{prop:veps:stab}.
Let $T>0$ and fix $t\in(0,T]$. We recall that $u=v+\xi$ and $v$ satisfies \eqref{v:eqn}. Thus, for any $\eps>0$, we have
	\begin{align}\label{un:u:triangle}
		u(t;u_{0,n},f_n,\s_n)-u(t;u_0,f,\s)&=v(t;\De_n)-v(t;\De)+\xi_n-\xi\notag\\
				&=(v(t;\De_n)-v^\eps(t;\De_n^\eps))+(v^\eps(t;\De_n^\eps)-v^\eps(t;\De^\eps))
				\notag \\
				& \quad +(v^\eps(t;\De^\eps)-v(t;\De))+(\xi_n-\xi),
	\end{align}
for all $n\geq0$, where $v^\eps$ refers to the sequence identified in the proof of \cref{prop:cauchyinmeasure:Hm}.

For each $0\leq\ell\leq m+3$, consider the set $E(\ell,\rho,T)^c:=\{\Xi_\ell(T)<\rho\}\subseteq\Om$, where $\Xi_\ell(T)$ is defined by \eqref{def:ZF}. By \eqref{eq:mart:ineq}, it follows that $\lim_{\rho\goesto\infty}\Prb(E(\ell,\rho,T)^c)=1$. Given $\veps>0$, we choose $\rho_0=\rho_0(m,\veps, T)$, so that 
	\begin{align}\label{exceptional:set:A}
		\Prb(E(m+3,\rho,T))<{\veps}.
	\end{align}

By \cref{prop:cauchyinmeasure:Hm}, recall that there exists a sequence $(v^{{\eps}}(\cdotp;v_0^\eps,\eta^\eps,\xi))_{{\eps}>0}\subseteq C([0,T];H^m)$ satisfying \eqref{v:par:reg:eqn} such that $v^\eps(\cdotp;v_0^\eps,\eta^\eps,\xi)\goesto v(\cdotp;v_0,\eta,\xi)$ in $L^\infty_TH^m$ a.s., uniformly on bounded sets of $v_0\in{H}^m$, $\eta\in L^\infty_TH^m$.
Since $\Sob{u_{0,n}^\eps}{H^m}\leq\Sob{u_{0,n}}{H^m}\leq 2\Sob{u_0}{H^m}$, $\Sob{f_n^\eps}{H^m}\leq2\Sob{f}{H^m}$, $\Sob{\s_n}{H^{m+3}}\leq2\Sob{\s}{H^{m+3}}$, and $\eta$ is smooth in its arguments, given ${\de}>0$, there exists $\eps_0(\de)$, such that
	\begin{align}\label{veps:approx:de}
		\left(\sup_{n\geq0}\Sob{v^{\eps_{0}}(t;\De_n^{\eps_0})- v(t;\De_n)}{H^m}\right)
		<{{\de}}/3,\quad \text{a.s.}
	\end{align}
On the other hand, by \cref{prop:veps:stab} and the Poincar\'e inequality, we have, for all $\eps>0$, that
	\begin{align}\label{veps:stab:de}
		\Sob{v^{\eps_0}(t;u_{0,n}^{\eps_0})-v^{\eps_0}(t;u_0^{\eps_0})}{H^m}
		\leq C\frac{\Sob{u_{0,n}-u_0}{H^m}+\Sob{\eta_n-\eta}{L^2_TH^m}+\Sob{\xi_n-\xi}{L^2_TH^{m+1}}}{\eps_{{0}}^{2/5}},
	\end{align}
where $C$ depends only on $m,\de,\veps,\gam, T$, $\Sob{u_0}{H^m}$, $\Sob{f}{H^m}, \Sob{\xi}{L^\infty_TH^{m+3}}$, and in particular is independent of $n$.

Now observe that from \eqref{est:eta:gen}, we have
    \begin{align}
        \Sob{\eta_n-\eta}{L^2_TH^m}\leq C\left(\Sob{\xi_n-\xi}{L^2_TH^{m+3}}+\Sob{f_n-f}{H^m}\right),\notag
    \end{align}
where $C$ depends on $m,T,\Sob{\xi}{L^\infty_TH^{m+3}}$. Let us therefore choose $N=N(\de,\veps, \Sob{u_0}{H^m},\Sob{f}{H^m}, T)>0$, so that 
	\begin{align}\label{u0:approx:de}
		\Sob{u_{0,n}-u_0}{H^m}+        \Sob{f_n-f}{H^m}+\Sob{\xi_n-\xi}{L^2_TH^{m+3}}
		<\frac{\eps_{0}^{2/5}}{9C}\de,\quad n\geq N,
	\end{align}
whenever $\w\in E(m+3,\rho,T)^c$. Therefore, for $\w\in E(m+3,\rho,T)^c$ and for all $n\geq N$, upon combining \eqref{un:u:triangle}, \eqref{veps:approx:de}, \eqref{veps:stab:de}, and \eqref{u0:approx:de}, we obtain
	\begin{align}
		\Sob{u(\w,t;\De_n)-u(\w,t;\De)}{H^m}
		&\leq \Sob{v(\w,t;\De_n)-v^{\eps_0}(\w,t;\De_n^{\eps_0})}{H^m}+\Sob{v^{\eps_0}(\w,t;\De_n^{\eps_0})-v^{\eps_0}(\w,t;\De^{\eps_0})}{H^m}
		\notag \\ &\quad +\Sob{v^{\eps_0}(\w,t;\De^{\eps_0})-v(t;u_{0,0})}{H^m}\notag\\
		&\leq \frac{\de}3+\frac{\de}3
		+\frac{\de}3\leq \de.\notag
	\end{align}
Hence, $\Sob{u(\cdotp,t;\De_n)- u(\cdotp,t;\De)}{H^m}\goesto0$ on $E(m+3,\rho,T)^c$. Since $\Prb(E(m+3),\rho,T))\geq 1-\veps$ and $\veps>0$ was arbitrary, this shows that $u(\cdotp;t;\De_n)\goesto u(\cdotp,t;\De)$ in $H^m$ a.s.
\end{proof}

We are now ready to prove \cref{prop:exist:uniq}.

\begin{proof}[Proof of \cref{prop:exist:uniq}]
Observe that  by \cref{prop:wp:veqn} and \cref{prop:feller}, \cref{prop:exist:uniq} holds under the stronger assumption that $\s\in\bH^{m+3}$. Now suppose $u_0,f\in H^m$, $\s\in \bH^{m}$ and let $\xi=\s W$, where $W=\{W_k\}_{k\geq1}$ a sequence of independent, one-dimensional standard Brownian motions. Also let $\xi^\eps=\s^\eps W$, where $\{\s^\eps\}\subset\bH^{m+3}$ such that $\s^\eps\rightarrow \s$ in $\bH^{m}$. Then $\xi^\eps\in C([0,\infty);H^{m+3})$. Let $v^\eps\in C([0,\infty);H^m)$ denote the unique solution of \eqref{v:gen:eqn} corresponding to $u_0,f,\xi^\eps$ with $\eta^\eps$ given by \eqref{eq:eta}. The estimates derived in \cref{sect:Lyapunov} can then be carried out rigorously for $u^\eps$ with bounds independent of $\eps$ and constants depending only on $\Sob{u_0}{H^m},\Sob{f}{H^m},\Sob{\xi}{L^\infty_TH^m}$ (cf. \cref{thm:Lyapunov}). Since $u^\eps=v^\eps+\xi^\eps$, these estimates may then be used in place of \cref{lem:L2v}--\cref{lem:Hmv} to argue for the convergence of $v^\eps$ in $C([0,T);L^2)$, for all $T>0$, to the unique solution $v$ of \req{v:gen:eqn} corresponding $u_0,f,\xi$. In particular, $u=v+\xi$, where $u$ satisfies \eqref{eq:s:KdV} corresponding to $u_0, f,\s$. Since continuity with respect to $u_0,f$ is inherited in the limit, it follows that $u=v+\xi$ is also continuous with respect to $u_0,f$ in $H^m$. This completes the proof.
\end{proof}

\section*{Acknowledgments}
The work of Nathan Glatt-Holtz was partially supported under the
National Science Foundation grants DMS-1313272, DMS-1816551, DMS-2108790,
and under a Simons Foundation travel support award 515990. The work
of Vincent R. Martinez was in part supported by the National Science Foundation through DMS 2213363 and DMS 2206491, as well as the Dolciani Halloran Foundation. The authors would also like to thank Jonathan Mattingly and Peter Constantin for their encouragement and insightful discussions in the course of this work.

\begin{footnotesize}
\newcommand{\etalchar}[1]{$^{#1}$}
\providecommand{\bysame}{\leavevmode\hbox to3em{\hrulefill}\thinspace}
\providecommand{\MR}{\relax\ifhmode\unskip\space\fi MR }
\providecommand{\MRhref}[2]{%
  \href{http://www.ams.org/mathscinet-getitem?mr=#1}{#2}
}
\providecommand{\href}[2]{#2}

\end{footnotesize}

\vspace{.3in}
\begin{multicols}{2}
\noindent Nathan Glatt-Holtz\\ 
{\footnotesize
Department of Statistics\\
Indiana University--Bloomington\\
Web: \url{https://negh.pages.iu.edu/}\\
 Email: \url{negh@iu.edu}} \\[.2cm]
\noindent
Vincent R. Martinez\\
{\footnotesize
Department of Mathematics \& Statistics\\
CUNY Hunter College\\
and\\
Department of Mathematics\\
CUNY Graduate Center\\
Web: \url{http://math.hunter.cuny.edu/vmartine/}\\
 Email: \url{vrmartinez@hunter.cuny.edu}}

\columnbreak 
\noindent Geordie Richards\\ 
{\footnotesize
Department of Mathematics and Statistics\\
University of Guelph\\
Web: \url{http://www.geordierichards.com}\\
 Email: \url{grichards@uoguelph.ca}}\\[.2cm]

\end{multicols}


\begin{thebibliography}{FGHRT15}

\bibitem[ABS89]{AmickBonaSchonbek1989}
C.J. Amick, J.L. Bona, and M.E. Schonbek, \emph{{Decay of Solutions of Some
  Nonlinear Wave Equations}}, J. Differ. Equ. \textbf{81} (1989), 1--49.

\bibitem[BDKM96]{BonaDougalisKarakashianMcKinney1996}
J.L. Bona, V.A. Dougalis, O.A. Karakashian, and W.R. McKinney, \emph{{The
  effect of dissipation on solutions of the generalized Korteweg-de Vries
  equation}}, Journal of Computational and Applied Mathematics \textbf{74}
  (1996), 127--154.

\bibitem[Bil13]{billingsley2013convergence}
Patrick Billingsley, \emph{{Convergence of probability measures}}, John Wiley
  \& Sons, 2013.

\bibitem[BIT11]{BabinIlyinTiti2011}
A.V. Babin, A.A. Ilyin, and E.S. Titi, \emph{{On the {R}egularization
  {M}echanism for the {P}eriodic {K}orteweg-de {V}ries {E}quation}}, Comm. Pure
  Appl. Math. \textbf{LXIV} (2011), 0591--0648.

  
\bibitem[BFZ23]{BrzezniakFerrarioZanella2023}
Z. Brze{\'z}niak, B. Ferrario, and M. Zanella, \emph{Ergodic results for the stochastic nonlinear Schr{\"o}dinger equation with large damping}, Journal of Evolution Equations \textbf{23}(2023), no.~1, 19

\bibitem[BKS20]{ButkovskyKulikScheutzow2019}
O.~Butkovsky, A.~Kulik, and M.~Scheutzow, \emph{Generalized couplings and
  ergodic rates for spdes and other markov models}, The Annals of Applied
  Probability \textbf{30} (2020), no.~1, pp. 1--39.

\bibitem[Bou71]{Boussinesq1871}
J.~Boussinesq, \emph{{Th{\'e}orie d l'intumescence liquid appel{\'e}e onde
  solitaire ou de translation, se propageant dans un canal rectangulaire}},
  C.R. Acad. Sci. \textbf{72} (1871), 755--759.

\bibitem[Bou93]{Bourgain1993}
J.~Bourgain, \emph{{Fourier transform restriction phenomena for certain lattice
  subsets and applications to nonlinear evolution equations}}, Geom. Funct.
  Anal. \textbf{3} (1993), no.~3, 209--262.

\bibitem[BS75]{BonaSmith1975}
J.L. Bona and R.~Smith, \emph{{The Initial-Value Problem for the Korteweg-de
  Vries Equation}}, Philos. Trans. R. Soc. Lond., Ser. A, Math. Phys. Sci.
  \textbf{278} (1975), no.~1287, 555--601.

\bibitem[BS76]{BonaScott1976}
J.~Bona and R.~Scott, \emph{{Solutions of the Korteweg-de Vries equation in
  fractional Sobolev spaces}}, Duke Math. J. \textbf{43} (1976), no.~1, 87--99.

\bibitem[CGH12]{chekroun2012invariant}
M.D. Chekroun and N.E. Glatt-Holtz, \emph{Invariant measures for dissipative
  dynamical systems: Abstract results and applications}, Communications in
  Mathematical Physics \textbf{316} (2012), no.~3, 723--761.

\bibitem[CKS{\etalchar{+}}03]{CollianderKeelStaffilaniTakaokaTao2003}
J.~Colliander, M.~Keel, G.~Staffilani, H.~Takaoka, and T.~Tao, \emph{{Sharp
  global well-posedness for {K}d{V} and modified {K}d{V} on $\mathbb{R}$ and
  $\mathbb{T}$}}, J. Amer. Math. Soc. \textbf{16} (2003), no.~3, 705--749.

\bibitem[CR04]{CabralRosa2004}
M.~Cabral and R.~Rosa, \emph{{Chaos for a damped and forced {K}d{V} equation}},
  Physica D \textbf{192} (2004), 265--278.

\bibitem[CS13a]{ChehabSadaka2013a}
J.P. Chehab and G.~Sadaka, \emph{{Numerical study of a family of dissipative
  {K}d{V} equations}}, Comm. Pure Appl. Math. \textbf{12} (2013), no.~1,
  519--546.

\bibitem[CS13b]{ChehabSadaka2013b}
\bysame, \emph{{On damping rates of dissipative {K}d{V} equations}}, Discrete
  Contin. Dyn. Sys., Ser. S \textbf{6} (2013), no.~5, 1487--1506.

\bibitem[dBD98]{deBouardDebussche1998}
A.~de~Bouard and A.~Debussche, \emph{{On the {S}tochastic {K}orteweg-de {V}ries
  {E}quation}}, J. Funct. Anal. \textbf{154} (1998), 215--251.

\bibitem[dBDT99]{deBouardDebusscheTsutsumi1999}
A.~de~Bouard, A.~Debussche, and Y.~Tsutsumi, \emph{{White {N}oise {D}riven
  {K}orteweg-de {V}ries {E}quation}}, J. Funct. Anal. \textbf{169} (1999),
  532--558.

\bibitem[dBDT04]{deBouardDebusscheTsutsumi2004}
\bysame, \emph{{Periodic solutionsof the {K}orteweg-de {V}ries equation driven
  by white noise}}, SIAM J. Math. Anal. \textbf{36} (2004), no.~3, 815--855.

\bibitem[DO05]{DebusscheOdasso2005}
A.~Debussche and C.~Odasso, \emph{{Ergodicity for a weakly damped stochastic
  non-linear Schr{\"o}dinger equation}}, Journal of Evolution Equations
  \textbf{5} (2005), no.~3, 317--356.

\bibitem[DZ96]{DaPratoZabczyk1996}
G.~{Da Prato} and J.~Zabczyk, \emph{{Ergodicity for infinite-dimensional
  systems}}, {London Mathematical Society Lecture Note Series}, vol. 229,
  Cambridge University Press, Cambridge, 1996. \MR{MR1417491 (97k:60165)}

\bibitem[DZ14]{da2014stochastic}
Giuseppe {Da Prato} and Jerzy Zabczyk, \emph{{Stochastic equations in infinite
  dimensions}}, Cambridge university press, 2014.

\bibitem[EKZ18]{EkrenKukavicaZiane2018}
I.~Ekren, I.~Kukavica, and M.~Ziane, \emph{{Existence of invariant measures for
  the stochastic damped {K}d{V} equation}}, Indiana Math. J. \textbf{67}
  (2018), 1221--1254.

\bibitem[FGHRT15]{FoldesGlattHoltzRichardsThomann2013}
J.~F{\"o}ldes, N.~Glatt-Holtz, G.~Richards, and E.~Thomann, \emph{Ergodic and
  mixing properties of the boussinesq equations with a degenerate random
  forcing}, Journal of Functional Analysis \textbf{269} (2015), no.~8,
  2427--2504.

\bibitem[FM76]{flaschka1976}
H.~Flaschka and D.W. McLaughlin, \emph{{Canonically conjugate variables for the
  Korteweg-de Vries equation and the Toda lattice with periodic boundary
  conditions}}, Progress of Theoretical Physics \textbf{55} (1976), no.~2,
  438--456.

\bibitem[FMRT01]{foias2001navier}
C.~Foias, Oscar Manley, R.~Rosa, and R.~Temam, \emph{Navier-stokes equations
  and turbulence}, vol.~83, Cambridge University Press, 2001.

\bibitem[FP67]{FoiasProdi1967}
C.~Foias and G.~Prodi, \emph{{Sur le comportement global des solutions
  non-stationnaires des {\'e}quations de Navier-Stokes en dimension $2$}},
  Rendiconti del Seminario Matematico della Universit{\`a} di Padova
  \textbf{39} (1967), 1--34.

\bibitem[Gar71]{gardner1971}
C.S. Gardner, \emph{{Korteweg-de Vries equation and generalizations. IV. The
  Korteweg-de Vries equation as a Hamiltonian system}}, Journal of Mathematical
  Physics \textbf{12} (1971), no.~8, 1548--1551.

\bibitem[GH14]{GlattHoltz2014}
N.~E. Glatt-Holtz, \emph{{Notes on Statistically Invariant States in
  Stochastically Driven Fluid Flows}}, arXiv preprint arXiv:1410.8622 (2014),
  1--24.

\bibitem[Ghi88]{Ghidaglia1988}
J.~M. Ghidaglia, \emph{{Weakly damped forced Korteweg-de Vries equations behave
  as a finite dimensional dynamical system in the long time}}, Journal of
  Differential Equations \textbf{74} (1988), no.~2, 369--390.

\bibitem[GHM22]{GlattHoltzMondaini2020}
N.E. Glatt-Holtz and C.F. Mondaini, \emph{Mixing rates for hamiltonian monte
  carlo algorithms in finite and infinite dimensions}, Stochastics and Partial
  Differential Equations: Analysis and Computations \textbf{10} (2022), no.~4,
  1318--1391.

\bibitem[GHM24]{glatt2024long}
N.E. Glatt-Holtz and C.F. Mondaini, \emph{Long-term accuracy of numerical approximations of SPDEs with the stochastic Navier--Stokes equations as a paradigm}, IMA journal of numerical analysis, (2024), drae043, DOI:10.1093/imanum/drae043


\bibitem[GHMR17]{GlattHoltzMattinglyRichards2015}
N.~Glatt-Holtz, J.C. Mattingly, and G.~Richards, \emph{{On unique ergodicity in
  nonlinear stochastic partial differential equations}}, J. Stat. Phys.
  \textbf{166} (2017), no.~3-4, 618--649. \MR{3607584}

\bibitem[Gou00]{Goubet2000}
O.~Goubet, \emph{{Asymptotic smoothing effect for weakly damped forced
  {K}orteweg-de {V}riews equations}}, Discrete Contin. Dyn. Sys. \textbf{6}
  (2000), no.~3, 625--644.

\bibitem[Gou18]{Goubet2018}
\bysame, \emph{{Anayticity of the global attractor for the damped forced
  periodic {K}orteweg-de {V}ries equation}}, J. Differ. Equ. \textbf{264}
  (2018), 3052--3066.

\bibitem[HM06]{HairerMattingly2006}
M.~Hairer and J.C. Mattingly, \emph{{Ergodicity of the 2{D} {N}avier-{S}tokes
  equations with degenerate stochastic forcing}}, Ann. of Math. (2)
  \textbf{164} (2006), no.~3, 993--1032. \MR{2259251 (2008a:37095)}

\bibitem[HM08]{HairerMattingly2008}
\bysame, \emph{{Spectral gaps in Wasserstein distances and the 2D stochastic
  Navier--Stokes equations}}, The Annals of Probability \textbf{36} (2008),
  no.~6, 2050--2091.

\bibitem[HM11]{HairerMattingly2011}
\bysame, \emph{{A Theory of Hypoellipticity and Unique Ergodicity for
  Semilinear Stochastic PDEs}}, Electron. J. Probab. \textbf{16} (2011),
  no.~23, 658--738.

\bibitem[HMS11]{HairerMattinglyScheutzow2011}
M.~Hairer, J.C. Mattingly, and M.~Scheutzow, \emph{{Asymptotic coupling and a
  general form of Harris' theorem with applications to stochastic delay
  equations}}, Probab. Theory Relat. Fields (2011), no.~149, 223--259.

\bibitem[Joh70]{Johnson1970}
R.S. Johnson, \emph{{A non-linear equation incorporating damping and
  dispersion}}, J. Fluid Mech. \textbf{42} (1970), no.~1, 49--60.

\bibitem[JST17]{JollySadigovTiti2017}
M.S. Jolly, T.~Sadigov, and E.S. Titi, \emph{{Determining form and data
  assimilation algorithm for weakly damped and driven {K}orteweg-de {V}riews
  equation--{F}ourier modes case}}, Nonlinear Anal., Real World Appl.
  \textbf{36} (2017), 287--317.

\bibitem[Kat79]{Kato1979}
T.~Kato, \emph{{On the {K}orteweg-de {V}ries equation}}, Manuscripta Math.
  \textbf{29} (1979), 89--99.

\bibitem[KdV95]{KortewegdeVries1895}
D.~Korteweg and G.~de~Vries, \emph{{On the change of form of long waves
  advanced in a rectangular canal, and on a new type of long stationary wave}},
  Phil. Mag. \textbf{539} (1895), 422--43.

\bibitem[KLO12]{Komorowski2012a}
T.~Komorowski, C.~Landim, and S.~Olla, \emph{{Fluctuations in Markov processes:
  time symmetry and martingale approximation}}, vol. 345, Springer Science \&
  Business Media, 2012.

\bibitem[KMGZ70]{kruskal1970korteweg}
M.D. Kruskal, R.M. Miura, C.S. Gardner, and N.J. Zabusky,
  \emph{{Korteweg-deVries Equation and Generalizations. V. Uniqueness and
  Nonexistence of Polynomial Conservation Laws}}, Journal of Mathematical
  Physics \textbf{11} (1970), no.~3, 952--960.

\bibitem[KP03]{kappeler2003}
T.~Kappeler and J.~P{\"o}schel, \emph{{Kdv \& Kam}}, vol.~45, Springer Science
  \& Business Media, 2003.

\bibitem[KP08]{KP2008}
S.B. Kuksin and A.L. Piatnitski, \emph{{Khasminskii--Whitham averaging for
  randomly perturbed KdV equation}}, Journal de math{\'e}matiques pures et
  appliqu{\'e}es \textbf{89} (2008), no.~4, 400--428.

\bibitem[KPV91]{KenigPonceVega1991}
C.~E. Kenig, G.~Ponce, and L.~Vega, \emph{{Well-posedness of the initial value
  problem for the {K}orteweg-de {V}ries equation}}, J. Amer. Math. Soc.
  \textbf{4} (1991), no.~2, 323--347.

\bibitem[KPV96]{KenigPonceVega1996}
C.E. Kenig, G.~Ponce, and L.~Vega, \emph{{A bilinear estimate with applications
  to the {K}d{V} equation}}, J. Amer. Math. Soc. \textbf{9} (1996), no.~2,
  573--603.

\bibitem[Kry02]{krylov2002introduction}
N.V. Krylov, \emph{{Introduction to the theory of random processes}}, vol.~43,
  American Mathematical Soc., 2002.

\bibitem[KS12]{KuksinShirikyan2012}
S.~Kuksin and A.~Shirikyan, \emph{{Mathematics of Two-Dimensional Turbulence}},
  {Cambridge Tracts in Mathematics}, no. 194, Cambridge University Press, 2012.

\bibitem[KS14]{karatzas2014brownian}
I.~Karatzas and S.~Shreve, \emph{{Brownian Motion and Stochastic Calculus}},
  vol. 113, Springer, 2014.

\bibitem[KT06]{KappelerTopalev2006}
T.~Kappeler and P.~Topalov, \emph{{Global wellposedness of {K}d{V} in
  ${H}^{-1}(\mathbb{T},\mathbb{R})$}}, Duke Math. J. \textbf{135} (2006),
  no.~2, 327--360.

\bibitem[Kuk00]{Kuksin2000}
S.B. Kuksin, \emph{{Analysis of hamiltonian PDEs}}, vol.~19, Clarendon Press,
  2000.

\bibitem[Kuk10]{Kuksin2010}
\bysame, \emph{{Damped-driven {K}d{V} and effective equations for long-time
  behaviour of its solutions}}, Geom. Funct. Anal. \textbf{20} (2010),
  1431--1463.

\bibitem[Kul17]{Kulik2017}
A.~Kulik, \emph{{Ergodic Behavior of Markov Processes: With Applications to
  Limit Theorems}}, vol.~67, Walter de Gruyter GmbH \& Co KG, 2017.

\bibitem[KV19]{KillipVisan2019}
R.~Killip and M.~Visan, \emph{{Kd{V} is well-posed in ${H}^{-1}$}}, Ann. Math.
  \textbf{190} (2019), 249--305.

\bibitem[KW12]{Komorowski2012b}
T.~Komorowski and A.~Walczuk, \emph{{Central limit theorem for Markov processes
  with spectral gap in the Wasserstein metric}}, Stochastic Processes and their
  Applications \textbf{122} (2012), no.~5, 2155--2184.

\bibitem[Lax76]{lax1976}
P.D Lax, \emph{{Almost periodic solutions of the KdV equation}}, SIAM review
  \textbf{18} (1976), no.~3, 351--375.

\bibitem[Mat99]{MattinglyLargeViscosity1999}
J.C. Mattingly, \emph{{Ergodicity of 2{D} {N}avier-{S}tokes {E}quations with
  {R}andom {F}orcing and {L}arge {V}iscosity}}, Comm. Math. Phys. \textbf{206}
  (1999), 273--288.

\bibitem[Mat03]{Mattingly2003}
\bysame, \emph{On recent progress for the stochastic {N}avier {S}tokes
  equations}, Journ\'{e}es ``\'{E}quations aux {D}\'{e}riv\'{e}es
  {P}artielles'', Univ. Nantes, Nantes, 2003, pp.~Exp. No. XI, 52. \MR{2050597}

\bibitem[MGK68]{miura1968II}
R.M. Miura, C.S. Gardner, and M.D. Kruskal, \emph{{Korteweg-de Vries equation
  and generalizations. II. Existence of conservation laws and constants of
  motion}}, Journal of Mathematical physics \textbf{9} (1968), no.~8,
  1204--1209.

\bibitem[Miu68]{miura1968I}
R.M. Miura, \emph{{Korteweg-de Vries equation and generalizations. I. A
  remarkable explicit nonlinear transformation}}, Journal of Mathematical
  Physics \textbf{9} (1968), no.~8, 1202--1204.

\bibitem[Miu76]{Miura1976}
\bysame, \emph{{The {K}orteweg-de {V}ries equation: {A} survey of recent
  results}}, SIAM Review \textbf{18} (1976), no.~3, 412--459.

\bibitem[Mol11]{Molinet2011}
L.~Molinet, \emph{{A note on ill posedness for the {K}d{V} equation}}, Differ.
  Integral Equ. \textbf{24} (2011), no.~7--8, 759--765.

\bibitem[Mol12]{Molinet2012}
\bysame, \emph{{Sharp ill-posedness results for the {K}d{V} and m{K}d{V}
  equations on the torus}}, Adv. Math. \textbf{230} (2012), no.~4--6,
  1895--1930.

\bibitem[MR97]{MoiseRosa1997}
I.~Moise and R.~Rosa, \emph{{On the regularity of the global attractor of a
  weakly damped, forced {K}orteweg-de {V}ries equation}}, Adv. Differ. Equ.
  \textbf{2} (1997), no.~2, 257--296.

\bibitem[MR13]{MarsdenRatiu2013}
J.E. Marsden and T.S. Ratiu, \emph{{Introduction to mechanics and symmetry: a
  basic exposition of classical mechanical systems}}, vol.~17, Springer Science
  \& Business Media, 2013.

\bibitem[MT76]{mckean1976}
H.P. McKean and E.~Trubowitz, \emph{{Hill's operator and hyperelliptic function
  theory in the presence of infinitely many branch points}}, Communications on
  Pure and Applied Mathematics \textbf{29} (1976), no.~2, 143--226.

\bibitem[Oh10]{oh2010periodic}
T.~Oh, \emph{{Periodic stochastic Korteweg--de Vries equation with additive
  space-time white noise}}, Analysis \& PDE \textbf{2} (2010), no.~3, 281--304.

\bibitem[PR07]{prevot2007concise}
C.~Pr{\'e}v{\^o}t and M.~R{\"o}ckner, \emph{{A concise course on stochastic
  partial differential equations}}, vol. 1905, Springer, 2007.

\bibitem[RY99]{RevuzYor1999}
D.~Revuz and M.~Yor, \emph{{Continuous martingales and {B}rownian motion}},
  third ed., {Grundlehren der Mathematischen Wissenschaften [Fundamental
  Principles of Mathematical Sciences]}, vol. 293, Springer-Verlag, Berlin,
  1999. \MR{1725357 (2000h:60050)}

\bibitem[Sch96]{SchmalfussBook}
B.~Schmalfuss, \emph{A random fixed point theorem based on lyapunov exponents},
  Random Comput. Dynamics IV (1996), 257--268.

\bibitem[Sjo70]{Sjoberg1970}
A.~Sjoberg, \emph{{On the {K}orteweg-de {V}ries equation: existence and
  uniqueness}}, J. Math. Anal. Appl. \textbf{29} (1970), 569--579.

\bibitem[ST76]{SautTemam1976}
J.C. Saut and R.~Temam, \emph{{Remarks on the Korteweg-de Vries Equation}},
  Israel Journal of Math. \textbf{24} (1976), no.~1, 78--87.

\bibitem[Tem69]{Temam1969}
R.~Temam, \emph{{Sur un probl{\'e}me non lin{\'e}aire}}, J. Math. Pures Appl.
  \textbf{48} (1969), 159--172.

\bibitem[Tem97]{Temam1997}
\bysame, \emph{{Infinite-dimensional dynamical systems in mechanics and
  physics}}, second ed., {Applied Mathematical Sciences}, vol.~68,
  Springer-Verlag, New York, 1997. \MR{MR1441312 (98b:58056)}

\bibitem[Vil08]{villani2008optimal}
C.~Villani, \emph{Optimal transport: old and new}, vol. 338, Springer Science
  \& Business Media, 2008.

\bibitem[ZF71]{zakharov1971}
V.E. Zakharov and L.D. Faddeev, \emph{{Korteweg-de Vries equation: A completely
  integrable Hamiltonian system}}, Functional analysis and its applications
  \textbf{5} (1971), no.~4, 280--287.

\bibitem[ZPS09]{ZahiboPelinovskySergeeva2009}
N.~Zahibo, E.~Pelinovsky, and A.~Sergeeva, \emph{{Weakly damped {K}d{V} soliton
  dynamics with random force}}, Chaos Solitons Fractals \textbf{39} (2009),
  1645--1650.

\end{thebibliography}
\end{document}